%% file: nonKissingComplex.tex
\documentclass{amsart}

\usepackage[T1]{fontenc}
\usepackage{enumerate, amsmath, amsfonts, amssymb, amsthm, mathrsfs, wasysym, graphics, graphicx, xcolor, url, hyperref, hypcap, xargs, multicol, overpic, pdflscape, multirow, hvfloat, minibox, accents, array, xifthen, ae, aecompl, blkarray, pifont, mathtools, a4wide}
\usepackage{marginnote}
\hypersetup{colorlinks=true, citecolor=darkblue, linkcolor=darkblue}
\usepackage[all]{xy}
\usepackage[bottom]{footmisc}
\usepackage{tikz}
\usetikzlibrary{trees, decorations, decorations.markings, shapes, arrows, matrix, calc, fit, intersections, patterns, angles}
\graphicspath{{figures/}}
\makeatletter\def\input@path{{figures/}}\makeatother
\usepackage{caption}
\newtheorem{theorem}{Theorem}[part]
\newtheorem{corollary}[theorem]{Corollary}
\newtheorem{proposition}[theorem]{Proposition}
\newtheorem{lemma}[theorem]{Lemma}

\newtheorem*{theorem*}{Theorem}

\theoremstyle{definition}
\newtheorem{definition}[theorem]{Definition}
\newtheorem{example}[theorem]{Example}
\newtheorem{remark}[theorem]{Remark}
\newtheorem{question}[theorem]{Question}
\newtheorem{notation}[theorem]{Notation}

\newcommand{\R}{\mathbb{R}} 
\newcommand{\N}{\mathbb{N}} 
\newcommand{\Z}{\mathbb{Z}} 
\newcommand{\fS}{\mathfrak{S}} 
\newcommand{\cS}{\mathbb{S}} 
\renewcommand{\b}[1]{\mathbf{#1}} 

\newcommand{\set}[2]{\left\{ #1 \;\middle|\; #2 \right\}} 
\newcommand{\bigset}[2]{\big\{ #1 \;\big|\; #2 \big\}} 
\newcommand{\ssm}{\smallsetminus} 
\newcommand{\dotprod}[2]{\left\langle \, #1 \; \middle| \; #2 \, \right\rangle} 
\newcommand{\symdif}{\,\triangle\,} 
\newcommand{\one}{{1\!\!1}} 
\newcommand{\eqdef}{\mbox{\,\raisebox{0.2ex}{\scriptsize\ensuremath{\mathrm:}}\ensuremath{=}\,}} 

\DeclareMathOperator{\ascents}{asc} 
\DeclareMathOperator{\descents}{des} 

\newcommand{\fref}[1]{Figure~\ref{#1}} 
\newcommand{\ie}{\textit{i.e.}~} 
\newcommand{\eg}{\textit{e.g.}~} 
\newcommand{\apriori}{\textit{a priori}} 
\newcommand{\viceversa}{\textit{vice versa}} 
\definecolor{darkblue}{rgb}{0,0,0.7} 
\definecolor{green}{RGB}{57,181,74} 
\definecolor{violet}{RGB}{147,39,143} 
\newcommand{\darkblue}{\color{darkblue}} 
\newcommand{\defn}[1]{\textsl{\darkblue #1}} 
\newcommand{\para}[1]{\medskip\noindent\textbf{#1.}} 
\DeclareRobustCommand{\exmAn}{\overrightarrow{A_n}} 

\usepackage{todonotes}

\newcommand{\blossom}{^\text{\ding{96}}} 
\newcommand{\blinkers}[1]{_{\LEFTCIRCLE \!\! #1 \!\! \RIGHTCIRCLE}} 
\newcommand{\indeg}{\mathrm{indeg}} 
\newcommand{\outdeg}{\mathrm{outdeg}} 

\newcommand{\strings}{\mathcal{S}} 
\newcommand{\distinguishableStrings}{\mathcal{S}_\mathrm{dist}} 
\newcommand{\walks}{\mathcal{W}} 
\newcommand{\straightWalks}{\mathcal{W}_\mathrm{str}} 
\newcommand{\bendingWalks}{\mathcal{W}_\mathrm{bend}} 
\newcommand{\NKWalks}{\mathcal{W}_\mathrm{nk}} 
\newcommand{\peaks}[1]{\mathsf{peaks}(#1)} 
\newcommand{\deeps}[1]{\mathsf{deeps}(#1)} 
\newcommand{\distinguishedWalk}[2]{\mathsf{dw}(#1,#2)} 
\newcommand{\distinguishedArrows}[2]{\mathsf{da}(#1,#2)} 
\newcommand{\distinguishedString}[2]{\mathsf{ds}(#1,#2)} 
\newcommand{\distinguishedSign}[2]{\varepsilon(#1,#2)} 

\newcommand{\kn}{\kappa} 
\newcommand{\KN}{\textsc{kn}} 
\newcommandx{\NKC}[1][1=\bar Q]{\mathcal{K}_{\mathrm{nk}}(#1)} 
\newcommandx{\RNKC}[1][1=\bar Q]{\mathcal{C}_{\mathrm{nk}}(#1)} 
\newcommandx{\NKL}[1][1=\bar Q]{\mathcal{L}_{\mathrm{nk}}(#1)} 
\newcommandx{\NKG}[1][1=\bar Q]{\mathcal{G}_{\mathrm{nk}}(#1)} 
\newcommandx{\NFC}[1][1=\bar Q]{\mathcal{C}_{\mathrm{nf}}(#1)} 
\newcommand{\peak}{\mathrm{peak}} 
\newcommand{\deep}{\mathrm{deep}} 
\newcommand{\reversed}[1]{#1^{\mathrm{rev}}} 
\renewcommand{\top}{\mathrm{top}} 
\newcommand{\bottom}{\mathrm{bot}} 

\newcommand{\meet}{\wedge} 
\newcommand{\join}{\vee} 
\newcommand{\bigMeet}{\bigwedge} 
\newcommand{\bigJoin}{\bigvee} 
\newcommand{\closure}[1]{#1^{\mathrm{cl}}} 
\newcommand{\coclosure}[1]{#1^{\mathrm{cocl}}} 
\newcommand{\Bicl}[1]{\mathsf{Bic}(#1)} 
\newcommand{\projDown}{\pi_\downarrow} 
\newcommand{\projUp}{\pi^\uparrow} 
\newcommand{\JI}{\mathsf{JI}} 
\newcommand{\MI}{\mathsf{MI}} 
\newcommand{\Cong}{\mathsf{Cong}} 
\newcommand{\con}{\mathrm{con}} 
\newcommand{\ji}{\mathsf{ji}} 
\newcommand{\mi}{\mathsf{mi}} 

\newcommand{\gvector}[1]{\mathbf{g}(#1)} 
\newcommand{\gvectors}[1]{\mathbf{g}(#1)} 
\newcommandx{\gvectorFan}[1][1=\bar Q]{\mathcal{F}^\mathbf{g}(#1)} 
\newcommand{\cvector}[2]{\mathbf{c}(#1 \in #2)} 
\newcommand{\cvectors}[1]{\mathbf{c}(#1)} 
\newcommandx{\allcvectors}[1][1=\bar Q]{\mathbf{C}(#1)} 
\newcommandx{\cvectorFan}[1][1=\bar Q]{\mathcal{F}^\mathbf{c}(#1)} 
\newcommand{\point}[1]{\mathbf{p}(#1)} 
\newcommand{\HS}[1]{\mathbf{H}^{\le}(#1)} 
\newcommand{\Hyp}[1]{\mathbf{H}^{=}(#1)} 
\newcommandx{\Asso}[2][1=\bar Q,2={}]{\mathsf{Asso}^{#2}(#1)} 
\newcommandx{\Zono}[2][1=\bar Q,2={}]{\mathsf{Zono}^{#2}(#1)} 
\newcommand{\Fan}{\mathcal{F}} 
\newcommand{\multiplicityVector}{\b{m}} 

\newcommand{\Hom}[1]{\operatorname{{\rm Hom}}_{#1}}
\newcommand{\Ext}[1]{\operatorname{{\rm Ext}}_{#1}}
\newcommandx{\AR}[1][1=\bar Q]{\mathrm{AR}(#1)} 
\newcommandx{\tTC}[1][1=\bar Q]{\mathcal{K}^{\textrm{s$\tau$-tilt}}(#1)} 
\newcommand{\rep}{\operatorname{{\rm rep}}}
\newcommand{\proj}{\operatorname{{\rm proj}}}

\makeatletter
\def\part{\@startsection{part}{1}%
\z@{.7\linespacing\@plus\linespacing}{.8\linespacing}%
{\LARGE\sffamily\centering}}
\@addtoreset{section}{part}
\makeatother

\makeatletter
\def\l@section{\@tocline{1}{2pt}{0pc}{}{}}
\makeatother
\let\oldtocpart=\tocpart
\renewcommand{\tocpart}[2]{\bf\large\oldtocpart{#1}{#2}}
\let\oldtocsection=\tocsection
\renewcommand{\tocsection}[2]{\bf\oldtocsection{#1}{#2}}

\title{Non-kissing complexes and tau-tilting for gentle algebras}

\thanks{The three authors were partially supported by the French ANR grant SC3A~(15\,CE40\,0004\,01).}

\author{Yann Palu}
\address[Yann Palu]{LAMFA, Universit\'e Picardie Jules Verne, Amiens}
\email{yann.palu@u-picardie.fr}
\urladdr{\url{http://www.lamfa.u-picardie.fr/palu/}}

\author{Vincent Pilaud}
\address[Vincent Pilaud]{CNRS \& LIX, \'Ecole Polytechnique, Palaiseau}
\email{vincent.pilaud@lix.polytechnique.fr}
\urladdr{\url{http://www.lix.polytechnique.fr/~pilaud/}}

\author{Pierre-Guy Plamondon}
\address[Pierre-Guy Plamondon]{Universit\'e Paris-Saclay, UVSQ, CNRS, Laboratoire de Math\'ematiques de Versailles}
\email{pierre-guy.plamondon@uvsq.fr}
\urladdr{\url{https://www.imo.universite-paris-saclay.fr/~plamondon/}}

\begin{document}

\begin{abstract}
We interpret the support $\tau$-tilting complex of any gentle bound quiver as the non-kissing complex of walks on its blossoming quiver.
Particularly relevant examples were previously studied for quivers defined by a subset of the grid or by a dissection of a polygon.
We then focus on the case when the non-kissing complex is finite.
We show that the graph of increasing flips on its facets is the Hasse diagram of a congruence-uniform lattice.
Finally, we study its $\b{g}$-vector fan and prove that it is the normal fan of a non-kissing associahedron.
\end{abstract}


\maketitle

\vspace*{.3cm}
\tableofcontents

\vspace*{-.9cm}
\enlargethispage{.4cm}


\newpage
\section*{Introduction}

\subsection*{Non-kissing complex and support $\tau$-tilting complex}

The non-kissing complex is the simplicial complex of a compatibility relation, called non-kissing, on paths in a fixed shape of a grid.
It was introduced by T.~K.~Petersen, P.~Pylyavskyy and D.~Speyer in~\cite{PetersenPylyavskyySpeyer} for a staircase shape, studied by F.~Santos, C.~Stump and V.~Welker~\cite{SantosStumpWelker} for rectangular shapes, and extended by T.~McConville in~\cite{McConville} for arbitrary shapes.
This complex is known to be a simplicial sphere, and it was actually realized as a polytope using successive edge stellations and suspensions in~\cite[Sect.~4]{McConville}.
Moreover, the dual graph of the non-kissing complex has a natural orientation which equips its facets with a lattice structure~\cite[Thm. 1.1, Sect.~5--8]{McConville}.
Further lattice-theoretic and geometric aspects of this complex were recently developed by A.~Garver and T.~McConville in~\cite{GarverMcConville-grid}.

The interest in non-kissing complexes is motivated by relevant instances arising from particular shapes.
As already observed in~\cite[Sect.~10]{McConville}, when the shape is a ribbon, the non-kissing complex is an associahedron (the simplicial complex of dissections of a polygon), and the non-kissing lattice is a type~$A$ Cambrian lattice of N.~Reading~\cite{Reading-CambrianLattices}.
In particular, the straight ribbon corresponds to the Tamari lattice, an object at the heart of a deep research area~\cite{TamariFestschrift}.
When the shape is a rectangle (or even a staircase), the non-kissing complex was studied in~\cite{PetersenPylyavskyySpeyer, SantosStumpWelker} as the Grassmann associahedron, in connection to non-crossing subsets of~$[n]$.

Other instances of such complexes arise naturally from the representation theory of associative algebras.  
The notion of support $\tau$-tilting module over an algebra was introduced by T.~Adachi, O.~Iyama and I.~Reiten in \cite{AdachiIyamaReiten}, and has proved to be a successful generalization of tilting and cluster-tilting theory.  
Over a given algebra, indecomposable $\tau$-rigid modules form a complex. 
For an account of the various algebraic interpretations of this complex, we refer the reader to \cite{BrustleYang}.  
For example, in the case of the path algebra of a quiver which is a straight line, the support $\tau$-tilting complex is, again, an associahedron.

The aim of this paper is twofold.
On the one hand, we provide a realization of any non-kissing complex as the support $\tau$-tilting complex of a well-chosen associative algebra.
The algebras that occur are certain gentle algebras, a special case of the well-studied string algebras of M.~C.~R.~Butler and C.~Ringel~\cite{ButlerRingel}.
On the other hand, we extend the notion of non-kissing complex and show that the support $\tau$-tilting complex of any gentle algebra can be interpreted as an extended non-kissing complex.
More precisely, starting from any gentle bound quiver~$\bar Q = (Q,I)$, we attach additional incoming and outgoing arrows, called blossoms, to make each initial vertex \mbox{$4$-valent}. The resulting gentle bound quiver~$\bar Q\blossom = (Q\blossom, I\blossom)$ is called the blossoming quiver of~$\bar Q$ (see Section~\ref{subsec:blossomingQuiver}). 
We then define a combinatorial notion of compatibility on walks in $\bar Q \blossom$, called the non-kissing relation (see Section~\ref{sec:nonKissingComplex}).
The motivating result of this paper is Theorem \ref{thm:nkc/sttiltc}.

\begin{theorem*}
For any gentle bound quiver~$\bar Q = (Q,I)$, the non-kissing complex of walks in the blossoming quiver~$\bar Q \blossom$ is isomorphic to the support $\tau$-tilting complex of the gentle algebra~$kQ/I$.
\end{theorem*}

In short, to any walk in $\bar Q \blossom$ corresponds a representation of $\bar Q$, and this correspondence takes non-kissing walks to $\tau$-compatible representations.
This connection between the combinatorially-flavored non-kissing complex and the algebraically-flavoured support $\tau$-tilting complex opens a bridge to go back and forth between the two worlds.
It allows us, for instance, to combinatorially define mutation of support $\tau$-tilting modules (see Section \ref{subsec:flips}).
This seems worthwhile, as the mutation of support $\tau$-tilting modules over an arbitrary algebra is generally difficult to carry out~explicitly.
Table~\ref{table:dictionary} provides a dictionary to translate between the algebraic notions on the support $\tau$-tilting complex of a gentle algebra and the combinatorial notions on the non-kissing complex of its blossoming quiver.

\newcommandx{\multiLinesBox}[2][2=6cm]{
\begin{minipage}{#2}\vspace{.1cm}\begin{center}
#1
\end{center}\vspace{0cm}\end{minipage}
}

\begin{table}
	\capstart{}
    \centerline{
    \renewcommand{\arraystretch}{1.2}
    \begin{tabular}{|c|c|c|}
    \hline
    String modules on~$\bar Q$ &
    Walks on~$\bar Q\blossom$  &
    Reference
    \\ \hline \hline
    almost positive strings (Def.~\ref{def:stringsBands}) &
    bending walks (Def.~\ref{def:straightBended}) &
    Fig.~\ref{exm:exmBijectionStringsWalks1} \& \ref{fig:exmAuslanderReitenQuiver}
    \\ \hline
    ?\footnotemark[1] &
    straight walks (Def.~\ref{def:straightBended}) &
    Fig.~\ref{exm:exmBijectionStringsWalks1} \& \ref{fig:exmAuslanderReitenQuiver}
    \\ \hline
    \multiLinesBox{$P(v)$ and $P(v)[1]$ \\ $A=kQ/I$ and~$A[1]$ \\ ~} &
    \multiLinesBox{$v_\peak$ and $v_\deep$ \\ $F_\peak$ and~$F_\deep$ \\ (Exm.~\ref{example: Fdeep Fpeak} and Def.~\ref{def: deep walk})} &
    Lem.~\ref{lem:bijections-STTC-NKC}
    \\ \hline
    ?\footnotemark[2] &
    countercurrent order~$\prec_\alpha$ &
    ---
    \\ \hline
    $\tau$-compatibility (Def.~\ref{def: tau-compatibility}) &
    non-kissing (Def.~\ref{def: kissing}) &
    Thm.~\ref{thm:nkc/sttiltc}
    \\ \hline
    support~$\tau$-tilting object (Def.~\ref{def: tau-rigid and stautilt}) &
    maximal non-kissing collection of walks &
    Thm.~\ref{thm:nkc/sttiltc}
    \\ \hline
    support~$\tau$-tilting complex~$\tTC$ (Def.~\ref{def: stautilt complex}) &
    reduced non-kissing complex~$\RNKC$ (Def.~\ref{def: nKc}) &
    Thm.~\ref{thm:nkc/sttiltc}
    \\ \hline
    \multiLinesBox{(left) mutation of support \\ $\tau$-tilting objects (Thm.~\ref{thm:mutationstautilts}~\cite{AdachiIyamaReiten})} &
    \multiLinesBox{(increasing) flip \\ (Prop.~\ref{prop:flip} and Def.~\ref{def:flip} \& \ref{def:increasingFlip})} &
    Thm.~\ref{thm:nkc/sttiltc}
    \\ \hline
    lattice of torsion classes &
    non-kissing lattice~$\NKL$ (Thm.~\ref{thm:lattice}) &
    Thm.~\ref{thm:nkc/sttiltc}
    \\ \hline
    bricks &
    \multiLinesBox{distinguishable strings (Def.~\ref{def:distinguishedSubstring})} &
    Prop.~\ref{prop:characterizationDistinguishableStrings}
    \\ \hline
    collections of pairwise Hom-orthogonal bricks &
    non-friendly complex (Def.~\ref{def: friendly}) &
    Prop.~\ref{prop:characterizationDistinguishableStrings}
    \\ \hline
    \multiLinesBox{bijection between indecomposable~$\tau$-rigid representations and bricks (Thm. 6.1 in~\cite{DemonetIyamaJasso})} &
    \multiLinesBox{bijection between non-kissing walks that are not peak walks and distinguishable strings (Prop.~\ref{prop:bijectionDistinguishableStringsWalks})} &
    \multiLinesBox{Conjectural \\ cf.~Prop.~\ref{prop:bijectionDistinguishableStringsJI}}[2cm]
    \\ \hline
    \multiLinesBox{Bongartz cocompletion of a~$\tau$-rigid indecomposable representation $N=M(\rho)$ whose associated brick is~$M(\sigma)$. Moreover $\sigma = \distinguishedWalk{\omega(\rho)}{F}$} &
    \multiLinesBox{$F=\eta\left(\closure{\Sigma_\bottom(\sigma)}\right)$, where $\sigma$ is a distinguishable string} &
    Rem.~\ref{rem: Bongartz}
    \\ \hline
    Bongartz completion &
    $\eta\left(\closure{\Sigma_\top(\sigma)}\right)$ &
    Rem.~\ref{rem: Bongartz}
    \\ \hline
    $\b{g}$-vector or index of a module (Def.~\ref{definition: g-vector of representation}) &
    $\b{g}$-vector of a walk (Def.~\ref{def: g-vectors for walks}) &
    Rem.~\ref{rem: g-vectors coincide}
    \\ \hline
    $\dim \Hom{A} \big( M(\sigma), \tau M(\sigma') \big)$ (Sec.~\ref{subsec:tautilting}) &
    kissing number~$\kappa \big( \omega(\sigma), \omega(\sigma') \big)$ (Def.~\ref{def: kissing number}) &
    Rem.~\ref{rem:KNvsTau}
    \\ \hline
    \end{tabular}
	\renewcommand{\arraystretch}{1}
    }
    \caption{Dictionary between algebraic notions on the support $\tau$-tilting complex and combinatorial notions on the non-kissing complex.}
    \label{table:dictionary}    
\end{table}

\subsection*{The non-kissing lattice and the non-kissing associahedron}

Besides this algebraic connection, we study the non-kissing complex on any gentle bound quiver with a triple perspective.

On the combinatorial side, we provide a purely combinatorial proof that the non-kissing complex over a gentle bound quiver is a pseudomanifold.
This involves in particular a precise description of flips inspired from~\cite[Sect.~3]{McConville}.

On the lattice-theoretical side, we consider the graph of increasing flips between non-kissing facets.
When the quiver is a directed path, this increasing flip graph is the Hasse diagram of the classical Tamari lattice.
In general, when $\bar Q$ is $\tau$-tilting finite, the transitive closure of the increasing flip graph is isomorphic to the lattice of torsion classes, and we therefore call it the non-kissing lattice.
Our main lattice-theoretic result is the following extension of a remarkable result of T.~McConville~\cite{McConville}.

\begin{theorem*}
For any gentle bound quiver~$\bar Q$ whose non-kissing complex is finite, the non-kissing lattice is congruence-uniform.
In fact, the non-kissing lattice is both a lattice quotient and a sublattice of a lattice of biclosed sets on strings in~$\bar Q$.
\end{theorem*}
\footnotetext[1]{The straight walks should morally correspond to some projective-injective objects in an exact category. We plan to further investigate the situation depicted in Figure~\ref{fig:exmAuslanderReitenQuiver}.}
\footnotetext[2]{We do not know of an algebraic interpretation of the countercurrent order yet. We nonetheless included it in the table above since we think it might be of interest to find one.}
\noindent
It is then natural to study the canonical join complex of the non-kissing lattice, which we call the non-friendly complex following~\cite{GarverMcConville-grid}.
This complex corresponds to the classical non-crossing partitions in Catalan combinatorics.

Finally, on the geometric side, we study realizations of non-kissing complexes as polyhedral fans and polytopes.
Geometric prototypes are given by the type~$A$ Cambrian fans of N.~Reading and D.~Speyer~\cite{ReadingSpeyer} and their polytopal realization by C.~Hohlweg and C.~Lange~\cite{HohlwegLange}.
We define $\b{g}$- and $\b{c}$-vectors for walks in the non-kissing complex of any gentle bound quiver and provide a simple combinatorial proof of the following geometric statement.

\begin{theorem*}
Consider a gentle bound quiver~$\bar Q$ whose non-kissing complex is finite.
Then the \mbox{$\b{g}$-vectors} of the walks in the blossoming quiver~$\bar Q\blossom$ support a complete simplicial fan which realizes the non-kissing complex of~$\bar Q$.
Moreover, this fan is the normal fan of a polytope, called the non-kissing associahedron.
\end{theorem*}

\noindent
The non-kissing complex provides an instrumental combinatorial model to understand the linear dependences among $\b{g}$-vectors of adjacent $\b{g}$-vector cones, which leads to simple polytopal realizations.
In fact, this understanding even extends to arbitrary $\tau$-tilting finite algebras.


\subsection*{Grid and dissection bound quivers}

Particular instances of gentle bound quivers enable us to connect previous results and to answer open conjectures from previous works:
\begin{enumerate}[(i)]
\item \textbf{Grid quivers}: Our combinatorial and lattice-theoretic results are directly inspired from that of~\cite{PetersenPylyavskyySpeyer, SantosStumpWelker, McConville, GarverMcConville-grid} for quivers arising from the $\Z^2$ grid. Our approach provides an interpretation of the non-kissing complex of a grid in terms of $\tau$-tilting theory of gentle algebras and answers open questions of~\cite{GarverMcConville-grid} concerning its geometry.
\item \textbf{Dissection quivers}: Each dissection of a polygon is classically associated with a gentle quiver. Its non-kissing complex is isomorphic to the accordion complex of the dissection, defined and developed in the works of Y.~Baryshnikov~\cite{Baryshnikov}, F.~Chapoton~\cite{Chapoton-quadrangulations}, A.~Garver and T.~McConville~\cite{GarverMcConville} and T.~Manneville and V.~Pilaud~\cite{MannevillePilaud-accordion}. As far as we know, this provides the first connection between non-kissing complexes and accordion complexes.
\item \textbf{Path quivers}: When the quiver is a path, the non-kissing complex is a simplicial associahedron, the non-kissing lattice is a type~$A$ Cambrian lattice of N.~Reading~\cite{Reading-CambrianLattices} (including the classical Tamari lattice), and the non-kissing associahedron is an associahedron of C.~Hohlweg and C.~Lange~\cite{HohlwegLange} (including the classical associahedron of J.-L.~Loday~\cite{Loday}).
\end{enumerate}

\subsection*{Organization}

The paper is organized as follows.
Part~\ref{part:algebra} gathers the algebraic aspects of the paper.
In Section~\ref{sec:recollectionsStringAlgebras}, we recall elements of the representation theory of bound quiver algebras in general and of string algebras in particular.
The definition and essential properties of support \mbox{$\tau$-tilting} modules are given in Section~\ref{subsec:tautilting}.  
The combinatorics of strings and bands is also introduced in Sections~\ref{subsec:stringsBands} and~\ref{subsec:stringBandModules}, and used in \ref{sec:stringModules} to derive combinatorial descriptions of $\b{g}$-vectors and $\tau$-compatibility for string modules.

Part~\ref{part:combinatorics} connects the support $\tau$-tilting complex to the non-kissing complex.
Sections~\ref{sec:blossomingQuivers} and \ref{sec:nonKissingComplex} are devoted to our main definitions: those of the blossoming quiver of a gentle bound quiver, and of its non-kissing complex.  
Theorem \ref{thm:nkc/sttiltc} linking the non-kissing complex to the support $\tau$-tilting complex is then proved in Section \ref{sec:nkcvsttc}.

Part~\ref{part:lattice} is devoted to the non-kissing lattice.
After a brief reminder on lattice congruences and congruence-uniform lattices in Section~\ref{sec:latticeTheory}, we define a closure operator on strings of~$\bar Q$ and study the lattice of biclosed sets in Section~\ref{sec:biclosedStrings}.
We then introduce a lattice congruence on biclosed sets in Section~\ref{sec:latticeCongruence} and show that the non-kissing lattice is isomorphic to the lattice quotient of this congruence in Section~\ref{sec:biclosedSetsToNKF}.
Finally, we describe the join-irreducible elements and the canonical join complex of the non-kissing lattice in Section~\ref{sec:nonFriendlyComplex}.

We construct geometric realizations of the non-kissing complex in Part~\ref{part:geometry}.
After recalling classical definitions and characterizations of simplicial fans and polytopal realizations in Section~\ref{sec:polyhedralGeometry}, we define the $\b{g}$- and $\b{c}$-vectors in Section~\ref{subsec:gcvectors}, construct the $\b{g}$-vector fan in Section~\ref{subsec:gvectorFan} and the non-kissing associahedron in Section~\ref{subsec:associahedron}.
Further geometric topics are discussed in Sections~\ref{subsec:geomLattice}~and~\ref{subsec:zonotope}.

\subsection*{Connections to recent or ongoing projects}

In an advanced stage of our project, we became aware of recent projects of other research groups which partially intersect or complete our work:
\begin{enumerate}[(i)]
\item First and foremost, T.~Br\"ustle, G.~Douville, K.~Mousavand, H.~Thomas and E.~Y\i{}ld\i{}r\i{}m also observed independently in~\cite{BrustleDouvilleMousavandThomasYildirim} that the support $\tau$-tilting complex of a gentle bound quiver is isomorphic to the non-kissing complex of its blossoming quiver (that they call fringed quiver). Their work is more oriented on the algebraic aspects of this connection.
\item Our paper is largely inspired from the paper of T.~McConville~\cite{McConville} and a preliminary version of the recent preprint of A.~Garver and T.~McConville~\cite{GarverMcConville-grid}. Although we follow the same construction for the $\b{g}$-vector fan (but in the general context of gentle bound quivers), our proof is independent and instrumental in the construction of the non-kissing associahedron.
\item Recently, L.~Demonet, O.~Iyama, N.~Reading, I.~Reiten and H.~Thomas proved that the lattice of torsion classes of any $\tau$-tilting finite algebra is congruence-uniform~\cite{DemonetIyamaReadingReitenThomas}. It seems however that a general notion of biclosed sets still remains to be found.
\end{enumerate}
We are grateful to all our colleagues for sharing their ideas with us at different preliminary stages.

\subsection*{Conventions}

We conclude by outlining certain conventions that might be unusual for the reader.
We let~$[n] \eqdef \{1, \dots, n\}$.
If~$Q$ is a quiver and~$k$ is a field, we denote by~$kQ$ the path algebra of~$Q$.
We try to denote vertices by~$u,v,w$, arrows by~$\alpha, \beta, \gamma$, strings by~$\rho, \sigma, \tau$, and walks by~$\omega, \lambda, \mu, \nu$.
Arrows are composed from left to right: if \raisebox{-.04cm}{\includegraphics[scale=.3]{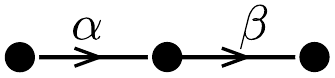}} is a quiver, then~$\alpha\beta$ is a path while~$\beta\alpha$ is not.
Modules over an algebra are assumed to be right modules.
By these conventions, a module over~$kQ$ is equivalent to a representation of~$Q$.
Finally, we also stick to the following picture conventions.
The Hasse diagram of a poset is represented bottom-up and a $3$-dimensional fan is represented by the stereographic projection of its trace on the unit sphere.
Many pictures are difficult to visualize without colors; we refer the reader to the online~version.


\newpage
\part{String modules}
\label{part:algebra}

In this section, we first fix our notations (\ref{subsec:pathAlgebra}) on quivers, path algebras, their representations, and their support $\tau$-tilting complexes (\ref{subsec:tautilting}).
We then focus on string algebras (\ref{subsec:stringGentleAlgebra} to \ref{subsec:morphismsStringModules}): following~\cite{ButlerRingel}, we present a combinatorial description of the indecomposable representations of a string algebra and of the Auslander-Reiten translation~$\tau$.
We then provide a convenient combinatorial description of the $\tau$-compatibility between string modules (\ref{sec:stringModules}).

\section{Recollections on string algebras}
\label{sec:recollectionsStringAlgebras}

\subsection{Path algebras}
\label{subsec:pathAlgebra}

We start by recalling standard notions on quivers, paths, and path algebras.
We refer to the textbooks of I.~Assem, D.~Simson and A.~Skowro\'nski~\cite{AssemSimsonSkowronski} and R.~Schiffler~\cite{Schiffler} for more details.

A \defn{quiver}~$Q$ is an oriented graph, represented by a quadruple~$(Q_0, Q_1, s, t)$ where~$Q_0$ is the set of \defn{vertices}, $Q_1$ is the set of \defn{arrows}, and~$s, t : Q_1 \to Q_0$ are the \defn{source map} and the \defn{target map} respectively.
Throughout, we always assume that~$Q$ is \defn{finite}, that is, it has only finitely many vertices and arrows.
We usually denote the vertices of~$Q$ by~$u,v,w$ and the arrows of~$Q$~by~$\alpha, \beta, \gamma$.

A \defn{path} in~$Q$ is a sequence~$\pi = \alpha_1 \dots \alpha_\ell$ of~$Q_1$ such that~$t(\alpha_k) = s(\alpha_{k+1})$ for all~${k \in [\ell-1]}$ (note that we compose arrows from left to right).
The \defn{source} (resp.~\defn{target}) of~$\pi$ is~$s(\pi) \eqdef s(\alpha_1)$ (resp.~$t(\pi) \eqdef t(\alpha_\ell)$), and the \defn{length} of~$\pi$ is~$\ell(\pi) \eqdef \ell$.
Note that for each vertex~$v \in Q_0$, we also consider the \defn{path~$\varepsilon_v$ of length zero} with source and target~$s(\varepsilon_v) = t(\varepsilon_v) = v$.

Let~$k$ denote a field.
The \defn{path algebra} of~$Q$ is the $k$-algebra~$kQ$ generated by all paths in~$Q$, and where the product of two paths is defined by~$p \cdot q = pq$ (concatenation) if~$s(p) = t(q)$ and~$p \cdot q = 0$ if~$s(p) \ne t(q)$.
Note that~$kQ$ is graded by~$kQ = \bigoplus_{\ell \ge 0} kQ_\ell$ where~$kQ_\ell$ is the subvector space generated by paths of length~$\ell$.
Observe also that~$kQ$ is finite dimensional if and only if~$Q$ is acyclic (meaning that it has no oriented cycle).

The \defn{arrow ideal} of~$kQ$ is the ideal~$R \eqdef \bigoplus_{\ell \ge 1} kQ_\ell$ generated by the arrows of~$Q$.
A two-sided ideal~$I$ of~$kQ$ is \defn{admissible} if there exists~$m \ge 2$ such that~$R^m \subseteq I \subseteq R^2$ (the condition~$R^m \subseteq I$ ensures that~$kQ/I$ is finite dimensional, and the condition~$I \subseteq R^2$ ensures that~$Q$ can be recovered from~$kQ/I$).
The pair~$\bar Q \eqdef (Q,I)$ is then called a \defn{bound quiver} and the quotient~$kQ/I$ is a \defn{bound quiver algebra} and is finite dimensional.

Two quivers~$Q = (Q_0, Q_1, s, t)$ and~$Q' = (Q_0', Q_1', s', t')$ are said to be \defn{isomorphic} if there exists $\phi \eqdef (\phi_0, \phi_1)$ where~$\phi_0 : Q_0 \to Q_0'$ and~$\phi_1 : Q_1 \to Q_1'$ are bijections such that~${s'(\phi_1(\alpha)) = \phi_0(s(\alpha))}$ and~${t'(\phi_1(\alpha)) = \phi_0(t(\alpha))}$ for all~$\alpha \in Q_1$.
The bijection~$\phi_1 : Q_1 \to Q_1'$ naturally extends to a bijection on paths by~$\phi_1(\alpha_1 \cdots \alpha_\ell) = \phi_1(\alpha_1) \cdots \phi_1(\alpha_\ell)$, and thus to a morphism~$\phi_1 : kQ \to kQ'$ of path algebras by linearity.
We say that two bound quivers~$\bar Q = (Q,I)$ and~$\bar Q' = (Q',I')$ are \defn{isomorphic} if there there is an isomorphism~$\phi = (\phi_0, \phi_1)$ between the quivers~$Q$ and~$Q'$ such that~$\phi_1(I) = I'$.

For a quiver~$Q = (Q_0, Q_1, s, t)$, define its \defn{reversed quiver} by~$\reversed{Q} \eqdef (Q_0, Q_1, t, s)$.
For a path~$\pi$ on~$Q$, define its \defn{reversed path} on~$\reversed{Q}$ by~$\reversed{(\alpha_1 \cdots \alpha_\ell)} = \alpha_\ell \cdots \alpha_1$, and extend it by linearity to an anti-morphism~$kQ \to k\reversed{Q}$ of path algebras.
We denote by~$\reversed{I}$ the direct image of~$I$ under this morphism.
The \defn{reversed bound quiver} of a bound quiver~$\bar Q = (Q,I)$ is the bound quiver~$\reversed{\bar Q} \eqdef (\reversed{Q},\reversed{I})$.

\subsection{Auslander-Reiten theory and $\tau$-tilting theory}
\label{subsec:tautilting}

This short section is mainly meant to fix notations and recall some definitions.  The interested reader is referred to \cite{AssemSimsonSkowronski, Schiffler} for detailed accounts of Auslander-Reiten theory, and to \cite{AdachiIyamaReiten} for $\tau$-tilting theory.

\begin{definition}
Let~$k$ be a field.
A \defn{representation} of a bound quiver~$\bar Q = (Q,I)$ is a pair ${M = \big( (M_v)_{v\in Q_0}, (M_\alpha)_{\alpha\in Q_1} \big)}$, where
\begin{itemize}
\item $M_v$ is a~$k$-vector space for all~$v\in Q_0$, and
\item $M_\alpha: M_{s(\alpha)} \to M_{t(\alpha)}$ is a linear map for all~$\alpha \in Q_1$, such that $M_\pi := \sum M_{\alpha_{\ell}} \circ \cdots \circ M_{\alpha_1} = 0$  for any element~$\pi =\sum \alpha_1 \dots \alpha_\ell$ in~$I$.
\end{itemize}
We will always assume that representations are finite-dimensional, that is, that all the spaces~$M_v$ are finite-dimensional.

A \defn{morphism of representations}~$f:M\to N$ from a representation~$M$ to a representation~$N$ is a tuple~$f=(f_v)_{v\in Q_0}$, where
\begin{itemize}
\item $f_v:M_v\to N_v$ is a linear map for all~$v\in Q_0$, and
\item for all~$\alpha\in Q_1$, the equality~$N_\alpha \circ f_{s(\alpha)} = f_{t(\alpha)} \circ M_\alpha$ holds.
\end{itemize}
\end{definition}

Representations of a bound quiver~$\bar Q$, together with morphisms of representations, form a category~$\rep \bar Q$.
This category is equivalent to the category of finite-dimensional modules over the algebra~$kQ/I$.
As consequences of this,~$\rep \bar Q$ is abelian, has enough projective and injective objects, and thus every object of $\rep \bar Q$ admits a projective and an injective resolution.
Projective and injective representations can be explicitly computed (see, for instance, \cite[Chapter III.2]{AssemSimsonSkowronski}).
In particular, there is a canonical bijection~$v\mapsto P(v)$ from~$Q_0$ to a set of representatives for the isoclasses of indecomposable projective representations of~$\bar Q$.
Another consequence is that any representation can be written in a unique way as a direct sum of indecomposable representations, up to isomorphism.
In other words, the category~$\rep \bar Q$ satisfies the \mbox{Azuyama\,--\,Krull\,--\,Remak\,--\,Schmidt} Theorem.
In view of the above, we will often use the words ``module'' and ``representation'' to mean the same thing.

An important operation on representations is the Auslander-Reiten translation (see \cite[Chap.~IV]{AssemSimsonSkowronski}).  For an indecomposable representation~$M$, let 
\[
 P_1^M \stackrel{f^M}{\to} P_0^M \to M \to 0 
\]
be a minimal projective presentation.  Applying the \defn{Nakayama functor}~$\nu = D\circ \Hom{kQ/I}(-, kQ/I)$, where~$D = \Hom{k}(-, k)$ is the standard vector space duality, one gets a morphism between injective representations
\[
\nu P_1^M \stackrel{\nu f^M}{\longrightarrow} \nu P_0^M.
\]
\begin{definition}
 \begin{enumerate}
  \item Using the above notation, the \defn{Auslander-Reiten translation}~$\tau M$ of~$M$ is defined to be the kernel of~$\nu f^M$.
  \item \cite{AdachiIyamaReiten} A representation~$M$ is \defn{$\tau$-rigid} if~$\Hom{\bar Q}(M, \tau M) = 0$.
 \end{enumerate}
\end{definition} 

It is convenient to introduce the following notation and terminology.  If~$P$ is a projective representation of~$\bar Q$, then the associated \defn{shifted projective} is the symbol~$P[1]$.  We treat~$P[1]$ as an object, and allow ourselves to form direct sums of shifted projectives with other shifted projectives and with ordinary representations of~$\bar Q$.  By convention, the shifted projectives are $\tau$-rigid.

\begin{definition}
\label{def: tau-rigid and stautilt}
With this notation, and following \cite{AdachiIyamaReiten}, we say that~$M\oplus P[1]$ is
\begin{enumerate}
 \item a \defn{$\tau$-rigid representation} if
   \begin{itemize}
     \item $M$ is a representation of~$\bar Q$ and~$P$ is a projective representation of~$\bar Q$,
     \item $M$ is~$\tau$-rigid, and
     \item $\Hom{\bar Q}(P,M) = 0$.
   \end{itemize}
 \item a \defn{support~$\tau$-tilting representation} if moreover the number of pairwise non-isomorphic indecomposable summands of $M\oplus P$ is the number of vertices of~$Q$.
\end{enumerate}
Note that~$A = kQ/I$ and~$A[1]$ are always support~$\tau$-tilting modules.
\end{definition}

\begin{definition}\label{def: tau-compatibility}
 Two representations or shifted projectives~$M, N$ of~$\bar Q$ are said to be \defn{$\tau$-compatible} precisely when $M\oplus N$ is a~$\tau$-rigid representation of~$\bar Q$.
\end{definition}

\begin{remark}
In \cite{AdachiIyamaReiten}, the notation~$(M,P)$ is used instead of~$M\oplus P[1]$.  Our choice of notation is motivated by the link between~$\tau$-rigid representations and so-called $2$-terms silting complexes.
\end{remark}

\begin{remark}
Let~$v\in Q_0$, then~$\Hom{\bar Q} \big( P(v), M \big) = 0$ if and only if~$M_v=0$. More generally, if~$v_1, \ldots, v_r$ are vertices of~$Q$, then~$\Hom{\bar Q}\big( {\oplus_{j \in [r]} P(v_j)}, M \big) = 0$ if and only if~$M_{v_j}=0$ for all~$j \in [r]$. 
\end{remark}

One of the most important results on support~$\tau$-tilting objects is the existence and uniqueness of a \defn{mutation}.  

\begin{theorem}[\cite{AdachiIyamaReiten}]\label{thm:mutationstautilts}
Let~$N = \bigoplus_{i \in [n]} N_i$ be a basic support~$\tau$-tilting object, where each~$N_i$ is either an indecomposable representation or a shifted indecomposable projective representation.
Then for any~${i \in [n]}$, there exists a unique indecomposable representation or shifted indecomposable projective representation~$N'_i$ not isomorphic to~$N_i$ and such that
\[
\mu_i(N) \eqdef N'_i \oplus \bigoplus_{j\neq i} N_j
\]
is a support~$\tau$-tilting object (called mutation of~$N$ at~$N_i$).
\end{theorem}

\begin{definition}\label{def: stautilt complex}
Let~$\bar Q = (Q,I)$ be a bound quiver and $A = kQ/I$ be its bound algebra.
\begin{itemize}
\item The \defn{support $\tau$-tilting complex}~$\tTC$ is the simplicial complex whose vertices are the isomorphism classes of indecomposable $\tau$-rigid or shifted projective representations of~$A$ and whose faces are the collections of representations~$\{M_1, \dots, M_k\}$ such that $M_1 \oplus \dots \oplus M_k$ is $\tau$-rigid.
\item The \defn{support $\tau$-tilting graph}~$\tTC^*$ is the dual graph of~$\tTC$ with a vertex for each isomorphism class of support $\tau$-tilting modules over~$A$ and an edge between two vertices~$T$ and~$U$ if and only if~$U$ is a mutation of~$T$.
\end{itemize}
\end{definition}

\begin{example}
 Two examples of support $\tau$-tilting graphs and support $\tau$-tilting complexes are given in Figure~\ref{fig:exmtTC}.
\end{example}

\subsection{String and gentle algebras}
\label{subsec:stringGentleAlgebra}

String algebras are a class of path algebras defined by generators and particularly nice relations.
The study of their representations goes back to \cite{GelfandPonomarev}, see also \cite{DonovanFreislich} and \cite{WaldWaschbusch}.
We will follow the general framework established in \cite{ButlerRingel}.

\begin{definition}[\cite{ButlerRingel}]
A \defn{string bound quiver}~$\bar Q \eqdef (Q,I)$ is a bound quiver such that
\begin{itemize}
\item each vertex~$v \in Q_0$ has at most two incoming and two outgoing arrows,
\item the ideal~$I$ is generated by paths,
\item for any arrow~$\beta \in Q_1$, there is at most one arrow~$\alpha \in Q_1$ such that~$t(\alpha) = s(\beta)$ and~$\alpha\beta\notin I$,
\item for any arrow~$\beta \in Q_1$, there is at most one arrow~$\gamma \in Q_1$ such that~$t(\beta) = s(\gamma)$ and~$\beta\gamma\notin I$.
\end{itemize}
The algebra~$kQ/I$ associated to a string bound quiver~$\bar Q$ is called a \defn{string algebra}.
\end{definition}

\begin{example}
\label{exm:stringAlgebras}
We have represented in \fref{fig:exmQuiver} four examples of string bound quivers.
Throughout the paper, we draw \raisebox{-.04cm}{\includegraphics[scale=.3]{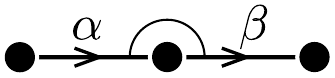}} to indicate that~$\alpha\beta \in I$.

\begin{figure}[b]
	\capstart
	\centerline{
		\begin{tabular}{c@{\quad}@{\quad}c@{\quad}@{\quad}c@{\quad}@{\quad}c}
			$Q = \vcenter{\hbox{\includegraphics[scale=.5]{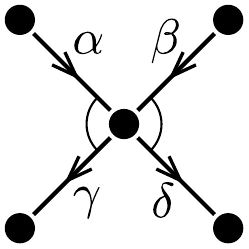}}}$ &
			$Q = \vcenter{\hbox{\includegraphics[scale=.5]{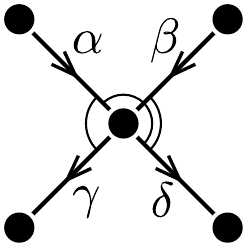}}}$ &
			$Q = \vcenter{\hbox{\includegraphics[scale=.5]{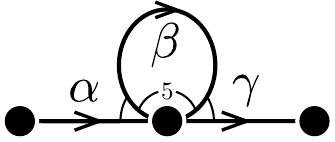}}}$ &
			$Q = \vcenter{\hbox{\includegraphics[scale=.5]{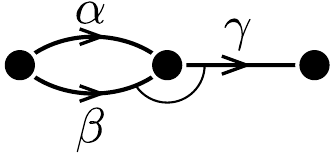}}}$ \\[1.1cm]
			$I = (\alpha\beta, \gamma\delta)$ & 
			$I = (\alpha\beta, \gamma\delta, \alpha\delta)$ & 
			$I = (\alpha\beta, \beta\gamma, \beta^5)$ & 
			$I = (\beta\gamma)$
		\end{tabular}
	}
	\caption{Four string bound quivers. The leftmost and rightmost are gentle bound quivers.}
	\label{fig:exmQuiver}
\end{figure}
\end{example}

\begin{definition}[\cite{ButlerRingel}]
\label{def:gentleQuiver}
A \defn{gentle bound quiver}~$\bar Q \eqdef (Q,I)$ is a string bound quiver such that
\begin{itemize}
\item $I$ is generated by paths of length exactly two,
\item for any arrow~$\beta \in Q_1$, there is at most one arrow~$\alpha \in Q_1$ such that~$t(\alpha) = s(\beta)$ and~$\alpha\beta \in I$,
\item for any arrow~$\beta \in Q_1$, there is at most one arrow~$\gamma \in Q_1$ such that~$t(\beta) = s(\gamma)$ and~$\beta\gamma \in I$.
\end{itemize}
The algebra~$kQ/I$ associated to a gentle bound quiver~$\bar Q$ is called a \defn{gentle algebra}.
\end{definition}

\begin{example}
In Example~\ref{exm:stringAlgebras}, only the first and last bound quivers are gentle.
\end{example}

String algebras enjoy a particularly nice representation theory: they are tame algebras (as proved in~\cite{ButlerRingel}) and their indecomposable representations are completely classified.  We will make heavy use of this classification, which we recall in the next two sections.

\subsection{Strings and bands}
\label{subsec:stringsBands}

For any arrow~$\alpha$ of any quiver~$Q$, define a formal inverse~$\alpha^{-1}$ with the properties that~$s(\alpha^{-1}) = t(\alpha)$, $t(\alpha^{-1}) = s(\alpha)$, $\alpha^{-1}\alpha = \varepsilon_{t(\alpha)}$ and~$\alpha\alpha^{-1} = \varepsilon_{s(\alpha)}$, where~$\varepsilon_v$ is the path of length zero starting and ending at the vertex~$v \in Q_0$.  Furthermore, let~$(\alpha^{-1})^{-1}=\alpha$; thus~$(-)^{-1}$ is an involution on the set of arrows and formal inverses of arrows of~$Q$.  We extend this involution to the set of all paths involving arrows or their formal inverses by further setting~$(\varepsilon_v)^{-1} = \varepsilon_v$ for all vertices~$v \in Q_0$.

\begin{definition}
\label{def:stringsBands}
Let~$\bar Q = (Q,I)$ be a string bound quiver.
\begin{enumerate}
\item A \defn{string} for~$\bar Q$ is a word of the form
\(
\rho = \alpha_1^{\varepsilon_1}\alpha_2^{\varepsilon_2}\cdots \alpha_\ell^{\varepsilon_\ell},
\)
where
	\begin{itemize}
	\item $\alpha_i \in Q_1$ and~$\varepsilon_i \in \{-1,1\}$ for all~$i \in [\ell]$,
	\item $t(\alpha_i^{\varepsilon_i}) = s(\alpha_{i+1}^{\varepsilon_{i+1}})$ for all~$i \in [\ell-1]$,
	\item there is no path~$\pi \in I$ such that~$\pi$ or~$\pi^{-1}$ appears as a factor of~$\rho$, and
	\item $\rho$ is reduced, in the sense that no factor~$\alpha\alpha^{-1}$ or~$\alpha^{-1}\alpha$ appears in~$\rho$, for~$\alpha \in Q_1$.
	\end{itemize}
The integer~$\ell$ is called the \defn{length} of the string~$\rho$.
Moreover, we still denote by~$s(\rho) \eqdef s(\alpha_1^{\varepsilon_1})$ and~$t(\rho) \eqdef t(\alpha_\ell^{\varepsilon_\ell})$ the source and target of~$\rho$.
For each vertex~$v \in Q_0$, there is also a \defn{string of length zero}, denoted by~$\varepsilon_v$, that starts and ends at~$v$.
We denote by~$\strings(\bar Q)$ the set of strings on~$\bar Q$.
We usually denote strings by~$\rho, \sigma, \tau$.
For reasons that will be explicit later, we often implicitly identify the two inverse strings~$\rho$ and~$\rho^{-1}$.
More formally, an \defn{undirected string} is a pair~$\{\rho, \rho^{-1}\}$ and we let~$\strings^\pm(\bar Q) \eqdef \set{\{\rho, \rho^{-1}\}}{\rho \in \strings(\bar Q)}$ denote the collection of all undirected strings of~$\bar Q$.

\item We call \defn{negative simple string} a formal word of length zero of the form $-v$, where $v$ is any vertex of $Q_0$. We denote by
\(
\strings^\pm_{\ge -1}(\bar Q) \eqdef \strings^\pm(\bar Q) \sqcup \set{-v}{v \in Q_0}
\)
the set of \defn{almost positive unoriented strings}.
Note that~$\varepsilon_v$ and~$-v$ are different almost positive strings.

\item A \defn{band} for~$\bar Q$ is a string~$\varpi$ of length at least one such that
	\begin{itemize}
	\item $s(\varpi) = t(\varpi)$,
	\item all powers of~$\varpi$ are strings, and
	\item $\varpi$ is not itself a power of a strictly smaller string.
	\end{itemize}
      
\end{enumerate}
\end{definition}

\begin{example}
\label{exm:aStringAlgebra}
For the algebra
\[
\vcenter{\hbox{\includegraphics[scale=.5]{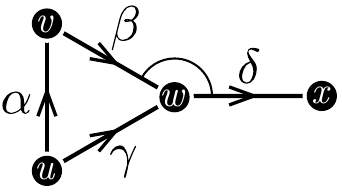}}} \qquad I = (\beta\delta)
\]
some almost positive strings are~$-u, -v, -w, -x, \varepsilon_u$, $\varepsilon_v$, $\varepsilon_w$, $\varepsilon_x$, $\alpha^\pm$, $\beta^\pm$, $\gamma^\pm$, $\delta^\pm$, $\gamma\delta$, $\beta\gamma^{-1}$, etc. The only bands are~$\beta\gamma^{-1}\alpha$, $\alpha\beta\gamma^{-1}$, $\gamma^{-1}\alpha\beta$ and their inverses.
\end{example}

\begin{notation}
The following notation, taken from \cite{ButlerRingel}, will be useful for dealing with strings.
For a given string~$\rho = \alpha_1^{\varepsilon_1}\alpha_2^{\varepsilon_2}\cdots \alpha_\ell^{\varepsilon_\ell}$, we draw~$\rho$ as follows:
  \begin{itemize}
    \item draw all arrows~$\alpha_1, \ldots, \alpha_\ell$ from left to right;
    \item draw all arrows pointing downwards.
  \end{itemize}
Moreover, the strings~$\varepsilon_v$ of length zero are depicted simply as~$v$.
Be aware that even if the string is depicted linearly from left to right, it might have cycles since some substrings can be repeated along the string.
\end{notation}

\begin{example}
\label{exm:drawingString}
For the algebra defined in Example~\ref{exm:aStringAlgebra}, the string~$\rho = \gamma\beta^{-1}\alpha^{-1}\gamma\beta^{-1}\alpha^{-1}\gamma\delta$ is drawn as follows:
\[
\rho = \;\vcenter{\hbox{\includegraphics[scale=0.3]{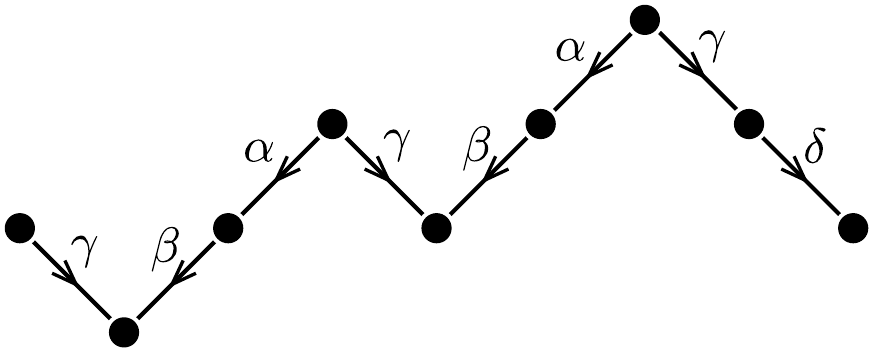}}}
\]
\end{example}

\begin{definition}
A \defn{substring} of a string~$\rho$ is a factor of~$\rho$.
Note that the position of the substring in~$\rho$ is important and that a word can appear as distinct substrings of~$\rho$ in different positions.
Observe also that we allow substrings of length~$0$, which we will abusively call \defn{vertices} of~$\rho$.
If~$u,v$ are two vertices of~$\rho$, we denote by~$\rho[u,v]$ the substring of~$\rho$ between~$u$ and~$v$.
We also abbreviate~$\rho[s(\rho),v]$ by~$\rho[\,\cdot\,,v]$ and~$\rho[u,t(\rho)]$ by~$\rho[u,\cdot\,]$.
A substring of~$\rho$ is \defn{strict} if it is distinct from~$\rho$ itself.
Finally, we denote by~$\Sigma(\rho)$ the set of all substrings~of~$\rho$.
\end{definition}

\begin{definition}
\label{def:topBottom}
A substring~$\rho = \alpha_i^{\varepsilon_i} \cdots \alpha_j^{\varepsilon_j}$ of a string~$\sigma = \alpha_1^{\varepsilon_1} \cdots \alpha_\ell^{\varepsilon_\ell}$ of~$\bar Q$ is said to be:

\begin{itemize}    
\item \defn{on top of~$\sigma$} (or a \defn{top substring of~$\sigma$}) if~$\sigma$ either ends or has an outgoing arrow at each endpoint of~$\rho$, \ie if $i = 1$ or~$\varepsilon_{i-1} = -1$, and $j = \ell$ or~$\varepsilon_{j+1} = 1$.
\item \defn{at the bottom of~$\rho$} (or a \defn{bottom substring of~$\rho$}) if~$\sigma$ either ends or has an incoming arrow at each endpoint of~$\rho$, \ie if $i = 1$ or~$\varepsilon_{i-1} = 1$, and $j = \ell$ or~$\varepsilon_{j+1} = -1$.
\end{itemize}
We denote by $\Sigma_\top(\rho)$ and~$\Sigma_\bottom(\rho)$ the sets of top and bottom substrings of~$\rho$ respectively.
Note that~$\rho$ is both in~$\Sigma_\top(\rho)$ and in~$\Sigma_\bottom(\rho)$.
\end{definition}

\begin{example}
\fref{fig:exmTopBottom} represents examples of top and bottom substrings.

\begin{figure}[t]
	\capstart
	\centerline{\includegraphics[scale=.3]{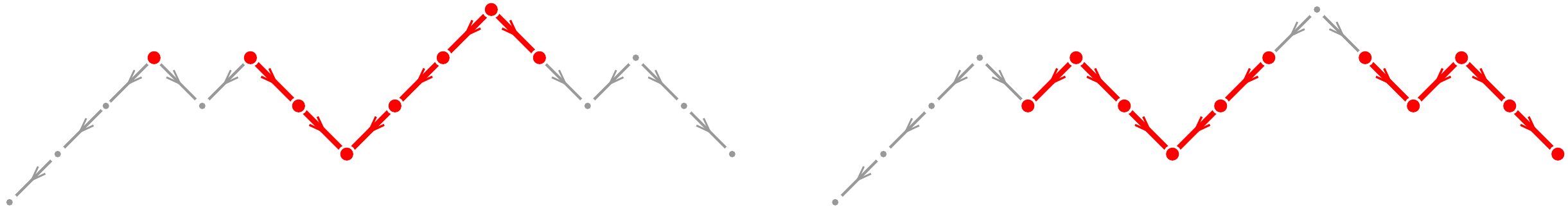}}
	\caption{Two top substrings (left) and two bottom substrings (right). Note that substrings of~$\rho$ can be reduced to a vertex and can end at an endpoint of~$\rho$.}
	\label{fig:exmTopBottom}
\end{figure}
\end{example}

\begin{definition}
A \defn{peak} (resp.~\defn{deep}) of a string~$\rho$ is a vertex~$v$ of~$\rho$ (substring of length~$0$) which is on top (resp.~at the bottom) of~$\rho$.
We call \defn{corners} of~$\rho$ the peaks and the deeps of~$\rho$.
A corner is \defn{strict} if it is not an endpoint of~$\rho$.
In other words, $v$ is a strict peak (resp.~deep) if the two incident arrows of~$\rho$ at~$v$ are both outgoing (resp.~incoming).
\end{definition}

\begin{example}
The string of \fref{fig:exmTopBottom} has~$4$ peaks and~$5$ deeps ($3$ of which are strict).
\end{example}

\begin{remark}
\label{rem:reverseStrings}
Note that reversing the quiver preserves the set of strings and their substrings but exchanges top with bottom:
$\Sigma(\reversed{\rho}) = \Sigma(\reversed{\rho})$ but~$\Sigma_\bottom(\reversed{\rho}) = \Sigma_\top(\rho)$ and~$\Sigma_\top(\reversed{\rho}) = \Sigma_\bottom(\rho)$.
\end{remark}

\subsection{String and band modules}
\label{subsec:stringBandModules}

Following~\cite{ButlerRingel}, we now define $A$-modules corresponding to the strings and bands of a string algebra~$A = kQ/I$.

\begin{definition}[\cite{ButlerRingel}]
Let~$\rho = \alpha_1^{\varepsilon_1}\alpha_2^{\varepsilon_2}\cdots \alpha_\ell^{\varepsilon_\ell}$ be a string for a string algebra~$A=kQ/I$.  The \defn{string module}~$M(\rho)$ is the $A$-module defined as a representation of~$Q$ as follows.

  \begin{itemize}
    \item Let~$v_0 = s(\alpha_1^{\varepsilon_1})$, and~$v_m = t(\alpha_m^{\varepsilon_m})$ for each~$m \in [\ell]$.
    
    \item For each vertex~$u \in Q_0$, let~$M(\rho)_u$ be the vector space with basis given by~$\{x_m \ | \ v_m = u\}$.
    
    \item For each arrow~$\beta$ of~$Q$, the linear map~$M(\rho)_\beta : M(\rho)_{s(\beta)} \to M(\rho)_{t(\beta)}$ is defined on the basis of~$M(\rho)_{s(\beta)}$ by
     \[
        M(\rho)_\beta(x_m) =
        \begin{cases}
			x_{m-1} & \text{if~$\alpha_{m} = \beta$ and~$\varepsilon_{m} = -1$}, \\
			x_{m+1} & \text{if~$\alpha_{m+1} = \beta$ and~$\varepsilon_{m+1} = 1$}, \\
            0 & \text{otherwise}.
       \end{cases}
     \]
  \end{itemize}
\end{definition}

It follows from the definition that for any string~$\rho$, the string modules~$M(\rho)$ and~$M(\rho^{-1})$ are isomorphic.
This explains why we will often consider undirected strings.

\begin{example}
For the algebra given in Example~\ref{exm:aStringAlgebra}, let~$\rho = \gamma\beta^{-1}\alpha^{-1}\gamma\beta^{-1}\alpha^{-1}\gamma\delta$.
Then~$M(\rho)$ is given by

\[
	\begin{overpic}[scale=0.6]{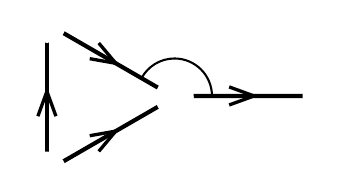}
	\put(8,2){$k^3$}
	\put(8,47){$k^2$}
	\put(47,25){$k^3$}
	\put(90,25){$k$}
	\put(-20,23){$\big[\!\begin{smallmatrix} 0 & 1 & 0 \\ 0 & 0 & 1 \end{smallmatrix}\!\big]$}
	\put(30,50){$\left[\!\begin{smallmatrix} 1 & 0 \\ 0 & 1 \\ 0 & 0 \end{smallmatrix}\!\right]$}
	\put(30,0){$\left[\!\begin{smallmatrix} 1 & 0 & 0 \\ 0 & 1 & 0 \\ 0 & 0 & 1 \end{smallmatrix}\!\right]$}
	\put(60,17){$[\!\begin{smallmatrix} 0 & 0 & 1 \end{smallmatrix}\!]$}
	\end{overpic}
\]
\vspace{-.2cm}
\end{example}

\begin{notation}
 \begin{itemize}
  \item For convenience, even though~$0$ is not a string by our definition, we will define~$M(0)$ to be the zero module.
  \item For any vertex~$v \in Q_0$, the projective~$P(v)$ is $M(v_\peak)$ where~$v_\peak$ is the maximal string of~$\bar Q$ with a single peak at~$v$.
  \item We will also make use of the following convention: for any~$v\in Q_0$, we let~$M(-v)$ be the shifted projective~$P(v)[1]$.
 \end{itemize}
\end{notation}

\begin{definition}[\cite{ButlerRingel}]
Let~$\varpi = \alpha_1^{\varepsilon_1}\alpha_2^{\varepsilon_2}\cdots \alpha_\ell^{\varepsilon_\ell}$ be a band for a string algebra~$A=kQ/I$.
Moreover, let~$\lambda\in k^*$, and let~$d \in \N_{>0}$.
The \defn{band module}~$M(\varpi, \lambda, d)$ is defined as follows.
\begin{itemize}
\item For each vertex~$v \in Q_0$, the vector space~$M(\varpi, \lambda, d)_v$ as a copy of $k^d$ for each $0 \le m \le \ell$ such that $v_m = v$.
\item For each $m \in [\ell-1]$, the linear map $M(\varpi, \lambda, d)_{\alpha_m}$ sends the copy of $k^d$ associated to $s(\alpha_m)$ to that associated to $t(\alpha_m)$ via the identity map.
\item The linear map~$M(\varpi, \lambda, d)_{\alpha_\ell}$ sends the copy of $k^d$ associated to $s(\alpha_{\ell})$ to that associated to $t(\alpha_\ell)$ via the matrix~$J_d(\lambda^{\varepsilon_\ell})$, where
     \[
        J_d(\lambda) =
        \begin{pmatrix} 
        \lambda & 0 & \cdots & \cdots & 0 \\
        1 & \lambda & \ddots & & \vdots \\
        0 & 1 & \ddots & \ddots & \vdots \\
        \vdots & \ddots & \ddots & \lambda & 0 \\
        0 & \cdots & 0 & 1 & \lambda                                        
        \end{pmatrix}
     \]
    is the~$d\times d$ Jordan block of type~$\lambda$, where we identify $t(\alpha_{\ell}^{\varepsilon_{\ell}})$ with $s(\alpha_{1}^{\varepsilon_1})$.
  \end{itemize}
\end{definition}

It follows from the definition that the band modules~$M(\varpi, \lambda, d)$ and~$M(\varpi^{-1}, \lambda^{-1}, d)$ are isomorphic.  Moreover, if~$\varpi$ and~$\varpi'$ are two bands that are \defn{cyclically equivalent}, in the sense that one is obtained from the other by cyclically permuting the arrows that constitute it, then we also have that~$M(\varpi, \lambda, d)$ and~$M(\varpi', \lambda, d)$ are isomorphic.

\begin{example}
For the algebra defined in Example~\ref{exm:aStringAlgebra}, let~$\varpi=\alpha\beta\gamma^{-1}$.  Then~$\varpi$ is a band, and the following are band modules

\[
	M(\varpi, \lambda, 1) = \vcenter{\hbox{
	\begin{overpic}[scale=0.6]{exmQuiver5bis}
	\put(10,2){$k$}
	\put(10,47){$k$}
	\put(49,25){$k$}
	\put(90,25){$0$}
	\put(2,23){$1$}
	\put(30,46){$1$}
	\put(30,4){$\lambda^{-1}$}
	\put(68,17){$0$}
	\end{overpic}}}
	\hspace{1cm}
	M(\varpi, \lambda, 3) = \vcenter{\hbox{
	\begin{overpic}[scale=0.6]{exmQuiver5bis}
	\put(8,2){$k^3$}
	\put(8,47){$k^3$}
	\put(47,25){$k^3$}
	\put(90,25){$0$}
	\put(0,23){$\mathrm{Id}$}
	\put(30,46){$\mathrm{Id}$}
	\put(23,-5){$\left[\!\begin{smallmatrix} \lambda^{-1} & 0 & 0 \\ 1 & \lambda^{-1} & 0 \\ 0 & 1 & \lambda^{-1} \end{smallmatrix}\!\right]$}
	\put(68,17){$0$}
	\end{overpic}}}
\]
\vspace{.5cm}
\end{example}

\begin{theorem}[{\cite[p.\,161]{ButlerRingel}}]
Assume that~$k$ is algebraically closed and consider a string algebra~$A$.
Then the string and band modules over~$A$ form a complete list of indecomposable $A$-modules, up to isomorphism.
Moreover,
  \begin{itemize}
    \item A string module is never isomorphic to a band module.
    \item Two string modules~$M(\rho)$ and~$M(\rho')$ are isomorphic if and only if~$\rho' = \rho^{\pm 1}$.
    \item Two band modules~$M(\varpi,\lambda, d)$ and~$M(\varpi', \lambda', d')$ are isomorphic if and only if~$d=d'$, and  
       \begin{itemize}
         \item either $\varpi$ is cyclically equivalent to~$\varpi'$, and~$\lambda = \lambda'$,
         \item or~$\varpi^{-1}$ is cyclically equivalent to~$\varpi'$, and~$\lambda^{-1} = \lambda'$.
       \end{itemize}
  \end{itemize}
\end{theorem}

\subsection{Auslander-Reiten translation of string modules}
\label{subsec:ARtranslationStringModules}

Our goal in this section is to describe the Auslander-Reiten translation in a string algebra.
This description requires some additional definitions and notations.

\begin{definition}
\label{def:peaksDeeps}
We say that a string~$\rho$
  \begin{itemize}
    \item \defn{starts} (resp.~\defn{ends}) \defn{on a peak} if there is no arrow~$\alpha$ such that~$\alpha \rho$ (resp.~$\rho\alpha^{-1}$) is a string,
    \item \defn{starts} (resp.~\defn{ends}) \defn{in a deep} if there is no arrow~$\alpha$ such that~$\alpha^{-1} \rho$ (resp.~$\rho \alpha$) is a string.
  \end{itemize}
\end{definition} 

\begin{remark}
Definition~\ref{def:peaksDeeps} is best understood by drawing strings.  Starting (or ending) on a peak means that one cannot add an arrow at the start (or the end) of~$\rho$ such that the starting point (or ending point) of the new string would be higher in the picture than that of~$\rho$.  The same applies when replacing ``on a peak'' by ``in a deep'' and  ``higher'' by ``lower''.
\end{remark}

\begin{example}
The string in Example~\ref{exm:drawingString} starts and ends on a peak, and also ends in a deep.
\end{example}

\begin{definition}
\label{def:addingRemovingHooksCohooks}
  \begin{enumerate}
    \item Let~$\rho$ be a string that does not start on a peak. Let~$\alpha, \alpha_1, \ldots, \alpha_r$ be arrows such that~$_h \rho \eqdef \alpha_r^{-1}\cdots \alpha_1^{-1} \alpha \rho$ starts in a deep.  We say that~$_h \rho$ is obtained from~$\rho$ by \defn{adding a hook at the start of~$\rho$}.

    \item Dually, let~$\rho$ be a string that does not end on a peak.  Define $\rho_h \eqdef (_h ( \rho )^{-1} )^{-1}$ to be the string obtained by \defn{adding a hook at the end of~$\rho$}.


    \item Let~$\rho$ be a string that does not start in a deep. Let~$\alpha, \alpha_1, \ldots, \alpha_r$ be arrows such that ${_c \rho =  \alpha_r \cdots \alpha_1 \alpha^{-1} \rho}$ starts on a peak. We say that~$_c \rho$ is obtained from~$\rho$ by \defn{adding a cohook at the start of~$\rho$}.

    \item Dually, let~$\rho$ be a string that does not end in a deep.  Define $\rho_c \eqdef (_c ( \rho )^{-1} )^{-1}$ to be the string obtained by \defn{adding a cohook at the end of~$\rho$}.

  \end{enumerate}
\end{definition}

\begin{remark}
Note that in Definition~\ref{def:addingRemovingHooksCohooks}:
  \begin{enumerate}
    \item Unless~$\rho$ has length zero, the strings~$_h\rho$, $\rho_h$, $_c\rho$ and~$\rho_c$, when they are defined, are uniquely determined. If~$\rho$ has length zero, then one can exchange~$_h\rho$ with $\rho_h$ and~$_c\rho$ with~$\rho_c$.
    \item In all $4$ cases, it is possible that~$r=0$, meaning that the added hook (resp.~cohook) is restricted to a single incoming (resp.~outgoing) arrow.
  \end{enumerate}    
\end{remark}

\begin{remark}
In pictures, hooks and cohooks always look like

\bigskip
\centerline{
\begin{overpic}[scale=0.3]{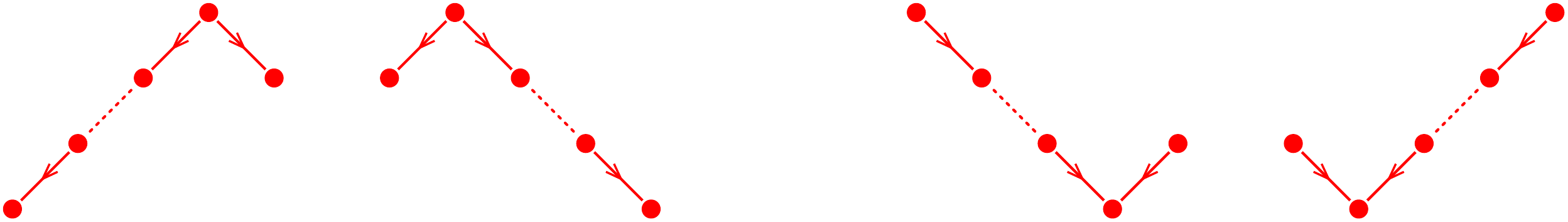}
\put(18,-5){hooks}
\put(74,-5){cohooks}
\end{overpic}
}
\vspace{.8cm}
\end{remark}

\begin{example}
In the quiver of Example~\ref{exm:aStringAlgebra}, the string~$\rho = \gamma\beta^{-1}\alpha^{-1}\gamma\beta^{-1}\alpha^{-1}\gamma\delta$ of Example~\ref{exm:drawingString} does not start on a deep, and we have
\[
\rho = \;\vcenter{\hbox{\includegraphics[scale=0.3]{exmString1}}}
\hspace{1.3cm}
_c\rho = \;\vcenter{\hbox{\includegraphics[scale=0.3]{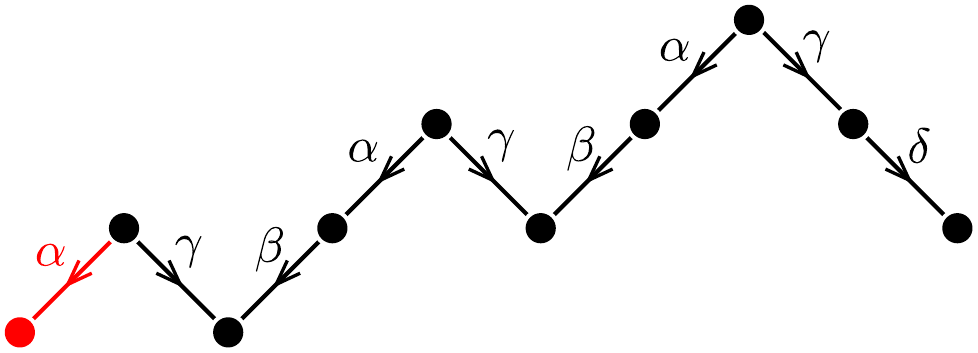}}}
\]
Still in the same quiver of Example~\ref{exm:aStringAlgebra}, the string~$\rho = \delta$ does not start on a peak, and we have
\[
\rho = \;\vcenter{\hbox{\includegraphics[scale=0.3]{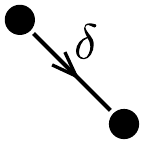}}}
\hspace{1.3cm}
_h\rho = \;\vcenter{\hbox{\includegraphics[scale=0.3]{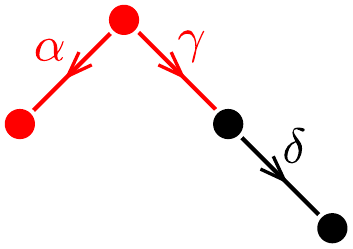}}}
\]
\end{example}

\begin{lemma}[and Definition]
\enlargethispage{.1cm}
Let~$\rho$ be a string for a string bound quiver~$\bar Q$.

  \begin{enumerate}
    \item If~$\rho$ starts on a peak and is not a path in~$\bar Q$, then there exists a unique string~$\rho'$ such that~$\rho'$ does not start in a deep, and~$\rho = \phantom{}_c\rho'$. We say that~$\rho'$ is obtained by \defn{removing a cohook at the start of~$\rho$} and denote it by~$_{c^{-1}}\rho$. If~$\rho$ is a path in~$\bar Q$, then we put~$_{c^{-1}}\rho =0$.
    
    \item Dually, if~$\rho$ ends on a peak and is not the inverse of a path in~$\bar Q$, then define the string obtained by \defn{removing a cohook at the end of~$\rho$} as~${\rho_{c^{-1}} \eqdef (_{c^{-1}} (\rho^{-1}) )^{-1}}$.
    
    
    \item If~$\rho$ starts in a deep and is not the inverse of a path in~$\bar Q$, then there exists a unique string~$\rho'$ such that~$\rho'$ does not start on a peak, and~$\rho = \phantom{} _h\rho'$. We say that~$\rho'$ is obtained by \defn{removing a hook at the start of~$\rho$} and denote it by~$_{h^{-1}}\rho$. If~$\rho$ is a the inverse of a path in~$\bar Q$, then we put~$_{h^{-1}}\rho =0$.
    
    \item Dually, if~$\rho$ ends in a deep and is not a path in~$\bar Q$, then define the string obtained by \defn{removing a hook at the end of~$\rho$} as~${\rho_{h^{-1}} \eqdef (_{h^{-1}} (\rho^{-1}) )^{-1}}$.
    
  \end{enumerate}
\end{lemma}

\begin{example}
For the algebra defined in Example~\ref{exm:aStringAlgebra}, the string
\[
\rho = \gamma\beta^{-1}\alpha^{-1}\gamma\beta^{-1}\alpha^{-1}\gamma\delta =  \;\vcenter{\hbox{\includegraphics[scale=0.3]{exmString1}}}
\]
starts and ends on a peak, and also ends in a deep. Moreover, we have
\[
_{c^{-1}}\rho = \;\vcenter{\hbox{\includegraphics[scale=0.3]{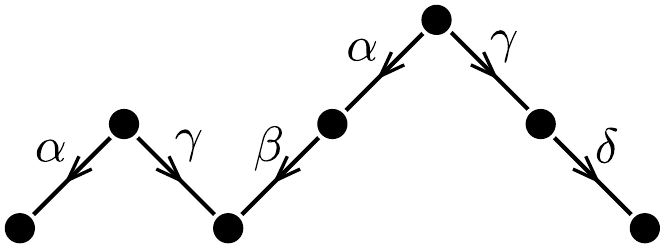}}}
\qquad
\rho_{h^{-1}} = \;\vcenter{\hbox{\includegraphics[scale=0.3]{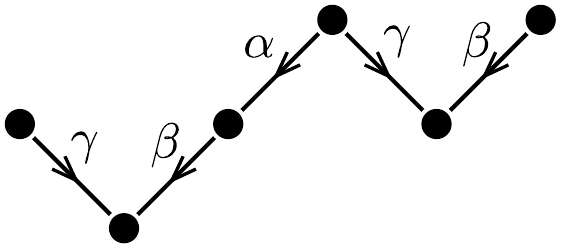}}}
\qquad
\rho_{c^{-1}} = \;\vcenter{\hbox{\includegraphics[scale=0.3]{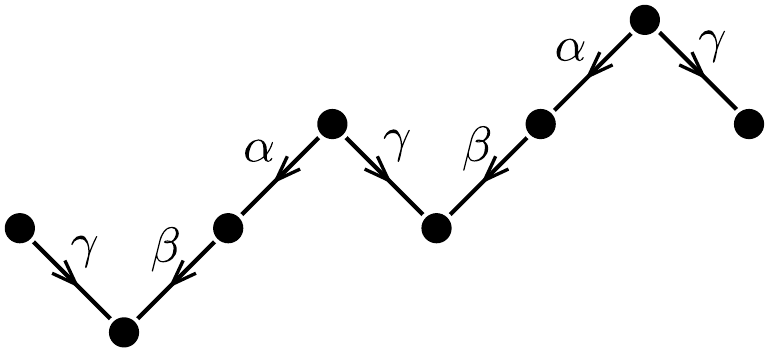}}}
\]
\end{example}

We now have all definitions to describe the Auslander-Reiten translate of a string module.

\begin{theorem}[\cite{ButlerRingel}]
\label{theorem: tau for strings}
Consider a string algebra~$A = kQ/I$.
The Auslander-Reiten translation transforms string modules to string modules: for a string~$\rho$ of~$\bar Q$, we have~${\tau \big( M(\rho) \big) = M(\rho')}$, where the string~$\rho'$ is obtained by following the table below.
\\[.3cm]
	\centerline{
	\begin{tabular}{|c|c|c|} 
      \hline 
      & $\rho$ does not start in a deep & $\rho$ starts in a deep \\
      \hline
      $\rho$ does not end in a deep & $\rho' = {}_c\rho_c$ & $\rho' = {}_{h^{-1}}(\rho_c)$ \\
      \hline
      $\rho$ ends in a deep & $\rho' = (_c\rho)_{h^{-1}}$ & $\rho' = {}_{h^{-1}}\rho_{h^{-1}}$ \\
      \hline
    \end{tabular}
    }
\\[.3cm]
Dually, we have~$\tau^{-1} \big( M(\rho) \big) = M(\rho'')$, where~$\rho''$ is the string given in the table below.
\\[.3cm]
	\centerline{
    \begin{tabular}{|c|c|c|} 
      \hline 
      & $\rho$ does not start on a peak & $\rho$ starts on a peak \\
      \hline
      $\rho$ does not end on a peak & $\rho'' = {}_h\rho_h$ & $\rho'' = {}_{c^{-1}}(\rho_h)$ \\
      \hline
      $\rho$ ends on a peak & $\rho'' = (_h\rho)_{c^{-1}}$ & $\rho'' = {}_{c^{-1}}\rho_{c^{-1}}$ \\
      \hline
    \end{tabular}
	}
\\[.3cm]
Finally, both the Auslander-Reiten translation~$\tau$ and its inverse~$\tau^{-1}$ preserve the band modules: for a band~$\varpi$ of~$\bar Q$, we have~$\tau \big( M(\varpi, \lambda, d) \big) = M(\varpi, \lambda, d)$ and $\tau^{-1} \big( M(\varpi, \lambda, d) \big)=M(\varpi, \lambda, d)$.
\end{theorem}

\subsection{Morphisms between string modules}
\label{subsec:morphismsStringModules}

Morphisms between string modules can be read directly from the combinatorics of strings.  The following result is folklore in the theory.

\begin{proposition}
\label{prop:morphismsStringModules}
Let~$\rho$ and~$\rho'$ be two strings for a string algebra~$A$.
Then the dimension of the $k$-vector space~$\Hom{A} \big( M(\rho), M(\rho') \big)$ is equal to the number of pairs~$(\sigma, \sigma')$, where~$\sigma$ and~$\sigma'$ are substrings of~$\rho$ and~$\rho'$, respectively, such that~$\sigma$ is on top of~$\rho$,~$\sigma'$ is at the bottom of~$\rho'$ and there is an equality of strings~$\sigma' = \sigma^{\pm 1}$.

\end{proposition}


\section{$\b{g}$-vectors and $\tau$-compatibility for string modules}
\label{sec:stringModules}

\subsection{$\b{g}$-vectors}
\label{subsec:gvectors}

The following definition takes its roots in the additive categorification of cluster algebras, see \cite{DehyKeller};
it also appears naturally in the theory of $\tau$-tilting, see \cite[Sect.~5.1]{AdachiIyamaReiten}.

\begin{definition}
\label{definition: g-vector of representation}
 Let $M$ be a representation of a bound quiver $\bar Q$, and let
 \[
  P^M_1 \to P^M_0 \to M \to 0
 \]
 be a minimal projective presentation of $M$.  
 The \defn{$\b{g}$-vector} of $M$ is the element 
 \[
  \b{g}(M) = [P_0^M­] - [P_1^M]
 \]
 of the Grothendieck group $K_0(\proj \bar Q)$ of the additive category $\proj \bar Q$ of projective representations of $\bar Q$. 
 If $P$ is a projective representation, then the $\b{g}$-vector of the shifted projective $P[1]$ is $\b{g}(P[1]) \eqdef -[P]$.
\end{definition}

\begin{remark}
 A more practical way to think of $\b{g}$-vectors is by noting that $K_0(\proj \bar Q)$ is isomorphic to the free abelian group
 $\bigoplus_{v\in Q_0} \Z [P(v)] \cong \Z^{Q_0}$.  Thus, in the notation of Definition \ref{definition: g-vector of representation},
 if we put $P_0^M = \bigoplus_{v\in Q_0}P(v)^{\oplus a_v}$ and $P_1^M = \bigoplus_{v\in Q_0}P(v)^{\oplus b_v}$, 
 then
 \[
  \b{g}(M) = (a_v - b_v)_{v\in Q_0} \in \Z^{Q_0}.
 \] 
 Moreover, if $P = \bigoplus_{v\in Q_0}P(v)^{\oplus n_v}$, then $\b{g}(P[1]) = (-n_v)_{v\in Q_0}$.
\end{remark}

An application of the combinatorics of strings allows us to compute the $\b{g}$-vector of any string module.

\begin{proposition}
 Let $\bar Q$ be a string bound quiver, and let $\rho$ be a string for $\bar Q$.  
 Let $A$ be the set of vertices on top of $\rho$, and let $P(A) = \bigoplus_{v\in A} P(v)$.
 Let $B$ be the set of vertices at the bottom of $\rho$ which are not at the start or the end of $\rho$,
 and let $P(B) = \bigoplus_{v\in B} P(v)$.  Finally, let 
 \[
  R = \begin{cases}
          0 & \textrm{if $\rho$ starts and ends in a deep,} \\
          P({t(\alpha)})   & \textrm{if $\rho$ ends in a deep and $\alpha^{-1}\rho$ is a string, with $\alpha\in Q_1$,} \\
          P({t(\beta)})    & \textrm{if $\rho$ starts in a deep and $\rho\beta$ is a string, with $\beta\in Q_1$,} \\
          P({t(\alpha)}) \oplus P({t(\beta)})  & \textrm{if $\alpha^{-1}\rho\beta$ is a string, with $\alpha, \beta\in Q_1$.}
        \end{cases}
 \]
 Then a minimal projective presentation of the string module $M(\rho)$ has the form
 \[
  R\oplus P(B) \to P(A) \to M(\rho) \to 0.
 \]
\end{proposition}

\begin{corollary}
 The $\b{g}$-vector of a string module $M(\rho)$ is given by $\b{g}(M) = \b{a} - \b{b} - \b{r}$, where
 \begin{itemize}
  \item $\b{a} = (a_v)_{v\in Q_0}$, where $a_v$ is the number of times the vertex $v$ is found at the top of $\rho$;
  \item $\b{b} = (b_v)_{v\in Q_0}$, where $b_v$ is the number of times the vertex $v$ is found at the bottom of $\rho$,
        but not at its start or end;
        \medskip
  \item 
  		\(
         \b{r} = \begin{cases}
                   0 & \textrm{if $\rho$ starts and ends in a deep,} \\
                   \b{e}_{t(\alpha)}   & \textrm{if $\rho$ ends in a deep and $\alpha^{-1}\rho$ is a string, with $\alpha\in Q_1$,} \\
                   \b{e}_{t(\beta)}    & \textrm{if $\rho$ starts in a deep and $\rho\beta$ is a string, with $\beta\in Q_1$,} \\
                   \b{e}_{t(\alpha)} + \b{e}_{t(\beta)}  & \textrm{if $\alpha^{-1}\rho\beta$ is a string, with $\alpha, \beta\in Q_1$,}
                  \end{cases}
        \) \\
        \medskip
        where $(\b{e}_v)_{v \in Q_0}$ is the standard basis of~$\Z^{Q_0}$.
 \end{itemize}
\end{corollary}

\begin{definition}
Let $\bar Q$ be a string bound quiver.
\begin{enumerate}
  \item If $\rho$ is a string for $\bar Q$, we define its \defn{$\b{g}$-vector} as $\b{g}(\rho) \eqdef \b{g} \big( M(\rho) \big)$.
  \item If $-v$ is a negative simple string for $\bar Q$, we define its \defn{$\b{g}$-vector} as $\b{g}(-v) \eqdef \b{g} \big( P(v)[1] \big) = -\b{e}_v$.
\end{enumerate}
\end{definition}

\subsection{$\tau$-compatibility}
\label{subsec:compatibility}

Consider a string bound quiver~$\bar Q = (Q,I)$ and its string algebra~${A = kQ/I}$.
We now provide a combinatorial criterion for $\tau$-compatibility of string modules over~$A$.

\begin{notation}
In the diagrams of this section, we will make repeated use of the following notation:
\begin{itemize}
 \item waves~$\xy \ar@{~} (0,0);(7,0) \endxy$ mean ``any string'' (possibly of length 0);
 \item brackets~$( \hspace{1pc} )$ mean ``either empty or of the form given inside the brackets'';
 \item plain arrows have to be there;
 \item straight lines stand for hooks or cohooks and might thus be of any length (possibly~$0$);
 \item dotted arrows mean ``there exists''.
\end{itemize}
\end{notation}

\begin{definition}
\label{definition: dance and attract}
Let~$\rho$ and~$\rho'$ be two strings of~$\bar Q$. The string~$\rho$ is said to
\begin{itemize}
\item[(i)] \defn{attract}~$\rho'$ (or~$(\rho')^{-1}$) if the starting vertex~$v$ of~$\rho'$ is not a deep and there exists an arrow~$\alpha \in Q_1$ from~$v$ to~$u$ such that
	\begin{enumerate}[(a)]
	\item $\alpha^{-1}\rho'$ is a string of~$\bar Q$ and
	\item $\rho$ ends in~$u$ or contains an arrow starting at~$u$ and moreover, the arrow, if any, preceding this occurrence of $u$ in $\rho$ is not $\alpha$ (see Figure~\ref{fig:def-attracts});
	\end{enumerate}
\item[(ii)] \defn{reach} for~$\rho'$ (or~$(\rho')^{-1}$) if there is some common substring~$\xi$ on top of~$\rho$ and at the bottom of~$\rho'$ satisfying the swinging arms condition explained thereafter. In this definition only, we call \defn{left and right arms} of~$\rho$ (resp.~$\rho'$) the (at most) two arrows in the string~$\rho$ (resp.~$\rho'$) incident with~$\xi$. The \defn{swinging arms condition} (as illustrated in Figure~\ref{fig: swinging arms}) is the following: 
\begin{itemize}
\item if~$\rho'$ has no left arm, then~$\rho$ has a left arm~$\alpha$ and~$\alpha^{-1}\rho'$ is a string of~$\bar Q$,
\item if~$\rho'$ has no right arm, then~$\rho$ has a right arm~$\beta$ and~$\rho'\beta$ is a string of~$\bar Q$,
\end{itemize}
\end{itemize}
\begin{figure}[t]
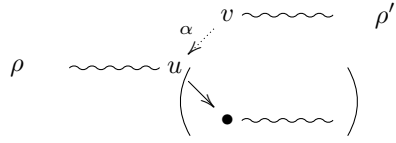

 	\capstart
\[
\xy
(28,7)*+{v}="14";
(42,7)="16";
(49,7)*+{\rho'}="17";
(0,0)*+{\rho}="0";
(7,0)="1";
(21,0)*+{u}="3";
(28,-7)*+{\bullet}="-4";
(42,-7)="-6";
{\ar@{~} "1";"3"};
{\ar "3";"-4"};
{\ar@{~} "-4";"-6"};
{\ar@{~} "14";"16"};
{\ar@{..>}_{\alpha} "14";"3"};
{\ar@/_.3pc/@{-} (23,0);(23,-9)};
{\ar@/^.3pc/@{-} (44,0);(44,-9)};
\endxy
\]
	\caption{The string~$\rho$ attracts~$\rho'$, as in Definition~\ref{definition: dance and attract} (i).}
	\label{fig:def-attracts}
\end{figure}

\begin{figure}[t]
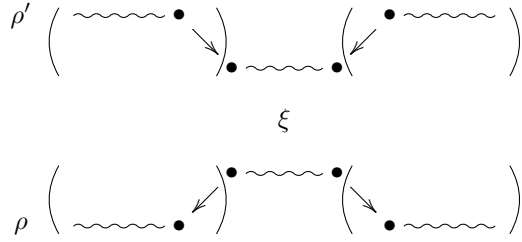

 	\capstart
\[
\xy
(-35,14)*+{\rho'};
(-28,14)="-14";
(-14,14)*+{\bullet}="-12";
(-7,7)*+{\bullet}="-11";
(7,7)*+{\bullet}="11";
(14,14)*+{\bullet}="12";
(28,14)="14";
(0,0)*+{\xi};
(-35,-14)*+{\rho};
(-28,-14)="-4";
(-14,-14)*+{\bullet}="-2";
(-7,-7)*+{\bullet}="-1";
(7,-7)*+{\bullet}="1";
(14,-14)*+{\bullet}="2";
(28,-14)="4";
{\ar@{~} "-14";"-12"};
{\ar "-12";"-11"};
{\ar@{~} "-11";"11"};
{\ar "12";"11"};
{\ar@{~} "12";"14"};
{\ar@/_.3pc/@{-} (-30,15);(-30,6)};
{\ar@/^.3pc/@{-} (-9,15);(-9,6)};
{\ar@/^.3pc/@{-} (30,15);(30,6)};
{\ar@/_.3pc/@{-} (9,15);(9,6)};
{\ar@/^.3pc/@{-} (-30,-15);(-30,-6)};
{\ar@/_.3pc/@{-} (-9,-15);(-9,-6)};
{\ar@/_.3pc/@{-} (30,-15);(30,-6)};
{\ar@/^.3pc/@{-} (9,-15);(9,-6)};
{\ar@{~} "-4";"-2"};
{\ar "-1";"-2"};
{\ar@{~} "-1";"1"};
{\ar "1";"2"};
{\ar@{~} "2";"4"};
\endxy
\]
	\caption{The string~$\rho$ reaches for~$\rho'$, as in Definition~\ref{definition: dance and attract} (ii).}
	\label{fig:def-dances}
\end{figure}

\begin{figure}[t]
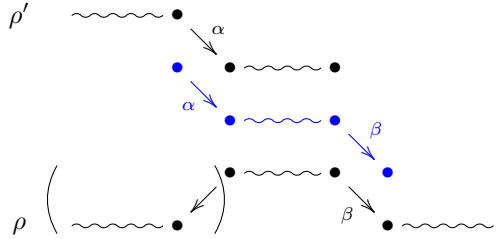

 	\capstart
\[
\xy
(-35,14)*+{\rho'};
(-28,14)="-14";
(-14,14)*+{\bullet}="-12";
(-7,7)*+{\bullet}="-11";
(7,7)*+{\bullet}="11";
(-14,7)*+{\textcolor{blue}{\bullet}}="a";
(-7,0)*+{\textcolor{blue}{\bullet}}="b";
(7,0)*+{\textcolor{blue}{\bullet}}="c";
(14,-7)*+{\textcolor{blue}{\bullet}}="d";
(-35,-14)*+{\rho};
(-28,-14)="-4";
(-14,-14)*+{\bullet}="-2";
(-7,-7)*+{\bullet}="-1";
(7,-7)*+{\bullet}="1";
(14,-14)*+{\bullet}="2";
(28,-14)="4";
{\ar@{~} "-14";"-12"};
{\ar^\alpha "-12";"-11"};
{\ar@{~} "-11";"11"};
{\ar@/^.3pc/@{-} (-30,-15);(-30,-6)};
{\ar@/_.3pc/@{-} (-9,-15);(-9,-6)};
{\ar@{~} "-4";"-2"};
{\ar "-1";"-2"};
{\ar@{~} "-1";"1"};
{\ar_\beta "1";"2"};
{\ar@{~} "2";"4"};
{\textcolor{blue}{\ar_\alpha "a";"b"}};
{\textcolor{blue}{\ar@{~} "b";"c"}};
{\textcolor{blue}{\ar^\beta "c";"d"}};
\endxy
\]
	\caption{The swinging arms condition, as in Definition~\ref{definition: dance and attract} (ii).}
	\label{fig: swinging arms}
\end{figure}
\end{definition}

\begin{remark}
If~$\xi$ is not reduced to a vertex, then the conditions~$\alpha^{-1}\rho'$ (resp.~$\rho'\beta$) are strings follows from the fact that~$\alpha^{-1}\xi$ (resp.~$\xi\beta$) is a substring of~$\rho$. If~$\xi$ is reduced to a vertex, then this condition is non-empty since it requires~$\alpha\beta \notin I$.
\end{remark}

We extend the relation of dancing or attracting to the set~$\strings^\pm_{\ge -1}(\bar Q) \eqdef \strings^\pm(\bar Q) \sqcup \set{-v}{v \in Q_0}$ of almost positive strings as follows.

\begin{definition}
\label{def:negativeSimpleDance}
Let $\sigma$ be a string in $\strings^\pm(\bar Q)$ and let~$v$ be a vertex of~$\bar Q$.
We say that the negative simple string $-v$ reaches for the string $\sigma$ if~$v$ is a vertex of~$\sigma$.
On the other hand, no almost positive string reaches for the negative simple string~$-v$.
\end{definition}

\begin{notation}
\label{not:explanationNegativeSimpleDance}
In the proposition below, we will make use of the following convention.
If~$P[1]$ is a shifted projective, then 
\begin{itemize}
\item $\Hom{A} (M, \tau P[1]) = \Hom{A} (P,M)$ for any representation or shifted projective~$M$,
\item $\Hom{A} (P[1], \tau M) = 0$ for any representation~$M$.
\end{itemize}
These choices are motivated by the isomorphism
\[
\Hom{A} (M, \tau N) \cong \Hom{K^b(\operatorname{proj} A)}(P_N, P_M[1]),
\]
where~$P_M$ and~$P_N$ are the minimal projective presentations of the representations~$M$ and~$N$.
\end{notation}

\begin{remark}
While writing up a first version of the paper, we noticed that a result similar to the following proposition can be found in~\cite{EiseleJanssensRaedschelders}. 
Since our notation and statement differ from those of~\cite{EiseleJanssensRaedschelders}, we nonetheless include a proof. We also note that, in the proof below, we work directly with modules over~$A$, while the proof in~\cite{EiseleJanssensRaedschelders} uses two-term complexes of projectives.
\end{remark}

\begin{proposition}
\label{proposition: attract or dance and tau compatibility}
Let~$\rho$ and~$\rho'$ be two almost positive strings for the string algebra~$A$ and let~$M(\rho)$ and~$M(\rho')$ be the associated string modules.
Then~$\Hom{A} \big( M(\rho), \tau M(\rho') \big) \ne 0$ if and only if~$\rho$ attracts~$\rho'$ or reaches for~$\rho'$.
\end{proposition}

\begin{proof}
The case where $\rho$ or~$\rho'$ is a negative simple string directly follows from Definition~\ref{def:negativeSimpleDance} and Notation~\ref{not:explanationNegativeSimpleDance}.
We therefore assume that~$\rho$ and~$\rho'$ are positive strings.
We let~$\rho''$ be the string corresponding to~$\tau M(\rho')$ obtained from~$\rho'$ by the rules of Theorem~\ref{theorem: tau for strings}.

Assume first that~$\rho$ attracts~$\rho'$ (Figure~\ref{fig:rho attracts rho'}), and let~$\alpha$ be as in Definition~\ref{definition: dance and attract}.
Since~$\rho'$ does not start in a deep,~$\rho''$ starts at some substring~$\alpha_1\cdots\alpha_r$, where~$\alpha_r\cdots\alpha_1\alpha^{-1}$ is the cohook at the start of~$\rho'$.
Let~$0\leq k\leq r$ be maximal such that~$\alpha_k\cdots\alpha_1$ is a substring (possibly a vertex if $k=0$) of~$\rho$ ending in~$u$, where~$u$ is as in Definition~\ref{definition: dance and attract}.
By definition of a cohook,~$\alpha_k\cdots\alpha_1$ has to be on top of~$\rho$, and is at the bottom of~$\rho''$.
By Proposition~\ref{prop:morphismsStringModules}, this implies the existence of a non-zero morphism from~$M(\rho)$ to~$M(\rho'') = \tau M(\rho')$.

\begin{figure}[h]
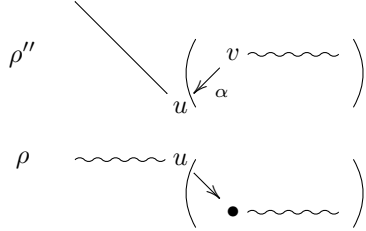

\[
\xy
(7,14)="peak";
(0,7)*+{\rho''}="10";
(28,7)*+{v}="14";
(42,7)="16";
(21,0)*+{u}="3";
(0,-7)*+{\rho}="-0";
(7,-7)="-1";
(21,-7)*+{u}="-3";
(28,-14)*+{\bullet}="-14";
(42,-14)="-16";
{\ar@{~} "-1";"-3"};
{\ar "-3";"-14"};
{\ar@{~} "-14";"-16"};
{\ar@{~} "14";"16"};
{\ar^{\alpha} "14";"3"};
{\ar@{-} "peak";"3"};
{\ar@/_.3pc/@{-} (23,9);(23,0)};
{\ar@/^.3pc/@{-} (44,9);(44,0)};
{\ar@/_.3pc/@{-} (23,-7);(23,-16)};
{\ar@/^.3pc/@{-} (44,-7);(44,-16)};
\endxy
\]
    \caption{If the string~$\rho$ attracts~$\rho'$, then~$\rho''$ is of the form given above.}
    \vspace{-.5cm}
    \label{fig:rho attracts rho'}
\end{figure}

Assume now that~$\rho$ reaches for~$\rho'$, and let~$\xi$ be as in Definition~\ref{definition: dance and attract}.
We distinguish four cases:

\begin{enumerate}[(a)]
\item If~$\rho'=\xi$ (see Figure~\ref{fig:dance case a}). Then~$\rho'$ does not start and does not end in a deep so that ${\rho''= \alpha_r\cdots\alpha_1\alpha^{-1}\rho'\beta\beta_1^{-1}\cdots\beta_s^{-1}}$, where~$\alpha_r\cdots\alpha_1\alpha^{-1}$ and~$\beta\beta_1^{-1}\cdots\beta_s^{-1}$ are cohooks for~$\rho'$.
 By assumption, ~$\alpha^{-1}\rho'\beta$ is a substring of~$\rho$. Let~$0\leq i\leq r$ and~$0\leq j\leq s$ be maximal such that~$\alpha_i\cdots\alpha_1\alpha^{-1}\rho'\beta\beta_1^{-1}\cdots\beta_j^{-1}~$ is a substring of~$\rho$.
 This substring is on top of~$\rho$ and at the bottom of~$\rho''$, which shows that~$\Hom{A} \big( M(\rho), \tau M(\rho') \big) \ne 0$ by Proposition~\ref{prop:morphismsStringModules}.

\begin{figure}[h]
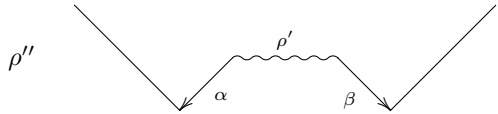

 	\capstart
\[
\xy
(7,14)="lpeak";
(63,14)="rpeak";
(0,7)*+{\rho''}="10";
(28,7)="14";
(42,7)="16";
(21,0)="3";
(49,0)="7";
{\ar@{~}^{\rho'} "14";"16"};
{\ar^{\alpha} "14";"3"};
{\ar@{-} "lpeak";"3"};
{\ar_{\beta} "16";"7"};
{\ar@{-} "rpeak";"7"};
\endxy
\]
    \caption{The string~$\rho''$, when~$\rho$ reaches for~$\rho'$ and~$\rho'=\xi$ as in case (a).}
    \vspace{-.5cm}
    \label{fig:dance case a}
\end{figure}

\item If there are arrows~$\alpha, \beta$ such that~$\alpha\xi\beta^{-1}$ is a substring of~$\rho'$ (see Figure~\ref{fig:dance case b}).
 Then~$\rho''$ contains~$(\alpha)\xi(\beta^{-1})$ where the brackets have to be understood as follows:~$\rho''$ either contains~$\xi\beta^{-1}$ or ends in~$\xi$ and either contains~$\alpha\xi$ or starts at~$\xi$.
 In any case,~$\xi$ is at the bottom of~$\rho''$. Since~$\xi$ is at the top of~$\rho$ by assumption, we have~$\Hom{A} \big( M(\rho), \tau M(\rho') \big)\neq 0$ by Proposition~\ref{prop:morphismsStringModules}.

\begin{figure}[h]
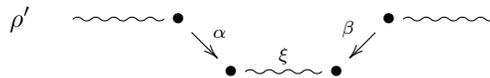

 	\capstart
\[
\xy
(-35,14)*+{\rho'};
(-28,14)="-14";
(-14,14)*+{\bullet}="-12";
(-7,7)*+{\bullet}="-11";
(7,7)*+{\bullet}="11";
(14,14)*+{\bullet}="12";
(28,14)="14";
{\ar@{~} "-14";"-12"};
{\ar^{\alpha} "-12";"-11"};
{\ar@{~}^{\xi} "-11";"11"};
{\ar_{\beta} "12";"11"};
{\ar@{~} "12";"14"};
\endxy
\]
    \caption{When~$\rho$ reaches for~$\rho'$ and~$\rho'$ contains a substring~$\alpha\xi\beta^{-1}$ as in case (b).}
    \vspace{-.5cm}
    \label{fig:dance case b}
\end{figure}

\item If~$\rho'$ is of the form~$\xi\beta^{-1}\sigma$ (for some~$\beta\in Q_1$ and some substring~$\sigma$) and satisfies the swinging arms condition of Definition~\ref{definition: dance and attract}\,(ii) (see Figure~\ref{fig:dance case c}).
 Then the beginning of~$\rho''$ is~$c\xi(\beta^{-1})$, where~$c$ is a cohook for~$\rho'$. Note that the swinging arms condition ensures that the cohook~$c$ ends in~$\alpha^{-1}$.
 One concludes, similarly as in case (a), that there is some substring on top of~$\rho$ which is at the bottom of~$\rho''$, so that~$\Hom{A} \big( M(\rho), \tau M(\rho') \big) \ne 0$ by Proposition~\ref{prop:morphismsStringModules}.
 
 \begin{figure}[h]
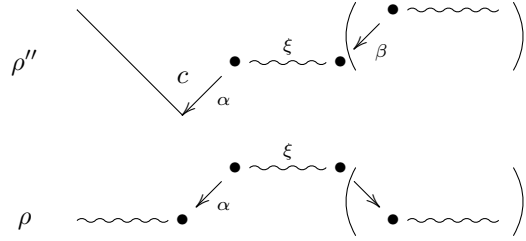

 	\capstart
\[
\xy
(-35,7)*+{\rho''};
(-28,14)="peak";
(-14,5)*+{c};
(-14,0)="-12";
(-7,7)*+{\bullet}="-11";
(7,7)*+{\bullet}="11";
(14,14)*+{\bullet}="12";
(28,14)="14";
(-35,-14)*+{\rho};
(-28,-14)="-4";
(-14,-14)*+{\bullet}="-2";
(-7,-7)*+{\bullet}="-1";
(7,-7)*+{\bullet}="1";
(14,-14)*+{\bullet}="2";
(28,-14)="4";
{\ar@{-} "peak";"-12"};
{\ar^{\alpha} "-11";"-12"};
{\ar@{~}^\xi "-11";"11"};
{\ar^{\beta} "12";"11"};
{\ar@{~} "12";"14"};
{\ar@/^.3pc/@{-} (30,15);(30,6)};
{\ar@/_.3pc/@{-} (9,15);(9,6)};
{\ar@/_.3pc/@{-} (30,-16);(30,-7)};
{\ar@/^.3pc/@{-} (9,-16);(9,-7)};
{\ar@{~} "-4";"-2"};
{\ar^{\alpha} "-1";"-2"};
{\ar@{~}^\xi "-1";"1"};
{\ar "1";"2"};
{\ar@{~} "2";"4"};
\endxy
\]
    \caption{When~$\rho$ reaches for~$\rho'$ and~$\rho'=\xi\beta^{-1}\sigma$ as in case (c).}
    \vspace{-.5cm}
    \label{fig:dance case c}
\end{figure}

\item If~$\rho'$ is of the form~$\sigma\alpha\xi$, then~$(\rho')^{-1}$ is as in case (c), so that~$\Hom{A} \big( M(\rho), \tau M(\rho') \big) \ne 0$.
\end{enumerate}
This concludes the proof that~$\Hom{A} \big( M(\rho), \tau M(\rho') \big) \ne 0$ when~$\rho$ attracts~$\rho'$ or reaches for~$\rho'$.

\bigskip
Conversely, assume that~$\Hom{A} \big( M(\rho), \tau M(\rho') \big) \ne 0$.
Then there is some substring~$\nu$ on top of~$\rho$ and at the bottom of~$\rho''$.
Let~$\xi$ be the longest substring common to~$\nu$ and~$\rho'$.
We again distinguish three cases:

\begin{enumerate}[(a)]
\item If~$\xi$ is empty (see Figure~\ref{fig:not tau compatible case a}).
The substring~$\nu$ has to be contained in some cohook for~$\rho'$.
We may assume that it is a left cohook.
The string~$\rho''$ is of the form~$\rho''= \alpha_r\cdots\alpha_1(\alpha^{-1}\sigma)$ where~$\alpha_r\cdots\alpha_1\alpha^{-1}$ is a left cohook for~$\rho'$.
(We note that removing a right hook to~$\rho'_c$ might remove~$\alpha^{-1}$ from~$\rho''$, which explains the brackets).
In particular,~$\rho'$ does not end in a deep.
Moreover,~$\nu$ does not intersect~$\rho'$ so that it is of the form~$\alpha_i\cdots\alpha_1$, for some~$0\leq i \leq r$.
Here, we use the following convention: let~$v$ be the source of~$\alpha$ and~$u$ be its target; then~$i=0$ means that~$\nu = u$.
Since~$\rho'$ starts at~$v$ and~$\nu$ is on top of~$\rho$, we deduce that~$\rho$ attracts~$\rho'$.

\begin{figure}[h]
 	\capstart
\[
\xy
(0,0)*+{\rho''};
(7,7)="1";
(11.8,2.2)*+{w}="2";
(21,-7)*+{u}="3";
(28,0)*+{v}="4";
(42,0)="6";
{\ar@{-} "1";"2"};
{\ar@{-}_{\nu} "2";"3"};
{\ar^{\alpha} "4";"3"};
{\ar@{~} "4";"6"};
{\ar@/^.3pc/@{-} (23,-9);(23,2)};
{\ar@/_.3pc/@{-} (43,-9);(43,2)};
(32,7)*+{v\in \rho'\ssm\nu};
(55,0)*+{\rho};
(63,-4)="11";
(70,-4)="12";
(77,2)*+{w}="13";
(88,-7)*+{u}="14";
(95,-14)*+{\bullet}="15";
(104,-14)="16";
{\ar@{~} "11";"12"};
{\ar "13";"12"};
{\ar@{-}^{\nu} "13";"14"};
{\ar "14";"15"};
{\ar@{~} "15";"16"};
{\ar@/^.3pc/@{-} (62,-6);(62,2)};
{\ar@/_.3pc/@{-} (75,-6);(75,2)};
{\ar@/_.3pc/@{-} (90,-7);(90,-16)};
{\ar@/^.3pc/@{-} (105,-7);(105,-16)};
\endxy
\]
    \caption{When~$\Hom{A} \big( M(\rho), \tau M(\rho') \big) \ne 0$, case (a).}
    \label{fig:not tau compatible case a}
\end{figure}

\item If~$\xi=\nu$. By Lemma~\ref{lem:xiEqualsNu} below and the fact that~$\xi$ is on top of~$\rho$, we have that~$\rho$ reaches for~$\rho'$. 

\item Assume now that~$\xi$ is a non-empty strict substring of~$\nu$. Since~$\xi$ is not all of~$\nu$, at least one cohook was added to~$\rho'$ when forming~$\rho''$. We distinguish two cases:

\begin{enumerate}[(c1)]
\item If~$\nu$ intersects exactly one cohook that was added to~$\rho'$ (see Figure~\ref{fig: not tau compatible case c1}).
By symmetry, we may assume that it is a left cohook.
In that case,~$\rho'$ does not start in a deep and~$\rho''$ is of the form~$\rho''=\alpha_r\cdots\alpha_1\alpha^{-1}\xi(\beta^{-1}\sigma)$ for some substring~$\sigma$, some arrow~$\beta$ and some cohook~$\alpha_r\cdots\alpha_1\alpha^{-1}$.
Moreover, there is some~$0\leq i\leq r$ such that~$\nu=\alpha_i\cdots\alpha_1\alpha^{-1}\xi$.
The proof of Lemma~\ref{lem:xiEqualsNu} shows that~$\rho'$ is of the form~$\xi\beta^{-1}\sigma'$.
Since~$\nu$ is on top of~$\rho$, we have~$\rho=\rho_1\alpha^{-1}\xi(\beta\rho_2)$ for some substrings~$\rho_1,\rho_2$ and some arrow~$\beta$.
This shows that~$\rho$ reaches for~$\rho'$ (we note that the swinging arms condition is satisfied).

\begin{figure}[h]
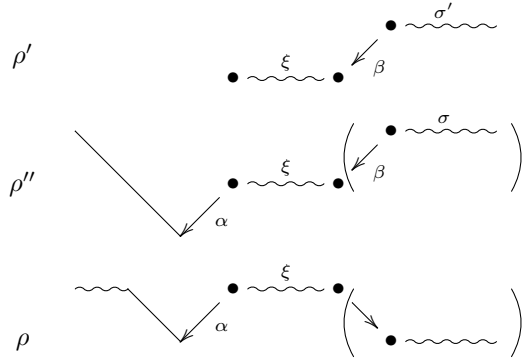

 	\capstart
\[
\xy
(-35,7)*+{\rho''};
(-28,14)="peak";
(-14,0)="-12";
(-7,7)*+{\bullet}="-11";
(7,7)*+{\bullet}="11";
(14,14)*+{\bullet}="12";
(28,14)="14";
(-35,-14)*+{\rho};
(-28,-7)="-4";
(-21,-7)="-3";
(-14,-14)="-2";
(-7,-7)*+{\bullet}="-1";
(7,-7)*+{\bullet}="1";
(14,-14)*+{\bullet}="2";
(28,-14)="4";
{\ar@{-} "peak";"-12"};
{\ar^{\alpha} "-11";"-12"};
{\ar@{~}^\xi "-11";"11"};
{\ar^{\beta} "12";"11"};
{\ar@{~}^\sigma "12";"14"};
{\ar@/^.3pc/@{-} (30,15);(30,6)};
{\ar@/_.3pc/@{-} (9,15);(9,6)};
{\ar@/_.3pc/@{-} (30,-16);(30,-7)};
{\ar@/^.3pc/@{-} (9,-16);(9,-7)};
{\ar@{~} "-4";"-3"};
{\ar@{-} "-3";"-2"};
{\ar^{\alpha} "-1";"-2"};
{\ar@{~}^\xi "-1";"1"};
{\ar "1";"2"};
{\ar@{~} "2";"4"};
(-35,24)*+{\rho'};
(-7,21)*+{\bullet}="-21";
(7,21)*+{\bullet}="21";
(14,28)*+{\bullet}="22";
(28,28)="24";
{\ar@{~}^\xi "-21";"21"};
{\ar^{\beta} "22";"21"};
{\ar@{~}^{\sigma'} "22";"24"};
\endxy
\]
    \caption{Case (c1): the string~$\nu$ intersects exactly one cohook.}
    \label{fig: not tau compatible case c1}
    \vspace{-.3cm}
\end{figure}

\item Assume that~$\nu$ intersects two cohooks that where added to~$\rho'$ (see Figure~\ref{fig: not tau compatible case c2}).
In that case,~$\rho''$ is of the form~$\rho''=\alpha_r\cdots\alpha_1\alpha^{-1}\xi\beta\beta_1^{-1}\cdots\beta_s^{-1}$ where~$\alpha_r\cdots\alpha_1\alpha^{-1}$ and~$\beta\beta_1^{-1}\cdots\beta_s^{-1}$ are cohooks for~$\xi$.
Then~$\rho'=\xi$ does not start nor end in a deep.
Moreover,~$\nu$ is at the bottom of~$\rho''$ so that it contains~$\alpha^{-1}\xi\beta$.
The string~$\rho$ also contains~$\alpha^{-1}\xi\beta$ and therefore reaches for~$\xi=\rho'$.

\begin{figure}[h]
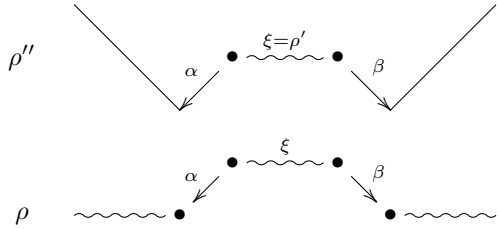

 	\capstart
\[
\xy
(-35,7)*+{\rho''};
(-28,14)="lpeak";
(-14,0)="-12";
(-7,7)*+{\bullet}="-11";
(7,7)*+{\bullet}="11";
(14,0)="12";
(28,14)="rpeak";
(-35,-14)*+{\rho};
(-28,-14)="-4";
(-14,-14)*+{\bullet}="-2";
(-7,-7)*+{\bullet}="-1";
(7,-7)*+{\bullet}="1";
(14,-14)*+{\bullet}="2";
(28,-14)="4";
{\ar@{-} "lpeak";"-12"};
{\ar_{\alpha} "-11";"-12"};
{\ar@{~}^{\xi=\rho'} "-11";"11"};
{\ar^{\beta} "11";"12"};
{\ar@{-} "12";"rpeak"};
{\ar@{~} "-4";"-2"};
{\ar_{\alpha} "-1";"-2"};
{\ar@{~}^\xi "-1";"1"};
{\ar^\beta "1";"2"};
{\ar@{~} "2";"4"};
\endxy
\]
    \caption{Case (c2): the string~$\nu$ intersects two cohooks.}
    \label{fig: not tau compatible case c2}
\end{figure}
\vspace*{-1.2cm}
\qedhere
\end{enumerate}
\end{enumerate}
\end{proof}

\begin{lemma}\label{lem:xiEqualsNu}
Let~$\rho'$ be a string for~$A$ and~$\rho''$ be the string obtained when computing~$\tau M(\rho')$ as in Theorem~\ref{theorem: tau for strings}.
Assume that~$\xi$ is a substring common to~$\rho'$ and~$\rho''$, and that~$\xi$ is at the bottom of~$\rho''$.
Then~$\rho'$ has a substring of the form~$\alpha\xi\beta^{-1}$, for some arrows~$\alpha, \beta$.
\end{lemma}

\begin{figure}[h]
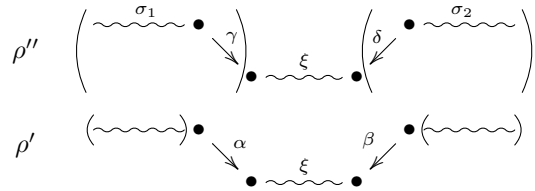

 	\capstart
\[
\xy
(-37,10.5)*+{\rho''};
(-28,14)="-14";
(-14,14)*+{\bullet}="-12";
(-7,7)*+{\bullet}="-11";
(7,7)*+{\bullet}="11";
(14,14)*+{\bullet}="12";
(28,14)="14";
{\ar@{~}^{\sigma_1} "-14";"-12"};
{\ar^{\!\!\gamma} "-12";"-11"};
{\ar@{~}^{\xi} "-11";"11"};
{\ar_{\delta\!\!} "12";"11"};
{\ar@{~}^{\sigma_2} "12";"14"};
{\ar@/^.3pc/@{-} (-29,5);(-29,16)};
{\ar@/_.3pc/@{-} (-9,5);(-9,16)};
{\ar@/^.3pc/@{-} (9,5);(9,16)};
{\ar@/_.3pc/@{-} (29,5);(29,16)};
(-37,-2)*+{\rho'};
(-28,0)="-4";
(-14,0)*+{\bullet}="-2";
(-7,-7)*+{\bullet}="-1";
(7,-7)*+{\bullet}="1";
(14,0)*+{\bullet}="2";
(28,0)="4";
{\ar@{~} "-4";"-2"};
{\ar^{\alpha} "-2";"-1"};
{\ar@{~}^{\xi} "-1";"1"};
{\ar_{\beta} "2";"1"};
{\ar@{~} "2";"4"};
{\ar@/^.2pc/@{-} (-28,-2);(-28,2)};
{\ar@/_.2pc/@{-} (-16.5,-2);(-16.5,2)};
{\ar@/_.2pc/@{-} (28,-2);(28,2)};
{\ar@/^.2pc/@{-} (16.5,-2);(16.5,2)};
\endxy
\]
    \caption{A substring~$\xi$ at the bottom of~$\rho''$ which is also a substring of~$\rho$.}
    \label{fig: xi equals nu}
    \vspace{-.3cm}
\end{figure}

\begin{proof}
 Since~$\xi$ is at the bottom of~$\rho''$, we have~$\rho''=(\sigma_1\gamma)\xi(\delta^{-1}\sigma_2)$, for some substrings~$\sigma_1,\sigma_2$ and arrows~$\gamma,\delta$.
 If the substring~$\gamma$ (resp.~$\delta^{-1}$) belongs to the string~$\rho''$, then it already belongs to the string~$\rho'$.
 Indeed, adding a cohook to~$\rho'$ cannot create~$\gamma$ nor~$\delta^{-1}$ as in the string~$\rho''$.
 If~$\gamma$ does not appear in the string~$\rho''$, then~$\rho''$ starts by~$\xi$.
 Since~$\xi$ also appears in~$\rho'$, the string~$\rho'$ is of the form~$h\xi\sigma_3$ for some substring~$\sigma_3$ and some hook~$h$.
 In particular,~$\rho'$ contain a substring of the form~$\alpha\xi$.
 In case~$\delta^{-1}$ does not appear in~$\rho''$, a similar argument shows that some~$\xi\beta^{-1}$ is a substring of~$\rho'$.
\end{proof}

\begin{corollary}
\label{coro: stautilt complex from strings}
The support $\tau$-tilting complex of~$\bar Q$ is isomorphic to the clique complex of the graph whose vertices are the unoriented strings of $\tau$-rigid modules of~$\bar Q$ and whose edges link two strings which do not attract nor reach for each other.
\end{corollary}

\begin{remark}
The criterion of Proposition \ref{proposition: attract or dance and tau compatibility} for $\tau$-compatibility between string modules can be compared with a similar criterion obtained recently in \cite{CanakciPauksztelloSchroll} for the existence of extensions between two string modules.
The similarity between the two criteria is explained by the Auslander--Reiten formula asserting that $\Ext{}^1(N,M)$ is a quotient of the space $\Hom{}(M,\tau N)$ (see \cite[Thm.~IV.2.13]{AssemSimsonSkowronski}).
\end{remark}

\begin{example}
Two examples of support~$\tau$-tilting complexes of gentle algebras, illustrating Corollary~\ref{coro: stautilt complex from strings}, are shown in Figure~\ref{fig:exmtTC}.
The support~$\tau$-tilting complexes of the gentle algebras for all connected gentle bound quivers with~$2$ vertices are illustrated on \fref{fig:all2}.

\begin{figure}[h]
	\capstart
	\centerline{\includegraphics[scale=.5]{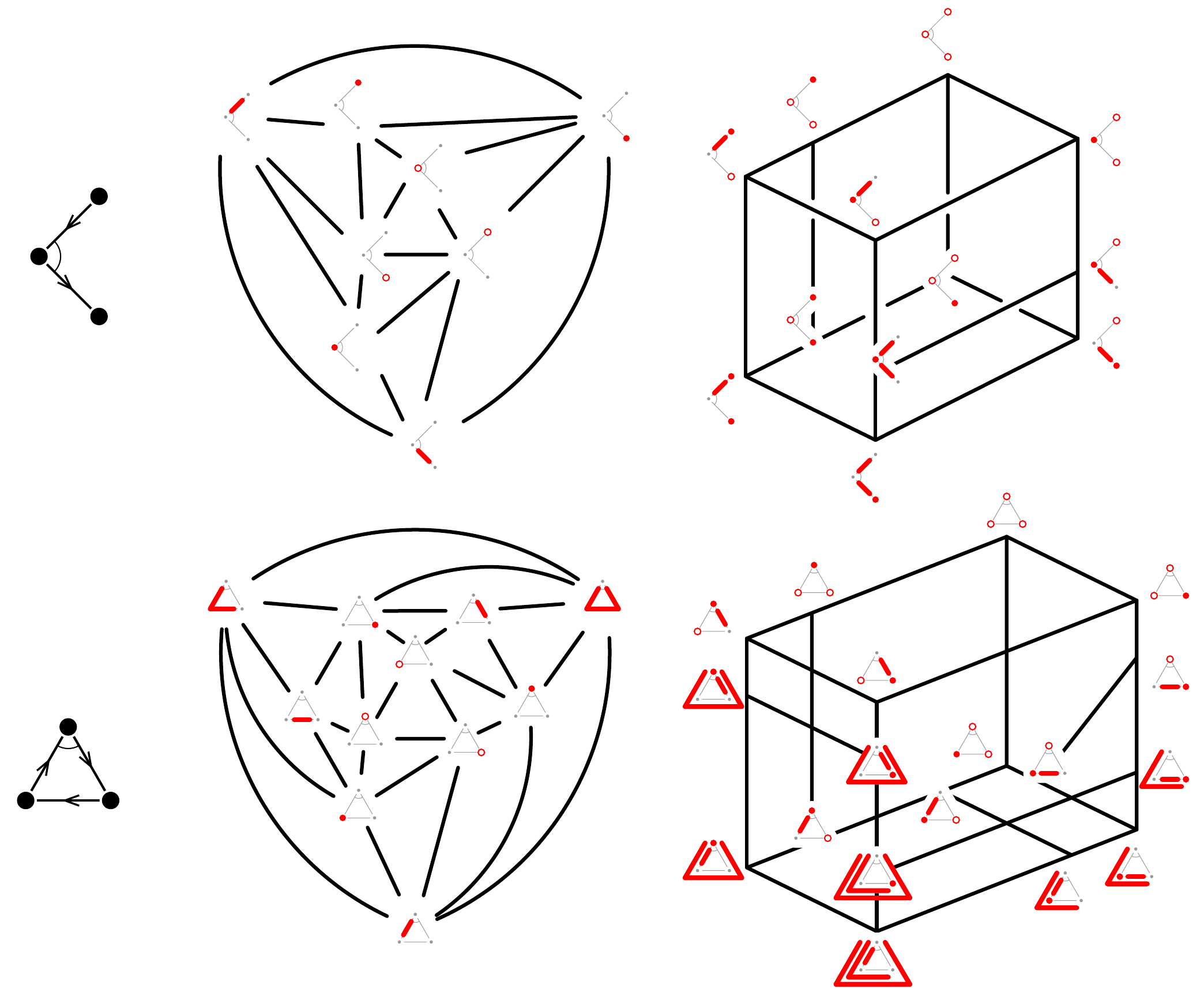}}
	\caption{The support $\tau$-tilting complex~$\tTC$ (left) and its dual graph (right). The string module~$M(\rho)$ is represented by the string~$\rho$ of~$\bar Q$ and the preprojective module~$P(v)[1]$ is represented by a hollow vertex at~$v$.}
	\label{fig:exmtTC}
\end{figure}
\end{example}


\clearpage
\part{The non-kissing complex}
\label{part:combinatorics}

This section focuses on the non-kissing complex of a gentle bound quiver, defined as the clique complex of a non-kissing relation among walks in a blossoming version of the quiver~(\ref{sec:blossomingQuivers}).
We prove combinatorially that this complex is a pseudomanifold~(\ref{sec:nonKissingComplex}), introducing along the way the definitions of distinguished walks, arrows and substrings~(\ref{subsec:distinguished}) and the operation of flip~(\ref{subsec:flips}) that will be fundamental tools all throughout the paper.
We then prove the motivating result of this paper, namely that the support $\tau$-tilting complex is isomorphic to the non-kissing complex for any gentle bound quiver~(\ref{sec:nkcvsttc}).


\section{Blossoming quivers}
\label{sec:blossomingQuivers}

\subsection{Blossoming quiver of a gentle bound quiver}
\label{subsec:blossomingQuiver}

Let~${\bar Q = (Q,I)}$ be a gentle bound quiver.
Denote by~$\indeg(v) \eqdef |t^{-1}(v)|$ and~$\outdeg(v) \eqdef |s^{-1}(v)|$ the \defn{in-} and \defn{out-degrees} of a vertex~${v \in Q_0}$ and by~$\deg(v) \eqdef \indeg(v) + \outdeg(v)$ its \defn{degree}.
We start by completing~$\bar Q$ such that each vertex gets four neighbors.

\begin{definition}
The \defn{blossoming quiver} of a gentle bound quiver~${\bar Q = (Q,I)}$ is the gentle bound quiver~$\bar Q\blossom \eqdef (Q\blossom, I\blossom)$ obtained by adding at each vertex~$v \in Q_0$:
\begin{itemize}
\item $2-\indeg(v)$ incoming arrows and~$2-\outdeg(v)$ outgoing arrows,
\item relations such that vertex~$v$ fulfills the gentle bound quiver conditions of Definition~\ref{def:gentleQuiver}.
\end{itemize}
Note that each vertex of~$Q_0$ has precisely two incoming and two outgoing arrows in~$Q\blossom$.
The vertices of~$Q\blossom_0 \ssm Q_0$ are called \defn{blossom vertices}, and the arrows~$Q\blossom_1 \ssm Q_1$ are called \defn{blossom arrows}.
\end{definition}

\begin{example}
A gentle bound quiver~$\bar Q$ and its blossoming gentle bound quiver~$\bar Q\blossom$ are represented in \fref{fig:exmBlossomingQuiver}.
\vspace*{-.4cm}

\begin{figure}[h]
	\capstart
	\centerline{\includegraphics[scale=.7]{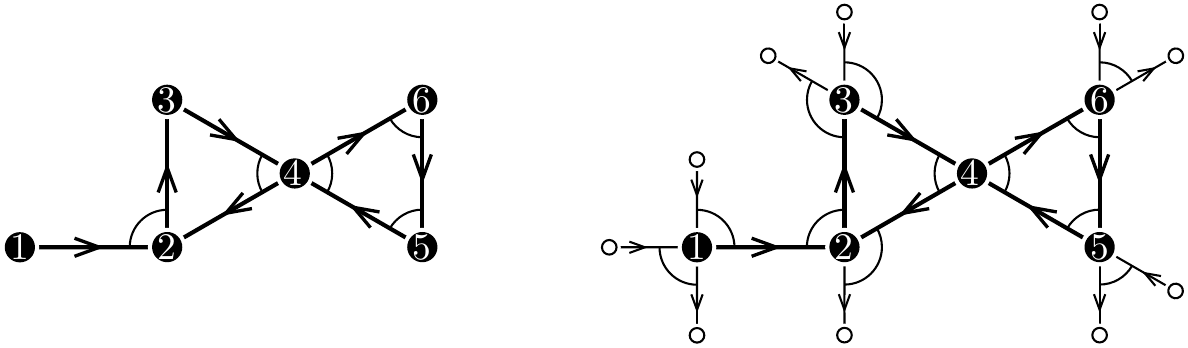}}
	\caption{A gentle bound quiver~$\bar Q$~(left) and its blossoming gentle bound quiver~$\bar Q\blossom$~(right). Blossom vertices are white and blossom edges are thin.}
	\label{fig:exmBlossomingQuiver}
\end{figure}
\end{example}

The following lemma is immediate and left to the reader.

\begin{lemma}
The number of vertices and arrows of the blossoming quiver~$Q\blossom$ are given by
\[
|Q\blossom_0| = |Q_0| + 2\Delta = 5|Q_0|-2|Q_1|
\qquad\text{while}\qquad
|Q\blossom_1| = |Q_1| + 2\Delta = 4|Q_0|-|Q_1|,
\]
where
\[
\Delta \eqdef 2|Q_0| - |Q_1| = \frac{1}{2} \sum_{v \in Q_0} \big( 4-\deg(v) \big) = \sum_{v \in Q_0} \big( 2-\indeg(v) \big) = \sum_{v \in Q_0} \big( 2-\outdeg(v) \big)
\]
is the number of incoming (or equivalently outgoing) blossom arrows of~$\bar Q\blossom$.
\end{lemma}

%

Note that any arrow~$\alpha \in Q_1$ is incident with~$6$ other arrows of~$Q\blossom_1$.
Therefore, any~$\alpha \in Q_1$ can be continued to a string both at its source and target in two ways (switching orientation or not).
It follows that a string in~$\bar Q\blossom$ is maximal if and only if it joins two blossom vertices of~$\bar Q\blossom$.

\begin{definition}
A \defn{walk} of~$\bar Q$ is a maximal string in~$\bar Q\blossom$ (thus joining two blossom vertices of~$\bar Q\blossom$).
We denote by~$\walks(\bar Q)$ the set of all walks on~$\bar Q$.
Note that~$\walks(\bar Q)$ is finite if and only if~$\bar Q$ has a relation in any (not necessarily oriented) cycle.
As for strings, we often implicitly identify the two inverse walks~$\omega$ and~$\omega^{-1}$, and we let~$\walks^\pm(\bar Q) \eqdef \set{\{\omega, \omega^{-1}\}}{\omega \in \walks(\bar Q)}$ denote the set of undirected~walks.
\end{definition}

\begin{example}
\fref{fig:exmFacet} shows many examples of walks on the gentle quiver of \fref{fig:exmBlossomingQuiver}.
\end{example}

\begin{definition}
\label{def:substrings}
We call \defn{substring} of a walk~$\omega \in \walks(\bar Q)$ the strings of~$\strings(\bar Q)$ which are factors of~$\omega$.
We insist on the fact that we forbid the blossom endpoints of~$\omega$ to belong to a substring of~$\omega$.
We denote by~$\Sigma(\omega)$ the set of substrings of~$\omega$ and by~$\Sigma_\top(\omega)$ and~$\Sigma_\bottom(\omega)$ the sets of top and bottom substrings of~$\omega$ respectively.
We use the same notations for undirected walks (of course, substrings of an undirected walk are undirected).
\end{definition}

\begin{definition}
\label{def:straightBended}
A walk~$\omega$ is \defn{straight} if it has no corner (\ie if~$\omega$ or~$\omega^{-1}$ is a path in~$\bar Q\blossom$), and \defn{bending} otherwise. Let $\straightWalks(\bar Q)$ and~$\bendingWalks(\bar Q)$ denote the sets of straight and bending walks of~$\bar Q$ and let~$\straightWalks^\pm(\bar Q) \eqdef \set{\{\omega, \omega^{-1}\}}{\omega \in \straightWalks(\bar Q)}$ and~$\bendingWalks^\pm(\bar Q) \eqdef \set{\{\omega, \omega^{-1}\}}{\omega \in \bendingWalks(\bar Q)}$.
\end{definition}

\begin{definition}
\label{def: deep walk}
A \defn{peak walk} (resp.~\defn{deep walk}) is a walk that switches orientation only once, at a peak (resp.~deep). For~$v \in Q_0$, we denote by~$v_\peak$ the peak walk with peak at~$v$ and by~$v_\deep$ the deep walk with deep~at~$v$.
\end{definition}

\begin{remark}
\label{rem:reverseBlossomingQuiver}
Observe that blossoming and reversing commute:~$(\reversed{\bar Q})\blossom = \reversed{(\bar Q\blossom)}$.
In particular, the undirected walks in~$\bar Q$ coincide with the undirected walks in~$\reversed{\bar Q}$.
However, bottom and top substrings are reversed: $\Sigma_\bottom(\reversed{\omega}) = \reversed{\big( \Sigma_\top(\omega) \big)}$ and~$\Sigma_\top(\reversed{\omega}) = \reversed{\big( \Sigma_\bottom(\omega) \big)}$.
\end{remark}

\subsection{Dissection and grid bound quivers}
\label{subsec:dissectionGridQuivers}

\begin{figure}[b]
	\capstart
	\centerline{$\vcenter{\hbox{\includegraphics[scale=1.2]{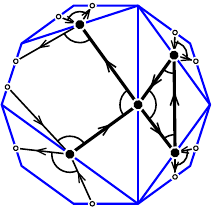}}}$ \hspace*{1.5cm}~$\vcenter{\hbox{\includegraphics[scale=.45]{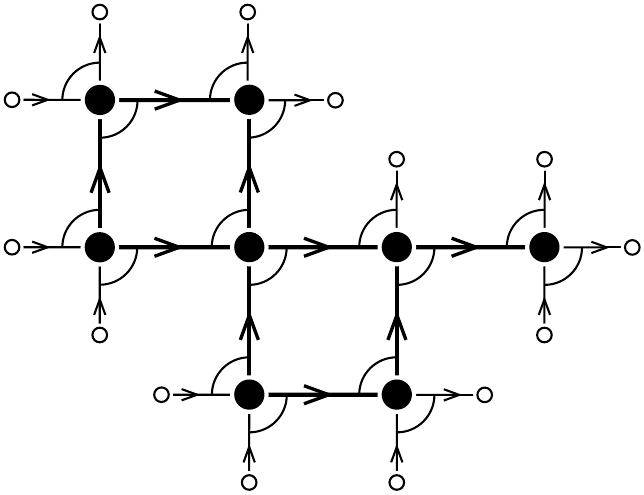}}}$}
	\caption{A dissection bound quiver~$\bar Q(D)$~(left) and a grid bound quiver~$\bar Q(L)$~(right), together with their blossoming quivers.}
	\label{fig:dissectionGridQuivers}
\end{figure}

\enlargethispage{.8cm}
Throughout the paper, we will illustrate our results with two specific families of gentle bound quivers:

\para{Dissection quivers}
The following is motivated by the works of Y.~Baryshnikov~\cite{Baryshnikov}, F.~Chapoton~\cite{Chapoton-quadrangulations}, and A.~Garver and T.~McConville~\cite{GarverMcConville}.
Consider a convex polygon~$P$ and a \defn{dissection}~$D$ of~$P$, \ie a collection of pairwise non-crossing internal diagonals of~$P$.
The \defn{dissection bound quiver} of~$D$ is the gentle bound quiver~$\bar Q(D)$ with a vertex for each diagonal of~$D$, with an arrow connecting any two diagonals of~$D$ that are counterclockwise consecutive edges of a cell of~$D$, and with a relation between any two arrows connecting three counterclockwise consecutive edges of a cell of~$D$.
Note that the blossoming quiver~$\bar Q\blossom(D)$ of~$\bar Q(D)$ is obtained similarly by considering as well the boundary edges of the polygon~$P$ (however, observe that a boundary edge whose endpoints are both incident to a diagonal of~$D$ should be duplicated into two distinct blossom vertices).
See \fref{fig:dissectionGridQuivers}\,(left).
Note that the walks of~$\bar Q(D)$ correspond to the diagonals joining two blossom vertices of~$\bar Q\blossom(D)$ and intersecting a connected set of diagonals of~$D$.
These diagonals are called \defn{accordion diagonals} of~$D$ in~\cite{MannevillePilaud-accordion} (they correspond to accordions in~$D$).

\para{Grid quivers}
The following is motivated by the works of T.~K.~Petersen, P.~Pylyavskyy and D.~Speyer for staircase shapes~\cite{PetersenPylyavskyySpeyer}, of F.~Santos, C.~Stump and V.~Welker on Grassmann associahedra~\cite{SantosStumpWelker}, and of T.~McConville~\cite{McConville} and A.~Garver and T.~McConville~\cite{GarverMcConville-grid} for arbitrary grid bound quivers.
Consider a connected subset~$L$ of the integer grid~$\Z^2$.
The \defn{grid bound quiver} of~$L$ is the gentle bound quiver~$\bar Q(L)$ obtained as the induced subgraph of the grid where edges are oriented in the direction of the increasing coordinates (meaning north and east), and with a relation for each length~$2$ path with a vertical and a horizontal steps.
The blossoming quiver~$\bar Q\blossom(L)$ of~$\bar Q(L)$ is obtained similarly by considering as well the neighbors of~$L$ (however, observe that a vertex of the grid with two or more neighbors in~$L$ should be duplicated into two or more distinct blossom vertices).
See \fref{fig:dissectionGridQuivers}\,(right) for an illustration.
Note that the walks of~$\bar Q(L)$ correspond to the maximal south-east lattice paths in~$\bar Q\blossom(L)$.

\para{Oriented paths}
Some quivers are simultaneously dissection and grid bound quivers. This is in particular the case of any oriented path, which is both the dissection bound quiver of an acyclic triangulation (with no internal triangle) and the grid bound quiver of a ribbon of~$\Z^2$. This is illustrated in \fref{fig:exmBijectionAssociahedron}\,(left).

\begin{figure}[t]
	\capstart
	\centerline{\includegraphics[scale=1.4]{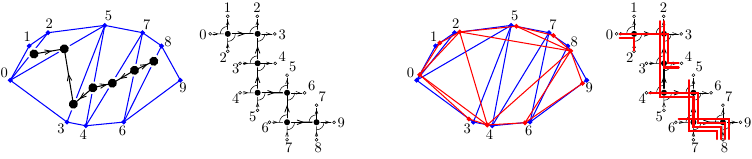}}
	\caption{An oriented path~$\exmAn$ is simultaneously a triangulation and a grid bound quiver (left). A facet of the reduced non-kissing complex~$\RNKC[\exmAn]$ can be seen equivalently as a triangulation or as a collection of non-kissing south-east paths (right).}
	\label{fig:exmBijectionAssociahedron}
\end{figure}


\section{The non-kissing complex}
\label{sec:nonKissingComplex}

\enlargethispage{.2cm}
We now define and study first elementary properties of the non-kissing complex of a quiver.
This section is inspired from previous works on grid bound quivers~\cite{PetersenPylyavskyySpeyer, SantosStumpWelker, McConville, GarverMcConville}.

\subsection{The non-kissing complex}

Let~$\bar Q = (Q,I)$ be a gentle bound quiver and~$\bar Q\blossom = (Q\blossom,I\blossom)$ denote its blossoming quiver.
The following definition is schematically illustrated in \fref{fig:kissingCrossing}\,(left).

\begin{definition}\label{def: kissing}
Let~$\omega,\omega' \in \walks^\pm(\bar Q)$ be two undirected walks on~$\bar Q$.
We say that~$\omega$ \defn{kisses}~$\omega'$ if~$\Sigma_\top(\omega) \cap \Sigma_\bottom(\omega') \ne \varnothing$, \ie if there exist a common substring~$\sigma \in \strings^\pm(\bar Q)$ of~$\omega$ and~$\omega'$ such that the arrows of~$\omega$ incident to~$\sigma$ are both outgoing while the arrows of~$\omega'$ incident to~$\sigma$ are both incoming.
We say that~$\omega$ and~$\omega'$ are \defn{kissing} if~$\omega$ kisses~$\omega'$ or~$\omega'$ kisses~$\omega$ (or both).
\end{definition}

Note that we authorize the situation where~$\sigma$ is reduced to a vertex~$v$, meaning that~$v$ is a peak of~$\omega$ and a deep of~$\omega'$.
Observe also that~$\omega$ can kiss~$\omega'$ several times, that~$\omega$ and~$\omega'$ can mutually kiss, and that~$\omega$ can kiss itself.
We denote by~$\NKWalks(\bar Q)$ the walks on~$\bar Q$ which are not self-kissing and by~$\NKWalks^\pm(\bar Q) \eqdef \set{\{\omega, \omega^{-1}\}}{\omega \in \NKWalks(\bar Q)}$ their undirected version.

\begin{figure}[b]
	\capstart
	\centerline{\includegraphics[scale=1]{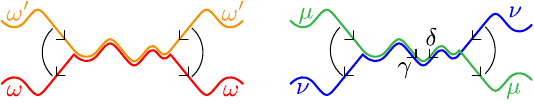}}
	\caption{A schematic representation of two kissing walks (left) and two non-kissing walks~(right) having a common substring: $\omega$~kisses~$\omega'$ (left) while~$\mu$~does not kiss~$\nu$ (right). On the right, we have~$\mu \prec_\gamma \nu$ and~$\nu \prec_\delta \mu$.}
	\label{fig:kissingCrossing}
\end{figure}

\begin{example}
\fref{fig:exmFacet}~(left) provides examples of kissing and non-kissing walks.
\begin{figure}[t]
	\capstart
	\centerline{\includegraphics[scale=.7]{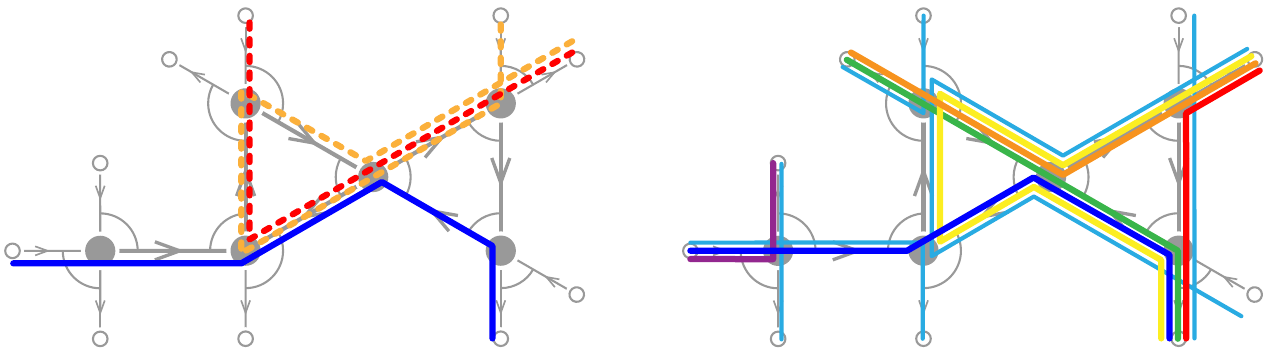}}
	\caption{(Left) Three walks on the bound quiver~$\bar Q$ of \fref{fig:exmBlossomingQuiver}: the dotted red and orange walks are non-kissing, but both kiss the plain blue walk. (Right) A maximal collection~$F$ of pairwise non-kissing walks on~$\bar Q$: the thin light blue walks are straight, the others~are~bending.}
	\label{fig:exmFacet}
\end{figure}
\end{example}

\begin{definition}\label{def: nKc}
The \defn{non-kissing complex} of~$\bar Q$ is the simplicial complex~$\NKC$ whose faces are the collections of pairwise non-kissing walks of~$\NKWalks^\pm(\bar Q)$.
Note that self-kissing walks never appear in~$\NKC$ by definition.
In contrast, no straight walk can kiss another walk by definition, so that they appear in all facets of~$\NKC$.
We thus consider the \defn{reduced non-kissing complex}~$\RNKC$ to be the deletion of all straight walks from~$\NKC$.
\end{definition}

\begin{example}
\fref{fig:exmFacet}\,(right) shows a facet~$F$ of the non-kissing complex~$\NKC[\bar Q]$ for the bound quiver~$\bar Q$ of \fref{fig:exmBlossomingQuiver}.
The reduced non-kissing complexes of all connected gentle bound quivers with~$2$ vertices are illustrated on \fref{fig:all2}.
The reduced non-kissing complexes of two gentle bound quivers with~$3$ vertices are illustrated on \fref{fig:exmNKC}\,(left).
\end{example}

\begin{example}\label{example: Fdeep Fpeak}
For any vertices~$v,w \in Q_0$:
\begin{itemize}
\item the two peak walks~$v_\peak$ and~$w_\peak$ are non-kissing,
\item the two deep walks~$v_\deep$ and~$w_\deep$ are non-kissing,
\item the peak walk~$v_\peak$ kisses the deep walk~$w_\deep$ if and only if there is an oriented path from~$v$ to~$w$ in~$\bar Q$.
\end{itemize}
Therefore, the sets~$F_\peak \eqdef \set{v_\peak}{v \in Q_0}$ and~$F_\deep \eqdef \set{v_\deep}{v \in Q_0}$ are both in~$\RNKC$.
It will follow from Corollary~\ref{coro:pure} that~$F_\peak$ and~$F_\deep$ are in fact both non-kissing facets of~$\RNKC$ that we call the \defn{peak facet} and the \defn{deep facet}.
\end{example}

\begin{example}
The non-kissing complex of the dissection and grid bound quivers introduced in Section~\ref{subsec:dissectionGridQuivers} have been studied in details in the literature:
\begin{itemize}
\item For a dissection~$D$, the non-kissing complex~$\RNKC[\bar Q(D)]$ was studied by A.~Garver and T.~McConville~\cite{GarverMcConville} and T.~Manneville and V.~Pilaud~\cite{MannevillePilaud-accordion}, motivated by preliminary works of Y.~Baryshnikov~\cite{Baryshnikov} and F.~Chapoton~\cite{Chapoton-quadrangulations} on quadrangulations. The faces of~$\RNKC[\bar Q(D)]$ correspond to sets of pairwise non-crossing accordion diagonals~of~$D$.
\item For a subset~$L$ of the integer grid~$\Z^2$, the non-kissing complex~$\RNKC[\bar Q(L)]$ was studied by T.~K.~Petersen, P.~Pylyavskyy and D.~Speyer for staircase shapes~\cite{PetersenPylyavskyySpeyer}, by F.~Santos, C.~Stump and V.~Welker on rectangular shapes~\cite{SantosStumpWelker}, and by T.~McConville~\cite{McConville} and A.~Garver and T.~McConville~\cite{GarverMcConville-grid} on arbitrary grid bound quivers. 
\item The non-kissing complex of an oriented path is a simplicial associahedron. \fref{fig:exmBijectionAssociahedron}\,(right) shows the correspondence between triangulations of an $(n+3)$-gon and facets of the non-kissing complex of an $n$-path. The simplicial associahedron was introduced in early works of D.~Tamari~\cite{Tamari} and J.~Stasheff~\cite{Stasheff}, in connection to associative schemes and topological motivations.
\end{itemize}
\end{example}

\begin{figure}[t]
	\capstart
	\centerline{\includegraphics[scale=.5]{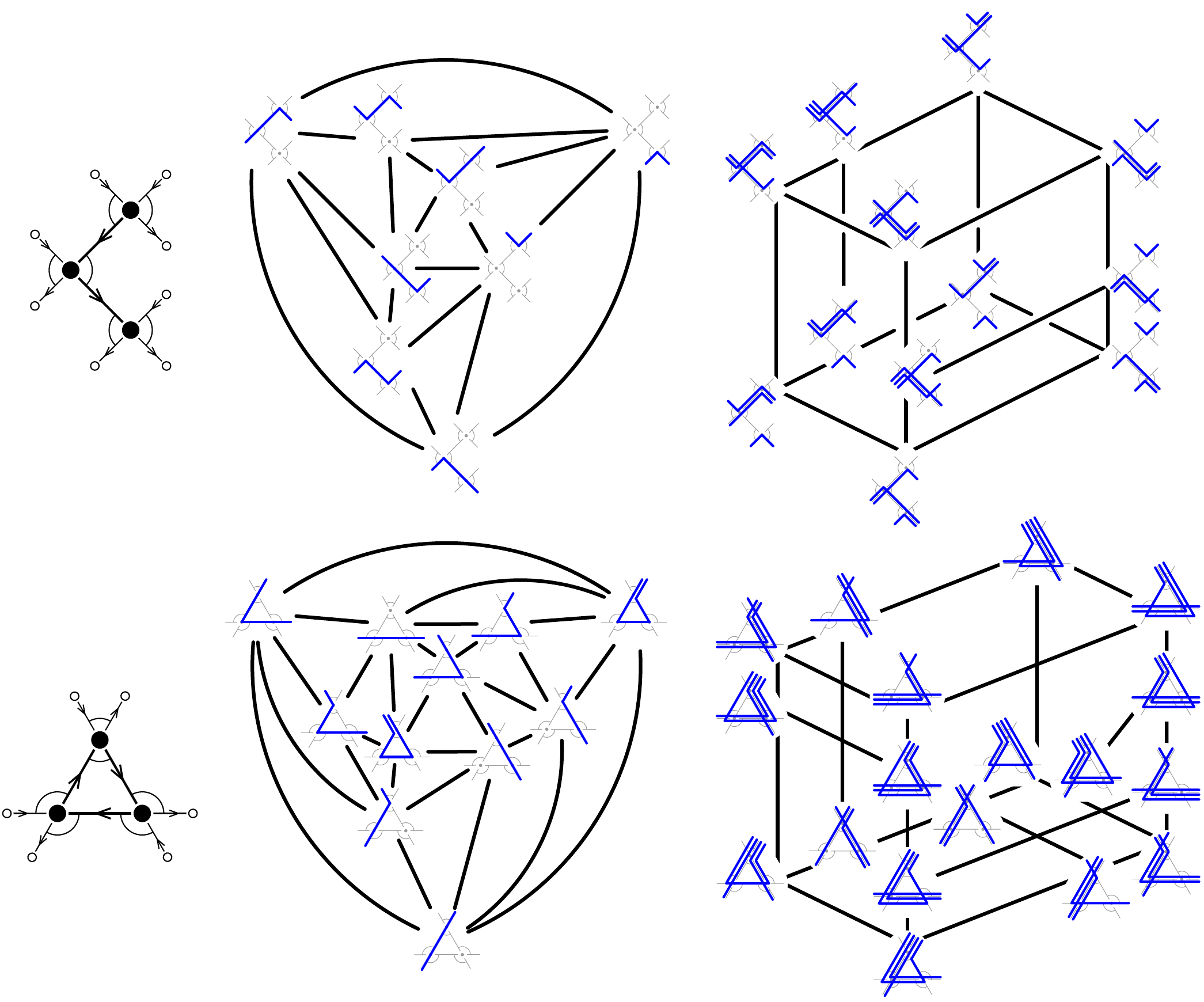}}
	\caption{The reduced non-kissing complex~$\RNKC$ (left) and the non-kissing oriented flip graph~$\NKG$ (right). The graph is oriented from bottom to top: the peak facet~$F_\peak$ appears at the bottom while the deep facet~$F_\deep$ appear on top.}
	\label{fig:exmNKC}
\end{figure}

\begin{remark}
\label{rem:reverseNKC}
For any gentle bound quivers~$\bar Q, \bar Q'$, if~$\bar Q'$ is isomorphic to either~$\bar Q$ or~$\reversed{\bar Q}$, then the non-kissing complexes~$\NKC[\bar Q]$ and~$\NKC[\bar Q']$ are isomorphic simplicial complexes.
\end{remark}

\subsection{Distinguished walks, arrows and substrings}
\label{subsec:distinguished}

Our next task is to show that the non-kissing complex~$\NKC$ is \defn{pure}, \ie that all maximal non-kissing facets of~$\NKC$ have the same number of walks.
We follow the proof of~\cite[Thm~3.2\,(1--2)]{McConville}, but we adapt the arguments to cover our level of generality.
We introduce along the way the convenient notions of distinguished walks, arrows and substrings that will be used throughout the paper.

\begin{definition}
A \defn{marked walk}~$\omega_\star$ is a walk~$\omega = \alpha_1^{\varepsilon_1} \cdots \alpha_\ell^{\varepsilon_\ell} \in \walks(\bar Q)$ with a marked arrow~$\alpha_i^{\varepsilon_i}$.
Note that if~$\omega$ contains several occurrences of~$\alpha_i^\pm$, only one occurrence is marked.
\end{definition}

\begin{definition}
Consider an arrow~$\alpha \in Q\blossom_1$ and any two distinct non-kissing walks~$\mu_\star, \nu_\star$ marked at an occurrence of~$\alpha^\pm$.
Let~$\sigma$ denote their maximal common substring containing that occurrence of~$\alpha$.
Since~$\mu_\star \ne \nu_\star$, their common substring~$\sigma$ is strict, so that~$\mu_\star$ and~$\nu_\star$ split at least at one endpoint of~$\sigma$.
We define the \defn{countercurrent order at~$\alpha$} by~$\mu_\star \prec_\alpha \nu_\star$ when~$\mu_\star$ enters and/or exits~$\sigma$ in the direction pointed by~$\alpha$, while~$\nu_\star$ enters and/or exits~$\sigma$ in the direction opposite to~$\alpha$.
\end{definition}

\begin{example}
On the schematic illustration of \fref{fig:kissingCrossing}\,(right), the walks~$\mu$ and~$\nu$ are non-kissing and for the walks~$\mu_\star, \nu_\star$ marked at~$\gamma$ (resp.~at~$\delta$), we have~$\mu_\star \prec_\gamma \nu_\star$ (resp.~$\nu_\star \prec_\delta \mu_\star$).
\end{example}

\begin{remark}
Note that the countercurrent order~$\prec_\alpha$
\begin{itemize}
\item is well-defined: if both endpoints of~$\sigma$ are in~$Q_0$, then $\mu_\star$ (resp.~$\nu_\star$) enters and exits~$\sigma$ in the same direction since~$\mu_\star$ and~$\nu_\star$ are non-kissing;
\item is independent of the orientation of~$\mu_\star$ and~$\nu_\star$:~${\mu_\star \prec_\alpha \nu_\star \!\iff\! (\mu_\star)^{-1} \prec_\alpha \nu_\star \!\iff\! \mu_\star \prec_\alpha (\nu_\star)^{-1}}$; we thus also consider~$\prec_\alpha$ as a relation on undirected marked walks;
\item depends on which occurrence of~$\alpha$ is marked on~$\mu$ and~$\nu$.
\end{itemize}
\end{remark}


\begin{lemma}
For any face~$F$ of~$\NKC$, the countercurrent order~$\prec_\alpha$ defines a total order on the walks of~$F$ marked at an occurrence of~$\alpha$.
\end{lemma}

\begin{proof}
We prove that~$\prec_\alpha$ is transitive.
Let~$\lambda_\star$, $\mu_\star$, $\nu_\star$ be three pairwise non-kissing walks marked at an occurrence of~$\alpha$ such that~${\lambda_\star \prec_\alpha \mu_\star \prec_\alpha \nu_\star}$. Let~$v$ be a vertex where~$\lambda_\star$ and~$\mu_\star$ split. Then
\begin{itemize}
\item either~$\nu_\star$ contains~$v$, then it must leave~$v$ via the same arrow as~$\mu_\star$ (otherwise~$\nu_\star \prec_\alpha \mu_\star$), so that~$\lambda_\star \prec_\alpha \nu_\star$.
\item or~$\nu_\star$ does not contain~$v$, then~$\nu_\star$ and~$\mu_\star$ split at a vertex~$w$ between~$\alpha$ and~$v$. Since~${\mu_\star \prec_\alpha \nu_\star}$, $\mu$ (and thus~$\lambda$) leaves~$w$ with an arrow in the same direction as~$\alpha$, while~$\nu$ leaves~$w$ with an arrow opposite to the direction of~$\alpha$, so that~$\lambda_\star \prec_\alpha \nu_\star$.
\end{itemize}
Moreover, $\prec_\alpha$ is antisymmetric by definition.
Finally, any two distinct walks of~$F$ marked at an occurrence of~$\alpha$ are comparable in~$\prec_\alpha$.
Therefore, the reflexive closure of~$\prec_\alpha$ defines a total order on the walks of~$F$ marked at an occurrence of~$\alpha$.
\end{proof}

\begin{notation}
In all pictures, we represent the walks passing through~$\alpha$ in increasing order of~$\prec_\alpha$ from the left side to the right side of~$\alpha$. See for example \fref{fig:exmFacet}\,(right).
\end{notation}

Consider now an arrow~$\alpha \in Q\blossom_1$ and a non-kissing face~$F \in \NKC$ containing the straight walk passing through~$\alpha$.
Let~$F_\alpha$ be the set of walks of~$F$ containing~$\alpha$, marked at an occurrence of~$\alpha$.
Note that~$F_\alpha \ne \varnothing$ since it contains at least the straight walk passing through~$\alpha$, marked at~$\alpha$.
Observe that this straight walk is always the minimal marked walk of~$F_\alpha$ for the order~$\prec_\alpha$, since it always agrees with the direction of~$\alpha$.
On the other hand, Proposition~\ref{prop:nkFacetsAreFinite}, below, shows that we can consider the maximal marked walk of~$F_\alpha$ for the order~$\prec_\alpha$.
Intuitively, it is the most countercurrent walk passing through~$\alpha$ with respect to the direction of~$\alpha$.

\begin{lemma}\label{lem:distArrowExist}
Let~$F$ be a finite face of~$\NKC$, let~$\omega$ be a walk in~$F$, and let~$\alpha$ be an arrow on~$\omega$.
Then, there is an arrow~$\beta$ oriented in the same direction as~$\alpha$ in~$\omega$ such that~$\omega$ (marked at~$\beta$) is maximal in~$F$ for~$\prec_\beta$.
\end{lemma}

\begin{proof}
If~$\omega$ is not maximal in~$F$ for~$\prec_\alpha$, consider~${\mu \eqdef \min_{\prec_\alpha} \set{\nu \in F_\alpha}{\omega \prec_\alpha \nu}}$.
Let~$\sigma$ denote the maximal common substring of~$\omega$ and~$\mu$ containing~$\alpha$.
Since~$\omega$ and~$\mu$ are distinct, they split at an endpoint of~$\sigma$.
At this endpoint, the arrow~$\beta$ of~$\omega$ not in~$\sigma$ points in the same direction as~$\alpha$ since~$\omega \prec_\alpha \mu$.
Let~$\lambda$ be maximal in~$F$ for~$\prec_\beta$.
If~$\lambda \ne \omega$, then~$\lambda$ either kisses~$\mu$ or contradicts the minimality of~$\mu$ among the walks~$\set{\nu \in F_\alpha}{\omega \prec_\alpha \nu}$.
See \fref{fig:pure}\,(left) for an illustration.
We conclude that~$\lambda = \omega$ has a distinguished arrow~$\beta$ in the direction of~$\alpha$.
\end{proof}

\begin{proposition}\label{prop:nkFacetsAreFinite}
The faces of the non-kissing complex have finite cardinality.
\end{proposition}

\begin{proof}
By Lemma~\ref{lem:distArrowExist}, the cardinality of any finite non-kissing face is bounded by~$|Q_1\blossom|$.
Thus, any collection of~$|Q_1\blossom|+1$ walks contains a kiss and all non-kissing faces are finite.
\end{proof}

\begin{definition}
For an arrow~$\alpha \in Q\blossom_1$ and a non-kissing face~$F \in \NKC$ containing the straight walk passing through~$\alpha$, consider the marked walk~$\omega_\star \eqdef \max_{\prec_\alpha} F_\alpha$ and let~$\omega$ be the walk obtained by forgetting the mark in~$\omega_\star$. We say that~$\omega$ is the \defn{distinguished walk} of~$F$ at~$\alpha$ and that~$\alpha$ is a \defn{distinguished arrow} on~$\omega \in F$.
We write
\[
\distinguishedWalk{\alpha}{F} \eqdef \max\nolimits_{\prec_\alpha} F_\alpha
\qquad\text{and}\qquad
\distinguishedArrows{\omega}{F} \eqdef \set{\alpha \in \omega}{\distinguishedWalk{\alpha}{F} = \omega}.
\]
\end{definition}

\begin{example}
Distinguished arrows and walks are illustrated on \fref{fig:exmFlip}\,(left).
\end{example}

\begin{example}
For any~$v \in Q_0$, the distinguished arrows of the peak walk~$v_\peak$ in the peak facet~$F_\peak$ are the two outgoing arrows at~$v$, and the distinguished arrows of the deep walk~$v_\deep$ in the deep facet~$F_\deep$ are the two incoming arrows at~$v$.
\end{example}

The following statement extends~\cite[Thm.~3.2]{McConville}.

\begin{proposition}
\label{prop:distinguishedArrows}
For any non-kissing facet~$F \in \NKC$, each bending (resp.~straight) walk of~$F$ contains precisely~$2$ (resp.~$1$) distinguished arrows in~$F$ pointing in opposite directions (resp.~in the direction~of~the~walk).
\end{proposition}

\begin{proof}
In view of Lemma~\ref{lem:distArrowExist}, it remains to prove that each walk of~$F$ contains at most one distinguished arrow in each direction.
Assume that a walk~$\omega$ has two distinguished arrows~$\alpha, \beta$ pointing in the same direction.
Orient~$\omega$ in this same direction and assume without loss of generality that~$\alpha$ appears before~$\beta$ along~$\omega$.
Let~$\gamma$ be the incoming arrow at~$v \eqdef s(\beta)$ such that~$\gamma\beta \in I$.
Let~$\nu \eqdef \distinguishedWalk{\gamma}{F}$ denote the distinguished walk of~$F$ at~$\gamma$, considered oriented backward~$\gamma$.
We denote by~$\sigma$ the maximal common substring of~$\omega$ and~$\nu$ (it can be reduced to~$\varepsilon_v$).
Note that~$\sigma$ cannot contain~$\alpha$ since otherwise we would have~$\omega \prec_\alpha \nu$ contradicting the assumption that~$\omega = \distinguishedWalk{\alpha}{F}$. 
Since~$\omega$ and~$\nu$ are non-kissing, $\omega$ (resp.~$\nu$) enters~$\sigma$ with an incoming (resp.~outgoing) arrow.
Let~$\mu \eqdef \omega[\,\cdot\,, v] \, \nu[v, \cdot\,]$ be the walk connecting~$s(\omega)$ to~$t(\nu)$ via~$v$, using the prefix of~$\omega$ ending at~$v$ followed by the suffix of~$\nu$ starting at~$v$.
All these notations are illustrated in \fref{fig:pure}\,(right).
Note that $\omega \prec_\alpha \mu$ so that $\mu \notin F$ by maximality of~$\omega$ in~$\prec_\alpha$.
We claim that~$\mu$ is non-kissing with any walk of~$F$, contradicting the maximality~of~$F$.

Assume for the sake of contradiction that there is~$\lambda \in F$ such that~$\lambda$ and~$\mu$ are kissing.
Let~$\tau$ be a maximal common substring of~$\lambda$ and~$\mu$ where they kiss.
Since~$\lambda$ does not kiss~$\omega$ nor~$\nu$ (as they all belong to~$F$), we obtain that~$\tau$ contains~$v$ and strictly contains~$\sigma$.
We now distinguish two cases (see \fref{fig:pure}\,(right)):
\begin{enumerate}[(i)]
\item If~$\lambda$ passes through~$\beta$, then it enters~$\tau$ with an outgoing arrow (since it kisses~$\mu$), contradicting the maximality of~$\omega$~in~$\prec_\beta$.
\item If~$\lambda$ passes through~$\gamma$, then it enters~$\sigma$ with an incoming arrow (since~$\tau$ strictly contains~$\sigma$), contradicting the maximality of~$\nu$ in~$\prec_\gamma$. 
\end{enumerate}
We get a contradiction in all cases, therefore~$\mu$ is non-kissing with any other walk~of~$F$.
\begin{figure}[t]
	\capstart
	\centerline{\includegraphics[scale=1]{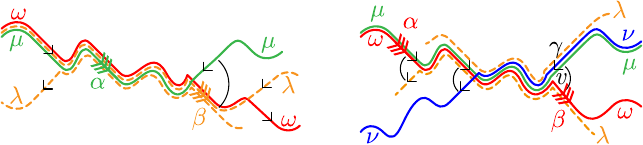}}
	\caption{A schematic illustration of the proofs of Lemma~\ref{lem:distArrowExist} and Proposition~\ref{prop:distinguishedArrows}: for each arrow~$\alpha$ on~$\omega$, there is a distinguished arrow on~$\omega$ in the direction of~$\alpha$ (left), and~$\omega$ cannot have two distinguished arrows~$\alpha, \beta$ in the same direction (right). See the proof of Proposition~\ref{prop:distinguishedArrows} for a detailed description.}
	\label{fig:pure}
\end{figure}
\end{proof}

\begin{corollary}
\label{coro:pure}
The non-kissing complex~$\NKC$ is pure of dimension~$|Q_0| + \Delta = 3|Q_0|-|Q_1|$. The reduced non-kissing complex~$\RNKC$ is pure of dimension~$|Q_0|$.
\end{corollary}

\begin{proof}
Let~$b$ (resp.~$s$) be the number of bending (resp.~straight) walks in a non-kissing facet~$F$ of~$\NKC$.
Since each straight walk uses two blossom arrows of~$\bar Q\blossom$, we have~$s = \Delta = 2|Q_0|-|Q_1|$.
From Proposition~\ref{prop:distinguishedArrows}, we obtain by double counting that
\[
2b+s = |Q\blossom_1| = 4|Q_0|-|Q_1|,
\qquad\text{thus}\qquad
b = |Q_0|
\qquad\text{and}\qquad
b+s = 3|Q_0|-|Q_1|.
\qedhere
\]
\end{proof}

\pagebreak

Since each bending walk in a non-kissing facet of~$\RNKC$ has two distinguished arrows, we can consider the substring they surround.

\begin{definition}
\label{def:distinguishedSubstring}
The \defn{distinguished string}~$\distinguishedString{\omega}{F}$ of a bending walk~$\omega$ in a non-kissing facet~$F$ of~$\NKC$ is the substring of~$\omega$ located in between the two distinguished arrows of~$\omega \in F$.
A string~${\sigma \in \strings^\pm(\bar Q)}$ is \defn{distinguishable} if there exists a walk~$\omega$ in a non-kissing facet~$F \in \NKC$ such that~$\sigma = \distinguishedString{\omega}{F}$ is the distinguished string of~$\omega \in F$.
We denote by~$\distinguishableStrings^\pm(\bar Q)$ the set of distinguishable strings~of~$\bar Q$.
\end{definition}

\begin{example}
For any~$v \in Q_1$, we have~$\distinguishedString{v_\peak}{F_\peak} = \distinguishedString{v_\deep}{F_\deep} = \varepsilon_v$.
\end{example}

Using material from Part~\ref{part:lattice}, we will characterize distinguishable strings of~$\bar Q$ in Section~\ref{subsec:distinguishableStrings}.
We close this section by a technical lemma about distinguished arrows, which will be helpful to understand the different possible situations in Proposition~\ref{prop:flip}.

\begin{lemma}
\label{lem:weirdDistinguishedArrows}
Consider a bending walk~$\omega$ in a non-kissing facet~$F \in \NKC$.
\begin{enumerate}[(i)]
\item If a distinguished arrow~$\alpha$ of~$\omega$ is a loop, then there is a string~$\sigma$ such that~${\omega = \sigma \alpha \sigma^{-1}}$. In particular, $\omega$ cannot be distinguished at two loops.
\item If the distinguished arrows~$\alpha, \beta$ of~$\omega$ are not loops, then~$\alpha \beta \notin I$ and~$\beta \alpha \notin I$, when these concatenations make sense.
\end{enumerate}
\end{lemma}

\begin{proof}
For~(i), write~$\omega = \rho \sigma \alpha \sigma^{-1} \tau$ for some strings~$\rho, \sigma, \tau$ with~$\sigma$ maximal, and assume that~$\rho$ and~$\tau$ are non-empty.
Then the last arrow~$\gamma$ of~$\rho$ points towards~$\sigma$ (otherwise, by maximality of~$\sigma$, the first arrow of~$\tau$ would point outwards~$\sigma$ so that~$\omega$ would be self-kissing).
Let~$\lambda$ be the distinguished walk of~$F$ at~$\gamma$ and let~$\xi$ be the substring of~$\sigma$ contained in~$\lambda$.
Since~$\omega$ is distinguished at~$\alpha$, and since~$\alpha$ is a loop, $\xi$ is a strict substring of~$\sigma$.
Since~$\omega$ and~$\lambda$ are non-kissing, $\lambda$ (resp.~$\omega$) leaves~$\xi$ with an outgoing (resp.~incoming) arrow.
This implies that~$\lambda \prec_\gamma \omega$, a contradiction.
This shows that~$\rho = \varnothing = \tau$~and~$\omega = \sigma \alpha \sigma^{-1}$.

For~(ii), assume for example that~$\alpha\beta \in I$ and that~$\omega = \rho \alpha \sigma \tau \beta^{-1} \sigma \xi$ for some strings~$\rho, \sigma, \tau, \xi$ with~$\sigma$ maximal (we leave it to the reader to check that the other cases are similar).
Since~$\omega$ is not self-kissing, the first arrow~$\gamma$ of~$\xi$ points towards~$\sigma$.
Let~$\lambda$ be the distinguished walk of~$F$ at~$\gamma$ and let~$\zeta$ be the substring of~$\sigma$ contained in~$\lambda$.
Since~$\omega$ is distinguished at~$\alpha$ and~$\beta$, $\tau$ is a strict substring of~$\sigma$.
Since~$\omega$ and~$\lambda$ are non-kissing, $\lambda$ (resp.~$\omega$) leaves~$\zeta$ with an outgoing (resp.~incoming) arrow.
This implies that~$\lambda \prec_\gamma \omega$, a contradiction.
This shows~(ii).
\end{proof}

\subsection{Flips}
\label{subsec:flips}

\begin{figure}[t]
	\capstart
	\centerline{\includegraphics[scale=.7]{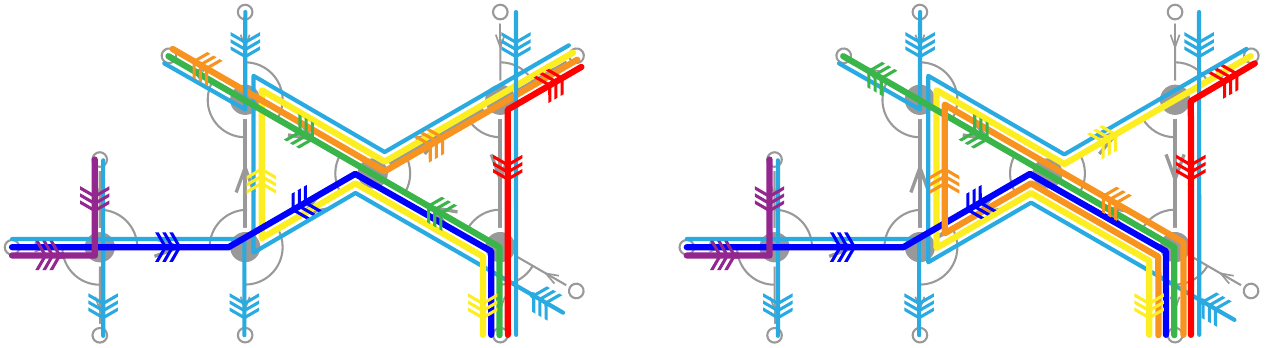}}
	\caption{The flip of the orange walk~$\omega$ in the facet~$F$ of \fref{fig:exmFacet}. The other two walks~$\mu, \nu$ of~$F$ involved in the flip are the yellow and green walks. The distinguished arrows on each walk of~$F$ are marked with triple arrows.}
	\label{fig:exmFlip}
\end{figure}

We now show that the non-kissing complex~$\RNKC$ is \defn{thin}, \ie that any codimension~$1$ non-kissing face of~$\RNKC$ is contained in exactly two non-kissing facets of~$\RNKC$.
Here again, it is more convenient to work with the unreduced non-kissing complex~$\NKC$.
We adapt the proof of~\cite[Thm~3.2\,(3)]{McConville}.
The following statement is schematized in \fref{fig:flip}.

\pagebreak

\begin{proposition}
\label{prop:flip}
Consider a non-kissing facet~$F \in \NKC$ and a bending walk~$\omega \in F$.
Let~$\alpha$ and~$\beta$ be the distinguished arrows of~$\omega$, and~$\sigma$ be the distinguished substring of~$\omega$, which splits~$\omega$ into~$\omega = \rho \sigma \tau$.
Let~$\alpha'$ and~$\beta'$ be the other two arrows of~$Q_1\blossom$ incident to the endpoints of~$\sigma$ and such that~$\alpha'\alpha \in I$ or~$\alpha\alpha' \in I$, and $\beta'\beta \in I$ or~$\beta\beta' \in I$.
Let~$\mu \eqdef \distinguishedWalk{\alpha'}{F \ssm \{\omega\}}$ and~$\nu \eqdef \distinguishedWalk{\beta'}{F \ssm \{\omega\}}$ be the distinguished walks of~$F \ssm \{\omega\}$ at~$\alpha'$ and~$\beta'$ respectively.
Then
\begin{enumerate}[(i)]
\item The walk~$\mu$ splits into~$\mu = \rho' \sigma \tau$ and the walk~$\nu$ splits into~$\nu = \rho \sigma \tau'$.
\item The walk~$\omega' \eqdef \rho' \sigma \tau'$ is kissing~$\omega$ but no other~walk~of~$F$. Moreover, $\omega'$ is the only other (undirected) walk besides~$\omega$ which is not kissing any other walk of~$F \ssm \{\omega\}$.
\end{enumerate}
\end{proposition}

\begin{proof}
Let~$v$ (resp.~$w$) be the endpoint of~$\sigma$ incident to~$\alpha$ and~$\alpha'$ (resp.~to~$\beta$ and~$\beta'$).
The walk~$\mu$ splits into~$\rho' \mu'$ where~$t(\rho') = v$ and~$\rho'$ finishes with the marked occurrence of~$\alpha'$ along~$\mu$.
Consider the walk~$\lambda \eqdef \rho' \sigma \tau$.
We claim that~$\lambda$ is non-kissing with all walks of~$F$.
Indeed, assume that there is a walk~$\theta \in F$ kissing~$\lambda$, and let~$\xi$ denote a maximal common substring of~$\lambda$ and~$\theta$ where they kiss.
Since~$\theta$ does not kiss~$\mu$ nor~$\omega$, the substring~$\xi$ must contain~$v$.
We now distinguish two cases:
\begin{itemize}
\item If~$\theta$ contains~$\alpha$, then it enters~$\xi$ with an outgoing arrow and thus leaves~$\xi$ with an outgoing arrow as well, thus contradicting the maximality of~$\omega$ for~$\prec_\alpha$.
\item If~$\theta$ contains~$\alpha'$, then it enters~$\sigma\tau$ with an incoming arrow and thus leaves~$\sigma\tau$ with an outgoing arrow (otherwise it would kiss~$\omega$). Therefore, it enters and leaves~$\xi$ with an outgoing arrow, thus contradicting the maximality of~$\mu$~for~$\prec_{\alpha'}$.
\end{itemize}
Since both cases raise a contradiction, we obtained that~$\lambda$ is non-kissing with all walks of~$F$.
Since~$F$ is a maximal non-kissing set of walks, we obtain that~$\lambda \in F$.
Assume for contradiction that~$\lambda \ne \mu$, and let~$\xi$ be the maximal common substring of~$\mu$ and~$\omega$ containing~$v$.
Since~$\mu$ and~$\omega$ are non-kissing, and~$\mu$ enters~$\xi$ with an incoming arrow, it leaves~$\xi$ with an incoming arrow as well.
This implies that~$\mu \prec_\alpha \lambda$, contradicting the maximality of~$\mu$ for~$\prec_\alpha$.
We conclude that~$\lambda = \mu$, so that the walk~$\mu$ splits into~$\rho' \sigma \tau$ and similarly the walk~$\nu$ splits into~${\rho \sigma \tau'}$.
In particular, the string~$\sigma$ is common to~$\omega$ and~$\mu$ and~$\nu$.

We now show that no walk of~$F \ssm \{\omega\}$ can kiss~$\omega'$.
Let~$\lambda$ be a walk of~$F$ kissing~$\omega'$, and let~$\xi$ denote a maximal common substring of~$\lambda$ and~$\omega'$ where they kiss.
Since~$\lambda \in F$, it does not kiss~$\mu$ or~$\nu$, which implies that~$\xi$ contains~$\sigma$.
Assume first that~$\lambda$ enters~$\sigma$ at~$v$ by the arrow~$\alpha'$.
Since~$\lambda$ and~$\omega$ are non-kissing, it forces~$\lambda$ to leave~$\sigma$, and thus~$\xi$, at~$w$ by the arrow~$\beta$.
For~$\lambda$ and~$\omega'$ to kiss, $\lambda$ thus enters~$\xi$ with an arrow in the opposite direction of~$\beta$, which contradicts the maximality of~$\mu$ for~$\prec_{\alpha'}$.
By symmetry, we obtain that~$\lambda$ contains the arrows~$\alpha$ and~$\beta$.
Now by maximality of~$\omega$ for both orders~$\prec_\alpha$ and~$\prec_\beta$, the walks~$\lambda$ and~$\omega$ of~$F$ cannot split before~$\alpha$ or after~$\beta$.
We conclude that~$\lambda = \omega$.

Finally, we prove that~$\omega'$ is the only walk distinct from~$\omega$ which is non-kissing with any walk of~$F \ssm \{\omega\}$.
Let~$\lambda$ be a walk such that~$G = F \symdif \{\omega, \lambda\}$ is non-kissing.
Note that in the face~$F \ssm \{\omega\}$, the walk~$\mu$ is maximal for both~$\prec_{\alpha'}$ and~$\prec_\beta$, while the walk~$\nu$ is maximal for both~$\prec_\alpha$ and~$\prec_{\beta'}$.
Therefore, in the facet~$G$, the walk~$\lambda$ should have a distinguished arrow among~$\{\alpha', \beta\}$ and a distinguished arrow among~$\{\alpha, \beta'\}$.
Since~$\lambda$ cannot contain~$\{\alpha, \alpha'\}$ or~$\{\beta, \beta'\}$, this imposes that the distinguished arrows of~$\lambda$ are either~$\{\alpha, \beta\}$, or~$\{\alpha', \beta'\}$.
Assume for example that~$\lambda$ is distinguished at~$\{\alpha, \beta\}$.
If~$\lambda \ne \omega$, then~$\lambda$ and~$\omega$ are kissing by Corollary~\ref{coro:pure}.
This is incompatible with the maximality of~$\lambda$ for both~$\prec_\alpha$ and~$\prec_\beta$.
We conclude that~$\lambda = \omega$.
We obtain similarly that if~$\lambda$ is distinguished at~$\{\alpha', \beta'\}$, then~$\lambda = \omega'$.
\end{proof}

\begin{remark}
\label{rem:changeDistinguishedArrows}
For latter purposes, observe that one can update distinguished arrows when transforming~$F$ to~$F'$:
\begin{itemize}
\item On~$\omega$ and~$\omega'$, we always have
\[
\omega = \distinguishedWalk{\alpha}{F} = \distinguishedWalk{\beta}{F}
\qquad\text{and}\qquad
\omega' = \distinguishedWalk{\alpha'}{F'} = \distinguishedWalk{\beta'}{F'}.
\]

\item On~$\mu$ and~$\nu$, be aware that two different situations can happen:
\begin{itemize}
\item If~$\alpha = \alpha'$ is a loop, then it follows from Lemma~\ref{lem:weirdDistinguishedArrows} that
\[
\mu = \nu = \distinguishedWalk{\alpha}{F \ssm \{\omega\}} = \distinguishedWalk{\beta'}{F} = \distinguishedWalk{\beta}{F'}.
\]
(The situation is symmetric when~$\beta = \beta'$ is a loop).
\item If neither~$\alpha$ nor~$\beta$ are loops, then~$\alpha, \beta, \alpha', \beta'$ are all distinct by Lemma~\ref{lem:weirdDistinguishedArrows} and
\[
\mu = \distinguishedWalk{\alpha'}{F} = \distinguishedWalk{\beta}{F'}
\qquad\text{and}\qquad
\nu = \distinguishedWalk{\alpha}{F'} = \distinguishedWalk{\beta'}{F}.
\]
Note that it can still happen that~$\mu = \nu$ or that~$\mu$ or~$\nu$ are straight walks.
\end{itemize}
\item Finally, all the other distinguished arrows on the walks of~$F \cap F'$ are preserved.
\end{itemize}
\end{remark}


\begin{corollary}
The non-kissing complex~$\RNKC$ is a pseudomanifold without boundary.
\end{corollary}

\begin{definition}
\label{def:flip}
The non-kissing facet~$F' \eqdef F \symdif \{\omega, \omega'\}$ is obtained from~$F$ by the \defn{flip} of~$\omega$.
We say that the flip is \defn{supported} by the common substring~$\sigma = \distinguishedString{\omega}{F} = \distinguishedString{\omega'}{F'}$ of~$\omega$ and~$\omega'$.
The facets~$F$ and~$F'$ are \defn{adjacent facets}.
The (non-kissing) \defn{flip graph} of~$\bar Q$ is the dual graph~$\NKG$ of~$\RNKC$, whose vertices are the non-kissing facets of~$\RNKC$ and whose edges are pairs of adjacent non-kissing facets.
\end{definition}

\begin{example}
\fref{fig:exmFlip} illustrates a flip in the non-kissing complex of the quiver of \fref{fig:exmBlossomingQuiver}.

\begin{figure}[t]
	\capstart
	\centerline{\includegraphics[scale=1]{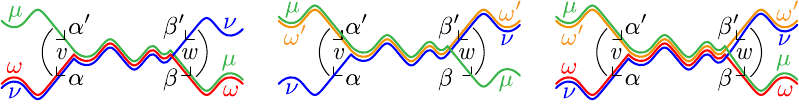}}
	\caption{A schematic representation of flips in the non-kissing complex (see Proposition~\ref{prop:flip}). The flip exchanges the walk~$\omega = \rho \sigma \tau$ in the facet~$F$ (left) with the walk~$\omega' = \rho' \sigma \tau'$ in the facet~$F'$ (middle), using the walks~$\mu = \rho' \sigma \tau = \distinguishedWalk{\alpha'}{F \ssm \{\omega\}}$ and~$\nu = \rho \sigma \tau' = \distinguishedWalk{\beta'}{F \ssm \{\omega\}}$. The walk~$\omega'$ kisses~$\omega$ but no other walk of~$F$ (right).}
	\label{fig:flip}
\end{figure}
\end{example}


\begin{definition}
\label{def:increasingFlip}
\enlargethispage{.3cm}
If the common substring~$\sigma$ of~$\omega$ and~$\omega'$ is on top of~$\omega$ and at the bottom of~$\omega'$, then we say that the flip~$F \to F'$ is \defn{increasing} and that the common substring~$\sigma$ is an \defn{ascent} of~$F$ and a \defn{descent} of~$F'$.
We denote by~$\ascents(F) \subseteq \strings^\pm(\bar Q)$ and $\descents(F) \subseteq \strings^\pm(\bar Q)$ the sets of ascents and descents of~$F$ respectively.
Otherwise,~$\sigma$ is a descent of~$F$ and an ascent of~$F'$ and the flip is called \defn{decreasing}.
\end{definition}

This provides a natural orientation of the non-kissing flip graph~$\NKG$.
We still denote by~$\NKG$ the oriented flip graph.
The purpose of Part~\ref{part:lattice} is to show that the non-kissing oriented flip graph is the Hasse diagram of a congruence-uniform lattice when~$\RNKC$ is finite.

\begin{example}
The non-kissing oriented flip graph is represented in \fref{fig:all2} for all connected gentle bound quivers with~$2$ vertices and in \fref{fig:exmNKC}\,(left) for two gentle bound quivers with~$3$ vertices.
In all pictures, graphs are oriented bottom-up.
\end{example}

\begin{example}
Oriented flip graphs were studied carefully in~\cite{GarverMcConville, MannevillePilaud-accordion, McConville} for the dissection and grid bound quivers introduced in Section~\ref{subsec:dissectionGridQuivers}.
Note that when the quiver is an oriented path~$\exmAn$, a flip in a facet of the non-kissing complex~$\NKC[\exmAn]$ corresponds to the classical flip of a diagonal in the corresponding triangulation (see \fref{fig:exmBijectionAssociahedron}).
Moreover, the flip is increasing if and only if the slope of the diagonal increases along the flip.
\end{example}

\begin{remark}
\label{rem:reverseFlipGraph}
Reversing all arrows reverses the orientation of the flips: $\NKG[\reversed{\bar Q}] = \reversed{\NKG[\bar Q]}$.
\end{remark}

%
%


\section{Non-kissing complexes versus support $\tau$-tilting complexes}
\label{sec:nkcvsttc}

In this section, we start the dictionary of Table~\ref{table:dictionary} between algebraic notions on the support $\tau$-tilting complex of a gentle algebra and the combinatorial notions on the non-kissing complex of its blossoming quiver.

\subsection{From walks to strings}
\label{subsec:walks2strings}

Recall that $\bar Q = (Q,I)$ is a gentle bound quiver.
Since~$\bar Q$ is a subquiver of~$\bar Q\blossom$, we can identify a string~$\sigma$ of~$\bar Q$ with a string~$\sigma\blossom$ of~$\bar Q\blossom$ with endpoints in~$Q_0$.
The following lemma is immediate from the definition of a blossoming quiver.

\begin{lemma}
\label{lem:bijections-STTC-NKC}
For a directed string~$\sigma \in \strings(\bar Q)$, the string~$\sigma\blossom$ does not start or end in a deep, so that we can define a bending walk~$\omega(\sigma) = \phantom{}_c (\sigma\blossom) \phantom{}_c \in \bendingWalks^\pm(\bar Q)$.
Conversely, for a directed bending walk~$\omega \in \bendingWalks^\pm(\bar Q)$ not in the deep facet~$F_\deep$, the string~$\phantom{}_{c^{-1}} \omega \phantom{}_{c^{-1}}$ has endpoints in~$Q_0$, so that it is a string~$\sigma(\omega) \in \strings(\bar Q)$.
The induced maps
\[
\strings^\pm(\bar Q) \xrightleftharpoons[\quad\mbox{$\sigma$}\;\quad]{\mbox{$\omega$}} \bendingWalks^\pm(\bar Q) \ssm F_\deep
\]
are inverse bijections.
Putting~$\omega(-v) = v_\deep$ and~$\sigma(v_\deep) = -v$, they extend to inverse bijections
\[
\strings_{\ge-1}^\pm(\bar Q) \xrightleftharpoons[\quad\mbox{$\sigma$}\;\quad]{\mbox{$\omega$}} \bendingWalks^\pm(\bar Q).
\]
\end{lemma}

\begin{example}
\label{exm:exmBijectionStringsWalks1}
\fref{fig:exmBijectionStringsWalks1} illustrates the map~$\sigma = \omega^{-1}$ for two quivers on $3$ vertices.

\begin{figure}[t]
	\capstart
	\centerline{\includegraphics[scale=.4]{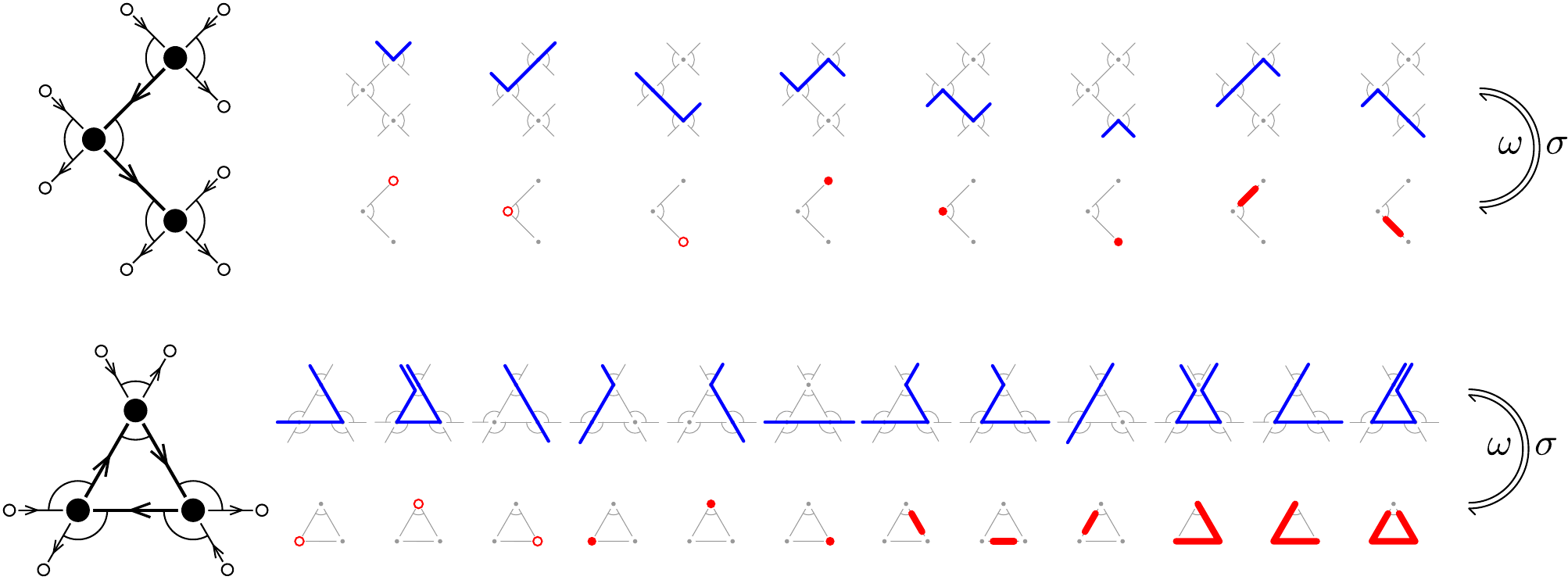}}
	\caption{The inverse bijections~$\sigma$ and~$\omega$ for two specific quivers.}
	\label{fig:exmBijectionStringsWalks1}
\end{figure}
\end{example}

\begin{remark}
In the case of gentle algebras arising from grids, the computation of the string module $M(\sigma)$ is made easier by the fact that the indecomposable representations are determined by their dimension vectors.
This dimension vector is computed by counting the number of kissings between~$\omega(\sigma)$ and the deep walks.
More precisely, for any string~$\sigma \in \strings^\pm(\bar Q)$ and any vertex~$v \in Q_0$, the dimension of~$M(\sigma)_v$ is the number of kisses between the deep walk~$v_\deep$ and the walk~$\omega(\sigma)$.
\end{remark}

\begin{lemma}\label{lem:attractDanceKiss}
Let~$\sigma$ and~$\sigma'$ be two strings for the gentle bound quiver~$\bar Q$.
Then~$\sigma$ attracts~$\sigma'$ or reaches for~$\sigma'$ if and only if~$\omega(\sigma)$ kisses~$\omega(\sigma')$.
\end{lemma}

\begin{proof}
In this proof, we freely identify a string~$\rho$ for~$\bar Q$ with the corresponding string~$\rho\blossom$ for~$\bar Q\blossom$.

Assume that~$\sigma$ attracts~$\sigma'$.
Then~$\omega(\sigma')$ is of the form~$\alpha_r\cdots\alpha_1\alpha^{-1}\sigma'\beta\beta_1^{-1}\cdots\beta_s^{-1}$ for some~${r \geq 1}$, ${s\geq 0}$ and cohooks~$\alpha_r\cdots\alpha_1\alpha^{-1}$ and~$\beta\beta_1^{-1}\cdots\beta_s^{-1}$.
We note that the arrow~$\alpha$ is in~$Q_1$ by Definition~\ref{definition: dance and attract} and that all arrows in~$\omega(\sigma')$ belong to~$Q_1$ except for the blossoming arrows~$\alpha_r$ and~$\beta_s$.
Since the arrow~$\alpha_r$ is blossoming, there is some substring~$\alpha_i\cdots\alpha_1$, with~$0\leq i \leq r$, which is on top of~$\omega(\sigma)$ (see Figure~\ref{fig: sigma attracts sigma'}).
Here again, we use the convention that~$\alpha_i\cdots\alpha_1 = t(\alpha)= u$ if~$i=0$. 
Moreover,~$\alpha_i\cdots\alpha_1$ is at the bottom of~$\omega(\sigma')$, so that the walk~$\omega(\sigma)$ kisses~$\omega(\sigma')$.

\begin{figure}[h]
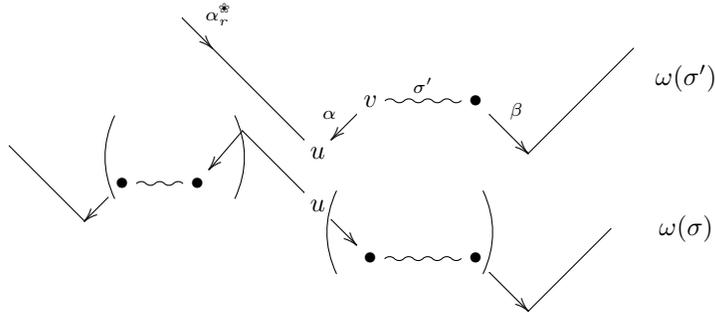

 	\capstart
\[
\xy
(35,10)*+{\omega(\sigma')};
(-28,14)="lpeak";
(-14,0)*+{u}="-12";
(-7,7)*+{v}="-11";
(7,7)*+{\bullet}="11";
(14,0)="12";
(28,14)="rpeak";
(35,-10)*+{\omega(\sigma)};
(-55,1)="1";
(-45,-9)="2";
(-40,-4)*+{\bullet}="3";
(-30,-4)*+{\bullet}="4";
(-24,3)="5";
(-14,-7)*+{u}="u";
(-7,-14)*+{\bullet}="a";
(7,-14)*+{\bullet}="b";
(14,-21)="c";
(25,-10)="d";
{\ar^{\alpha_r\blossom} (-32,18);"lpeak"};
{\ar@{-} "lpeak";"-12"};
{\ar_{\alpha} "-11";"-12"};
{\ar@{~}^{\sigma'} "-11";"11"};
{\ar^{\beta} "11";"12"};
{\ar@{-} "12";"rpeak"};
{\ar@{-} "1";"2"};
{\ar "3";"2"};
{\ar@{~} "3";"4"};
{\ar "5";"4"};
{\ar@{-} "5";"u"};
{\ar "u";"a"};
{\ar@{~} "a";"b"};
{\ar "b";"c"};
{\ar@{-} "c";"d"};
{\ar@/_.3pc/@{-} (-11.5,-5);(-11.5,-16)};
{\ar@/^.3pc/@{-} (8,-5);(8,-16)};
{\ar@/_.3pc/@{-} (-41,5);(-41,-6)};
{\ar@/^.3pc/@{-} (-25,5);(-25,-6)};
\endxy
\]
    \caption{The walks~$\omega(\sigma)$ and~$\omega(\sigma')$ when~$\sigma$ attracts~$\sigma'$.}
    \label{fig: sigma attracts sigma'}
\end{figure}

Assume that~$\sigma$ reaches for~$\sigma'$, and let~$\xi$ be the substring on top of~$\sigma$ and at the bottom of~$\sigma'$, as in Definition~\ref{definition: dance and attract}\,(ii).
We distinguish three cases:

\begin{enumerate}[(a)]
\item If~$\xi=\sigma'$, then~$\sigma = \sigma_1\alpha^{-1}\xi\beta\sigma_2$ for some substrings~$\sigma_1,\sigma_2$ and arrows~$\alpha,\beta\in Q_1$.
As a consequence,~$\omega(\sigma')$ is of the form~$\alpha_r\cdots\alpha_1\alpha^{-1}\sigma'\beta\beta_1^{-1}\cdots\beta_s^{-1}$ for some~$r,s\geq 1$, and cohooks~$\alpha_r\cdots\alpha_1\alpha^{-1}$ and~$\beta\beta_1^{-1}\cdots\beta_s^{-1}$ where~$\alpha_r,\beta_s$ are blossoming.
There exist~$0\leq i\leq r$ and~$0\leq j\leq s$ such that~$\alpha_i\cdots\alpha_1\alpha^{-1}\xi\beta\beta_1^{-1}\cdots\beta_j^{-1}$ is a substring on top of~$\omega(\sigma)$ (see Figure~\ref{fig: dance case a}).
Since it is at the bottom of~$\omega(\sigma')$, the walk~$\omega(\sigma)$ kisses~$\omega(\sigma')$.

\item If~$\sigma'$ is of the form~$\sigma'_1\alpha\xi\beta^{-1}\sigma'_2$, then~$\omega(\sigma')$ contains~$\alpha\xi\beta^{-1}$ as a substring.
Moreover,~$\xi$ being on top of~$\sigma$,~$\omega(\sigma)$ contains some substring of the form~$\gamma^{-1}\xi\delta$.
This shows that~$\omega(\sigma)$ kisses~$\omega(\sigma')$.

\item If~$\sigma'=\sigma'_1\alpha\xi$ for some substring~$\sigma'_1$ and some arrow~$\alpha$, then~$\sigma$ is of the form~$\sigma=(\sigma_1\gamma^{-1})\xi\beta\sigma_2$, where~$\alpha\xi\beta$ is a string for~$\bar Q$.
It follows that~$\omega(\sigma)$ kisses~$\omega(\sigma')$ along some substring~$\xi\beta\beta_1^{-1}\cdots\beta_j^{-1}$ (see Figure~\ref{fig: dance case c}), where~$\beta\beta_1^{-1}\cdots\beta_j^{-1}\cdots\beta_s^{-1}$ is a cohook for~$\sigma\blossom$ and~${j>0}$.
\end{enumerate}

We now assume that $\omega(\sigma)$ kisses $\omega(\sigma')$.
Let $\xi$ be a substring which is strictly on top of $\omega(\sigma)$ and strictly at the bottom of $\omega(\sigma')$.
Note that this implies that $\xi$ is also on top of $\sigma$ (but not necessarily strictly on top).
It is possible to show that $\sigma$ attracts or reaches for $\sigma'$ by using the same method as in the proof of Proposition~\ref{proposition: attract or dance and tau compatibility}.
However, we choose an alternative proof, showing instead that $\Hom{A} \big( M(\sigma), \tau M(\sigma') \big) \ne 0$, which implies the result by Proposition~\ref{proposition: attract or dance and tau compatibility}. 
Write $\omega(\sigma')=\alpha_r \cdots \alpha_1 \alpha^{-1} \sigma' \beta \beta_1^{-1} \cdots \beta_s^{-1}$ where $r,s\geq 0$, and where if $r\geq 1$, $\alpha_r$ is a blossoming arrow and $\alpha_1,\cdots,\alpha_{r-1},\alpha\in Q_1$ while if $r=0$ $\alpha$ is a blossoming arrow; and similarly for $\beta\beta_1^{-1}\cdots\beta_s^{-1}$.

Assume first that $r$ and $s$ are positive.
Then $\tau M(\sigma')$ is given by the string $\sigma''=\alpha_{r-1} \cdots \alpha_1 \alpha^{-1} \sigma'$ $\beta \beta_1^{-1} \cdots \beta_{s-1}^{-1}$.
Thus, $\xi$ is at the bottom of $\sigma''$ and $\Hom{A} \big( M(\sigma), \tau M(\sigma') \big) \neq 0$.

Assume that exactly one of $r,s$ equals $0$.
By symmetry, we may assume that $r=0$, which is equivalent to $\sigma'$ starting in a deep.
Since $\xi$ is strictly at the bottom of $\alpha^{-1}\sigma'\beta\beta_1^{-1}\cdots\beta_s^{-1}$, it is also strictly at the bottom of $\sigma'\beta\beta_1^{-1}\cdots\beta_s^{-1}$ (see Figure~\ref{fig: r=0}).
This implies that $\xi$ is at the bottom of $_{h^{-1}}(\sigma'\beta\beta_1^{-1}\cdots\beta_{s-1}^{-1})$, so that $\Hom{A} \big( M(\sigma), \tau M(\sigma') \big) \ne 0$.

\begin{figure}[h]
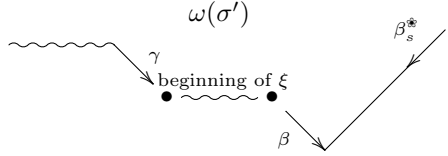

 	\capstart
\[
\xy
(-33,16)="-14";
(-28,21)="-13";
(-14,21)="-12";
(-7,14)*+{\bullet}="-11";
(0,25)*+{\omega(\sigma')};
(7,14)*+{\bullet}="11";
(14,7)="12";
(25,18)="13";
(30,23)="14";
{\ar@{~} "-13";"-12"};
{\ar^{\gamma} "-12";"-11"};
{\ar@{~}^{\text{ beginning of }\xi} "-11";"11"};
{\ar_\beta "11";"12"};
{\ar@{-} "12";"13"};
{\ar_{\beta_s\blossom} "14";"13"};
\endxy
\]
    \caption{The walk $\omega(\sigma')$ when $r=0$ and $s>0$. The arrow $\gamma$ ``protects'' $\xi$ from being deleted when removing a left hook for $\sigma'$.}
    \label{fig: r=0}
\end{figure}

Finally, assume that $r=0$ and $s=0$.
In that case, $\xi$ is strictly at the bottom of $\sigma'$.
It is thus at the bottom of $_{h^{-1}}(\sigma')_{h^{-1}}$ and this gives a non-zero morphism in $\Hom{A} \big( M(\sigma), \tau M(\sigma') \big)$.
\end{proof}

\begin{figure}[p]
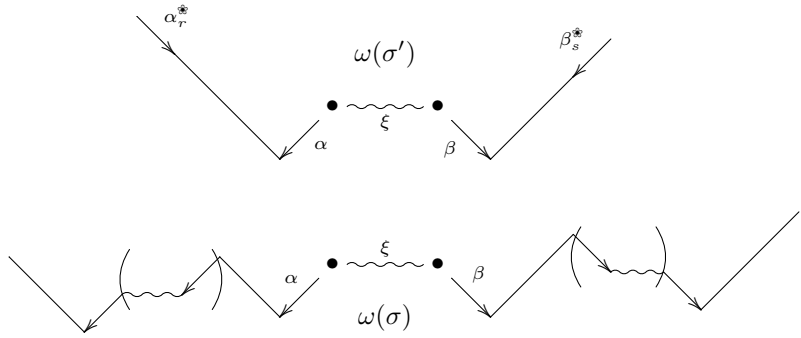

\[
\xy
(-33,26)="-14";
(-28,21)="-13";
(-14,7)="-12";
(-7,14)*+{\bullet}="-11";
(0,21)*+{\omega(\sigma')};
(7,14)*+{\bullet}="11";
(14,7)="12";
(25,18)="13";
(30,23)="14";
{\ar^{\alpha_r\blossom} "-14";"-13"};
{\ar@{-} "-13";"-12"};
{\ar^\alpha "-11";"-12"};
{\ar@{~}_\xi "-11";"11"};
{\ar_\beta "11";"12"};
{\ar@{-} "12";"13"};
{\ar_{\beta_s\blossom} "14";"13"};
(-50,-6)="-7";
(-40,-16)="-6";
(-35,-11)="-5";
(-27,-11)="-4";
(-22,-6)="-3";
(-14,-14)="-2";
(-7,-7)*+{\bullet}="-1";
(0,-14)*+{\omega(\sigma)};
(7,-7)*+{\bullet}="1";
(14,-14)="2";
(25,-3)="3";
(30,-8)="4";
(37,-8)="5";
(42,-13)="6";
(55,0)="7";
{\ar@{-} "-7";"-6"};
{\ar "-5";"-6"};
{\ar@{~} "-5";"-4"};
{\ar "-3";"-4"};
{\ar@{-} "-3";"-2"};
{\ar_\alpha "-1";"-2"};
{\ar@{~}^\xi "-1";"1"};
{\ar^\beta "1";"2"};
{\ar@{-} "2";"3"};
{\ar "3";"4"};
{\ar@{~} "4";"5"};
{\ar "5";"6"};
{\ar@{-} "6";"7"};
{\ar@/^.3pc/@{-} (-23,-5);(-23,-13)};
{\ar@/_.3pc/@{-} (-34,-5);(-34,-13)};
{\ar@/_.3pc/@{-} (26,-2);(26,-10)};
{\ar@/^.3pc/@{-} (36,-2);(36,-10)};
\endxy
\]
    \caption{The walks~$\omega(\sigma)$ and~$\omega(\sigma')$ when~$\sigma$ reaches for~$\sigma'$: case (a).}
    \vspace{1cm}
    \label{fig: dance case a}
\end{figure}

\begin{figure}[p]
\[
\xy
(-49,21)="-15";
(-35,7)="-14";
(-28,14)="-13";
(-14,14)="-12";
(-7,7)="-11";
(0,18)*+{\omega(\sigma')};
(7,7)="11";
(14,14)="12";
(28,14)="13";
(35,7)="14";
(49,21)="15";
{\ar@{-} "-15";"-14"};
{\ar "-13";"-14"};
{\ar@{~}^{\sigma'_1} "-13";"-12"};
{\ar^\alpha "-12";"-11"};
{\ar@{~}^\xi "-11";"11"};
{\ar_\beta "12";"11"};
{\ar@{~}^{\sigma'_2} "12";"13"};
{\ar "13";"14"};
{\ar@{-} "14";"15"};
(-49,-7)="-5";
(-35,-21)="-4";
(-28,-14)="-3";
(-14,-14)="-2";
(-7,-7)="-1";
(0,-18)*+{\omega(\sigma)};
(7,-7)="1";
(14,-14)="2";
(28,-14)="3";
(35,-21)="4";
(49,-7)="5";
{\ar@{-} "-5";"-4"};
{\ar "-3";"-4"};
{\ar@{~} "-3";"-2"};
{\ar "-1";"-2"};
{\ar@{~}^\xi "-1";"1"};
{\ar "1";"2"};
{\ar@{~} "2";"3"};
{\ar "3";"4"};
{\ar@{-} "4";"5"};
{\ar@/_.3pc/@{-} (8,-5);(8,-16)};
{\ar@/^.3pc/@{-} (27,-5);(27,-16)};
{\ar@/^.3pc/@{-} (-8,-5);(-8,-16)};
{\ar@/_.3pc/@{-} (-27,-5);(-27,-16)};
\endxy
\]
    \caption{The walks~$\omega(\sigma)$ and~$\omega(\sigma')$ when~$\sigma$ reaches for~$\sigma'$: case (b).}
    \vspace{1cm}
    \label{fig: dance case b}
\end{figure}

\begin{figure}[p]
\[
\xy
(-42,25)="-15";
(-33,16)="-14";
(-28,21)="-13";
(-14,21)="-12";
(-7,14)*+{\bullet}="-11";
(0,25)*+{\omega(\sigma')};
(7,14)*+{\bullet}="11";
(14,7)="12";
(25,18)="13";
(30,23)="14";
{\ar@{-} "-15";"-14"};
{\ar "-13";"-14"};
{\ar@{~}^{\sigma'_1} "-13";"-12"};
{\ar^\alpha "-12";"-11"};
{\ar@{~}_\xi "-11";"11"};
{\ar_\beta "11";"12"};
{\ar@{-} "12";"13"};
{\ar_{\beta_s\blossom} "14";"13"};
(-40,-9)="-5";
(-30,-19)="-4";
(-25,-14)="-3";
(-14,-14)="-2";
(-7,-7)*+{\bullet}="-1";
(0,-17)*+{\omega(\sigma)};
(7,-7)*+{\bullet}="1";
(14,-14)="2";
(25,-3)="3";
(30,-8)="4";
(37,-8)="5";
(42,-13)="6";
(55,0)="7";
{\ar@{-} "-5";"-4"};
{\ar "-3";"-4"};
{\ar@{~} "-3";"-2"};
{\ar "-1";"-2"};
{\ar@{~}^\xi "-1";"1"};
{\ar^\beta "1";"2"};
{\ar@{-} "2";"3"};
{\ar "3";"4"};
{\ar@{~} "4";"5"};
{\ar "5";"6"};
{\ar@{-} "6";"7"};
{\ar@/^.3pc/@{-} (-9,-7);(-9,-15)};
{\ar@/_.3pc/@{-} (-24,-7);(-24,-15)};
{\ar@/_.3pc/@{-} (26,-2);(26,-10)};
{\ar@/^.3pc/@{-} (36,-2);(36,-10)};
\endxy
\]
    \caption{The walks~$\omega(\sigma)$ and~$\omega(\sigma')$ when~$\sigma$ reaches for~$\sigma'$: case (c).}
    \label{fig: dance case c}
\end{figure}

We conclude this section with the observation which motivated this paper.

\begin{theorem}
\label{thm:nkc/sttiltc}
For any gentle quiver~$\bar Q = (Q,I)$, the maps
\[
\strings_{\ge-1}^\pm(\bar Q) \xrightleftharpoons[\quad\mbox{$\sigma$}\;\quad]{\mbox{$\omega$}} \bendingWalks^\pm(\bar Q)
\]
induce
\begin{itemize}
\item isomorphisms of simplicial complexes between the support $\tau$-tilting complex~$\tTC$ and the non-kissing complex~$\RNKC$,
\item isomorphisms of oriented graphs between the graph of left mutations in~$\tTC$ and the graph of increasing flips in~$\RNKC$.
\end{itemize}
\end{theorem}

\begin{proof}
The theorem is a direct consequence of Lemma~\ref{lem:attractDanceKiss}, except for the fact that increasing flips correspond to left mutations.
We use the notation of Proposition~\ref{prop:flip} and Figure~\ref{fig:flip}: let $F$ and~$F'$ be two adjacent facets such that the flip of~$F$ to~$F'$ is increasing, let $\omega$ and~$\omega'$ are the walks such that~$F \ssm \{\omega\} = F' \ssm \{\omega'\}$, let~$\{\alpha, \beta\} = \distinguishedArrows{\omega}{F}$ and~$\{\alpha', \beta'\} = \distinguishedArrows{\omega'}{F'}$, let~$\xi = \distinguishedString{\omega}{F} = \distinguishedString{\omega'}{F'}$, and let~$\mu = \distinguishedWalk{\alpha'}{F} = \distinguishedWalk{\beta}{F'}$ and~$\nu = \distinguishedWalk{\beta'}{F} = \distinguishedWalk{\alpha}{F'}$.
We consider the four associated strings~$\sigma(\omega), \sigma(\omega'), \sigma(\mu)$ and~$\sigma(\nu)$.
Note that~$\mu$ could be a straight walk, in which case we let~$\sigma(\mu)$ be~$\varnothing$ so that~$M \big( \sigma(\mu) \big) = 0$, and similarly for~$\nu$.
Since the statement is trivially satisfied when~$\omega'$ is a deep walk, we may assume that it is not the case.
We claim that the representation~$M \big( \sigma(\omega') \big)$ is a quotient of~$M \big( \sigma(\mu) \big) \oplus M \big( \sigma(\nu) \big)$, which implies the statement by~\cite[Def.--Prop.~2.28]{AdachiIyamaReiten}.
First note that the distinguished string~$\xi$ is also a substring of~$\omega$ (use~$\alpha$ and~$\beta$).
We distinguish five cases, depending on which of the distinguished arrows~$\alpha'$ and~$\beta'$ belongs to some cohook in~$\omega'$.
Assuming first that~$\alpha', \beta'$ both belong to~$\sigma(\omega')$, then there is a short exact sequence~$0\rightarrow M \big( \sigma(\omega) \big) \rightarrow M \big( \sigma(\mu) \big) \oplus M \big( \sigma(\nu) \big) \rightarrow M \big( \sigma(\omega') \big) \rightarrow 0$.
Assume next that~$\alpha'$ belongs to some (left) cohook for~$\omega'$, but that~$\beta'$ belongs to~$\sigma(\omega')$.
This implies that~$\nu$ is not a straight walk and that~$\sigma(\omega')$ is a substring on top of~$\sigma(\nu)$, thus giving a sujection~$M \big( \sigma(\nu) \big) \rightarrow M \big( \sigma(\omega') \big)$.
The symmetric case gives a surjection~$M \big( \sigma(\mu) \big) \rightarrow M \big( \sigma(\omega') \big)$.
Assume that the distinguished arrows~$\alpha', \beta'$ both belong to the same cohook.
By symmetry, we may assume that it is a left cohook.
In that case again,~$\nu$ is not a straight walk and~$\sigma(\omega')$ is a substring on top of~$\sigma(\nu)$.
Finally, assuming that~$\alpha'$ belongs to some left cohook, while~$\beta'$ belongs to some right cohook for~$\omega'$, the string~$\sigma(\omega')$ is easily seen to be on top of both strings~$\sigma(\mu)$ and~$\sigma(\nu)$.
\end{proof}

\begin{remark}
\begin{enumerate}
\item Considering walks on the blossoming gentle bound quiver allows to obtain a short exact sequence in~$\rep \bar Q\blossom$ associated with each left mutation:
\[
0\rightarrow M(\omega) \rightarrow M(\mu) \oplus M(\nu) \rightarrow M(\omega') \rightarrow 0,
\]
in the notations of Proposition~\ref{prop:flip} and Figure~\ref{fig:flip}.
We note that these extensions have a flavour very similar to the results of  \.{I}.~\c{C}anak\c{c}{\i} and S.~Schroll in~\cite{CanakciSchroll}.
In their paper, the middle term of an extension is obtained by resolving a crossing, while here, a kissing is resolved into a crossing.
\item One can check that, in all five cases of the proof above, the morphism $M \big( \sigma(\omega) \big) \rightarrow M \big( \sigma(\mu) \big) \oplus M \big( \sigma(\nu) \big)$ is a left-$T$ approximation with cokernel~$M \big( \sigma(\omega') \big)$, where $T$ is the direct sum of all~$M \big( \sigma(\lambda) \big)$ for $\lambda \in F \ssm \{\omega\}$ not a straight walk nor a deep walk.
 However, this approximation is not injective in general.
\end{enumerate}
\end{remark}

\begin{example}
The map~$\sigma$ of \fref{exm:exmBijectionStringsWalks1} sends the support~$\tau$-tilting complexes of \fref{fig:exmtTC}\,(left) to the non-kissing complexes of~\fref{fig:exmNKC}\,(left).
\end{example}

\subsection{From walks to AR-quivers}

In this section, we remark that, by using walks on $\bar Q\blossom$, it is possible to define combinatorially some translation quiver.
We show that this translation quiver contains enough information to recover the components of the Auslander--Reiten quiver of $\bar Q$ containing the string modules.
Two examples in the representation-finite case are given in Figure~\ref{fig:exmAuslanderReitenQuiver}.

\begin{proposition}\label{proposition: elementary moves and walks}
Let $\sigma\in\strings(\bar Q)$ be a string for $\bar Q$ and let $\sigma' = {}_c\sigma$ if $\sigma$ does not start in a deep or $\sigma' = {}_{h^{-1}}\sigma$ if $\sigma$ starts in a deep.
Then the walk $\omega(\sigma)$ starts in a deep and we have $\omega(\sigma') = {}_c\big(_{h^{-1}}\omega(\sigma)\big)$.
\end{proposition}

\begin{proof}
We first note that any walk in $Q\blossom$ starts either with a blossoming arrow or with the inverse of a blossoming arrow.
As a consequence, any walk in $Q\blossom$ starts in a deep.

We first assume that $\sigma$ does not start in a deep, so that $\sigma'=\,_c\sigma$.
Then $\omega(\sigma)$ is of the form $\alpha_{r+1} \blossom \alpha_r \cdots \alpha_1 \alpha^{-1}\sigma\nu$ where $r\geq -1$, $\alpha_r\cdots\alpha_1\alpha^{-1}$ is a left cohook for $\sigma$ in $\bar Q$, $\alpha_{r+1}\blossom$ is a blossoming arrow, and $\nu$ is a right cohook for $\sigma$ in $\bar Q\blossom$.
Moreover, $\sigma'=\alpha_r\cdots\alpha_1\alpha^{-1}\sigma$ and $\,_{h^{-1}}\omega(\sigma) = \alpha_r\cdots\alpha_1\alpha^{-1}\sigma\nu = \sigma'\nu$.
The result follows.

We now assume that $\sigma$ starts in a deep, so that $\sigma'=\,_{h^{-1}}\sigma$.
Write $\sigma = \mu\sigma'$ where $\mu$ is a hook in $\bar Q$.
Then we have $\omega(\sigma)=(\alpha\blossom)^{-1}\mu\sigma'\nu$ for some cohook $\nu$ in $\bar Q\blossom$ and some blossoming arrow $\alpha\blossom$ and $_{h^{-1}}\omega(\sigma)=\sigma'\nu$.
This implies that $\omega(\sigma') = {}_c\big(_{h^{-1}}\omega(\sigma)\big)$.
\end{proof}

\begin{remark}
Assume that $\bar Q$ is representation-finite.
It follows from Proposition~\ref{proposition: elementary moves and walks} and from its dual that the Auslander--Reiten quiver of $\bar Q$ can be computed by using walks in $\bar Q\blossom$ instead of strings in $\bar Q$.
The computation is made slightly easier by the fact that irreducible arrows always correspond to the same elementary move on the walks (replacing a cohook by a hook); while for strings two different elementary moves appear (adding a hook or removing a cohook).
Two examples of computations are shown in Figure~\ref{fig:exmAuslanderReitenQuiver}.
We note that one can slightly extend the AR-quiver by including straight walks and deep walks.
The translation quiver thus obtained resemble the AR-quiver of an exact category whose projective-injectives are the straight walks, whose non-injective projectives are the peak walks (corresponding to the projective representations of $\bar Q$) and whose non-projective injectives are the deep walks (corresponding to the shifted projective representations of $\bar Q$).
The category obtained by killing the projective-injectives seems to be related to two-term silting complexes since its indecomposable objects are the indecomposable representations of $\bar Q$ and the shifts of the indecomposable projective representations of $\bar Q$.
This remark might deserve further investigation.
We also note that, if $Q$ is of type $A_n$, the translation quiver described above is not the AR-quiver of the corresponding cluster category:
the arrows from shifted projectives to modules do not appear in this translation quiver.
\end{remark}

\captionsetup{width=1.5\textwidth}
\hvFloat[floatPos=h, capWidth=h, capPos=right, capAngle=90, objectAngle=90, capVPos=c, objectPos=c]{figure}
{\includegraphics[scale=.45]{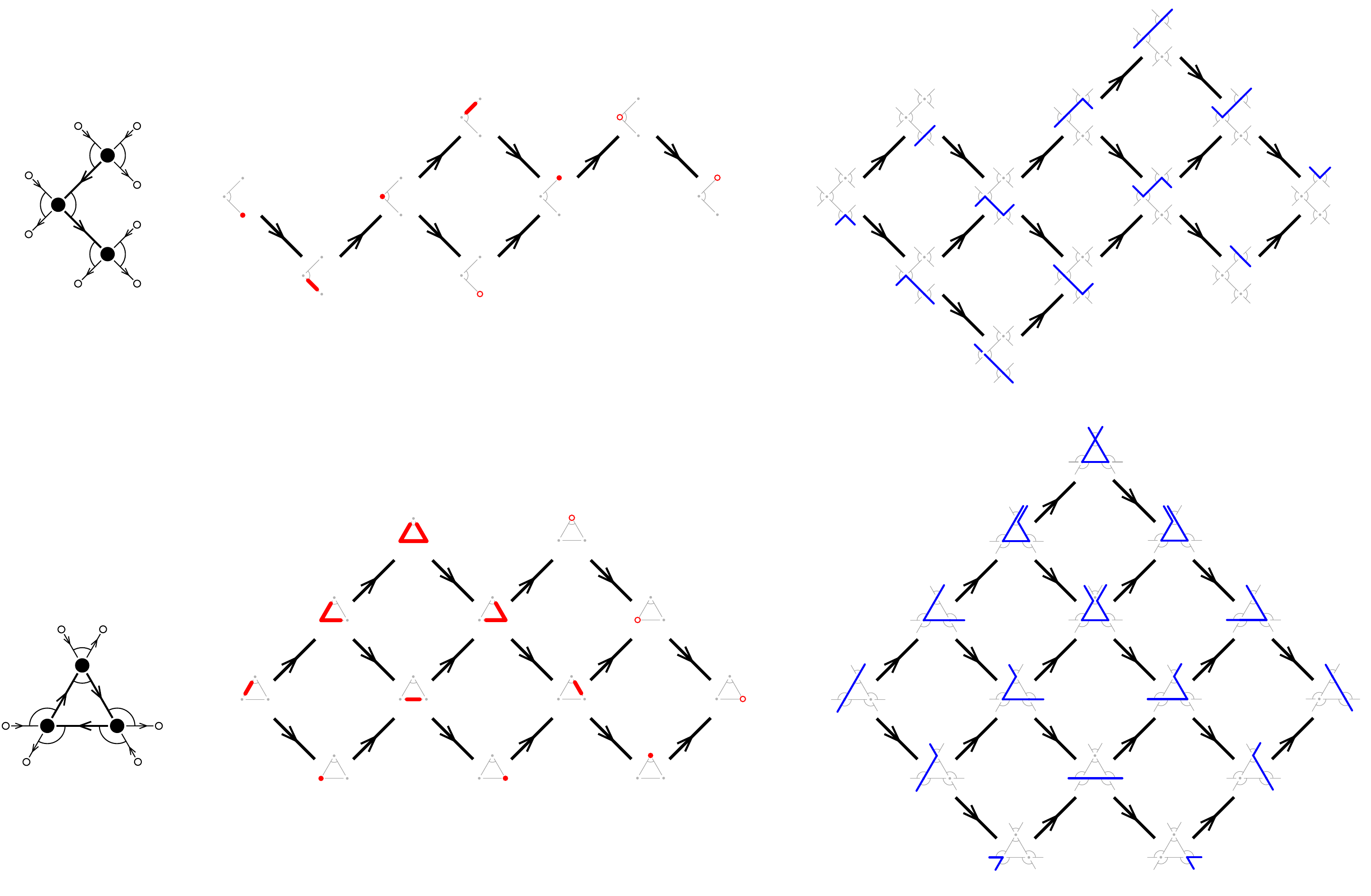}}
{The Auslander-Reiten quiver~$\AR$ for two gentle bound quivers on~$3$ vertices. By Sections~\ref{subsec:stringBandModules} and~\ref{subsec:walks2strings}, the indecomposable $\tau$-rigid or shifted projective representations of~$A = kQ/I$ can be identified either with strings of~$\bar Q$ (left) or with walks on~$\bar Q\blossom$ (right). The additional nodes on the walk picture (right) correspond to straight walks.}
{fig:exmAuslanderReitenQuiver}
\captionsetup{width=\textwidth}


\clearpage
\part{The non-kissing lattice}
\label{part:lattice}

This section focuses on lattice-theoretic properties of non-kissing complexes.
Referring to~\ref{sec:recollectionLattice} for proper definitions of the lattice-theoretic concepts involved, let us immediately state the main result of this section.

\begin{theorem}
\label{thm:lattice}
If~$\bar Q$ is a gentle bound quiver whose non-kissing complex~$\RNKC$ is finite, then the non-kissing oriented flip graph~$\NKG$ is the Hasse diagram of a congruence-uniform lattice.
\end{theorem}

\begin{example}
Theorem~\ref{thm:lattice} is already known for the dissection and grid bound quivers:
\begin{itemize}
\item For a dissection, it is shown in~\cite{GarverMcConville}.
\item For a subset of~$\Z^2$,  it is the central result of~\cite{McConville}.
\item For oriented paths, the non-kissing lattice is the (type~$A$) Cambrian lattice of N.~Reading~\cite{Reading-CambrianLattices}. This contains in particular the classical Tamari lattice for a directed path.
\end{itemize}
\end{example}

In fact, our proof of Theorem~\ref{thm:lattice} is an adaptation of the proof  developed by T.~McConville in~\cite{McConville}.
Namely, we define a closure operator on subsets of strings whose biclosed sets have a lattice structure~(\ref{sec:biclosedStrings}), define a natural lattice congruence on biclosed sets~(\ref{sec:latticeCongruence}), and show that the Hasse diagram of the lattice quotient is isomorphic to the non-kissing oriented flip graph~(\ref{sec:biclosedSetsToNKF}).
Finally (\ref{sec:nonFriendlyComplex}), we study the join-irreducible elements of the non-kissing lattice~$\NKL$ (Definition~\ref{def:NKL}) and the canonical join complex of~$\NKL$ which we call the non-friendly complex as in~\cite{GarverMcConville-grid}.

This section follows the same strategy as that of~\cite{McConville, GarverMcConville-grid} and many arguments are just translated to our context.
The reader will certainly appreciate that the definitions on arbitrary gentle bound quivers become somewhat uniform and do not depend on vertical and horizontal edges as in~\cite{McConville}.
However, there are two difficulties which make the proofs sometimes more technical: on the one hand, strings on arbitrary gentle bound quivers are not naturally oriented and can self-intersect which sometimes creates difficulties; on the other hand, only distinguishable strings correspond to join-irreducible elements in~$\NKL$.

In all this section, we assume that~$\RNKC$ is finite, so that~$\bar Q$ has finitely many strings and walks.
Equivalently, there is a relation of~$I$ in any non-kissing cycle of~$Q$ (oriented~or~not).

\section{Some algebraic notions on lattices}
\label{sec:recollectionLattice}

\subsection{Canonical join- and meet-representations}
\label{subsec:joinIrreducibleRep}

Consider a finite lattice~$(L,\le,\meet,\join)$ with minimal element~$\hat 0$ and maximal element~$\hat 1$.
Denote by~$x \lessdot y$ the \defn{cover relations} in~$L$, \ie $x < y$ such that there is no~$z \in L$ with~$x < z < y$.

\begin{definition}
An element~$x \in L$ is \defn{join-irreducible} (resp.~\defn{meet-irreducible}) if~$x \ne \hat 0$ (resp.~$x \ne \hat 1$) and~$x = y \join z$ (resp.~$x = y \meet z$) implies~$x = y$ or~$x = z$.
In other words, $x$ covers a unique element~$x_\star$ (resp.~is covered by a unique element~$x^\star$).
We denote by~$\JI(L)$ and~$\MI(L)$ the subposets induced by join and meet irreducible elements of~$L$.
\end{definition}

Join- and meet-irreducible elements are building blocks for join- and meet-representations.

\begin{definition}
A \defn{join-representation} of~$x \in L$ is a subset~$J \subseteq L$ such that~$x = \bigJoin J$.
Such a representation is \defn{irredundant} if~$x \ne \bigJoin J'$ for any strict subset~$J' \subsetneq J$.
The irredundant join-representations of an element~$x \in L$ are ordered by containement of the lower ideals of their elements, \ie~$J \le J'$ if and only if for any~$y \in J$ there exists~$y' \in J'$ such that~$y \le y'$ in~$L$.
The \defn{canonical join-representation} of~$x$ is the minimal irredundant join-representation of~$x$ for this order when it exists.
The \defn{canonical meet-representation} is defined dually.
\end{definition}

Note that the canonical join-representation does not always exists.
However, when it exists, its elements are necessarily join-irreducible.
The existence of canonical join- and meet-representations are guarantied by the following condition~\cite[Thm.~2.24]{FreeseNation}.

\begin{definition}
A lattice is \defn{semi-distributive} if
\[
x \join z = y \join z \Longrightarrow x \join z = (x \meet y) \join z
\qquad\text{and}\qquad
x \meet z = y \meet z \Longrightarrow x \meet z = (x \join y) \meet z
\]
for any~$x, y, z \in L$, or equivalently if any element admits canonical join- and meet-representations.
\end{definition}

The following characterization of the canonical join- and meet-representations in a semidistributive lattice was recently established by E.~Barnard~\cite{Barnard}.

\begin{proposition}[{\cite[Lem.~3.3]{Barnard}}]
\label{prop:canonicalJoinRepresentation}
Consider a semi-distributive lattice~$(L,\le,\join,\meet)$. Then
\begin{enumerate}[(i)]
\item For any cover relation~$x \lessdot y$ in~$L$, there exists a unique join-irreducible~$\ji(x \lessdot y) \in \JI(L)$ such that~$x \join \ji(x \lessdot y) = y$ and~$x \meet \ji(x \lessdot y) = \ji(x \lessdot y)_\star$.
\item The canonical join-representation of~$y \in L$ is given by
\[
y = \bigJoin_{\substack{x \in L \\ x \lessdot y}} \ji(x \lessdot y).
\]
\end{enumerate}
A dual statement also holds for canonical meet-representations.
\end{proposition}

Finally, as they are closed by subsets, canonical join- and meet-representations naturally define simplicial complexes.

\begin{definition}
The \defn{canonical join complex} of~$L$ is the simplicial complex on join-irreducible elements of~$L$ whose faces are the canonical join-representations of the elements of~$L$.
The \defn{canonical meet complex} is defined dually.
\end{definition}

Recall that a simplicial complex is \defn{flag} when its minimal non-faces are edges, \ie when it is the clique complex of its graph.
The following statement was proved by E.~Barnard~\cite{Barnard}.

\begin{theorem}[{\cite[Thm.~1.1]{Barnard}}]
The canonical join complex of a semi-distributive lattice~is~flag.
\end{theorem}

\subsection{Lattice congruences}
\label{sec:latticeTheory}

We now remind the reader with the definition of lattice congruences and refer to~\cite{Reading-LatticeCongruences, Reading-CambrianLattices} for further details.

\begin{definition}
An \defn{order congruence} is an equivalence relation~$\equiv$ on a poset~$P$ such that:
\begin{enumerate}[(i)]
\item Every equivalence class under~$\equiv$ is an interval of~$P$.
\item The projection~$\projDown : P \to P$ (resp.~$\projUp : P \to P$), which maps an element of~$P$ to the minimal (resp.~maximal) element of its equivalence class, is order preserving.
\end{enumerate}
The \defn{quotient}~$P/{\equiv}$ is a poset on the equivalence classes of~$\equiv$, where the order relation is defined by~$X \le Y$ in~$P/{\equiv}$ iff there exists representatives~$x \in X$ and~$y \in Y$ such that~$x \le y$ in~$P$. The quotient~$P/{\equiv}$ is isomorphic to the subposet of~$P$ induced by~$\projDown(P)$ (or equivalently by~$\projUp(P)$).

If moreover~$P$ is a finite lattice, then~$\equiv$ is automatically a \defn{lattice congruence}, meaning that it is compatible with meets and joins: for any~$x \equiv x'$ and~$y \equiv y'$, we have~$x \meet y \equiv x' \meet y'$ and~$x \join y \equiv x' \join y'$. The poset quotient~$P/{\equiv}$ then inherits a lattice structure where the meet~$X \meet Y$ (resp.~the join~$X \join Y$) of two congruence classes~$X$ and~$Y$ is the congruence classe of~$x \meet y$ (resp.~of~$x \join y$) for arbitrary representatives~$x \in X$ and~$y \in Y$.
\end{definition}

\begin{example}
A lattice congruence is illustrated on \fref{fig:exmLatticeQuotient}\,(left).
\end{example}

\begin{remark}
\label{rem:sublattice}
For a lattice congruence~$\equiv$ on a lattice~$L$, the subposet of~$L$ induced by~$\projDown(L)$ is automatically a join-sublattice of~$L$, meaning that for any~$x,y \in L$, we have~${\projDown(x) \join \projDown(y) \in \projDown(L)}$.
However, $\projDown(L)$ can fail to be a sublattice of~$L$, since it might happen that~$\projDown(x) \meet \projDown(y) \notin \projDown(L)$.
The dual remark holds for~$\projUp(L)$.
\end{remark}

Note that for any order congruence~$\equiv$ the projection maps~$\projDown$ and~$\projUp$ satisfy~${\projDown(x) \le x \le \projUp(x)}$, ${\projUp \circ \projUp = \projUp \circ \projDown = \projUp}$ and ${\projDown \circ \projDown = \projDown \circ \projUp = \projDown}$, and $\projUp$ and~$\projDown$ are order preserving.
The following lemma shows the reciprocal statement.
It is proved for example in~\cite[Lem.~4.2]{DermenjianHohlwegPilaud}.

\begin{lemma}
\label{lem:conditionsProjectionMaps}
If two maps~$\projUp : P \to P$ and~$\projDown : P \to P$ satisfy
\begin{enumerate}[(i)]
\item
\label{item:conditionsProjectionMapsSandwich}
$\projDown(x) \le x \le \projUp(x)$ for any element~$x \in P$,

\item
\label{item:conditionsProjectionMapsLastWins}
$\projUp \circ \projUp = \projUp \circ \projDown = \projUp$ and $\projDown \circ \projDown = \projDown \circ \projUp = \projDown$,

\item
\label{item:conditionsProjectionMapsOrderPreserving}
$\projUp$ and~$\projDown$ are order preserving,
\end{enumerate}
then the fibers of~$\projUp$ and~$\projDown$ coincide and the relation~$\equiv$ on~$P$ defined by
\[
x \equiv y \iff \projUp(x) = \projUp(y) \iff \projDown(x) = \projDown(y)
\]
is an order congruence on~$P$ with projection maps~$\projUp$ and~$\projDown$.
\end{lemma}

\begin{definition}
Given two lattice congruences~$\equiv_1, \equiv_2$ on a poset~$P$, we say that~$\equiv_1$ \defn{refines}~$\equiv_2$ (or that~$\equiv_2$ \defn{coarsens}~$\equiv_1$) when~$x \equiv_1 y$ implies~$x \equiv_2 y$.
The set of order congruences~$\Cong(P)$ on~$P$ is partially ordered by refinement.
The minimal element of~$\Cong(P)$ is the complete congruence ($x \equiv y$ for all~$x,y \in P$) while the maximal element of~$\Cong(P)$ is the empty congruence ($x \not\equiv y$ for all~$x \ne y \in P$).
\end{definition}

\subsection{Congruence-uniform lattices}
\label{subsec:congruenceUniformLattices}

A lattice~$L$ is \defn{distributive} when~$(x \join y) \meet z = (x \meet z) \join (y \meet z)$ for any~$x,y,z \in L$ (which implies as well the dual equality~$(x \meet y) \join z = (x \join z) \meet (y \join z)$).
Birkhoff's representation theorem for distributive lattices affirms that any distributive lattice can be represented as the lattice of order ideals of its subposet~$\JI(L)$ of join-irreducible elements.

This result plays a particularly relevant role with respect to lattice congruences.
Namely, the refinement poset~$\Cong(L)$ of lattice congruences on~$L$ is known to be a distributive lattice, which is thus represented by the lattice of order ideals of~$\JI(\Cong(L))$.
For any cover relation~$x \lessdot y$ in~$L$, there exists a finest lattice congruence~$\con(x,y)$ which identifies~$x$ and~$y$.
The congruence~$\con(x,y)$ is join-irreducible in~$\Cong(L)$.
In fact, all join-irreducible elements in~$\Cong(L)$ are actually of the form~$\con(x_\star,x)$ for some join irreducible~$x \in \JI(L)$.
In other words, the map~$x \mapsto \con(x_\star,x)$ defines a surjection from the join-irreducibles of~$\JI(L)$ to that of~$\JI(\Cong(L))$.
Dually, the map~${x \mapsto \con(x,x^\star)}$ defines a surjection from the meet-irreducibles of~$\MI(L)$ to that of~$\MI(\Cong(L))$.
In this section, we will be mainly interested in the following family of lattices, for which the join- and meet-irreducibles of~$\Cong(L)$ are understood directly from those of~$L$.

\begin{definition}
A lattice~$L$ is \defn{congruence-uniform} if the maps
\[
\begin{array}{ccccccc}
\JI(L) & \longrightarrow & \JI\big(\Cong(L)\big) &
\multirow{2}{*}{\qquad\text{and}\qquad} & 
\MI(L) & \longrightarrow & \MI\big(\Cong(L)\big) \\
x & \longmapsto & \con(x_\star,x) & & x & \longmapsto & \con(x,x^\star)
\end{array}
\]
are bijections.
\end{definition}

\begin{remark}
For finite lattices, it is known that 
\begin{itemize}
\item Congruence-uniformity implies semi-distributivity.
\item Congruence-uniformity and semi-distributivity are preserved by lattice quotients: if~$\equiv$ is a lattice congruence on a congruence-uniform (resp.~semi-distributive) lattice~$L$, then~$L/{\equiv}$ is also congruence-uniform (resp.~semi-distributive).
\end{itemize}
\end{remark}

\subsection{Closure operators and biclosed sets}

We now consider closure operators and the corresponding closed sets. In all this section, the ground set~$\cS$ is assumed to be finite.

\begin{definition}
\label{def:closureOperator}
A \defn{closure operator} on a finite set~$\cS$ is a map~$S \mapsto \closure{S}$ on subsets of~$\cS$ such that
\[
\closure{\varnothing} = \varnothing,
\qquad
S \subseteq \closure{S},
\qquad
\closure{(\closure{S})} = \closure{S},
\qquad\text{and}\qquad
S \subseteq T \Longrightarrow \closure{S} \subseteq \closure{T},
\]
for any~$S,T \subseteq \cS$.
A subset~$S \subseteq \cS$ is \defn{closed} if~$\closure{S} = S$, \defn{coclosed} if~$\cS \ssm S$ is closed, and \defn{biclosed} if it is both closed and coclosed.
We denote by~$\Bicl{\cS}$ the inclusion poset of biclosed subsets of~$\cS$.
Finally, we denote by~$\coclosure{S} \eqdef \cS \ssm \closure{(\cS \ssm S)}$ the \defn{coclosure} of~$S \subseteq \cS$.
\end{definition}

Note that~$\Bicl{\cS}$ always has minimal element~$\varnothing$ and maximal element~$\cS$.
We omit to mention the closure operator in the notation~$\Bicl{\cS}$ since it will be fixed once and for all.
The following result developed in~\cite[Sect.~5]{McConville} gives a very convenient criterion to show that the inclusion poset of biclosed sets of~$\cS$ is a congruence-uniform lattice.

\begin{theorem}[{\cite[Thm.~5.5]{McConville}}]
\label{thm:characterizationCongruenceUniform}
If~$(\cS, \vartriangleleft)$ is a poset with a closure operator~$S \mapsto \closure{S}$~such~that
\begin{enumerate}[(i)]
\item for all~$S,T \in \Bicl{\cS}$ with~$S \subsetneq T$, there exists~$\tau \in T \ssm S$ such that~$S \cup \{\tau\} \in \Bicl{\cS}$, 
\item $R \cup \closure{\big((S \cup T) \ssm R\big)} \in \Bicl{\cS}$ for~$R,S,T \in \Bicl{\cS}$ with~$R \subseteq S \cap T$, and
\item if~$\rho,\sigma,\tau \in \cS$ with~$\rho \in \closure{\{\sigma,\tau\}} \ssm \{\sigma,\tau\}$, then~$\sigma \vartriangleleft \rho$ and~$\tau \vartriangleleft \rho$,
\end{enumerate}
then the inclusion poset~$\Bicl{\cS}$ of biclosed sets of~$\cS$ is a congruence-uniform lattice.
\end{theorem}

\begin{remark}
\label{rem:meetJoinBiclosedSets}
Condition~(ii) of Theorem~\ref{thm:characterizationCongruenceUniform} applied to~$R = \varnothing$ implies that
\[
S_1, \dots, S_\ell \in \Bicl{\bar Q} \quad \Longrightarrow \quad \closure{\Big( \bigcup_{i \in [\ell]} S_i \Big)} \in \Bicl{\bar Q}.
\]
In particular the meet and join of biclosed sets~$S_1, \dots, S_\ell \in \Bicl{\bar Q}$ are given by
\[
\bigJoin_{i \in [\ell]} S_i = \closure{\Big( \bigcup_{i \in [\ell]} S_i \Big)}
\qquad\text{and}\qquad
\bigMeet_{i \in [\ell]} S_i = \cS \ssm \closure{\Big( \bigcup_{i \in [\ell]} \cS \ssm S_i \Big)} = \coclosure{\Big( \bigcap_{i \in [\ell]} S_i \Big)}.
\]
\end{remark}

\medskip
Unfortunately, the criterion Theorem~\ref{thm:characterizationCongruenceUniform} does not apply directly when singletons are not closed.
The issue comes from the following observation.

\begin{lemma}
\label{lem:loops}
For any closure operator~$S \mapsto \closure{S}$ on a ground set~$\cS$, any biclosed set~$S \in \Bicl{\cS}$ and any~$\sigma \in \cS$, we have~$\sigma \in S \iff \closure{\{\sigma\}} \subseteq S \iff \closure{\{\sigma\}} \cap S \ne \varnothing$.
\end{lemma}

\begin{proof}
Since~$S$ is closed, we have~$\sigma \in S \Longrightarrow \closure{\{\sigma\}} \subseteq S \Longrightarrow \closure{\{\sigma\}} \cap S \ne \varnothing$.
Since~$S$ is coclosed, we can apply the same argument to the complement of~$S$, and we obtain~$\sigma \notin S \Longrightarrow \closure{\{\sigma\}} \cap S = \varnothing$.
\end{proof}

In particular, the previous lemma implies that if~$\closure{\{\sigma\}} \ne \{\sigma\}$, then $\Bicl{\cS}$ cannot satisfy Condition~(i) of Theorem~\ref{thm:characterizationCongruenceUniform}.

The following statement is an adaptation of Theorem~\ref{thm:characterizationCongruenceUniform} to the case when singletons are not closed.

\begin{theorem}[Adapted from~{\cite[Thm.~5.2 \& Thm.~5.5]{McConville}}]
\label{thm:characterizationCongruenceUniform2}
If~$(\cS, \vartriangleleft)$ is a poset with a closure operator~$S \mapsto \closure{S}$~such~that
\begin{enumerate}[(i)]
\item for each cover relation~$S \subset T$ in~$\Bicl{\cS}$, there is a unique~$\tau \in (T \ssm S)$ such that~${S \cup \closure{\{\tau\}} = T}$,
\item $R \cup \closure{\big((S \cup T) \ssm R\big)} \in \Bicl{\cS}$ for~$R,S,T \in \Bicl{\cS}$ with~$R \subseteq S \cap T$, and
\item if~$\rho,\sigma,\tau \in \cS$ with~$\rho \in \closure{\{\sigma,\tau\}} \ssm \{\sigma,\tau\}$, then~$\sigma \vartriangleleft \rho$ and~$\tau \vartriangleleft \rho$,
\end{enumerate}
then the inclusion poset~$\Bicl{\cS}$ of biclosed sets of~$\cS$ is a congruence-uniform lattice.
\end{theorem}

\begin{proof}
We follow the proofs of~\cite[Thm.~5.2 \& Thm.~5.5]{McConville}. 
Let $R,S,T,U$ be biclosed.
Assume that $S$ and $T$ are distinct and both cover $R$, so that $R = S \wedge T$.
Let $\sigma \in S\ssm R$ and $\tau \in T \ssm R$ be such that $S = R \cup \closure{\{\sigma\}}$ and $T = R \cup \closure{\{\tau\}}$.
The conclusion of the theorem follows from the two points below:
\begin{enumerate}
 \item If $S \wedge U = T \wedge U$, then $(S\vee T) \wedge U = R \wedge U$,
 \item If $V$ is biclosed and such that $S \leq V \lessdot S \vee T$, then $S\vee T = V \cup \closure{\{\tau\}}$.
\end{enumerate}
Point (1) shows that~$\Bicl{\cS}$ is a semi-distributive lattice, while point (2) shows that it admits a CN-labelling, and is thus a congruence-uniform lattice (see~\cite[Thm.~4]{Reading-HyperplaneArrangement}).

In order to prove point (1), we will make use of the following:
\begin{itemize}
 \item[(0)] $S\vee T = R\cup\closure{\{\sigma,\tau\}}$.
\end{itemize}

We now prove the three points above:

\begin{enumerate}[(i)]
\item We have $R \cup \closure{(\closure{\{\sigma\}} \cup \closure{\{\tau\}} )} =   R \cup \closure{\{\sigma,\tau\}}$.
Condition (ii) thus shows that $R \cup \closure{\{\sigma,\tau\}}$ is biclosed.
It follows that $S\vee T = R \cup \closure{\{\sigma,\tau\}}$.

\item We let $S\wedge U$ equal $T\wedge U$. Assume $R\wedge U < (S\vee T)\wedge U$ and let $\rho \in (S\vee T) \ssm R$ satisfy $(R\wedge U) \cup \closure{\{\rho\}} \leq (S\vee T)\wedge U$ in $\Bicl{\cS}$.
Assume that $\rho\in\closure{\{\sigma\}}$.
Then, $\rho\in S$.
Since $\rho\in U$, we have ${\rho \in S \wedge U = R \wedge U}$, which is absurd.
We thus have $\rho \notin \closure{\{\sigma\}}$.
Since $S$ is coclosed, we have $\closure{\{\rho\}} \subseteq \cS\ssm S$.
In particular, $\sigma \notin \closure{\{\rho\}}$.
Similarly $\tau \notin \closure{\{\rho\}}$.
We deduce that ${\closure{\{\sigma\}} \cup \closure{\{\tau\}} \subseteq \cS\ssm \big( (R\wedge U) \cup \closure{\{\rho\}} \big)}$.
This is absurd because the latter set is closed and $\rho$ lies in~$\closure{( \closure{\{\sigma\}} \cup \closure{\{\tau\}} )}$ by (0).

\item Let $V$ be biclosed such that ${S \leq V \lessdot S \vee T}$.
There is a unique $\rho$ such that ${V \cup \closure{\{\rho\}} = S \vee T}$.
Then $\tau\notin V$ (otherwise $V$ would contain $\sigma$ and $\tau$ and thus $S\vee T$).
We thus have $\tau\in\closure{\{\rho\}}$, and $\closure{\{\tau\}}\subseteq\closure{\{\rho\}}$.
Assume that $\rho\notin\closure{\{\tau\}}$.
Then $\rho$ is not in $R\cup\closure{\{\tau\}}$ which is coclosed.
It follows that~$\closure{\{\rho\}}\subseteq \cS\ssm \left(R\cup\closure{\{\tau\}}\right)$ which is absurd.
We have thus proven that~$V \cup \closure{\{\rho\}} = V \cup \closure{\{\tau\}}$, so that~$\rho=\tau$.
\qedhere
\end{enumerate}
\end{proof}

\section{Biclosed sets of strings}
\label{sec:biclosedStrings}

Following~\cite[Sect.~6]{McConville}, we now define a closure operator on the set~$\strings^\pm(\bar Q)$ of undirected strings of~$\bar Q$.
First, for two oriented strings~$\sigma, \tau \in \strings(\bar Q)$, we define
\[
\sigma \circ \tau \eqdef \set{\sigma \alpha \tau}{\alpha \in Q_1 \text{ and } \sigma \alpha \tau \in \strings(\bar Q)}.
\]
This extends to sets~$S,T \subseteq \strings(\bar Q)$ by
\(
S \circ T \eqdef \bigcup_{\sigma \in S, \tau \in T} \sigma \circ \tau.
\)
In particular, for two undirected strings~$\sigma, \tau \in \strings^\pm(\bar Q)$, we can consider~$\sigma \circ \tau$ as the set of strings of~$\strings^\pm(\bar Q)$ obtained by joining an endpoint of~$\sigma$ to an endpoint of~$\tau$ with any arrow~$\alpha \in Q_1$.
Note that~$\sigma \circ \tau$ is not necessary a singleton: it may contain up to~$4$ strings and can also be empty (\eg when no arrow joins an endpoint of~$\sigma$ to an endpoint of~$\tau$).
For a set~$S$ of undirected strings of~$\strings^\pm(\bar Q)$, we then define
\[
\closure{S} \eqdef \bigcup_{\substack{\ell \in \N \\ \sigma_1, \dots, \sigma_\ell \in S}} \sigma_1 \circ \dots \circ \sigma_\ell.
\]
As there is no ambiguity, we abbreviate the notation~$\Bicl{\strings^\pm(\bar Q)}$ by~$\Bicl{\bar Q}$.

\begin{example}
\fref{fig:exmBiclosed} illustrates the notion of closed, coclosed and biclosed sets of strings.

\begin{figure}[t]
	\capstart
	\centerline{\includegraphics[scale=.45]{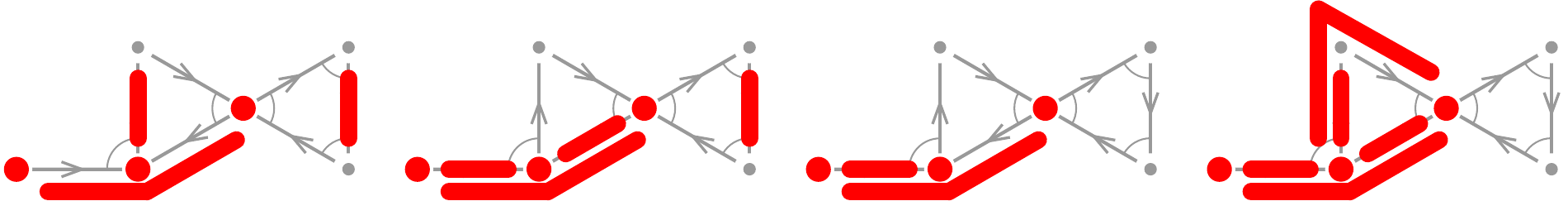}}
	\caption{Four sets of strings on the gentle bound quiver~$\bar Q$ of \fref{fig:exmBlossomingQuiver}\,(left). The first is neither closed nor coclosed, the second is closed but not coclosed, the third is coclosed but not closed, and the fourth is biclosed.}
	\label{fig:exmBiclosed}
\end{figure}
\end{example}

\begin{example}
\fref{fig:exmLatticeQuotient}\,(left) illustrates the inclusion poset of biclosed sets~$\Bicl{\bar Q}$ for a specific gentle bound quiver.
All posets are represented from bottom to top, with their minimal elements at the bottom and their maximal elements on top.
\end{example}

\begin{example}
When~$\bar Q$ is an oriented path with vertices linearly labeled by~$[n]$, a string is determined by its endpoint labels~$i \le j$ and can thus be represented by the pair~$(i,j+1)$.
This maps biclosed sets of strings to inversion sets of permutations of~$\fS_{n+1}$.
\end{example}

\begin{remark}
\label{rem:reverseBiclosed}
With the natural identification between undirected strings of a bound quiver~$\bar Q$ and its reversed bound quiver~$\reversed{\bar Q}$, a subset~$S$ is closed (resp.~coclosed, resp.~biclosed) in~$\strings^\pm(\bar Q)$ if and only if it is closed (resp.~coclosed, resp.~biclosed) in~$\strings^\pm(\reversed{\bar Q})$.
\end{remark}

\enlargethispage{.3cm}
To get used to the definition and to provide more examples of biclosed sets, let us immediately state the following lemma which will be used repeatedly.
Recall from Definition~\ref{def:topBottom} that a substring~$\sigma$ of a string~$\tau$ is a top (resp.~bottom) substring of~$\tau$ if $\tau$ either ends or has an outgoing (resp.~incoming) arrow at each endpoint of~$\sigma$.
The set of top (resp.~bottom) substrings of~$\sigma$ is denoted~$\Sigma_\top(\sigma)$ (resp.~$\Sigma_\bottom(\sigma)$).
Note that~$\tau$ is simultaneously a top and a bottom \mbox{substring of itself}.

\begin{lemma}
\label{lem:exmBiclosed}
For any string~$\sigma \in \strings^\pm(\bar Q)$, the closures~$\closure{\Sigma_\bottom(\sigma)}$ and~$\closure{\Sigma_\top(\sigma)}$ of the set of bottom and top substrings of~$\sigma$ are both biclosed.
Therefore,
\[
\closure{\Big( \bigcup_{i \in [\ell]} \Sigma_\bottom(\sigma_i) \Big)}
\quad\text{and}\quad
\coclosure{\Big( \bigcap_{i \in [\ell]} \closure{\Sigma_\top(\sigma_i)} \Big)}
\]
are both biclosed for any~$\sigma_1, \dots, \sigma_\ell \in \strings^\pm(\bar Q)$.
\end{lemma}

\begin{proof}
The set $\closure{\Sigma_\bottom(\sigma)}$ is closed by definition.
To prove that it is coclosed, consider~${\tau \in \closure{\Sigma_\bottom(\sigma)}}$.
There are substrings~$\tau_1, \dots, \tau_{\ell-1} \in \Sigma_\bottom(\sigma)$, arrows~$\alpha_1, \dots, \alpha_{\ell-1} \in Q_1$ and signs~${\varepsilon_1, \dots, \varepsilon_\ell \in \{-1,1\}}$ such that~$\tau = \tau_1 \alpha_1^{\varepsilon_1} \tau_2 \alpha_2^{\varepsilon_2} \cdots \alpha_{\ell-1}^{\varepsilon_{\ell-1}} \tau_\ell$.
Assume that~$\tau = \tau' \beta^\varepsilon \tau''$ for some strings~$\tau', \tau'' \in \strings(\bar Q)$, an arrow~$\beta \in Q_1$, and a sign~$\varepsilon$.
If~$\beta^\varepsilon$ is one of the arrows~$\alpha_k^{\varepsilon_k}$, then both~$\tau'$ and~$\tau''$ are in~$\closure{\Sigma_\bottom(\sigma)}$.
We can therefore assume that there exists~$k \in [\ell]$ such that~$\beta^\varepsilon$ appears inside the string~$\tau_k$.
Write~$\tau_k = \tau_k' \beta^\varepsilon \tau_k''$.
We distinguish two cases:
\begin{itemize}
\item If~$\varepsilon = -1$, then~$\tau_k'$ is a bottom substring of~$\tau_k$ which is a bottom substring of~$\sigma$, so that~$\tau_k' \in \Sigma_\bottom(\sigma)$ and thus~$\tau' = \tau_1 \alpha_1^{\varepsilon_1} \dots \alpha_{k-1}^{\varepsilon_{k-1}} \tau_k' \in \closure{\Sigma_\bottom(\sigma)}$.
\item Similarly, if~$\varepsilon = 1$, then~$\tau''_k$ is a bottom substring of~$\tau_k$ which is a bottom substring of~$\sigma$, so that~$\tau_k'' \in \Sigma_\bottom(\sigma)$ and thus~$\tau'' = \tau_k'' \alpha_k^{\varepsilon_k} \dots \alpha_{\ell-1}^{\varepsilon_{\ell-1}} \tau_\ell \in \closure{\Sigma_\bottom(\sigma)}$.
\end{itemize}
We conclude that~$\closure{\Sigma_\bottom(\sigma)}$ is coclosed, and thus biclosed.
By Remark~\ref{rem:meetJoinBiclosedSets}, $\closure{\big( \bigcup_{i \in [\ell]} \Sigma_\bottom(\sigma_i) \big)}$ is also biclosed for any~$\sigma_1, \dots, \sigma_\ell \in \strings^\pm(\bar Q)$.
The proof is similar for~$\closure{\Sigma_\top(\sigma)}$ or follows by duality from Remarks~\ref{rem:reverseBlossomingQuiver} and~\ref{rem:reverseBiclosed}.
\end{proof}

The next result is the keystone of this section.

\begin{theorem}
For any gentle bound quiver~$\bar Q$ with finitely many strings, the inclusion poset of biclosed sets~$\Bicl{\bar Q}$ is a congruence-uniform lattice.
\end{theorem}

\begin{proof}
The proof is inspired from that of~\cite[Thm.~6.5]{McConville}.
The main difficulty is that~$\Bicl{\bar Q}$ does not satisfy Condition~(i) of Theorem~\ref{thm:characterizationCongruenceUniform} when~$\bar Q$ has a loop.
For example, for the quiver~$\bar Q$ with a single vertex~$v$, a loop~$\alpha$, and the relation~$\alpha^2 = 0$, the biclosed sets of strings are precisely~$\varnothing$ and~$\{\varepsilon_v, \alpha\}$.
We therefore need the adapted criterion of Theorem~\ref{thm:characterizationCongruenceUniform2}.
For a string~$\sigma \in \strings^\pm(\bar Q)$, the closure~$\closure{\{\sigma\}}$ can be of two types:
\begin{itemize}
\item if there is a loop~$\alpha \in Q_1$ such that~$\sigma \alpha \in \strings^\pm(\bar Q)$, then~$\closure{\{\sigma\}} = \{\sigma, \sigma \alpha \sigma^{-1}\}$,
\item otherwise, $\closure{\{\sigma\}} = \{\sigma\}$.
\end{itemize}
In particular, for any two strings~$\sigma, \tau$, the equality~$\closure{\{\sigma\}}=\closure{\{\tau\}}$ implies $\sigma=\tau$.
We note that $\strings^\pm(\bar Q)$ is partially ordered by~$\sigma \vartriangleleft \tau$ if and only if~$\sigma$ is a substring of~$\tau$.

\medskip\noindent
{\bf (i)}
Consider a cover relation~$S \subset T$ in~$\Bicl{\bar Q}$.
If~$\tau$ is a minimal length string in~$T \ssm S$, then for any~$\sigma, \sigma' \in \strings^\pm(\bar Q)$ with~$\tau \in \sigma \circ \sigma'$, either~$\sigma \in S$ or~$\sigma' \in S$.
Consider now a maximal length string~$\tau$ in~$T \ssm S$ such that for any~$\sigma, \sigma' \in \strings^\pm(\bar Q)$ with~$\tau \in \sigma \circ \sigma'$, either~$\sigma \in S$ or~$\sigma' \in S$.
We claim that~$S \cup \closure{\{\tau\}} \in \Bicl{\bar Q}$.
Observe first that~$S \cup \closure{\{\tau\}}$ is coclosed by the property defining~$\tau$.
Assume by means of contradiction that~$S \cup \closure{\{\tau\}}$ is not closed.
Since~$S$ and~$\closure{\{\tau\}}$ are both closed, it implies that there exists~$\sigma \in S$ such that~$\sigma \circ \tau \not\subseteq S$.
Consider a string~$\sigma \in S$ of minimal length such that~$\sigma \circ \tau \not\subseteq S$.
Let~$\rho \in (\sigma \circ \tau) \ssm S$.
Since~$S \cup \{\tau\} \subseteq T$ and~$T$ is closed, we have~$\rho \in T$.
By maximality of~$\tau$, we thus obtain that there exist~$\sigma', \tau' \in \strings^\pm(\bar Q) \ssm S$ such that~$\rho \in \sigma' \circ \tau'$.
Up to exchanging~$\sigma'$ and~$\tau'$, we have $\sigma \subseteq \sigma'$ and~$\tau \supseteq \tau'$ or \viceversa.
We distinguish these~two~cases:
\begin{itemize}
\item Either there exists a string~$\zeta$ such that~$\sigma' \in \sigma \circ \zeta$ and~$\tau \in \zeta \circ \tau'$. By definition of~$\tau$, we obtain that~$\zeta \in S$. Since~$\sigma, \zeta \in S$ while~$\sigma' \notin S$, this contradicts the closedness of~$S$.
\item Or there exists a string~$\zeta$ such that~$\sigma \in \sigma' \circ \zeta$ and~$\tau' \in \zeta \circ \tau$. We have~$\zeta \notin S$ (otherwise~${\tau, \zeta \in S}$ and~$\tau' \notin S$ would contradict the closedness of~$S$). We obtain a string~$\zeta \in S$ with~$\zeta \circ \tau \not\subseteq S$ and shorter than~$\sigma$, thus contradicting the minimality~of~$\sigma$.
\end{itemize}
Since we reach a contradiction in both cases, we conclude that~$S \cup \closure{\{\tau\}}$ is biclosed.
Finally, since $T$ covers~$S$ in~$\Bicl{\bar Q}$, we obtain that~$S \cup \closure{\{\tau\}} = T$.
The uniqueness is immediate since~$\closure{\{\tau\}} = T \ssm S$ determines~$\tau$ as observed above.
This concludes the proof of~(i).

\medskip\noindent
{\bf (ii)}
We proceed by induction on~$|\strings^\pm(\bar Q) \ssm R|$. The property clearly holds if~${R = S = T = \strings^\pm(\bar Q)}$. Consider thus~$R,S,T \in \Bicl{\bar Q}$ such that the property holds for any~$R' \supsetneq R$. The result is immediate when~$S \subseteq T$ (and similarly when~$S \supseteq T$) since~$R \cup \closure{((S \cup T) \ssm R)} = T \in \Bicl{\bar Q}$ in this case. Assume thus that~$S \ssm T \ne \varnothing$ and~$T \ssm S \ne \varnothing$. By~(i), since~$R \subsetneq S$ and~$R \subsetneq T$, there exist~$\sigma \in S \ssm R$ and~$\tau \in T \ssm R$ such that~$R \cup \closure{\{\sigma\}}$ and~$R \cup \closure{\{\tau\}}$ are biclosed. We claim that~$X \eqdef R \cup \closure{\{\sigma,\tau\}}$ is biclosed:
\begin{description}
\item[closed] Since~$R \cup \closure{\{\sigma\}}$ and~$R \cup \closure{\{\tau\}}$ are closed, we have~$\rho \circ \zeta \subseteq R$ for any~$\rho \in R$~and~${\zeta \in \{\sigma,\tau\}}$. An immediate induction thus shows that~$\rho \circ \zeta_1 \circ \dots \circ \zeta_\ell \subseteq R$ for any~$\zeta_1, \dots, \zeta_\ell \in \{\sigma, \tau\}$. Therefore, $R \circ \zeta \subseteq R$ for any~$\zeta \in \closure{\{\sigma,\tau\}}$. This implies that~$X$ is closed.
\item[coclosed] Assume by means of contradiction that~$X$ is not coclosed. Then there exists~$\sigma', \tau' \in \strings^\pm(\bar Q) \ssm X$ such that~$X \cap (\sigma' \circ \tau') \ne \varnothing$. Since~$R \cup \closure{\{\sigma\}}$ and~$R \cup \closure{\{\tau\}}$ are coclosed, we can assume that~$(\sigma \circ \tau) \cap (\sigma' \circ \tau') \ne \varnothing$. Therefore, up to exchanging~$\sigma'$ and~$\tau'$, we have either~$\sigma \subseteq \sigma'$ and~$\tau \supseteq \tau'$, or~$\sigma \supseteq \sigma'$ and~$\tau \subseteq \tau'$. Say for example that the former holds. Then there exists~$\theta \in \strings^\pm(\bar Q)$ such that~$\sigma' \in \sigma \circ \theta$ and~$\tau \in \theta \circ \tau'$. This implies that~$\theta \notin R$ (because~$R \cup \closure{\{\sigma\}}$ is closed) and that~$\theta \in R$ (because~$R \cup \closure{\{\tau\}}$ is coclosed and~$\theta \subsetneq \tau$), a contradiction.
\end{description}
We will now use the biclosed set~$X = R \cup \closure{\{\sigma, \tau\}}$ to prove that the property holds for~$T$.
We first apply our induction hypothesis to~$R' = R \cup \{\sigma\}$, $S' = S$ and~$T' = X$. We obtain that
\[
(R \cup \{\sigma\}) \cup \closure{\big( ( S \cup X ) \ssm (R \cup \{\sigma\}) \big)}
\]
is biclosed. However, we have
\[
S \cup X \subseteq (R \cup \{\sigma\}) \cup \closure{\big( ( S \cup X ) \ssm (R \cup \{\sigma\}) \big)} \subseteq \closure{(S \cup X)}.
\]
Since~$\closure{(S \cup X)}$ is the smallest closed set containing~$S \cup X$, we obtain that
\[
(R \cup \{\sigma\}) \cup \closure{\big( ( S \cup X ) \ssm (R \cup \{\sigma\}) \big)} = \closure{(S \cup X)}.
\]
Applying the induction hypothesis to~$R' = R \cup \{\tau\}$, $S' = X$ and~$T' = T$, we obtain similarly that
\[
(R \cup \{\tau\}) \cup \closure{\big( ( X \cup T ) \ssm (R \cup \{\tau\}) \big)} = \closure{(X \cup T)}.
\]
Applying our induction hypothesis to~$R' = X$, $S' = \closure{(S \cup X)}$ and~$T' = \closure{(X \cup T)}$, we obtain that
\[
X \cup \closure{\big( \big( \closure{(S \cup X)} \cup \closure{(X \cup T)} \big) \ssm X \big)}
\]
is biclosed. Since~$\closure{\big( \big( \closure{(S \cup X)} \cup \closure{(X \cup T)} \big) \ssm X \big)}$ contains~$S \cup T$ and is closed, it also contains~$\closure{(S \cup T)}$.
Moreover,
\begin{align*}
X \cup \closure{\big( \big( \closure{(S \cup X)} \cup \closure{(X \cup T)} \big) \ssm X \big)}
& \subseteq X \cup \closure{\big( \closure{\big( (S \cup T \cup X) \ssm R\big)} \ssm X \big)} \\
& \subseteq R \cup \closure{\big( \closure{\big( (S \cup T \cup X) \ssm R\big)} \ssm R \big)} \\
& = R \cup \closure{\big( (S \cup T \cup X) \ssm R \big)} \\
& = R \cup \closure{\big( (S \cup T) \ssm R \big)} \\
& \subseteq \closure{(S \cup T)}.
\end{align*}
Therefore, all these sets coincide and~$R \cup \closure{\big( (S \cup T) \ssm R \big)}$ is biclosed, concluding the proof of~(ii).

\medskip\noindent
{\bf (iii)}
Consider~$\sigma, \tau \in \strings^\pm(\bar Q)$.
Any string~$\rho \in \closure{\{\sigma, \tau\}} \ssm \{\sigma, \tau\}$ is obtained by concatenations of copies of~$\sigma$ and~$\tau$ with arrows of~$Q_1$.
We thus immediately obtain that~$\sigma$ and~$\tau$ are substrings of~$\rho$.
\end{proof}

\section{Lattice congruence on biclosed sets}
\label{sec:latticeCongruence}

We now define a lattice congruence on biclosed sets of strings.
Our definition is borrowed from~\cite[Sect.~7]{McConville}.

\begin{definition}
\label{def:downUpProjections}
For a biclosed set~$S \in \Bicl{\bar Q}$, define
\[
\projDown(S) \eqdef \set{\sigma \in \strings^\pm(\bar Q)}{\Sigma_\bottom(\sigma) \subseteq S}
\qquad\text{and}\qquad
\projUp(S) \eqdef \set{\sigma \in \strings^\pm(\bar Q)}{\Sigma_\top(\sigma) \cap S \ne \varnothing}.
\]
\end{definition}

\begin{example}
The down and up projections~$\projDown$ and~$\projUp$ are illustrated on \fref{fig:exmProjections}.

\begin{figure}[t]
	\capstart
	\centerline{
	\begin{overpic}[scale=.45]{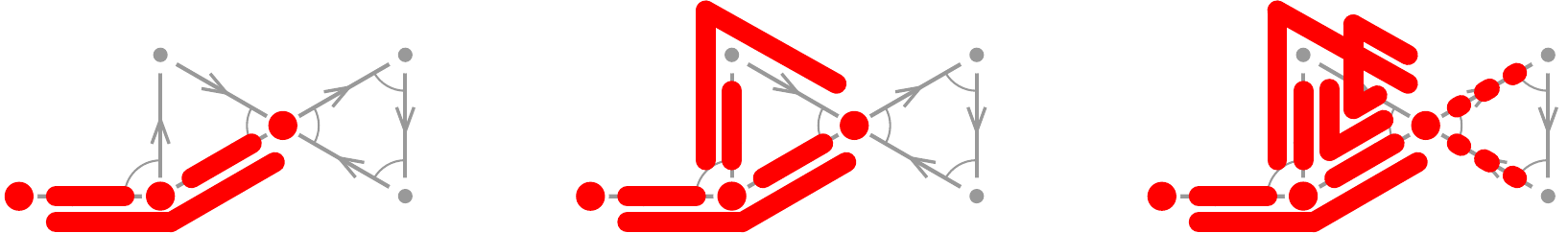}
			\put(30,6){$\longleftarrow$}
			\put(31,8){$\projDown$}
			\put(68,6){$\longrightarrow$}
			\put(69,8){$\projUp$}
			\put(12,-3){$\projDown(S)$}
			\put(50,-3){$S$}
			\put(85,-3){$\projUp(S)$}
	\end{overpic}
	}
	\vspace{.3cm}
	\caption{A biclosed set of strings~$S \in \Bicl{\bar Q}$ (middle), where~$\bar Q$ is the bound quiver of~\fref{fig:exmBlossomingQuiver}, and its images~$\projDown(S)$ (left) and~$\projUp(S)$ (right) by the down and up projections of Definition~\ref{def:downUpProjections}. For space reasons, $\projUp(S)$ is only represented partially: the remaining strings are obtained by adding independently the two dotted arrows to all strings containing their left endpoint.}
	\label{fig:exmProjections}
\end{figure}
\end{example}

Following Remarks~\ref{rem:reverseBlossomingQuiver}, \ref{rem:reverseNKC}, and~\ref{rem:reverseBiclosed}, we start with a simple observation.

\begin{lemma}
\label{rem:reversedProjections}
For any biclosed set~$S \in \Bicl{\bar Q}$,
\[
\strings^\pm(\bar Q) \ssm \projUp(S) = \projDown \big( \strings^\pm(\reversed{\bar Q}) \ssm S \big).
\]
\end{lemma}

\begin{proof}
On the one hand, $\strings^\pm(\bar Q) \ssm \projUp(S)$ is the set of strings of~$\strings^\pm(\bar Q)$ with no top substrings contained in~$S$, \ie whose top substrings are all in~$\strings^\pm(\bar Q) \ssm S$.
On the other hand, ${\projDown(\strings^\pm(\reversed{\bar Q}) \ssm S)}$ is the set of strings of~$\strings^\pm(\reversed{\bar Q})$ whose bottom substrings are all in~$\strings^\pm(\reversed{\bar Q}) \ssm S$.
Since reversing all arrows preserves the strings but exchanges top with bottom substrings, we conclude that the two sets coincide. 
\end{proof}

\begin{lemma}
For any biclosed set~$S \in \Bicl{\bar Q}$, the sets~$\projDown(S)$ and~$\projUp(S)$ are biclosed.
\end{lemma}

\begin{proof}
Let~$\rho, \sigma, \tau \in \strings^\pm(\bar Q)$ such that~$\rho \in \sigma \circ \tau$.
Assume first that~$\sigma, \tau \in \projDown(S)$.
Consider a bottom substring~$\rho'$ of~$\rho$.
We distinguish two cases:
\begin{itemize}
\item If~$\rho'$ is a substring of~$\sigma$ or~$\tau$, then it is a bottom substring and thus belongs to~$S$.
\item Otherwise, $\rho' \in \sigma' \circ \tau'$ for some substrings~$\sigma'$ of~$\sigma$ and~$\tau'$ of~$\tau$. These substrings are again bottom substrings so that~$\sigma' \in S$ and~$\tau' \in S$, which implies that~$\rho' \in S$ since~$S$ is closed.
\end{itemize}
In both cases, we showed that any bottom substring~$\rho'$ of~$\rho$ belongs to~$S$, so that~$\rho \in \projDown(S)$ and thus~$\projDown(S)$ is closed.

Assume now that~$\rho \in \projDown(S)$.
Since~$\rho \in \sigma \circ \tau$, either~$\sigma$ or~$\tau$ is a bottom substring of~$\rho$, say~$\sigma$ without loss of generality.
Any bottom substring of~$\sigma$ is then a bottom substring of~$\rho$ and thus belongs to~$S$ since~$\rho \in \projDown(S)$.
We conclude that~$\sigma \in \projDown(S)$ which proves that~$\projDown(S)$ is coclosed.

We thus proved that~$\projDown(S)$ is biclosed.
The proof is similar for~$\projUp(S)$ or follows by duality from Remarks~\ref{rem:reverseBlossomingQuiver} and~\ref{rem:reverseBiclosed} and Lemma~\ref{rem:reversedProjections}.
\end{proof}

\begin{proposition}
\label{prop:latticeCongruence}
The two maps~$\projDown : \Bicl{\bar Q} \to \Bicl{\bar Q}$ and~$\projUp : \Bicl{\bar Q} \to \Bicl{\bar Q}$ satisfy
\begin{enumerate}[(i)]
\item
$\projDown(S) \subseteq S \subseteq \projUp(S)$ for any element~$S \in \Bicl{\bar Q}$,

\item
$\projDown \circ \projDown = \projDown \circ \projUp = \projDown$ and $\projUp \circ \projUp = \projUp \circ \projDown = \projUp$,

\item
$\projDown$ and~$\projUp$ are order preserving.
\end{enumerate}
Therefore, the fibers of~$\projUp$ and~$\projDown$ coincide and the relation~$\equiv$ on~$\Bicl{\bar Q}$ defined by
\[
S \equiv T \iff \projDown(S) = \projDown(T) \iff \projUp(S) = \projUp(T)
\]
is an order congruence on~$\Bicl{\bar Q}$.
\end{proposition}

\begin{proof}
For~(i), we just observe that any string~$\tau$ is a top and a bottom substring of itself.

For~(iii), it follows from the definition that~$S \subseteq T$ implies~$\projDown(S) \subseteq \projDown(T)$ and~$\projUp(S) \subseteq \projUp(T)$.

For~(ii), we already obtain from~(i) and~(iii) that $\projDown(\projDown(S)) \subseteq \projDown(S)$ and~$\projDown(\projUp(S)) \supseteq \projDown(S)$.
We then observe that if~$\rho, \sigma, \tau \in \strings^\pm(\bar Q)$ are such that $\rho$ is a bottom substring of~$\sigma$ and~$\sigma$ is a bottom substring of~$\tau$, then~$\rho$ is a bottom substring of~$\tau$.
This shows that~$\projDown(\projDown(S)) \supseteq \projDown(S)$.
Finally, we prove that~$\projDown(\projUp(S)) \subseteq \projDown(S)$.
By means of contradiction, assume that there exists~${\tau \in \projDown(\projUp(S)) \ssm \projDown(S)}$.
Since~$\tau \notin \projDown(S)$, there exists a bottom substring~$\sigma$ of~$\tau$ which is not in~$S$.
We can assume that~$\sigma$ is an inclusion minimal such substring of~$\tau$.
Since~$\tau \in \projDown(\projUp(S))$, we have~$\sigma \in \projUp(S)$, so that there exists a top substring~$\rho$ of~$\sigma$ which belongs to~$S$.
The substring~$\rho$ decomposes~$\sigma$ into~$\sigma = \sigma' \circ \rho \circ \sigma''$, where~$\sigma'$ and~$\sigma''$ are bottom substrings of~$\sigma$, and thus of~$\tau$.
By minimality of~$\sigma$, we have~$\sigma' \in S$ and~$\sigma'' \in S$.
Since~$S$ is closed, we obtain that~$\sigma \in \sigma' \circ \rho \circ \sigma'' \subseteq S$, a contradiction.
We therefore proved that~$\projDown \circ \projDown = \projDown \circ \projUp = \projDown$.
Using the symmetry presented in Lemma~\ref{rem:reversedProjections}, this implies as well that~$\projUp \circ \projUp = \projUp \circ \projDown = \projUp$.
\end{proof}

\begin{example}
\fref{fig:exmLatticeQuotient}\,(left) illustrates this lattice congruence on~$\Bicl{\bar Q}$ for a specific gentle bound quiver.
Congruence classes are represented by blue rectangles.

\begin{figure}[t]
	\capstart
	\centerline{\includegraphics[width=1.1\textwidth]{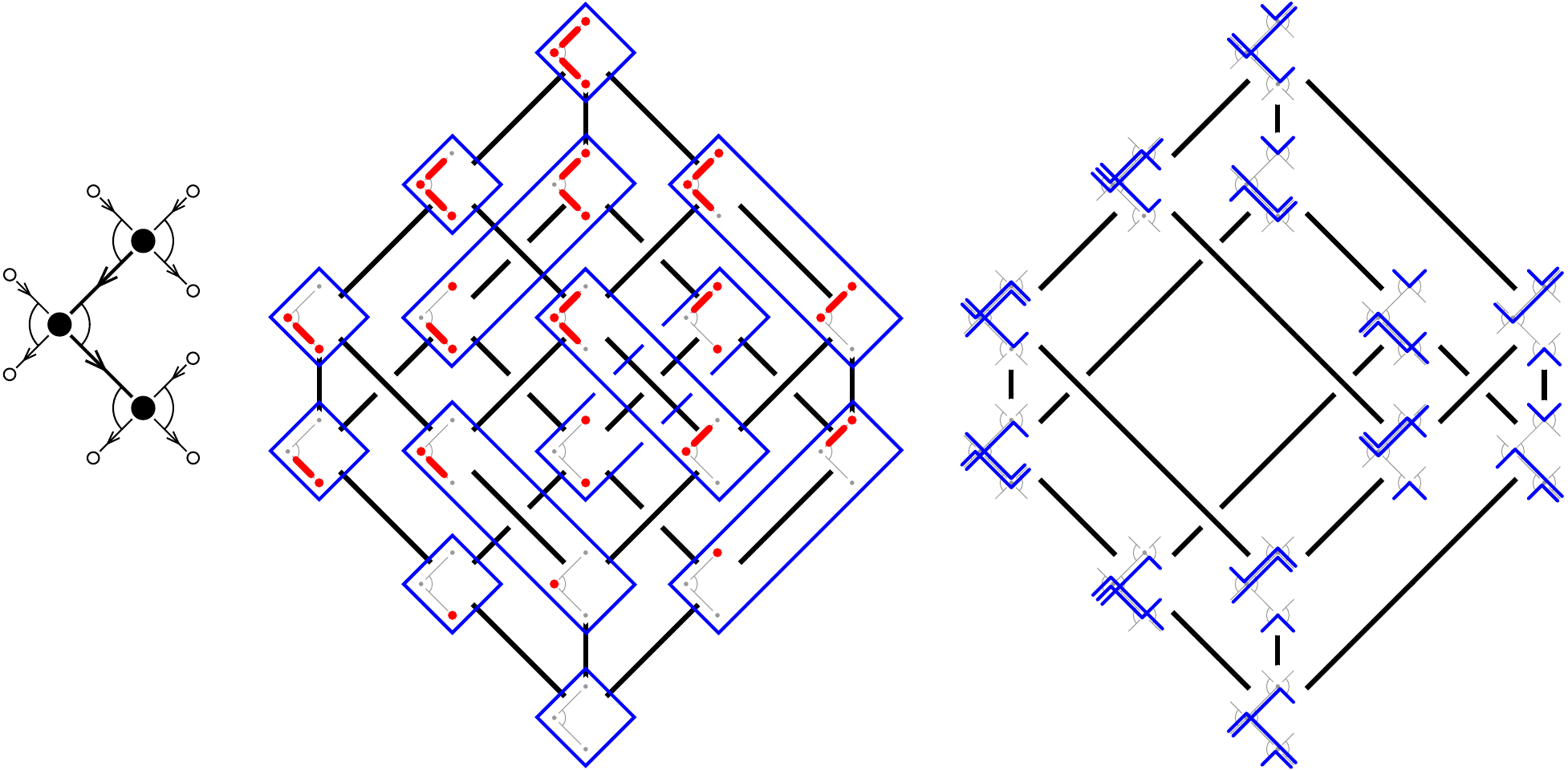}}
	\caption{The inclusion lattice of biclosed sets~$\Bicl{\bar Q}$ with congruence classes of~$\equiv$ in blue (left), and the corresponding lattice of increasing flips on facets of~$\NKC$ (right).}
	\label{fig:exmLatticeQuotient}
\end{figure}
\end{example}

To conclude this section, we show that the examples of biclosed sets~$S$ seen in Lemma~\ref{lem:exmBiclosed} are fixed by~$\projDown$ (resp.~$\projUp$).
In other words, they correspond to certain bottom (resp.~top) elements of the congruence classes of~$\equiv$.
We will see in Section~\ref{sec:biclosedSetsToNKF} that all bottom (resp.~top) elements of $\equiv$-classes are in fact of this form, and in Section~\ref{sec:nonFriendlyComplex} that one can in fact restrict to distinguishable~strings.

\begin{lemma}
\label{lem:fixProjDown}
For any string~$\sigma \in \strings^\pm(\bar Q)$, the biclosed set~$\closure{\Sigma_\bottom(\sigma)}$ is fixed by~$\projDown$ while~$\closure{\Sigma_\top(\sigma)}$ is fixed by~$\projUp$.
Therefore,
\[
\projDown \bigg( \closure{\Big( \bigcup_{i \in [\ell]} \Sigma_\bottom(\sigma_i) \Big)} \bigg) = \closure{\Big( \bigcup_{i \in [\ell]} \Sigma_\bottom(\sigma_i) \Big)}
\;\;\text{and}\;\;
\projUp \bigg( \coclosure{\Big( \bigcap_{i \in [\ell]} \closure{\Sigma_\top(\sigma_i)} \Big)} \bigg) = \coclosure{\Big( \bigcap_{i \in [\ell]} \closure{\Sigma_\top(\sigma_i)} \Big)}
\]
for any~$\sigma_1, \dots, \sigma_\ell \in \strings^\pm(\bar Q)$.
\end{lemma}

\begin{proof}
Let~$\tau \in \closure{\Sigma_\bottom(\sigma)}$ and~$\rho \in \Sigma_\bottom(\tau)$.
By definition, there exist strings~${\tau_1, \dots, \tau_\ell \in \Sigma_\bottom(\sigma)}$, arrows~$\alpha_1, \dots, \alpha_{\ell-1} \in Q_1$ and signs~$\varepsilon_1, \dots, \varepsilon_{\ell-1} \in \{-1,1\}$ such that~$\tau = \tau_1 \alpha_1^{\varepsilon_1} \dots \alpha_{\ell-1}^{\varepsilon_{\ell-1}} \tau_\ell$.
Since~${\rho \in \Sigma_\bottom(\tau)}$, there exist indices~$1 \le i < j \le \ell$ and substrings~$\tau_i' \in \Sigma_\bottom(\tau_i)$ and~${\tau_j' \in \Sigma_\bottom(\tau_j)}$ such that ${\rho = \tau_i' \alpha_i^{\varepsilon_i} \tau_{i+1} \dots \tau_{j-1} \alpha_{j-1}^{\varepsilon_{j-1}} \tau_j'}$.
We thus obtain~$\rho \in \closure{\Sigma_\bottom(\sigma)}$ so that~${\Sigma_\bottom(\tau) \subseteq \closure{\Sigma_\bottom(\sigma)}}$.
This shows that~$\closure{\Sigma_\bottom(\sigma)} \subseteq \projDown(\closure{\Sigma_\bottom(\sigma)})$ and thus the equality by Proposition~\ref{prop:latticeCongruence}\,(i).

Consider now~$\sigma_1, \dots, \sigma_\ell \in \strings^\pm(\bar Q)$.
We have
\[
\projDown \bigg( \bigJoin_{i\in [\ell]} \closure{\Sigma_\bottom(\sigma_i)} \bigg) = \bigJoin_{i \in [\ell]} \projDown \big( \closure{\Sigma_\bottom(\sigma_i)} \big) = \bigJoin_{i \in [\ell]} \closure{\Sigma_\bottom(\sigma_i)},
\]
so that
\[
\projDown \bigg( \closure{\Big( \bigcup_{i \in [\ell]} \Sigma_\bottom(\sigma_i) \Big)} \bigg) = \closure{\Big( \bigcup_{i \in [\ell]} \Sigma_\bottom(\sigma_i) \Big)}.
\]
The proof is similar for~$\closure{\Sigma_\top(\sigma)}$ or follows by duality from Remark~\ref{rem:reverseBlossomingQuiver} and Lemma~\ref{rem:reversedProjections}.
\end{proof}

\section{Biclosed sets of strings versus non-kissing facets}
\label{sec:biclosedSetsToNKF}

The goal of this section is to show that the non-kissing oriented flip graph~$\NKG$ is isomorphic to the Hasse diagram of the lattice quotient of the lattice of biclosed sets~$\Bicl{\bar Q}$ by the lattice congruence~$\equiv$.
For this, we first need to define inverse bijections between congruence classes of biclosed sets of~$\Bicl{\bar Q}$ and non-kissing facets of~$\NKC$.

\subsection{From biclosed sets of strings to non-kissing facets}

We first define a map~$\eta$ from biclosed sets in~$\Bicl{\bar Q}$ to subsets of walks of~$\walks^\pm(\bar Q)$.
Although slightly more technical to cover the general situation of gentle bound quivers, our definition is directly inspired by~\cite[Sect.~8]{McConville}.

\begin{definition}
\label{def:fromBiclToNKC}
For~$S \in \Bicl{\bar Q}$ and~$\alpha \in Q_1\blossom$, let
\(
{\omega(\alpha,S) \eqdef \alpha_{-\ell}^{\varepsilon_{-\ell}} \cdots \alpha_{-1}^{\varepsilon_{-1}} \cdot \alpha \cdot \alpha_1^{\varepsilon_1} \cdots \alpha_r^{\varepsilon_r}}
\)
be the directed walk containing~$\alpha$ defined by
\begin{itemize}
\item $\varepsilon_i = -1$ if the string~$\alpha_1^{\varepsilon_1} \cdots \alpha_{i-1}^{\varepsilon_{i-1}}$ belongs to~$S$, and~$\varepsilon_i = 1$ otherwise, for all~$i \in [r]$, and
\item $\varepsilon_{-i} = 1$ if the string~$\alpha_{-i+1}^{\varepsilon_{-i+1}} \cdots \alpha_{-1}^{\varepsilon_{-1}}$ belongs to~$S$, and~$\varepsilon_i = -1$ otherwise, for all~$i \in [\ell]$.
\end{itemize}
\end{definition}

\begin{remark}
Observe that~$\omega(\alpha,S)$ is well-defined and unique for any~$S \in \Bicl{\bar Q}$ and~$\alpha \in Q_1$.
Indeed, one can start from the string~$\alpha$ and let it grow in both directions according to the local rules of Definition~\ref{def:fromBiclToNKC} until it reaches some blossom vertices of~$\bar Q\blossom$.
Note that~$\ell = 0$ (resp.~$r = 0$) when~$\alpha$ is an incoming (resp.~outgoing) blossom arrow.
Note also that for~$i = 1$, we consider $\alpha_1^{\varepsilon_1} \dots \alpha_{i-1}^{\varepsilon_{i-1}}$ to be~$t(\alpha)$ and similarly~$\alpha_{-i+1}^{\varepsilon_{-i+1}} \dots \alpha_{-1}^{\varepsilon_{-1}}$ to be~$s(\alpha)$.
\end{remark}


\begin{lemma}
\label{lem:distinctDirectedWalks}
As directed walks, $\omega(\alpha,S) \ne \omega(\beta,S)$ for~$\alpha \ne \beta \in Q_1\blossom$.
\end{lemma}

\begin{proof}
Assume that~$\alpha \ne \beta \in Q_1\blossom$ are such that~$\omega \eqdef \omega(\alpha,S) = \omega(\beta,S)$.
We can assume without loss of generality that~$\alpha$ appears before~$\beta$ along~$\omega$.
Let~$\sigma$ denote the substring of~$\omega$ between~$\alpha$ and~$\beta$.
We obtain that~$\sigma \in S$ since~$\alpha$ is incoming at~$s(\sigma)$, and that~$\sigma \notin S$ since~$\beta$ is outgoing at~$t(\sigma)$, a contradiction.
\end{proof}

Recall from Definition~\ref{def:substrings} that a string~$\sigma \in \strings(\bar Q)$ is a \defn{top substring} (resp.~\defn{bottom substring}) of a walk~$\omega \in \walks(\bar Q)$ if the arrows of~$\omega \ssm \sigma$ incident to the endpoints of~$\sigma$ are both outgoing (resp.~incoming).

\begin{lemma}
\label{lem:positiveNegativeOrientation}
For any~$S \in \Bicl{\bar Q}$ and~$\alpha \in Q_1$, all bottom (resp.~top) substrings of~$\omega(\alpha,S)$ belong to~$S$ (resp.~to the complement of~$S$).
\end{lemma}

\begin{proof}
Write~$\omega(\alpha,S) = \alpha_{-\ell}^{\varepsilon_{-\ell}} \cdots \alpha_{-1}^{\varepsilon_{-1}} \alpha \alpha_1^{\varepsilon_1} \cdots \alpha_r^{\varepsilon_r}$ and consider a substring~$\sigma = \alpha_i^{\varepsilon_i} \cdots \alpha_j^{\varepsilon_j}$ ($i < j$).
Assume for instance that~$\sigma$ is a bottom substring of~$\omega(\alpha,S)$, \ie that~$\varepsilon_{i-1} = 1$ while~$\varepsilon_{j+1} = -1$.
We distinguished three cases:
\begin{itemize}
\item If~$j < 0$, then~$\alpha_i^{\varepsilon_i} \cdots \alpha_{-1}^{\varepsilon_{-1}} \in S$ (since~$\varepsilon_{i-1} = 1$) while~$\alpha_{j+2}^{\varepsilon_{j+2}} \cdots \alpha_{-1}^{\varepsilon_{-1}} \notin S$ (since~$\varepsilon_{j+1} = -1$).
\item If~$i < 0 < j$, then~$\alpha_i^{\varepsilon_i} \cdots \alpha_{-1}^{\varepsilon_{-1}} \in S$  (since~$\varepsilon_{i-1} = 1$) and~$\alpha_1^{\varepsilon_1} \cdots \alpha_j^{\varepsilon_j} \in S$ (since~${\varepsilon_{j+1} = -1}$).
\item If~$0 < i$, then~$\alpha_1^{\varepsilon_1} \cdots \alpha_{i-2}^{\varepsilon_{i-2}} \notin S$ (since~$\varepsilon_{i-1} = 1$) while~$\alpha_1^{\varepsilon_1} \cdots \alpha_j^{\varepsilon_j} \in S$ (since~$\varepsilon_{j+1} = -1$).
\end{itemize}
In all cases, we obtain that~$\sigma = \alpha_i^{\varepsilon_i} \cdots \alpha_j^{\varepsilon_j}$ must belong to~$S$ since~$S$ is biclosed.
The case of top substrings is symmetric, or can be deduced from the case of bottom substrings by Lemma~\ref{rem:reversedProjections}.
\end{proof}

\begin{corollary}
\label{coro:surjectionNonKissingWalks}
For any~$\alpha, \beta \in Q_1\blossom$, the walks~$\omega(\alpha,S)$ and~$\omega(\beta,S)$ (or its inverse) are non-kissing.
In particular, $\omega(\alpha,S)$ is not self-kissing.
\end{corollary}

\begin{proof}
Assume that~$\omega(\alpha,S)$ kisses~$\omega(\beta,S)$ along~$\sigma$, which is thus a maximal substring of~$\omega(\alpha,S)$ and~$\omega(\beta,S)$ (or its inverse).
By Lemma~\ref{lem:positiveNegativeOrientation}, we have~$\sigma \notin S$ since~$\sigma$ is a top substring of~$\omega(\alpha,S)$, while~$\sigma \in S$ since~$\sigma$ is a bottom substring of~$\omega(\beta,S)$, a contradiction.
\end{proof}

\begin{corollary}
\label{coro:eta}
The set~$\set{\omega(\alpha,S)}{\alpha \in Q_1\blossom}$ contains~$2|Q_0| - |Q_1|$ straight walks and~$|Q_0|$ pairs of inverse directed walks.
\end{corollary}

\begin{proof}
Denote by~$s$ the number of straight walks, by~$B$ the number of pairs of inverse bending walks and by~$b$ the remaining bending walks in the set~$\set{\omega(\alpha,S)}{\alpha \in Q_1\blossom}$.
Since this set contains~${|Q_1\blossom| = 4|Q_0| - |Q_1|}$ distinct directed walks by Lemma~\ref{lem:distinctDirectedWalks}, we have~$s + 2B + b = 4|Q_0| - |Q_1|$ and~$s \le 2|Q_0| - |Q_1|$.
Moreover, since all these walks are not self-kissing and pairwise non-kissing by Corollary~\ref{coro:surjectionNonKissingWalks}, we obtain that~$B + b \le |Q_0|$.
Therefore, Corollary~\ref{coro:pure} ensures that ${B = (s + 2B + b) - (s + B + b) \ge |Q_0|}$.
We conclude that~$B = |Q_0|$ and thus~$s = 2|Q_0| - |Q_1|$.
\end{proof}

In other words, beside the straight walks, we can group the walks of~$\set{\omega(\alpha,S)}{\alpha \in Q_1\blossom}$ into pairs of inverse walks.
This leads to the definition of the following map.

\begin{definition}
\label{def:eta}
We denote by~$\eta : \Bicl{\bar Q} \to 2^{\walks^\pm(\bar Q)}$ the map defined by
\[
\eta(S) \eqdef \set{\omega(\alpha,S)}{\alpha \in Q_1\blossom}_{ / \, (\omega \sim \omega^{-1})},
\]
where the quotient means that we consider the walks of~$\eta(S)$ as undirected.
\end{definition}

\begin{example}
Observe that~$\eta(\varnothing) = F_\deep$ and that~$\eta\big(\strings^\pm(\bar Q)\big) = F_\peak$.
\end{example}

\begin{example}
The map~$\eta$ is illustrated on \fref{fig:exmSurjection}\,(left) for the gentle bound quiver~$\bar Q$ of \fref{fig:exmBlossomingQuiver}\,(left).

\begin{figure}[t]
	\capstart
	\centerline{
	\begin{overpic}[scale=.45]{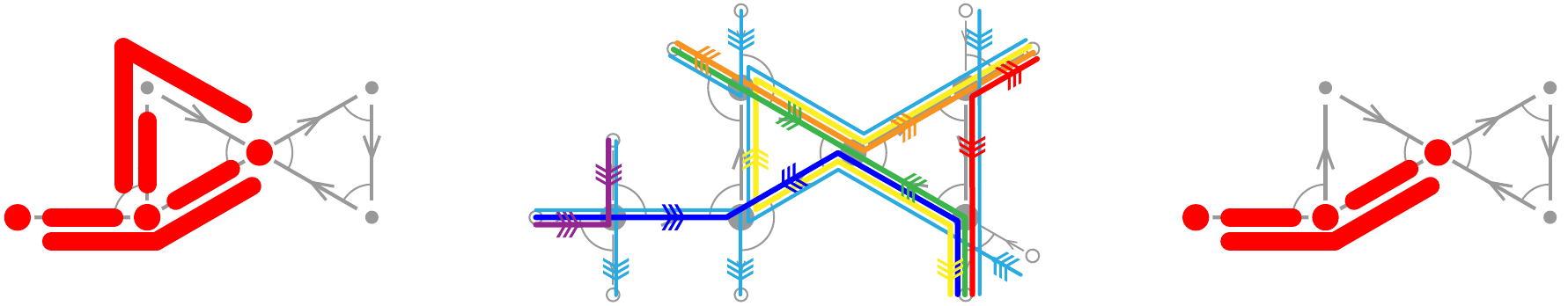}
			\put(28,8){$\longrightarrow$}
			\put(29,10){$\eta$}
			\put(69,8){$\longrightarrow$}
			\put(70,10){$\zeta$}
			\put(13,-3){$S$}
			\put(46,-3){$\eta(S)$}
			\put(77,-3){$\zeta(\eta(S)) = \projDown(S)$}
	\end{overpic}
	}
	\vspace{.3cm}
	\caption{The maps~$\eta$ of Definition~\ref{def:eta} from biclosed sets to non-kissing facets (left) and $\zeta$ of Definition~\ref{def:zeta} from non-kissing facets to biclosed sets (right).}
	\label{fig:exmSurjection}
\end{figure}
\end{example}

\begin{example}
\fref{fig:exmLatticeQuotient} shows the complete map~$\eta$ from the biclosed sets of~$\Bicl{\bar Q}$ to the non-kissing facets of~$\NKC$ for a specific gentle bound quiver.
Instead of drawing the image~$\eta(S)$ for each biclosed set~$S \in \Bicl{\bar Q}$, we have grouped in blue boxes the fibers of~$\eta$.
\end{example}

\enlargethispage{.2cm}
The following result recasts Corollaries~\ref{coro:surjectionNonKissingWalks} and~\ref{coro:eta} in terms of the map~$\eta$ of Definition~\ref{def:eta}.

\begin{proposition}
For any~$S \in \Bicl{\bar Q}$, the collection of undirected walks~$\eta(S)$ is a non-kissing facet of~$\NKC$.
\end{proposition}

We now identify the relation between an arrow~$\alpha \in Q_1\blossom$ and the walk~$\omega(\alpha,S)$ in the facet~$\eta(S)$.

\begin{lemma}
\label{lem:distinguishedWalkEta}
For any~$S \in \Bicl{\bar Q}$ and~$\alpha \in Q_1\blossom$, the walk~$\omega(\alpha,S)$ is the distinguished walk of the non-kissing facet~$\eta(S)$ at~$\alpha$:
\[
\omega(\alpha,S) = \distinguishedWalk{\alpha}{\eta(S)}.
\]
\end{lemma}

\begin{proof}
Let~$\alpha \ne \beta \in Q_1\blossom$ and assume that~$\alpha$ belongs to~$\omega(\beta,S)$.
If~$\alpha$ and~$\beta$ are opposite along~$\omega(\beta,S)$, then~$\omega(\beta,S)$ is not straight, thus there exists~$\beta' \in Q_1\blossom$ such that~${\omega(\beta,S) = \omega(\beta',S)^{-1}}$ by Corollary~\ref{coro:eta}.
We can thus assume that~$\alpha$ and~$\beta$ are both in the same direction as~$\omega(\beta,S)$.
Assume also for instance that~$\alpha$ appears before~$\beta$ along~$\omega(\beta,S)$, and let~$\rho$ denote the substring of~$\omega(\beta,S)$ between~$\alpha$ and~$\beta$.
By definition of~$\omega(\beta,S)$, we have~$\rho \in S$ which implies that~$\beta \notin \omega(\alpha,S)$.
Decompose~$\rho$ into~$\rho = \sigma \gamma^\varepsilon \tau$ for some strings~$\sigma, \tau \in \strings^\pm(\bar Q)$, arrow~$\gamma \in Q_1$ and sign~$\varepsilon \in \{-1,1\}$ such that~$\omega(\alpha,S)$ contains~$\sigma$ but not~$\sigma \gamma^\varepsilon$.
If~$\varepsilon = -1$, then~$\sigma \notin S$ and~$\tau \notin S$, while~$\rho \in S$ which contradicts that~$S$ is coclosed.
We conclude that~$\varepsilon = 1$ so that~$\omega(\beta,S) \prec_\alpha \omega(\alpha,S)$.
The other situation ($\alpha$ appears after~$\beta$ along~$\omega(\beta,S)$) is symmetric.
\end{proof}

\subsection{From non-kissing facets to biclosed sets of strings}

We now define a map~$\zeta$ from non-kissing facets of~$\NKC$ to sets of strings.
It is again similar to~\cite[Sect.~8]{McConville}.

\begin{definition}
\label{def:zeta}
Recall from Definition~\ref{def:substrings} that~$\Sigma_\bottom(\omega)$ denotes the set of bottom substrings of a walk~$\omega \in \walks^\pm(\bar Q)$ (\ie substrings~$\sigma \in \strings^\pm(\bar Q)$ of~$\omega$ such that the arrows of~$\omega \ssm \sigma$ incident to the endpoints of~$\sigma$ are both incoming).
For a non-kissing facet~$F \in \NKC$, define
\[
\zeta(F) \eqdef \closure{\Big( \bigcup_{\omega \in F} \Sigma_\bottom(\omega) \Big)}.
\]
\end{definition}

\begin{example}
The map~$\zeta$ is illustrated on \fref{fig:exmSurjection}\,(right)  for the gentle bound quiver~$\bar Q$ of \fref{fig:exmBlossomingQuiver}\,(left).
\end{example}

\begin{proposition}
For any non-kissing facet~$F \in \NKC$, the set of strings~$\zeta(F)$ is biclosed.
\end{proposition}

\begin{proof}
Observe first that for a walk~$\omega \in \walks^\pm(\bar Q)$, either~$\omega$ is a peak walk and it has no bottom substring, or we have~$\closure{\Sigma_\bottom(\omega)} = \closure{\Sigma_\bottom(\sigma)}$ where~$\sigma$ is the inclusion maximal bottom substring of~$\omega$.
It follows that~$\closure{\Sigma_\bottom(\omega)}$ is biclosed by Lemma~\ref{lem:exmBiclosed}.
Therefore, according to Remark~\ref{rem:meetJoinBiclosedSets},
\[
\zeta(F) = \closure{\Big( \bigcup_{\omega \in F} \Sigma_\bottom(\omega) \Big)} = \closure{\Big( \bigcup_{\omega \in F} \closure{\Sigma_\bottom(\omega)} \Big)} \in \Bicl{\bar Q}. \qedhere
\]
\end{proof}

Observe moreover that~$\projDown \big( \zeta(F) \big) = \zeta(F)$ for any non-kissing facet~$F \in \NKC$ by Lemma~\ref{lem:fixProjDown}.

\subsection{Projection map}

\enlargethispage{.4cm}
We have defined in the previous sections two functions
\[
\Bicl{\bar Q} \xrightleftharpoons[\quad\mbox{$\zeta$}\;\quad]{\mbox{$\eta$}} \NKC
\]
between biclosed sets of strings of~$\bar Q$ and non-kissing facets of~$\NKC$.
We now show that~$\eta$ is surjective and that~$\zeta$ provides a section of~$\eta$.

\begin{proposition}
\label{prop:surjection}
For any non-kissing facet~$F \in \NKC$, we have~$\eta \big( \zeta(F) \big) = F$.
\end{proposition}

\begin{proof}
Fix~$\alpha \in Q_1\blossom$ and let~$\omega \eqdef \distinguishedWalk{\alpha}{F}$ (resp.~$\omega' \eqdef \distinguishedWalk{\alpha}{\eta(\zeta(F))}$) denote the distinguished walk of~$F$ (resp.~$\eta(\zeta(F))$) at~$\alpha$.
Assume for the sake of contradiction that~$\omega \ne \omega'$.
Orient~$\omega$ and~$\omega'$ in the same direction as~$\alpha$ and assume for example that~$\omega$ and~$\omega'$ split after~$\alpha$ (if not, they split before~$\alpha$ and the argument is symmetric).
Let~$\sigma$ be the maximal common substring of~$\omega$ and~$\omega'$ after~$\alpha$.
Recall from Lemma~\ref{lem:distinguishedWalkEta} that~$\omega' = \omega(\alpha, \zeta(F))$ is the walk constructed by the procedure described in Definition~\ref{def:fromBiclToNKC} starting from~$\alpha$ in~$\zeta(F)$.
We distinguish two cases:
\begin{itemize}
\item Assume first that~$\omega \ssm \sigma$ has an incoming arrow at~$t(\sigma)$. Then~$\sigma$ is a bottom substring of~$\omega$, thus belongs to~$\zeta(F)$. This would imply that~$\omega' \ssm \sigma$ also has an incoming arrow at~$t(\sigma)$, thus contradicting the maximality of~$\sigma$ (even in the case that~$\sigma$ is a vertex).
\item Assume now that~$\omega \ssm \sigma$ has an outgoing arrow at~$t(\sigma)$. Since~$\omega$ and~$\omega'$ split at~$t(\sigma)$, this forces~$\omega' \ssm \sigma$ to have an incoming arrow at~$t(\sigma)$. Therefore, $\sigma \in \zeta(F)$. Thus, there exist bottom substrings~$\sigma_1, \dots, \sigma_\ell$ of walks~$\omega_1, \dots, \omega_\ell$ of~$F$, arrows~${\alpha_1, \dots, \alpha_{\ell-1} \in Q_1}$, and signs~$\varepsilon_1, \dots, \varepsilon_{\ell-1} \in \{-1,1\}$ such that~${\sigma = \sigma_1 \alpha_1^{\varepsilon_1} \cdots \alpha_{\ell-1}^{\varepsilon_{\ell-1}} \sigma_\ell}$. Note that~$\omega_1 = \omega$ and thus~$\varepsilon_1 = -1$, as otherwise, we would have~$\omega \prec_\alpha \omega_1$, contradicting the definition of~$\omega$. Let~$k \in [\ell]$ be minimal such that~$\sigma_1 \alpha_1^{\varepsilon_1} \dots \alpha_{k-1}^{\varepsilon_{k-1}} \sigma_k$ is not a bottom substring of~$\omega$ (note that, by assumption, $\sigma$ itself is not a bottom substring of~$\omega$, so that~$k$ is well defined). We thus have~$\varepsilon_{k-1} = -1$ while~$\varepsilon_k = 1$. This implies that~$\omega \ssm \sigma_k$ has two outgoing arrows at the endpoints of~$\sigma_k$, while~$\omega_k \ssm \sigma_k$ has two incoming arrows at the endpoints of~$\sigma_k$. This implies that~$\omega$ and~$\omega_k$ are kissing, contradicting the assumption that they both belong to the non-kissing facet~$F$ of~$\NKC$.
\end{itemize}
As we obtain a contradiction in both cases, we conclude that~$\omega = \omega'$, \ie $\distinguishedWalk{\alpha}{F} = \distinguishedWalk{\alpha}{\eta(\zeta(F))}$.
Since this holds for any arrow~$\alpha \in Q_1\blossom$ and any walk has at least one distinguished arrow, we conclude that~$F = \eta \big( \zeta(F) \big)$.
\end{proof}

\begin{proposition}
\label{prop:section}
For any biclosed set of strings~$S \in \Bicl{\bar Q}$, we have~$\zeta \big( \eta(S) \big) = \projDown(S)$.
\end{proposition}

\begin{proof}
Let~$S \in \Bicl{\bar Q}$ be a biclosed set of strings.
We first prove that~$\projDown(S) \subseteq \zeta \big( \eta(S) \big)$ by induction on the length of the strings in~$\projDown(S)$.
For this, consider a string~$\sigma \in \projDown(S)$ (by definition, all bottom substrings of~$\sigma$ belong to~$S$).
Orient~$\sigma$ in an arbitrary direction, and let~$\alpha \in Q_1\blossom$ be the incoming arrow at~$s(\alpha)$ such that~$\alpha \sigma \in \strings(\bar Q\blossom)$.
Let~$\omega \eqdef \omega(\alpha,S) = \distinguishedWalk{\alpha}{\eta(S)}$ be the walk of~$\eta(S)$ defined form~$\alpha$ (see Definition~\ref{def:fromBiclToNKC}), and let~$\tau$ be the maximal common substring of~$\sigma$ and~$\omega$.
We distinguish two cases:
\begin{itemize}
\item If~$\sigma = \tau$, then~$\omega \ssm \sigma$ has an incoming arrow at~$t(\sigma)$ (by Definition~\ref{def:fromBiclToNKC}). Therefore, $\sigma \in \Sigma_\bottom \big( \omega \big) \subseteq \zeta \big( \eta(S) \big)$.
\item Otherwise, we write~$\sigma = \tau \beta^\varepsilon \tau'$. If~$\varepsilon = -1$, $\tau$ is a bottom substring of~$\sigma$, and thus belong to~$S$, so that Definition~\ref{def:fromBiclToNKC} forces~$\omega$ to use~$\beta^{-1}$ as well, contradicting the maximality of~$\tau$. Therefore, $\sigma = \tau \beta \tau'$. It follows that the walk~$\omega$ has an incoming arrow at~$t(\tau)$, so that~$\tau \in \Sigma_\bottom(\omega) \subseteq \zeta(\eta(S))$. Moreover, $\tau'$ is a bottom substring of~$\sigma$, so that~$\tau' \in \projDown(S)$, which in turns implies by induction hypothesis that~$\tau' \in \zeta(\eta(S))$. Since~$\tau, \tau' \in \zeta(\eta(S))$ and~$\zeta(\eta(S))$ is closed, we conclude that~$\sigma \in \zeta(\eta(S))$.
\end{itemize}

\medskip
Conversely, we now prove the reverse inclusion~$\zeta \big( \eta(S) \big) \subseteq \projDown(S)$.
We showed in Lemma~\ref{lem:fixProjDown} that ${\zeta \big( \eta(S) \big) = \projDown \big( \zeta \big( \eta(S) \big) \big)}$. 
Moreover, as observed in Lemma~\ref{lem:positiveNegativeOrientation}, ${\Sigma_\bottom(\omega) \subseteq S}$ for any~$\omega \in \eta(S)$, so that~$\zeta \big( \eta(S) \big) \subseteq S$.
Since~$\projDown$ is order-preserving, we obtain~${\zeta \big( \eta(S) \big) = \projDown \big( \zeta \big( \eta(S) \big) \big) \subseteq \projDown (S)}$.
\end{proof}

\begin{example}
The maps~$\eta$ and~$\zeta$ are illustrated on \fref{fig:exmSurjection}\,(right).
Compare~$\zeta(\eta(S))$ on \fref{fig:exmSurjection}\,(right) with the down projection~$\projDown(S)$ on \fref{fig:exmProjections}\,(left).
\end{example}

\begin{corollary}
\label{coro:fibers}
$\eta(R) = \eta(S) \iff \projDown(R) = \projDown(S)$ for any biclosed sets~$R,S$. In other words, the fibers of~$\eta$ are precisely the equivalence classes of the lattice congruence of Proposition~\ref{prop:latticeCongruence}.
\end{corollary}

\begin{proof}
If~$\eta(R) = \eta(S)$, then by Proposition~\ref{prop:section}, we have~$\projDown(R) = \zeta(\eta(R)) = \zeta(\eta(S)) = \projDown(S)$.
Reciprocally, if~$\projDown(R) = \projDown(S)$, then by Propositions~\ref{prop:surjection} and~\ref{prop:section}, we have~$\eta(R) = \eta \circ \zeta \circ \eta(R) = \eta(\projDown(R)) = \eta(\projDown(S)) = \eta \circ \zeta \circ \eta(S) = \eta(S)$.
\end{proof}

\begin{corollary}\label{cor: interval biclosed}
For any non-kissing facet~$F \in \NKC$, the fiber~$\eta^{-1}(F)$ is the interval of~$\Bicl{\bar Q}$
\[
\eta^{-1}(F) = \bigg[ \closure{ \Big( \bigcup_{\omega \in F} \Sigma_\bottom(\omega) \Big) }, \coclosure{ \Big( \bigcup_{\omega \in F} \closure{\Sigma_\top(\omega)} \Big) } \bigg].
\]
\end{corollary}

\begin{remark}\label{rem: Bongartz}
 We note that Corollary~\ref{cor: interval biclosed} gives a combinatorial description of Bongartz (co)com\-pletions.
 This allows to reduce the algebraic computations on representations of $\bar Q$ to much easier computations involving walks on $\bar Q\blossom$.
\end{remark}

\subsection{Oriented graph isomorphisms}

To conclude our interpretation of the non-kissing oriented flip graph~$\NKG$ in terms of lattice quotients of biclosed sets, we still need to show that~$\NKG$ is the Hasse diagram of~$\Bicl{\bar Q}/{\equiv}$.
This is the purpose of the next statement similar to~\cite[Claim~8.10]{McConville}.

\begin{proposition}
\label{prop:covers}
For any non-kissing facet~$F' \in \RNKC$ and~$\sigma \in \zeta(F')$, there exists an increasing flip~$F \to F'$ supported by~$\sigma$ if and only if~$\zeta(F') \ssm \{\sigma\}$ is biclosed.
\end{proposition}

\begin{proof}
Throughout the proof, we denote~$S \eqdef \zeta(F') \ssm \{\sigma\}$ and~$S' \eqdef \zeta(F')$.
Consider an increasing flip~$F \to F'$ with~$F \ssm \{\omega\} = F' \ssm \{\omega'\}$ and~$\sigma = \distinguishedString{w}{F} = \distinguishedString{w'}{F'}$.
We prove that~$S$ is biclosed:
\begin{description}
\item[closed] Otherwise, there exist bottom substrings~$\sigma_1, \dots, \sigma_\ell$ of walks~$\omega_1, \dots, \omega_\ell$ of~$F'$, arrows~$\alpha_1, \dots, \alpha_{\ell-1} \in Q_1$ and signs~$\varepsilon_1, \dots, \varepsilon_{\ell-1} \in \{-1,1\}$ such that~$\sigma = \sigma_1 \alpha_1^{\varepsilon_1} \dots \alpha_{\ell-1}^{\varepsilon_{\ell-1}} \sigma_\ell$.
Assume that there is~$k \in [\ell-1]$ such that~$\varepsilon_k = 1$ and choose~$k$ minimal with that property.
Then~$\omega_k \ne \omega'$ and~$\omega$ has two outgoing arrows while~$\omega_k$ has two incoming arrows at the endpoints of~$\sigma_k$, so that~$\sigma_k \in \Sigma_\top(\omega) \cap \Sigma_\bottom(\omega_k)$, contradicting the non-kissingness of~$F$.
We obtain a similar contradiction assuming that there is~$k \in [\ell-1]$ such that~$\varepsilon_k = -1$ and choosing~$k$ maximal with that property.
This forces~$\ell = 1$, so that~$\sigma \in S$, a contradiction.

\item[coclosed] Otherwise, by coclosedness of~$S'$, we can assume that there is a string~${\tau \in \strings(\bar Q) \ssm S'}$, an arrow~$\alpha \in Q_1$ and a sign~$\varepsilon \in \{-1,1\}$ such that~$\sigma \alpha^\varepsilon \tau \in S'$ (the case~${\tau \alpha^\varepsilon \sigma \in S'}$ is symmetric).
Thus, there exist bottom substrings~$\sigma_1, \dots, \sigma_\ell$ of walks~$\omega_1, \dots, \omega_\ell$ of~$F'$, arrows~$\alpha_1, \dots, \alpha_{\ell-1} \in Q_1$ and signs~$\varepsilon_1, \dots, \varepsilon_{\ell-1} \in \{-1,1\}$ such that~$\sigma \alpha^\varepsilon \tau = \sigma_1 \alpha_1^{\varepsilon_1} \dots \alpha_{\ell-1}^{\varepsilon_{\ell-1}} \sigma_\ell$.
Since~$\tau \notin S'$, there is~$k \in [\ell]$ such that~$\alpha^\varepsilon$ belongs to~$\sigma_k$ and we write~$\sigma_k = \sigma'_k \alpha^\varepsilon \sigma''_k$.
Since~$\{\omega, \omega_1, \dots, \omega_k\}$ are non-kissing, we obtain that~$\varepsilon_1, \dots, \varepsilon_{k-1} = -1$, which in turn implies that~$\varepsilon = 1$ (otherwise~$\omega$ would kiss~$\omega_k$).
Since~$\varepsilon = 1$, we have~${\sigma''_k \in \Sigma_\bottom(\sigma_k) \subseteq \Sigma_\bottom(\omega_k)}$.
We thus obtain that~$\tau = \sigma''_k \alpha_k^{\varepsilon_k} \dots \alpha_{\ell-1}^{\varepsilon_{\ell-1}} \sigma_\ell \in S'$, a contradiction.
\end{description}

\medskip
Conversely, consider a string~$\sigma \in S' \eqdef \zeta(F')$ such that~$S \eqdef \zeta(F') \ssm \{\sigma\}$ is biclosed, and consider the facet~$F \eqdef \eta(S)$.
Let~$\alpha, \beta$ (resp.~$\alpha', \beta'$) denote the two outgoing (resp.~incoming) arrows at the endpoints of~$\sigma$ such that~$\alpha^{-1}\sigma\beta \in \strings(\bar Q\blossom)$ (resp.~$\alpha'\sigma\beta'^{-1} \in \strings(\bar Q\blossom)$).
Since~$S$ and~$S'$ only differ by~$\sigma$, the facets~$F = \eta(S)$ and~$F' = \eta(S') = \eta(S \cup \{\sigma\})$ can \apriori{} differ by $4$ directed walks by definition of the map~$\eta$: namely, the walks~$\omega(\alpha, S)$, $\omega(\alpha', S)$, $\omega(\beta, S)$ and~$\omega(\beta', S)$ of~$F$ may or may not be in~$F'$ and conversely the walks~$\omega(\alpha, S')$, $ \omega(\alpha', S')$, $\omega(\beta, S')$ and~$\omega(\beta', S')$ of~$F'$ may or may not be in~$F$.
We claim that
\begin{itemize}
\item $\omega(\alpha', S') = \omega(\beta', S')^{-1}$ is a bending walk~$\omega'$ of~$F' \ssm F$,
\item $\omega(\alpha', S) = \omega(\beta, S')$ and~$\omega(\beta', S) = \omega(\alpha, S')$.
\end{itemize}
This claim implies (by Corollary~\ref{coro:pure}) that~$\omega'$ is the single undirected walk of~$F' \ssm F$, and its distinguished arrows are~$\alpha'$ and~$\beta'$ (by Lemma~\ref{lem:distinguishedWalkEta}).
Therefore the non-kissing facets~$F$ and~$F'$ are adjacent and the flip~$F \to F'$ is increasing and supported by~$\sigma$.

To conclude, we just need to prove our claim.
Since~$\sigma \in S'$ and~$S = S' \ssm \{\sigma\}$ is coclosed, there is a walk~$\omega' \in F'$ such that~$\sigma \in \Sigma_\bottom(\omega')$.
Assume that~$\omega'$ is distinct from~$\lambda \eqdef \distinguishedWalk{\alpha'}{F'}$.
Note that~$\omega'$ and~$\lambda$ share the arrow~$\alpha'$, and thus its target vertex~$u \eqdef t(\alpha')$.
Let~$v$ denote the last common vertex of~$\omega'$ and~$\lambda$ after~$\alpha$, and let~$\tau \eqdef \omega'[u,v] = \lambda[u,v]$ denote their common substring between these two vertices.
Since~$\lambda \eqdef \distinguishedWalk{\alpha'}{F'}$, we know that~$\omega'$ (resp.~$\lambda$) leaves~$\tau$ with an outgoing (resp.~incoming) arrow at~$v$.
In particular, $\sigma \ne \tau$ and~$\tau \in \Sigma_\bottom(\lambda) \ssm \{\sigma\} \subseteq S$.
We distinguish two cases:
\begin{description}
\item[$\sigma \supset \tau$] Since~$\omega'$ leaves~$\tau$ with an outgoing arrow at~$v$, we have~$\sigma \ssm \tau \in \Sigma_\bottom(\omega') \ssm \{\sigma\} \subseteq S$. Therefore, we have~$\sigma \in \tau \circ (\sigma \ssm \tau)$ and~$\sigma \notin S$ while~$\tau \in S$ and~$\sigma \ssm \tau \in S$. This contradicts the assumption that~$S$ is closed.
\item[$\sigma \subset \tau$] Since~$\sigma \notin S$ while~$\tau \in S$ and~$S$ is coclosed, we have~$\tau \ssm \sigma \in S$. Therefore there exist bottom substrings~$\sigma_1, \dots, \sigma_\ell$ of walks~$\omega_1, \dots, \omega_\ell$ of~$F'$, arrows~$\alpha_1, \dots, \alpha_{\ell-1} \in Q_1$ and signs~$\varepsilon_1, \dots, \varepsilon_{\ell-1} \in \{-1,1\}$ such that~$\tau \ssm \sigma = \sigma_1 \alpha_1^{\varepsilon_1} \dots \alpha_{\ell-1}^{\varepsilon_{\ell-1}} \sigma_\ell$. Define~$\varepsilon_\ell = 1$ and let~$k \in [\ell]$ be minimal such that~$\varepsilon_k = 1$. Since~$\tau \ssm \sigma$ is a top substring of~$\omega'$, this implies that~$\omega_k$ kisses~$\omega'$. This contradicts the fact that~$F'$ is a non-kissing facet.
\end{description}
Since both cases are impossible, we conclude that~$\omega' = \distinguishedWalk{\alpha'}{F'} = \omega(\alpha', S')$ (by Lemma~\ref{lem:distinguishedWalkEta}).
By symmetry, we obtain that~$\omega' = \omega(\alpha', S') = \omega(\beta', S')$.
Finally, this implies that all bottom (resp.~top) subsegments of~$\sigma$ are contained in~$S'$ (resp.~in the complement of~$S'$).
We conclude by Definition~\ref{def:fromBiclToNKC} that~$\omega(\alpha', S) = \omega(\beta, S')$ and~$\omega(\beta', S) = \omega(\alpha, S')$.
\end{proof}

To conclude, we need a statement from~\cite{Reading-HopfAlgebras} which characterizes the cover relations in a lattice congruence.

\begin{proposition}[{\cite[Prop.~2.2]{Reading-HopfAlgebras}}]
\label{prop:characterizationCoversLatticeQuotient}
Consider a lattice congruence~$\equiv$ on a lattice~$L$ and two congruence classes~$X, Y \in L/{\equiv}$.
Then~$X$ covers~$Y$ in~$L/{\equiv}$ if and only if the minimal element of~$X$ covers some element of~$Y$ in~$L$.
\end{proposition}

This concludes our proof of Theorem~\ref{thm:lattice}, that we can now restate more precisely as follows.

\begin{corollary}
\label{coro:finalLatticeQuotient}
If~$\bar Q$ is a gentle bound quiver whose non-kissing complex~$\RNKC$ is finite, then the non-kissing oriented flip graph~$\NKG$ is the Hasse diagram of the lattice quotient~$\Bicl{\bar Q}/{\equiv}$ of the congruence-uniform lattice of biclosed sets of strings by the lattice congruence~$\equiv$.
\end{corollary}

\begin{proof}
Proposition~\ref{prop:latticeCongruence} affirms that the maps~$\projDown$ and~$\projUp$ of Definition~\ref{def:downUpProjections} define a lattice congruence~$\equiv$.
Corollary~\ref{coro:fibers} ensures that the congruence classes of~$\equiv$ are the fibers of~$\eta$.
Finally, Propositions~\ref{prop:covers} and~\ref{prop:characterizationCoversLatticeQuotient} ensure that the cover relations of~$\Bicl{\bar Q}/{\equiv}$ are precisely the increasing flips of~$\NKG$.
\end{proof}

\begin{example}
Corollary~\ref{coro:finalLatticeQuotient} is illustrated on \fref{fig:exmLatticeQuotient}, which shows a specific lattice quotient and the corresponding increasing flip graph.
\end{example}

\begin{definition}\label{def:NKL}
We call \defn{non-kissing lattice} and denote by~$\NKL$ the transitive closure of the non-kissing oriented flip graph~$\NKG$.
\end{definition}

\subsection{Sublattice}

As observed in Remark~\ref{rem:sublattice}, the subposet of~$\Bicl{\bar Q}$ induced by~$\projDown \big( \Bicl{\bar Q} \big)$ is a join-sublattice of~$\Bicl{\bar Q}$ but could fail to be a sublattice.
It turns out that there is sufficiently symmetry in the construction to ensures that it is indeed a sublattice.
A similar result was stated for grid bound quivers in~\cite[Lem.~4.5]{GarverMcConville-grid} and for dissection bound quivers in~\cite[Thm.~4.11]{GarverMcConville}.
See \fref{fig:exmSublattice} for an illustration.

\begin{proposition}
\label{prop:sublattice}
The non-kissing lattice~$\NKL$ is isomorphic to the sublattice of~$\Bicl{\bar Q}$ induced by~$\projDown \big( \Bicl{\bar Q} \big)$ (or dually by~$\projUp \big( \Bicl{\bar Q} \big)$). More precisely,
\begin{enumerate}[(i)]
\item $\zeta(F \meet F') = \zeta(F) \meet \zeta(F')$ and~$\zeta(F \join F') = \zeta(F) \join \zeta(F')$ for any non-kissing facets~${F, F' \in \NKC}$,
\item $\eta(S \meet S') = \eta(S) \meet \eta(S')$ and~$\eta(S \join S') = \eta(S) \join \eta(S')$ for any biclosed sets~$S, S' \in \projDown \big( \Bicl{\bar Q} \big)$.
\end{enumerate}
\end{proposition}

\begin{proof}
This follows from the reversing operation discussed in Remarks~\ref{rem:reverseFlipGraph}, \ref{rem:reverseBiclosed} and Lemma~\ref{rem:reversedProjections}.
Indeed, consider~$R, S \in \Bicl{\bar Q}$.
By Lemma~\ref{rem:reversedProjections}, we have
\begin{align*}
\projDown(R) \meet \projDown(S) 
& = \big( \strings^\pm(\bar Q) \ssm \projUp\big( \strings^\pm(\reversed{\bar Q}) \ssm R \big) \big) \meet \big( \strings^\pm(\bar Q) \ssm \projUp\big( \strings^\pm(\reversed{\bar Q}) \ssm S \big) \big) \\
& = \strings^\pm(\bar Q) \ssm \big( \projUp\big( \strings^\pm(\reversed{\bar Q}) \ssm R \big) \join \projUp\big( \strings^\pm(\reversed{\bar Q}) \ssm S \big) \big) \\
& = \strings^\pm(\bar Q) \ssm \projUp(T) = \projDown\big( \strings^\pm(\reversed{\bar Q}) \ssm T \big)
\end{align*}
for some~$T \in \Bicl{\reversed{\bar Q}} = \Bicl{\bar Q}$.
Therefore, $\projDown(R) \meet \projDown(S) \in \projDown \big( \Bicl{\bar Q} \big)$.
\end{proof}

\begin{example}
\fref{fig:exmSublattice} illustrates Proposition~\ref{prop:sublattice} for a specific gentle bound quiver.

\begin{figure}[t]
	\capstart
	\centerline{\includegraphics[width=1.1\textwidth]{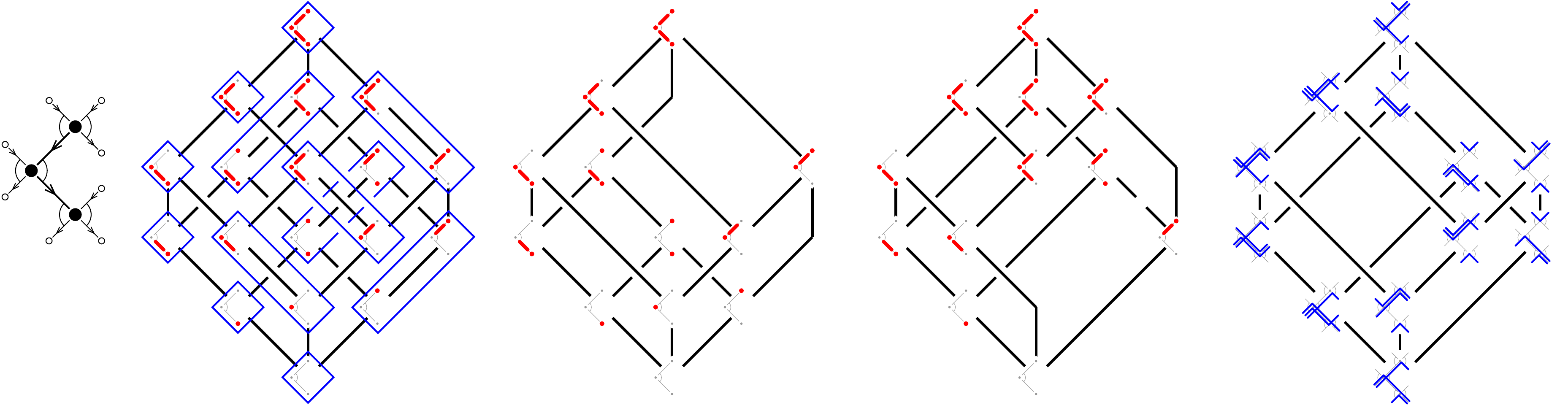}}
	\caption{The inclusion lattice of biclosed sets~$\Bicl{\bar Q}$ with congruence classes of~$\equiv$ in blue (left), the sublattice of~$\Bicl{\bar Q}$ induced by~$\projDown \big( \Bicl{\bar Q} \big)$ (middle left) and by~$\projUp \big( \Bicl{\bar Q} \big)$ (middle right), and the corresponding lattice of increasing flips on facets of~$\NKC$ (right).}
	\label{fig:exmSublattice}
\end{figure}
\end{example}

\section{Canonical join-representations in~$\NKL$ and the non-friendly complex}
\label{sec:nonFriendlyComplex}

As we have established that the non-kissing lattice~$\NKL$ is congruence-uniform, and thus semi-distributive, it is natural to study its canonical join complex.
Our first step is a bijective understanding of the join-irreducible elements of~$\NKL$.
They correspond to the distinguishable strings of~$\bar Q$, for which we first need an explicit characterization.

\subsection{Distinguishable strings}
\label{subsec:distinguishableStrings}

Recall Definition~\ref{def:distinguishedSubstring} for distinguishable strings.
We now have all tools to describe the distinguishable strings in a gentle bound quiver~$\bar Q$.
Our goal in this section is to prove the following statement.

\begin{proposition}
\label{prop:characterizationDistinguishableStrings}
A string~$\sigma \in \strings^\pm(\bar Q)$ is distinguishable if and only if $\Sigma_\bottom(\sigma) \cap \Sigma_\top(\sigma) = \{\sigma\}$.
\end{proposition}

\begin{proof}
Assume first that~$\sigma$ is distinguishable and consider~$\tau \in \Sigma_\bottom(\sigma) \cap \Sigma_\top(\sigma)$.
Let~$\tau_\bottom$ (resp.~$\tau_\top$) denote the bottom (resp.~top) copy of~$\tau$ seen as a substring of~$\sigma$.
If~$\tau \ne \sigma$, then $\tau_\bottom$ and~$\tau_\top$ are distinct as substrings of~$\sigma$ but coincide as strings of~$\bar Q$.
We distinguish two cases:
\begin{itemize}
\item If~$s(\tau_\bottom) = s(\sigma)$ and~$t(\tau_\top) = t(\sigma)$, then the two endpoints of~$\sigma$ coincide (otherwise, $\NKC$ would not be finite) so that $\sigma$ is not distinguishable (otherwise, the same arrow would be distinguished twice along a path).
\item Otherwise, we can assume for example that~$\tau_\top$ does not share an endpoint with~$\sigma$. Then for any walk~$\omega$ such that~$\sigma \in \Sigma_\bottom(\omega)$, we have~$\tau \in \Sigma_\bottom(\omega) \cap \Sigma_\top(\omega)$ so that~$\omega$ is self-kissing. It follows that~$\sigma$ is not distinguishable.
\end{itemize}
We conclude that if~$\sigma$ is distinguishable, then~$\Sigma_\bottom(\sigma) \cap \Sigma_\top(\sigma) = \{\sigma\}$.
The converse statement will follow from the next two lemmas.
\end{proof}

\begin{remark}
It follows from Propositions~\ref{prop:characterizationDistinguishableStrings} and~\ref{prop:morphismsStringModules} that a string~$\sigma$ is distinguishable if and only if~$\Hom{A} \big( M(\sigma), M(\sigma) \big) = k$.
Modules~$M$ such that~$\Hom{A}(M, M) = k$ are called \defn{bricks}.
\end{remark}

\begin{lemma}
\label{lem:sigmaBottomCapSigmaTop}
For any~$\sigma, \tau \in \strings^\pm(\bar Q)$, any string in~$\closure{\Sigma_\bottom(\sigma)} \cap \Sigma_\top(\tau)$ has at least one substring in~$\Sigma_\bottom(\sigma) \cap \Sigma_\top(\tau)$. In particular, we have
\begin{align*}
\Sigma_\bottom(\sigma) \cap \Sigma_\top(\sigma) = \{\sigma\} & \quad \Longrightarrow \quad \closure{\Sigma_\bottom(\sigma)} \cap \Sigma_\top(\sigma) = \{\sigma\} \\
\text{and}\qquad
\Sigma_\bottom(\sigma) \cap \Sigma_\top(\tau) = \varnothing & \quad \Longrightarrow \quad \closure{\Sigma_\bottom(\sigma)} \cap \Sigma_\top(\tau) = \varnothing.
\end{align*}
\end{lemma}

\begin{proof}
Let~$\rho \in \closure{\Sigma_\bottom(\sigma)} \cap \Sigma_\top(\tau)$.
As~$\rho \in \closure{\Sigma_\bottom(\sigma)}$, there exist substrings~$\rho_1, \dots, \rho_\ell \in \Sigma_\bottom(\sigma)$, arrows~$\alpha_1, \dots, \alpha_{\ell-1} \in Q_1$ and signs~$\varepsilon_1, \dots, \varepsilon_{\ell-1} \in \{-1,1\}$ such that~$\rho = \rho_1 \alpha_1^{\varepsilon_1} \dots \alpha_{\ell-1}^{\varepsilon_{\ell-1}} \rho_\ell$.
Assume that there exists~$k \in [\ell-1]$ such that~$\varepsilon_k = 1$ and choose~$k$ minimal for this property (the proof is similar if there exists~$k \in [\ell-1]$ such that~$\varepsilon_k = -1$).
If~$k > 1$, then~$\varepsilon_{k-1} = -1$ while~$\varepsilon_k = 1$ so that~$\rho_k \in \Sigma_\top(\tau)$.
If~$k = 1$, then~$\rho_1$ starts at the endpoint of~$\rho$ and finishes with an outgoing arrow, so that~$\rho_1$ is a top substring of~$\rho$ which is a top substring of~$\tau$, and thus~$\rho_1 \in \Sigma_\top(\tau)$.
In both cases, we found a substring of~$\rho$ in~$\Sigma_\bottom(\sigma) \cap \Sigma_\top(\tau)$.
\end{proof}

\begin{lemma}
\label{lem:bijectionDistinguishableStringsWalks}
Let~$\sigma \in \strings^\pm(\bar Q)$ be such that~$\Sigma_\bottom(\sigma) \cap \Sigma_\top(\sigma) = \{\sigma\}$.
Let~$\alpha', \beta'$ be the two incoming arrows at the endpoint of~$\sigma$ such that~$\alpha'\sigma\beta'^{-1} \in \strings(\bar Q\blossom)$.
Then the oriented walks~$\omega \big( \alpha', \closure{\Sigma_\bottom(\sigma)} \big)$ and~$\omega \big( \beta', \closure{\Sigma_\bottom(\sigma)} \big)$ both contain~$\sigma$ and are reverse to each other. Therefore the corresponding unoriented walk of~$\eta \big( \closure{\Sigma_\bottom(\sigma)} \big)$ has distinguished string~$\sigma$.
\end{lemma}

\begin{proof}
Denote for short~$S \eqdef \closure{\Sigma_\bottom(\sigma)}$ and~$F \eqdef \eta(S)$ the corresponding non-kissing facet. Note that~$F$ is well-defined since~$S$ is biclosed by Lemma~\ref{lem:exmBiclosed}.

We first prove that~$\omega(\alpha',S)$ and~$\omega(\beta',S)$ both contain~$\alpha' \sigma \beta'^{-1}$.
For this, write~$\sigma = \alpha_1^{\varepsilon_1} \dots \alpha_\ell^{\varepsilon_\ell}$ for some arrows~$\alpha_1, \dots, \alpha_\ell \in Q_1$ and signs~$\varepsilon_1, \dots, \varepsilon_\ell \in \{-1, 1\}$.
Define~$\sigma_i \eqdef \alpha_1^{\varepsilon_1} \dots \alpha_i^{\varepsilon_i}$ for~$i \in [\ell]$.
We prove by induction that~$\omega(\alpha',S)$ contains~$\sigma_i$.
Indeed, the base case ($i = 0$) is clear, so assume that it holds for~$i \in [\ell-1]$.
We have two situations:
\begin{itemize}
\item If~$\varepsilon_{i+1} = -1$, then~$\sigma_i \in \Sigma_\bottom(\sigma) \subseteq S$ so that~$\omega(\alpha',S)$ contains~$\alpha_{i+1}^{-1}$ and thus~$\sigma_{i+1}$.
\item If~$\varepsilon_{i+1} = 1$, then~$\sigma_i \in \Sigma_\top(\sigma)$. Since~$S \cap \Sigma_\top(\sigma) = \{\sigma\}$ by Lemma~\ref{lem:sigmaBottomCapSigmaTop}, we obtain that~$\sigma_i \notin S$ so that~$\omega(\alpha',S)$ contains~$\alpha_{i+1}$ and thus~$\sigma_{i+1}$.
\end{itemize}
Finally, since~$\sigma \in S$, we obtain that~$\omega(\alpha',S)$ contains~$\alpha' \sigma \beta'^{-1}$.
By symmetry, we thus obtained that~$\omega(\alpha',S)$ and~$\omega(\beta',S)$ both contain~$\alpha' \sigma \beta'^{-1}$.

To conclude, assume that~$\omega(\alpha',S)$ and~$\omega(\beta',S)$ are not reverse to each other, and let~$\tau$ be a maximal common substring of them containing~$\alpha' \sigma \beta'^{-1}$.
Write~$\tau = \tau' \alpha' \sigma \beta'^{-1} \tau''$ and assume for example that~$t(\tau'')$ is not a blossom.
Since~$\tau'' \in S \iff \sigma \beta^{-1} \tau'' \in S$, the two walks~$\omega(\alpha',S)$ and~$\omega(\beta',S)$ should have the same behavior at~$t(\tau'')$, contradicting the maximality of~$\tau$.
\end{proof}

In fact, Lemma~\ref{lem:bijectionDistinguishableStringsWalks} even provides a bijection from distinguishable strings to walks which are not self-kissing and not in the peak facet.
We will prove that this map is bijective in Proposition~\ref{prop:bijectionDistinguishableStringsWalks}.
Before that, we need the following statement.

\begin{lemma}
\label{lem:bijectionWalksDistinguishableStrings}
For any bending walk~$\omega \in \NKWalks^\pm(\bar Q)$ which is not self-kissing and not in the peak facet, the unique inclusion minimal string in~$\set{\sigma \in \Sigma_\bottom(\omega)}{\closure{\Sigma_\bottom(\omega)} = \closure{\Sigma_\bottom(\sigma)}}$ is distinguishable.
\end{lemma}

\begin{proof}
Let~$X \eqdef \set{\sigma \in \Sigma_\bottom(\omega)}{\closure{\Sigma_\bottom(\omega)} = \closure{\Sigma_\bottom(\sigma)}}$.
Since~$\omega$ is not in the peak facet, it has at least one bottom substring.
It follows that~$X$ is non-empty since it contains the inclusion maximal bottom substring of~$\omega$.
Consider now any inclusion minimal string~$\sigma$ in~$X$ (there could be several).

Assume that~$\sigma$ is not distinguishable.
By Lemma~\ref{lem:bijectionDistinguishableStringsWalks}, there exists~${\tau \in \Sigma_\bottom(\sigma) \cap \Sigma_\top(\sigma)}$ distinct from~$\sigma$.
Since~$\sigma$ is a bottom substring of a non-self-kissing walk~$\omega$, the substring~$\tau$ must be a prefix or a suffix of~$\sigma$, say for example a prefix.
We write~$\sigma = \tau \alpha^\varepsilon \rho$ for an arrow~$\alpha \in Q_1$, a sign~$\varepsilon \in \{-1,1\}$ and a substring~$\rho$ of~$\sigma$.
If~$\varepsilon = -1$, then~$\tau$ would be a top and a bottom substring of~$\omega$, a contradiction.
Therefore~$\varepsilon = 1$ so that we have~$\tau \in \Sigma_\bottom(\rho)$ and thus~$\sigma \in \closure{\Sigma_\bottom(\rho)}$, hence~$\closure{\Sigma_\bottom(\rho)} = \closure{\Sigma_\bottom(\sigma)} = \closure{\Sigma_\bottom(\omega)}$, which would contradict the minimality of~$\sigma$.

Finally, observe that~$\sigma$ is the unique descent of~$\eta \big( \closure{\Sigma_\bottom(\sigma)} \big) = \eta \big( \closure{\Sigma_\bottom(\omega)} \big)$ as will be established in Lemma~\ref{lem:jiJI}.
\end{proof}

We gather the constructions of Lemmas~\ref{lem:bijectionDistinguishableStringsWalks} and~\ref{lem:bijectionWalksDistinguishableStrings} in the following definition.

\begin{definition}
\label{def:bijectionWalksDistinguishableStrings}
For a distinguishable string~$\sigma \in \distinguishableStrings^\pm(\bar Q)$, we denote by $\omega_\bottom(\sigma) \in \NKWalks^\pm(\bar Q)$ the undirected walk constructed in Lemma~\ref{lem:bijectionDistinguishableStringsWalks}, meaning that~$\omega_\bottom(\sigma) = \{\omega(\alpha',S), \omega(\beta',S)\}$ where~$S = \closure{\Sigma_\bottom(\sigma)}$ and~$\alpha', \beta'$ are such that~$\alpha'\sigma\beta'^{-1} \in \strings(\bar Q\blossom)$.

Conversely, for a bending non-self-kissing walk~$\omega \in \NKWalks^\pm(\bar Q)$ not in the peak facet, we denote by~$\sigma_\bottom(\omega)$ the inclusion minimal string in~$\set{\sigma \in \Sigma_\bottom(\omega)}{\closure{\Sigma_\bottom(\omega)} = \closure{\Sigma_\bottom(\sigma)}}$.
\end{definition}

\begin{proposition}
\label{prop:bijectionDistinguishableStringsWalks}
The maps~$\omega_\bottom$ and~$\sigma_\bottom$ are inverse bijections between the distinguishable strings of~$\bar Q$ and the bending non-self-kissing walks of~$\NKWalks^\pm(\bar Q)$ not in the peak facet~$F_\peak$.
\end{proposition}

\begin{proof}
The maps~$\omega_\bottom$ and~$\sigma_\bottom$ are well-defined by Lemmas~\ref{lem:bijectionDistinguishableStringsWalks} and~\ref{lem:bijectionWalksDistinguishableStrings}.
To see that they are inverse bijections, observe that for any~$\sigma \in \distinguishableStrings^\pm(\bar Q)$, we have\[
\closure{ \Sigma_\bottom \big( \sigma_\bottom \big( \omega_\bottom(\sigma) \big) \big)}
= \closure{\Sigma_\bottom \big( \omega_\bottom(\sigma) \big)}
= \closure{\Sigma_\bottom(\sigma)}.
\]
Indeed, the first equality follows from the definition of~$\sigma_\bottom$, while the second one is derived from~$\sigma \in \Sigma_\bottom \big( \omega_\bottom(\sigma) \big)$ and~$\omega_\bottom(\sigma) \in \eta \big( \closure{\Sigma_\bottom(\sigma)} \big)$.
Therefore, we obtain that~$\sigma_\bottom \big( \omega_\bottom(\sigma) \big) = \sigma$ is the unique descent of
\[
\eta \big( \closure{ \Sigma_\bottom \big( \sigma_\bottom \big( \omega_\bottom(\sigma) \big) \big)} \big) = \eta \big( \closure{\Sigma_\bottom(\sigma)} \big)
\]
as will be established in Lemma~\ref{lem:jiJI}.
\end{proof}

\begin{remark}
\label{rem:bijectionDistinguishableStringsWalks}
One can define dually the maps~$\sigma_\top$ and~$\omega_\top$, which provide inverse bijections between distinguishable strings of~$\bar Q$ and the bending walks of~$\walks^\pm(\bar Q)$ not in the deep facet~$F_\deep$.
We will use both~$\sigma_\bottom$ and~$\sigma_\top$ in Section~\ref{subsec:zonotope}.
\end{remark}

\begin{example}
\fref{fig:exmBijectionStringsWalks2} illustrates the maps~$\sigma_\top = \omega_\top^{-1}$ and $\sigma_\bottom = \omega_\bottom^{-1}$ for two quivers on $3$ vertices.

\begin{figure}[t]
	\capstart
	\centerline{\includegraphics[scale=.4]{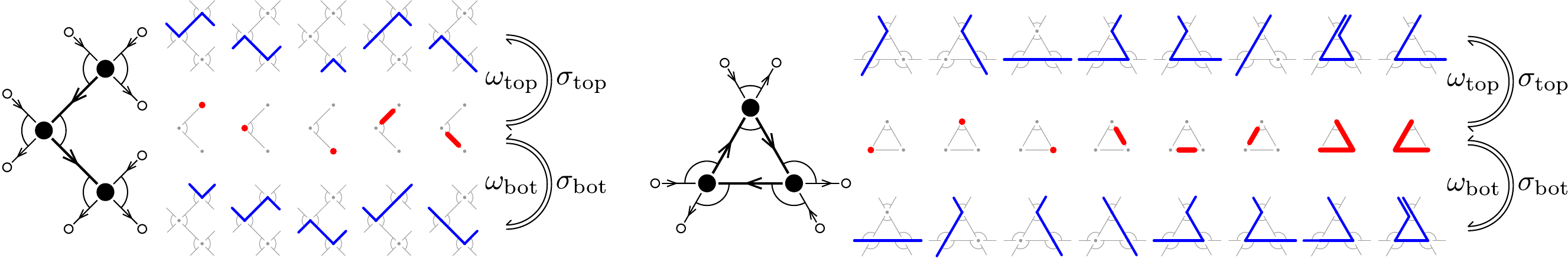}}
	\caption{The bijections~$\sigma_\top = \omega_\top^{-1}$ (top) and $\sigma_\bottom = \omega_\bottom^{-1}$ (bottom) for two specific quivers.}
	\label{fig:exmBijectionStringsWalks2}
\end{figure}
\end{example}

\begin{remark}
Assume that all strings of~$\bar Q$ are distinguishable, as happens for instance in grid and dissection quivers.
Then~$\omega_\bottom(\sigma)$ (resp.~$\omega_\top(\sigma)$) is obtained by adding a hook (resp.~a cohook) at each endpoint of~$\sigma$ while $\sigma_\bottom(\omega)$ (resp.~$\sigma_\top(\omega)$) is obtained by deleting a hook (resp.~a cohook) at each blossom of~$\omega$.
With the notations of Section~\ref{subsec:ARtranslationStringModules},
\[
\omega_\bottom(\sigma) = {}_h\sigma_h
\qquad\text{and}\qquad
\omega_\top(\sigma) = {}_c\sigma_c.
\]
In other words, when all strings of~$\bar Q$ are distinguishable, the bijections~$\omega_\bottom$ and~$\sigma_\bottom$ where already considered in Lemma~\ref{lem:bijections-STTC-NKC}.
As illustrated in \fref{fig:exmBijectionStringsWalks2}\,(right), the bijections of Lemma~\ref{lem:bijections-STTC-NKC} and Definition~\ref{def:bijectionWalksDistinguishableStrings} no longer coincide in general when not all strings are distinguishable.
\end{remark}

\subsection{Join-irreducibles in~$\NKL$}

In order to describe the canonical join-representations in the next section, we now completely describe the join-irreducible elements in~$\NKL$.
We first consider the following join-irreducible elements of~$\Bicl{\bar Q}$.

\begin{lemma}
\label{lem:exmBiclosedJI}
For any distinguishable string~$\sigma \in \strings^\pm(\bar Q)$, the biclosed set~$\closure{\Sigma_\bottom(\sigma)}$ is join-irreducible in~$\Bicl{\bar Q}$.
\end{lemma}

\begin{proof}
We have seen in Lemma~\ref{lem:exmBiclosed} that~$\closure{\Sigma_\bottom(\sigma)}$ is biclosed.
Assume that~$\closure{\Sigma_\bottom(\sigma)} = S \join T$ for some~$S,T \in \Bicl{\bar Q}$.
Recall from Remark~\ref{rem:meetJoinBiclosedSets} that~$S \join T = \closure{(S \cup T)}$.

Assume first that~$\sigma \notin S \cup T$.
Since~$\sigma \in \Sigma_\bottom(\sigma) \subseteq \closure{\Sigma_\bottom(\sigma)} = S \join T = \closure{(S \cup T)}$, there exists strings~$\sigma_1, \dots, \sigma_\ell \in \closure{\Sigma_\bottom(\sigma)} \ssm \{\sigma\}$, arrows~$\alpha_1, \dots, \alpha_{\ell-1} \in Q_1$ and signs~$\varepsilon_1, \dots, \varepsilon_{\ell-1} \in \{-1,1\}$ such that~$\sigma = \sigma_1 \alpha_1^{\varepsilon_1} \dots \alpha_{\ell-1}^{\varepsilon_{\ell-1}} \sigma_\ell$.
Assume that there exists~$k \in [\ell-1]$ such that~$\varepsilon_k = 1$ and choose~$k$ maximal for this property.
Then~$\sigma_{k-1} \in \closure{\Sigma_\bottom(\sigma)} \cap \Sigma_\top(\sigma) = \{\sigma\}$ by Proposition~\ref{prop:characterizationDistinguishableStrings} and Lemma~\ref{lem:sigmaBottomCapSigmaTop}.
We obtain a similar contradiction if there exists~$k \in [\ell-1]$ such that~$\varepsilon_k = -1$.
We thus obtain that~$\sigma \in S \cup T$ and we can thus assume by symmetry that~$\sigma \in S$.

Assume now that there exists~$\tau \in \Sigma_\bottom(\sigma) \ssm S$.
Since~$\tau \ne \sigma$, we can write either~$\sigma = \sigma' \alpha \tau \beta^{-1} \sigma''$ or~$\sigma = \tau \beta^{-1} \sigma''$ or~$\sigma = \sigma' \alpha \tau$ for some arrows~$\alpha, \beta \in Q_1$ and strings~$\sigma', \sigma'' \in \Sigma_\top(\sigma)$.
Assume for example the first situation, the other two are similar.
Since~$\sigma$ is distinguishable and~${\sigma', \sigma'' \in \Sigma_\top(\sigma)}$, Lemma~\ref{lem:sigmaBottomCapSigmaTop} ensures that~$\sigma'$ and~$\sigma''$ do not belong to~$\closure{\Sigma_\bottom(\sigma)}$ and thus to~$S$.
We thus obtain that~$\sigma \in \sigma' \circ \tau \circ \sigma''$ with~$\sigma \in S$ while~$\sigma', \tau, \sigma \notin S$ which contradicts the coclosedness of~$S$.

We conclude that~$\Sigma_\bottom(\sigma) \subseteq S$. Since~$S$ is closed, we thus obtain that~$S = \closure{\Sigma_\bottom(\sigma)}$, which proves that~$\closure{\Sigma_\bottom(\sigma)}$ is join-irreducible.
\end{proof}

Recall the notion of descent from Definition~\ref{def:increasingFlip}.

\begin{lemma}
\label{lem:jiJI}
Any distinguishable string~$\sigma \in \distinguishableStrings^\pm(\bar Q)$ is the unique descent of the non-kissing facet~$\eta \big( \closure{\Sigma_\bottom(\sigma)} \big)$. In other words, $\eta \big( \closure{\Sigma_\bottom(\sigma)} \big)$ is join-irreducible in the non-kissing lattice~$\NKL$.
\end{lemma}

\begin{proof}
Assume that~$\eta \big( \closure{\Sigma_\bottom(\sigma)} \big) = F \join G$ for some non-kissing facets~$F, G \in \NKC$.
By Lemma~\ref{lem:fixProjDown} and Propositions~\ref{prop:section} and~\ref{prop:sublattice}, we have
\[
\closure{\Sigma_\bottom(\sigma)} = \projDown \big( \closure{\Sigma_\bottom(\sigma)} \big) = \zeta \big( \eta \big( \closure{\Sigma_\bottom(\sigma)} \big) \big) = \zeta(F \join G) = \zeta(F) \join \zeta(G).
\]
Since~$\closure{\Sigma_\bottom(\sigma)}$ is join-irreducible in~$\Bicl{\bar Q}$ by Lemma~\ref{lem:exmBiclosedJI}, we can assume that~$\closure{\Sigma_\bottom(\sigma)} = \zeta(F)$ so that~$\eta \big( \closure{\Sigma_\bottom(\sigma)} \big) = \eta(\zeta(F)) = F$ by Proposition~\ref{prop:surjection}.
Thus, $\eta \big( \closure{\Sigma_\bottom(\sigma)} \big) \in \JI \big( \NKL \big)$.
\end{proof}

This motivates the definition of the following map~$\ji : \distinguishableStrings^\pm(\bar Q) \to \JI \big( \NKL \big)$.

\begin{definition}
For a distinguishable string~$\sigma \in \distinguishableStrings^\pm(\bar Q)$, we define the non-kissing facet
\[
\ji(\sigma) \eqdef \eta \big( \closure{\Sigma_\bottom(\sigma)} \big) \in \JI \big( \NKL \big).
\]
\end{definition}

\begin{proposition}
\label{prop:bijectionDistinguishableStringsJI}
The map~$\ji : \sigma \mapsto \ji(\sigma)$ defines a bijection between the distinguishable strings of~$\distinguishableStrings^\pm(\bar Q)$ and the join-irreducible elements of the non-kissing lattice~$\NKL$.
\end{proposition}

\begin{proof}
To see that~$\ji$ is injective, observe that~$\sigma$ is the unique descent of~$\ji(\sigma)$ since~$\ji(\sigma)$ is join-irreducible.
To see that~$\ji$ is surjective, consider a join-irreducible~$F$ in~$\NKL$.
By Propositions~\ref{prop:surjection} and~\ref{prop:sublattice}, we have
\[
F = \eta \big( \zeta(F) \big) = \eta \Big( \bigJoin_{\omega \in F} \closure{\Sigma_\bottom(\omega)} \Big) = \bigJoin_{\omega \in F} \eta \big( \closure{\Sigma_\bottom(\omega)} \big).
\]
Since~$F$ is join-irreducible, we obtain that there exists~$\omega \in F$ such that~$F = \eta \big( \closure{\Sigma_\bottom(\omega)} \big)$.
By Lemma~\ref{lem:bijectionWalksDistinguishableStrings}, there exists a distinguishable string~$\sigma = \sigma_\bottom(\omega)$ such that~$\closure{\Sigma_\bottom(\omega)} = \closure{\Sigma_\bottom(\sigma)}$.
We conclude that~$F = \eta \big( \closure{\Sigma_\bottom(\sigma)} \big)$ so that~$\ji$ is surjective.
\end{proof}

\begin{example}
On \fref{fig:exmLatticeQuotient}, observe that the bottom elements in the congruence classes corresponding to join-irreducible facets of~$\NKL$ are precisely the sets of bottom substrings of the $5$ strings of~$\bar Q$.
\end{example}

\subsection{Canonical join-representations in~$\NKL$}

Extending similar results on grid bound quivers~\cite{GarverMcConville-grid} and on dissection bound quivers~\cite{GarverMcConville}, we now consider canonical join-represen\-tations in~$\NKL$.
To apply Proposition~\ref{prop:canonicalJoinRepresentation}, we first need to understand the join-irreducible corresponding to each cover relation that appears in  Proposition~\ref{prop:canonicalJoinRepresentation}\,(i).

\begin{lemma}
Let~$F \to F'$ be an increasing flip in~$\NKG$ supported by~$\sigma$. Then
\[
F \join \ji(\sigma) = F'
\qquad\text{and}\qquad
F \meet \ji(\sigma) = \ji(\sigma)_\star,
\]
where~$\ji(\sigma)_\star$ denotes the only non-kissing facet covered by the join-irreducible facet~$\ji(\sigma)$.
\end{lemma}

\begin{proof}
Applying Proposition~\ref{prop:covers} to the increasing flip~$F \to F'$ supported by~$\sigma$, we have
\[
\zeta \big( \ji(\sigma) \big) = \closure{\Sigma_\bottom(\sigma)} \subseteq \zeta(F')
\qquad\text{and}\qquad
\zeta(F) = \zeta(F') \ssm \{\sigma\}.
\]
We thus obtain that~$\ji(\sigma) \le F'$ so that~$F \ne F \join \ji(\sigma) \le F'$. Since~$F$ covers~$F'$ in~$\NKL$, we obtain the first equality
\(
F \join \ji(\sigma) = F'.
\)
Moreover, since~$\zeta \big( \ji(\sigma) \big) \subseteq \zeta(F')$ and~$\zeta(F) = \zeta(F') \ssm \{\sigma\}$, 
\[
\zeta(F) \meet \zeta \big( \ji(\sigma) \big) = \big( \zeta(F') \ssm \{\sigma\} \big) \cap \zeta \big( \ji(\sigma) \big) = \zeta \big( \ji(\sigma) \big) \ssm \{\sigma\} = \zeta \big( \ji(\sigma)_\star \big).
\]
where the last equality follows from Proposition~\ref{prop:covers} for the increasing flip~$\ji(\sigma)_\star \to \ji(\sigma)$ supported by~$\sigma$.
Therefore, we obtain from Propositions~\ref{prop:surjection} and~\ref{prop:sublattice} that
\[
\ji(\sigma)_\star
= \eta \big( \zeta(\sigma)_\star \big)
= \eta \big( \zeta(F) \meet \zeta \big( \ji(\sigma) \big) \big)
= \eta \big( \zeta(F) \big) \meet \eta \big( \zeta \big( \ji(\sigma) \big) \big)
= F \meet \ji(\sigma).
\qedhere
\]
\end{proof}

Together with Proposition~\ref{prop:canonicalJoinRepresentation}, this statement immediately implies the following description of the join-representation of any facet in~$\NKL$.

\begin{corollary}
\label{coro:joinRepresentations}
The canonical join-representation of a non-kissing facet~$F \in \NKC$ is
\[
F = \bigJoin_{\sigma \in \descents(F)} \ji(\sigma).
\]
\end{corollary}

\begin{remark}
Dually, the map~$\sigma \mapsto \mi(\sigma) \eqdef \eta \big( \closure{\Sigma_\top(\sigma)} \big)$ defines a bijection between the distinguishable strings of~$\distinguishableStrings^\pm(\bar Q)$ and the meet-irreducible elements of the non-kissing lattice~$\NKL$, and the canonical meet-representation of any non-kissing facet~$F \in \NKC$ is $F = \bigMeet_{\sigma \in \ascents(F)} \mi(\sigma)$.
\end{remark}

\subsection{The non-friendly complex}
\label{subsec:nonFriendlyComplex}

\enlargethispage{.4cm}
To conclude, we want to understand the canonical join complex of the non-kissing lattice~$\NKL$.
In view of Corollary~\ref{coro:joinRepresentations}, this amounts to understanding which sets of strings are descent sets of non-kissing facets of~$\NKC$.
We start with a seemingly unrelated definition extending that of~\cite[Sect.~3.4]{GarverMcConville-grid}.

\begin{definition}\label{def: friendly}
Two strings~$\sigma, \tau \in \strings^\pm(\bar Q)$ are \defn{non-friendly} if
\[
\Sigma_\top(\sigma) \cap \Sigma_\bottom(\tau) = \varnothing = \Sigma_\bottom(\sigma) \cap \Sigma_\top(\tau).
\]
The \defn{non-friendly complex}~$\NFC$ is the simplicial complex of pairwise non-friendly subsets of distinguishable strings~of~$\bar Q$.
\end{definition}

\begin{example}
The non-friendly complexes of two gentle bound quivers with~$3$ vertices are illustrated on \fref{fig:exmNFC}.
\begin{figure}[t]
	\capstart
	\centerline{\includegraphics[scale=.4]{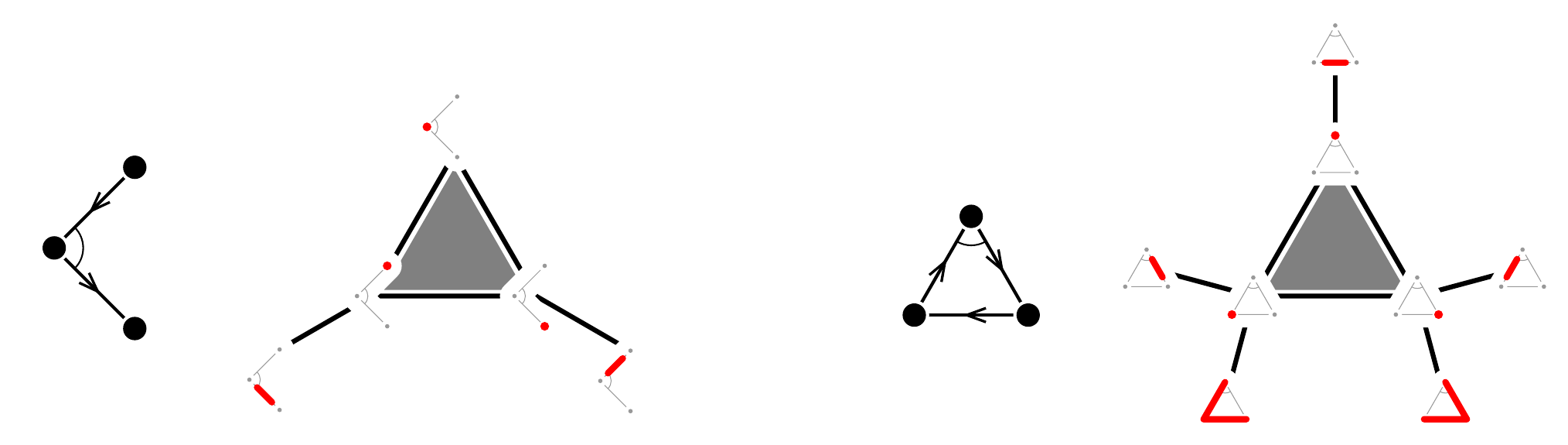}}
	\caption{The non-friendly complex~$\NFC$ for two specific quivers.}
	\label{fig:exmNFC}
\end{figure}
\end{example}

\begin{example}
When~$\bar Q$ is a path on~$n$ vertices, oriented from left to right, then the non-friendly subsets of strings of~$\strings^\pm(\bar Q)$ are in bijection with non-crossing partitions of~$[n+1]$.
The bijection is illustrated in \fref{fig:exmNCP}.
\begin{figure}[t]
	\capstart
	\centerline{\includegraphics[scale=.4]{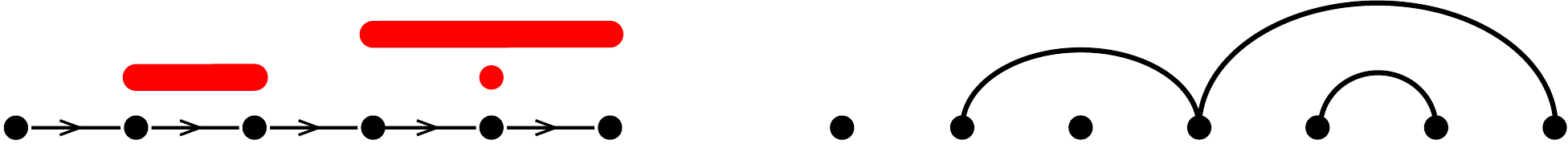}}
	\caption{The bijection between pairwise non-friendly set of strings of an oriented path on $n$ vertices and non-crossing partitions of~$[n+1]$.}
	\label{fig:exmNCP}
\end{figure}
\end{example}

\begin{theorem}
The following assertions are equivalent for a subset~$\Sigma$ of~$\distinguishableStrings^\pm(\bar Q)$:
\begin{enumerate}[(i)]
\item $\Sigma$ is pairwise non-friendly.
\item $\ji(\Sigma) \eqdef \set{\ji(\sigma)}{\sigma \in \Sigma}$ is the canonical join-representation of an element in~$\NKL$.
\item $\Sigma$ is the descent set of a non-kissing facet~$F \in \NKC$.
\end{enumerate}
Thus, the canonical join complex of~$\NKL$ is isomorphic to the non-friendly complex~$\NFC$.
\end{theorem}

\begin{proof}
We first prove~(i)$\iff$(ii), following exactly the same lines as the proof of~\cite[Thm.~4.9]{GarverMcConville-grid}.
Since both the non-friendly complex~$\NFC$ and the canonical join complex of~$\NKL$ are flag simplicial complexes, we just need to check the property for~$|\Sigma| \in \{1,2\}$.
The case $|\Sigma| = 1$ is Proposition~\ref{prop:bijectionDistinguishableStringsJI}.
Consider now the situation~$\Sigma = \{\sigma, \tau\}$.
Since~$\ji(\sigma) = \eta \big( \closure{\Sigma_\bottom(\sigma)} \big)$ and~$\eta$ is a surjective lattice homomorphism, $\ji(\sigma) \join \ji(\tau)$ is a canonical join representation in~$\NKL$ if and only if~$\Sigma_\bottom(\sigma) \join \Sigma_\bottom(\tau)$ is a canonical join-representation in~$\Bicl{\bar Q}$.
We thus just need to prove that the latter is equivalent to~$\sigma$ and~$\tau$ being non-friendly.

Assume first that~$\sigma$ and~$\tau$ are friendly and let for example~$\rho \in \Sigma_\top(\sigma) \cap \Sigma_\bottom(\tau)$.
Then~$\sigma$ decomposes into~$\sigma = \sigma' \alpha^{-1} \rho \beta \sigma''$, or~$\sigma = \rho \beta \sigma''$, or~$\sigma = \sigma' \alpha^{-1} \rho$ for some~$\sigma', \sigma'' \in \Sigma_\bottom(\sigma)$.
Assume for example the first situation, the other two are similar.
Since~$\rho \in \Sigma_\bottom(\tau)$ and~$\sigma', \sigma'' \in \Sigma_\bottom(\sigma)$, we have
\[
\closure{\Sigma_\bottom(\rho)} \subseteq \closure{\Sigma_\bottom(\tau)},
\qquad
\closure{\Sigma_\bottom(\sigma')} \subseteq \closure{\Sigma_\bottom(\sigma)}
\qquad\text{and}\qquad
\closure{\Sigma_\bottom(\sigma'')} \subseteq \closure{\Sigma_\bottom(\sigma)}.
\]
Therefore
\begin{align*}
\closure{\Sigma_\bottom(\sigma)} \join \closure{\Sigma_\bottom(\tau)}
& \subseteq \big( \closure{\Sigma_\bottom(\sigma')} \join \closure{\Sigma_\bottom(\rho)} \join \closure{\Sigma_\bottom(\sigma'')} \big) \join \closure{\Sigma_\bottom(\tau)} \\
& = \closure{\Sigma_\bottom(\sigma')} \join \closure{\Sigma_\bottom(\sigma'')} \join \closure{\Sigma_\bottom(\tau)} \\
& \subseteq \closure{\Sigma_\bottom(\sigma)} \join \closure{\Sigma_\bottom(\tau)}
\end{align*}
We thus obtain that
\[
\closure{\Sigma_\bottom(\sigma)} \join \closure{\Sigma_\bottom(\tau)} = \closure{\Sigma_\bottom(\sigma')} \join \closure{\Sigma_\bottom(\sigma'')} \join \closure{\Sigma_\bottom(\tau)}.
\]
Since~$\closure{\Sigma_\bottom(\sigma')} \subsetneq \closure{\Sigma_\bottom(\sigma)}$ and~$\closure{\Sigma_\bottom(\sigma'')} \subsetneq \closure{\Sigma_\bottom(\sigma)}$, we have obtained a smaller join-represen\-tation of~$\closure{\Sigma_\bottom(\sigma)} \join \closure{\Sigma_\bottom(\tau)}$. Therefore,~$\closure{\Sigma_\bottom(\sigma)} \join \closure{\Sigma_\bottom(\tau)}$ is not its canonical join-representation.

Conversely, assume that~$\closure{\Sigma_\bottom(\sigma)} \join \closure{\Sigma_\bottom(\tau)}$ is not its canonical join-representation, and let~${R \subseteq \distinguishableStrings^\pm(\sigma)}$ be such that
\[
\closure{\Sigma_\bottom(\sigma)} \join \closure{\Sigma_\bottom(\tau)} = \bigJoin_{\rho \in R} \closure{\Sigma_\bottom(\rho)}
\]
is the canonical join representation of~$\closure{\Sigma_\bottom(\sigma)} \join \closure{\Sigma_\bottom(\tau)}$.
Assume that there is no~$\rho \in R$ such that~$\sigma \in \Sigma_\bottom(\rho)$.
Since~$\sigma \in \closure{\Sigma_\bottom(\sigma)} \join \closure{\Sigma_\bottom(\tau)}$, we know that there exist strings~$\rho_1, \dots, \rho_\ell \in R$, substrings~$\xi_1 \in \Sigma_\bottom(\rho_1), \dots, \xi_\ell \in \Sigma_\bottom(\rho_\ell)$, arrows~$\alpha_1, \dots, \alpha_{\ell_1} \in Q_1$ and signs~${\varepsilon_1, \dots, \varepsilon_{\ell-1} \in \{-1,1\}}$ such that~$\sigma = \xi_1 \alpha_1^{\varepsilon_1} \dots \alpha_{\ell-1}^{\varepsilon_{\ell-1}} \xi_\ell$.
Assume that there is~$k \in [\ell-1]$ such that~$\varepsilon_k = 1$ and consider the prefix~$\xi \eqdef \xi_1 \alpha_1^{\varepsilon_1} \dots \alpha_{k-1}^{\varepsilon_{k-1}} \xi_k$ of~$\sigma$ (the proof for the other case is symmetric considering a suffix).
Since~$\xi \in \closure{\Sigma_\bottom(\sigma)} \join \closure{\Sigma_\bottom(\tau)}$, there exist strings~$\zeta_1, \dots, \zeta_m \in \Sigma_\bottom(\sigma) \cup \Sigma_\bottom(\tau)$, arrows~$\beta_1, \dots, \beta_{m-1} \in Q_1$ and signs~$\delta_1, \dots, \delta_{m-1} \in \{-1,1\}$ such that~$\xi = \zeta_1 \beta_1^{\delta_1} \dots \beta_{m-1}^{\delta_{m-1}} \zeta_m$.
Since~$\varepsilon_k = 1$, we have~$\zeta_m \notin \Sigma_\bottom(\sigma)$ so that~$\zeta_m \in \Sigma_\bottom(\tau)$.
Let~$n \in [m]$ be minimal such that~$\zeta \eqdef \zeta_n \beta_n^{\delta_n} \dots \beta_{m-1}^{\varepsilon_{m-1}} \zeta_m \in \closure{\Sigma_\bottom(\tau)}$.
We distinguish two cases:
\begin{itemize}
\item If~$n = 1$, then~$\zeta \in \Sigma_\top(\sigma)$ since~$\varepsilon_k = 1$.
\item If~$n > 1$, then~$\zeta_{n-1} \notin \Sigma_\bottom(\tau)$ (by minimality of~$n$), so that~$\zeta_{n-1} \in \Sigma_\bottom(\sigma)$, and we get~$\delta_{n-1} = -1$. Since~$\varepsilon_k = 1$, we thus obtain that~$\zeta \in \Sigma_\top(\sigma)$.
\end{itemize}
In both cases, we conclude that~$\zeta \in \Sigma_\top(\sigma) \cap \closure{\Sigma_\bottom(\tau)}$, so that~$\sigma$ and~$\tau$ are friendly along~$\zeta$ by Lemma~\ref{lem:sigmaBottomCapSigmaTop}.
This concludes the proof of~(i)$\iff$(ii). Finally, the equivalence (ii)$\iff$(iii) is proved by Corollary~\ref{coro:joinRepresentations}.
\end{proof}


\newpage
\part{The non-kissing associahedron}
\label{part:geometry}

In this section, we provide geometric realizations of the non-kissing complex~$\RNKC$ when it is finite.
Inspired by similar constructions of~\cite{MannevillePilaud-accordion} for dissection bound quivers and~\cite{GarverMcConville} for grid bound quivers, we first construct a complete simplicial fan~$\gvectorFan$ realizing~$\RNKC$~(\ref{subsec:gvectorFan}).
We then prove that~$\gvectorFan$ is the normal fan of a polytope~$\Asso$~(\ref{subsec:associahedron}) and that the graph of this polytope oriented in a suitable direction is the Hasse diagram of the non-kissing lattice~$\NKL$~(\ref{subsec:geomLattice}).
We then discuss the existence of a zonotope~$\Zono$ which could play the same role for~$\Asso$ as the classical permutahedron plays for the classical associahedron~(\ref{subsec:zonotope}).
Finally, we discuss certain transformations of the quiver~$Q$ that correspond to natural projections of~$\Asso$~(\ref{subsec:projections}).


\section{Recollections on polyhedral geometry}
\label{sec:polyhedralGeometry}

We briefly recall basic definitions and properties of polyhedral fans and polytopes, and refer to~\cite{Ziegler-polytopes} for a classical textbook on this topic.

\begin{definition}
A \defn{polyhedral cone} is a subset of~$\R^n$ defined equivalently as the positive span of finitely many vectors or as the intersection of finitely many closed linear halfspaces.
Throughout the paper, we write~$\R_{\ge0}\b{R}$ for the positive span of a set~$\b{R}$ of vectors of~$\R^n$.
The \defn{faces} of a cone~$C$ are the intersections of~$C$ with the supporting hyperplanes of~$C$.
The $1$-dimensional (resp.~codimension~$1$) faces of~$C$ are called~\defn{rays} (resp.~\defn{facets}) of~$C$.
A cone is \defn{simplicial} if it is generated by a set of independent vectors.
\end{definition}

\begin{definition}
A \defn{polyhedral fan} is a collection~$\Fan$ of polyhedral cones such that
\begin{itemize}
\item if~$C \in \Fan$ and~$F$ is a face of~$C$, then~$F \in \Fan$,
\item the intersection of any two cones of~$\Fan$ is a face of both.
\end{itemize}
A fan is \defn{simplicial} if all its cones are, and \defn{complete} if the union of its cones covers the ambient space~$\R^n$.
\end{definition}

The following statement characterizes complete simplicial fans.
A formal proof can be found \eg in~\cite[Coro.~4.5.20]{DeLoeraRambauSantos}.

\begin{proposition}
\label{prop:characterizationFan}
Consider a pseudomanifold~$\Delta$ on a finite vertex set~$V$ and a set of vectors $\big( \b{r}(v) \big)_{v \in V}$ of~$\R^n$.
For~$\triangle \in \Delta$, define the cone~$\b{r}(\triangle) \eqdef \bigset{\b{r}(v)}{v \in \triangle}$.
Then the collection of cones~${\bigset{\R_{\ge 0}\b{r}(\triangle)}{\triangle \in \Delta}}$ forms a complete simplicial fan if and only if
\begin{enumerate}
\item there exists a facet~$\triangle$ of~$\Delta$ such that~$\b{r}(\triangle)$ is a basis of~$\R^n$ and such that the open cones~$\R_{> 0}\b{r}(\triangle)$ and~$\R_{> 0}\b{r}(\triangle')$ are disjoint for any facet~$\triangle'$ of~$\Delta$ distinct from~$\triangle$;
\item for two adjacent facets~$\triangle, \triangle'$ of~$\Delta$ with~$\triangle \ssm \{v\} = \triangle' \ssm \{v'\}$, there is a linear dependence
\[
\alpha \, \b{r}(v) + \alpha' \, \b{r}(v') + \sum_{w \in \triangle \cap \triangle'} \beta_w \, \b{r}(w) = 0
\]
on~$\b{r}(\triangle \cup \triangle')$ where the coefficients~$\alpha$ and~$\alpha'$ have the same sign.
(When these conditions hold, these coefficients do not vanish and the linear dependence is unique up to rescaling.)
\end{enumerate}
\end{proposition}

\begin{definition}
A \defn{polytope} is a subset~$P$ of~$\R^n$ defined equivalently as the convex hull of finitely many points or as a bounded intersection of finitely many closed affine halfspaces.
The \defn{faces} of~$P$ are the intersections of~$P$ with its supporting hyperplanes.
The dimension~$0$ (resp.~dimension~$1$, resp.~codimension~$1$) faces of~$P$ are called \defn{vertices} (resp.~\defn{edges}, resp.~\defn{facets}) of~$P$.
\end{definition}

\begin{definition}
The (outer) \defn{normal cone} of a face~$F$ of~$P$ is the cone generated by the outer normal vectors of the facets of~$P$ containing~$F$.
The (outer) \defn{normal fan} of~$P$ is the collection of the (outer) normal cones of all its faces.
We say that a complete polyhedral fan in~$\R^n$ is \defn{polytopal} when it is the normal fan of a polytope of~$\R^n$.
\end{definition}

The following statement provides a characterization of polytopality of complete simplicial fans.
It is a reformulation of regularity of triangulations of vector configurations, introduced in the theory of secondary polytopes~\cite{GelfandKapranovZelevinsky}, see also~\cite{DeLoeraRambauSantos}.
We present here a convenient formulation from~\cite[Lem.~2.1]{ChapotonFominZelevinsky}.

\begin{proposition}
\label{prop:characterizationPolytopalFan}
Consider a pseudomanifold~$\Delta$ with vertex set~$V$ and a set of vectors~$\big( \b{r}(v) \big)_{v \in V}$ of~$\R^n$ such that~$\Fan \eqdef \bigset{\R_{\ge 0}\b{r}(\triangle)}{\triangle \in \Delta}$ is a complete simplicial fan in~$\R^n$.
Then the following are equivalent:
\begin{enumerate}
\item $\Fan$ is the normal fan of a simple polytope in~$\R^n$;
\item There exists a map~$h: V \to \R_{> 0}$ such that for any two adjacent facets~$\triangle, \triangle'$ of~$\Delta$ with~$\triangle \ssm \{v\} = \triangle' \ssm \{v'\}$, we have
\[
\alpha \, h(v) + \alpha' \, h(v') + \sum_{w \in \triangle \cap \triangle'} \beta_{w} \, h(w) > 0,
\]
where
\[
\alpha \, \b{r}(v) + \alpha' \, \b{r}(v') + \sum_{w \in \triangle \cap \triangle'} \beta_{w} \, \b{r}(w) = 0
\]
is the unique (up to rescaling) linear dependence with~$\alpha, \alpha' > 0$ between the rays of~${\b{r}(\triangle \cup \triangle')}$.
\end{enumerate}
Under these conditions, $\Fan$ is the normal fan of the polytope defined by
\[
\bigset{\b{x} \in \R^n}{\dotprod{\b{r}(v)}{\b{x}} \le h(v) \text{ for all } v \in V}.
\]
\end{proposition}


\section{Non-kissing geometry}
\label{sec:associahedron}

\subsection{$\b{g}$-vectors and $\b{c}$-vectors}
\label{subsec:gcvectors}

Let~$(\b{e}_v)_{v \in Q_0}$ be the standard basis of~$\R^{Q_0}$.
We first define an analog of the characteristic vector for multisets.

\begin{definition}
For a multiset~$V \eqdef \{v_1, \dots, v_m\}$ of vertices of~$Q_0$, we denote by~$\multiplicityVector_V \in \R^{Q_0}$ the \defn{multiplicity vector} of~$V$ defined~by
\[
\multiplicityVector_V \eqdef \sum_{i \in [m]} \b{e}_{v_i} = \sum_{v \in Q_0} |\set{i \in [m]}{v_i = v}| \, \b{e}_v.
\]
For a string~$\sigma \in \strings^\pm(\bar Q)$, we define~$\multiplicityVector_\sigma \eqdef \multiplicityVector_{V(\sigma)}$ where~$V(\sigma)$ is the multiset of vertices of~$\sigma$. 
\end{definition}

We now define two families of vectors that will play a crucial role in the geometric construction.

\begin{definition}\label{def: g-vectors for walks}
For a walk~$\omega \in \NKWalks^\pm(\bar Q)$, we denote by~$\peaks{\omega}$ (resp.~by~$\deeps{\omega}$) the multiset of vertices of~$Q_0$ corresponding to the peaks (resp.~deeps) of~$\omega$.
The \defn{$\b{g}$-vector} of a walk~$\omega$ is the vector~${\gvector{\omega} \in \R^{Q_0}}$ defined by
\[
\gvector{\omega} \eqdef \multiplicityVector_{\peaks{\omega}} - \multiplicityVector_{\deeps{\omega}}.
\]
For a set~$\Omega$ of walks, we let~$\gvectors{\Omega} \eqdef \set{\gvector{\omega}}{\omega \in \Omega}$.
Observe that~$\gvector{\omega} = 0$ for a straight walk~$\omega$.
\end{definition}

\begin{example}
The $\b{g}$-vector of the plain blue walk on \fref{fig:exmFacet}\,(left) is~$(0,-1,0,0,1,0)$ (see \fref{fig:exmBlossomingQuiver} for the labeling order of the vertices of the quiver~$Q$).
See \fref{fig:gcvectors} for more examples of $\b{g}$-vectors.
\end{example}

\begin{remark}\label{rem: g-vectors coincide}
One easily checks that if $\sigma$ is a string for a gentle bound quiver $\bar Q$, and $\omega(\sigma)$ the associated walk in the blossoming bound quiver $\bar Q\blossom$ as in Lemma~\ref{lem:bijections-STTC-NKC}, then $\gvector{\sigma}=\gvector{\omega(\sigma)}$.
\end{remark}

\begin{definition}\label{def: c-vectors}
Consider a bending walk~$\omega$ in a non-kissing facet~$F \in \RNKC$.
Recall from Proposition~\ref{prop:distinguishedArrows} that the walk~$\omega$ carries two distinguished arrows~$\distinguishedArrows{\omega}{F}$ surrounding its distinguished string~$\distinguishedString{\omega}{F}$ (see \fref{fig:gcvectors}).
Define
\[
\distinguishedSign{\omega}{F} \eqdef 
\begin{cases}
\phantom{-}1 & \text{if~$\distinguishedString{\omega}{F}$ is a top substring of~$\omega$ (\ie the arrows~$\distinguishedArrows{\omega}{F}$ point outside),} \\
-1 & \text{if~$\distinguishedString{\omega}{F}$ is a bottom substring of~$\omega$ (\ie if the arrows~$\distinguishedArrows{\omega}{F}$ point inside).}
\end{cases}
\]
The \defn{$\b{c}$-vector} of a walk~$\omega$ of~$F$ is the vector~${\cvector{\omega}{F} \in \R^{Q_0}}$ defined by
\[
\cvector{\omega}{F} \eqdef \distinguishedSign{\omega}{F} \, \multiplicityVector_{\distinguishedString{\omega}{F}}.
\]
We let~$\cvectors{F} \eqdef \set{\cvector{\omega}{F}}{\omega \in F}$ be the set of $\b{c}$-vectors of a non-kissing facet~$F \in \RNKC$, and~$\allcvectors \eqdef \bigcup_F \cvectors{F}$ be the set of all $\b{c}$-vectors of all non-kissing facets~${F \in \RNKC}$.
\end{definition}

\begin{example}
In the facet~$F$ of \fref{fig:exmFacet}\,(right), the $\b{c}$-vector of the blue walk is~$(0,-1,0,0,0,0)$ while the $\b{c}$-vector of the yellow path is~$(0,1,0,1,1,0)$ (see \fref{fig:exmBlossomingQuiver} for the labeling order of the vertices of the quiver~$Q$ and \fref{fig:gcvectors}\,(left) to see the distinguished arrows in the facet~$F$).
See \fref{fig:gcvectors} for more examples of $\b{c}$-vectors.

\begin{figure}[t]
	\capstart
	\centerline{$\raisebox{-1.5cm}{\begin{overpic}[scale=.7]{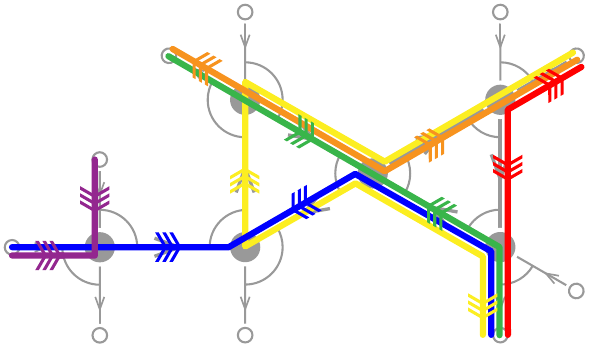}\put(60,4){$F$}\end{overpic}}$ \qquad \input{gcmatrices}}
	\caption{The $\b{g}$- and $\b{c}$-matrices of the facet~$F$ of \fref{fig:exmFacet}. Their columns correspond to the $\b{g}$- and $\b{c}$-vectors of the different walks of~$F$. See \fref{fig:exmBlossomingQuiver} for the labeling order of the vertices of the quiver~$Q$. Note that~$\gvectors{F} \cdot \cvectors{F}^T = \one$.}
	\label{fig:gcvectors}
\end{figure}
\end{example}

\begin{example}
\label{exm:gcvectors}
The $\b{g}$- and $\b{c}$-vectors of the peak and deep facets~$F_\peak$ and~$F_\deep$ are given by
\begin{gather*}
\gvector{v_\peak} = \b{e}_v, \quad 
\gvector{v_\deep} = -\b{e}_v, \quad
\cvector{v_\peak}{F_\peak} = \b{e}_v, \quad\text{and}\quad
\cvector{v_\deep}{F_\deep} = -\b{e}_v.
\end{gather*}
\end{example}

\begin{remark}
\label{rem:signCoherence}
The definitions of $\b{g}$- and $\b{c}$-vectors immediately imply the \defn{sign-coherence property}: for any non-kissing facet~$F \in \RNKC$,
\begin{itemize}
\item for any~$\omega \in F$, all coordinates of the $\b{c}$-vector~$\cvector{\omega}{F}$ have the same sign,
\item for~$v \in Q_0$, the~$v$th coordinates of all $\b{g}$-vectors~$\gvector{\omega}$ for~$\omega \in F$ have the same sign (otherwise, the corresponding walks would kiss at~$v$).
\end{itemize}
In particular, a $\b{c}$-vector is called \defn{positive} when all its coordinates are non-negative and \defn{negative} when all its coordinates are non-positive.
\end{remark}

\begin{remark}
\label{rem:setcvectors}
There is a subtle relation between $\b{c}$-vectors and strings. Indeed, $\b{c}$-vectors are multiplicity vectors of the vertex multisets of some strings of~$\bar Q$, but
\begin{itemize}
\item not all strings are distinguishable (see Definition~\ref{def:distinguishedSubstring} and Proposition~\ref{prop:characterizationDistinguishableStrings}), and
\item distinct distinguishable strings may have identical vertex multisets.
\end{itemize}
\end{remark}

The $\b{g}$- and $\b{c}$-vectors have been designed so that the following statement holds, as illustrated in \fref{fig:gcvectors}.

\begin{proposition}
\label{prop:gvectorscvectorsDualBases}
For any non-kissing facet~$F \in \RNKC$, the set of $\b{g}$-vectors~$\gvectors{F}$ and the set of $\b{c}$-vectors~$\cvectors{F}$ form dual bases. 
\end{proposition}

\begin{proof}
Consider two walks~$\mu, \nu \in F$.
By definition, the $\b{g}$-vector~$\gvector{\mu}$ counts the corners of~$\omega$, while the $\b{c}$-vector~$\cvector{\nu}{F}$ counts, up to a sign, the vertices on the distinguished string~$\distinguishedString{\nu}{F}$. 
To compute the scalar product~$\dotprod{\gvector{\mu}}{\cvector{\nu}{F}}$, we thus need to understand which corners of~$\mu$ lie on the distinguished string~$\distinguishedString{\nu}{F}$.

Assume first that~$\mu = \nu$. If the distinguished arrows~$\distinguishedArrows{\omega}{F}$ point outside, then~$\cvector{\mu}{F}$ is positive and there is one more peak than deeps on the distinguished string~$\distinguishedString{\mu}{F}$.
Otherwise, ${\cvector{\mu}{F}}$ is negative and there is one more deep than peaks on the distinguished string~$\distinguishedString{\mu}{F}$.
In both cases, we obtain~$\dotprod{\gvector{\mu}}{\cvector{\mu}{F}} = 1$.

Assume now that~$\mu \ne \nu$.
Consider a maximal common substring~$\sigma$ in the intersection of~$\mu$ with~$\nu$.
Since they have opposite directions, the distinguished arrows of~$\nu$ cannot both belong to~$\sigma$.
Therefore, the distinguished string~$\distinguishedString{\nu}{F}$ covers either the whole string~$\sigma$, or nothing, or an initial or final substring of~$\sigma$.
Moreover, the arrows of~$\mu$ incident to but not in~$\distinguishedString{\nu}{F}$ have the same direction along~$\mu$.
Therefore, the distinguished string~$\distinguishedString{\nu}{F}$ contains as many peaks as deeps of~$\mu$ in~$\sigma$.
We thus showed that~$\sigma$ does not contributes to~$\dotprod{\gvector{\mu}}{\cvector{\nu}{F}}$.
Summing the contribution of all common substrings of~$\mu$ and~$\nu$, we obtain that~$\dotprod{\gvector{\mu}}{\cvector{\nu}{F}} = 0$ when~$\mu \ne \nu$.
We conclude that~$\gvectors{F}$ and~$\cvectors{F}$ are dual bases.
\end{proof}

\subsection{The $\b{g}$-vector fan}
\label{subsec:gvectorFan}

We now restrict our attention to a gentle bound quiver~$\bar Q$ with a finite non-kissing complex~$\RNKC$.
We present a fan realization of~$\RNKC$ inspired by similar statements of T.~Manneville and V.~Pilaud in~\cite{MannevillePilaud-accordion} for dissection bound quivers and by A.~Garver and T.~McConville in~\cite{GarverMcConville-grid} for grid bound quivers.
As illustrated in \fref{fig:all2}, this construction actually provides fan realizations of all non-kissing complexes (finite or not).
Our elementary argument proves the finite case, the general case can be proven using~\cite{DehyKeller, DemonetIyamaJasso} (see~Remark~\ref{rem:linDep}).

\begin{theorem}
\label{thm:gvectorFan}
For a gentle bound quiver~$\bar Q$ with finite non-kissing complex~$\RNKC$, the collection of cones
\[
\gvectorFan \eqdef \bigset{\R_{\ge0} \gvectors{F}}{F \text{ non-kissing face of } \RNKC}
\]
forms a complete simplicial fan, that we call the \defn{$\b{g}$-vector fan} of~$\bar Q$.
\end{theorem}

\begin{proof}
We use the characterization of complete simplicial fans of Proposition~\ref{prop:characterizationFan}.
By Remark~\ref{rem:signCoherence}, the cone~$\R_{\ge0} \gvectors{F_\peak}$ is the only cone of~$\gvectorFan$ intersecting the interior of the positive orthant~$(\R_{\ge0})^{Q_0}$.
Consider now two adjacent non-kissing facets~$F,F' \in \RNKC$.
Let~$\omega \in F$ and~$\omega' \in F'$ be such that~${F \ssm \{\omega\} = F' \ssm \{\omega'\}}$, and let~$\mu$ and~$\nu$ be the two other walks involved in the flip as defined in Proposition~\ref{prop:flip} (see \fref{fig:flip}).
Note that a vertex~$v \in Q_0$ is a peak of none of (resp.~one of, resp.~both) the walks~$\omega, \omega'$ if and only if it is a peak of none of (resp.~one of, resp.~both) the walks~$\mu, \nu$.
The same holds for deeps.
We thus have the linear~dependence
\[
\gvector{\omega} + \gvector{\omega'} = \gvector{\mu} + \gvector{\nu}
\]
among the $\b{g}$-vectors involved in the flip.
Note that this equality holds even when~$\mu$ and~$\nu$ coincide, or when~$\mu$ or~$\nu$ are straight walks (in which case their $\b{g}$-vector vanishes).
This shows that~$\gvectorFan$ satisfies the two conditions of Proposition~\ref{prop:characterizationFan}, and thus shows that $\gvectorFan$ is a fan.
\end{proof}

\begin{remark}
The linear dependence~$\gvector{\omega} + \gvector{\omega'} = \gvector{\mu} + \gvector{\nu}$ relating the $\b{g}$-vectors of two adjacent non-kissing facets~$F,F' \in \RNKC$ shows that~${\det \big( \gvector{F} \big) = - \det \big( \gvector{F'} \big)}$.
Since the initial cone~$\R_{\ge0} \gvectors{F_\peak}$ is generated by the coordinate vectors (see Example~\ref{exm:gcvectors}), we obtain that~$\det \big( \gvector{F} \big) = \pm 1$ for all non-kissing facets~$F \in \RNKC$, so that the $\b{g}$-vector fan~$\gvectorFan$ is~\defn{smooth}.
\end{remark}

\begin{remark}
\label{rem:cvectorFan}
Call \defn{$\b{c}$-vector fan} the fan~$\cvectorFan$ defined by the arrangement of the hyperplanes orthogonal to the $\b{c}$-vectors of~$\allcvectors$.
Be aware that contrarily to the $\b{g}$-vector fan whose rays are the $\b{g}$-vectors, the $\b{c}$-vectors are not the rays but the normal vectors of the hyperplanes of the $\b{c}$-vector fan.
By Proposition~\ref{prop:gvectorscvectorsDualBases}, any non-maximal cone of~$\gvectorFan$ is supported by an hyperplane orthogonal to a $\b{c}$-vector of~$\allcvectors$.
The $\b{g}$-vector fan~$\gvectorFan$ thus coarsens the $\b{c}$-vector~fan~$\cvectorFan$.
\end{remark}

\begin{example}
To illustrate Theorem~\ref{thm:gvectorFan} and Remark~\ref{rem:cvectorFan}, we have represented in \fref{fig:exmFans} the $\b{c}$-vector fan~$\cvectorFan$ and the $\b{g}$-vector fan~$\gvectorFan$ for two specific gentle bound quivers.
We use the classical projection to represent~$3$-dimensional fans: the fan is intersected with the unit sphere and stereographically projected to the plane from the pole in direction~$(1,1,1)$.
More generally, \fref{fig:allFans3} gathers the $\b{c}$-vector fan~$\cvectorFan$ and the $\b{g}$-vector fan~$\gvectorFan$ for all possible connected simple gentle bound quivers with~$3$ vertices such that the non-kissing complex~$\RNKC$ is finite (\ie where there is a relation in any cycle, oriented or not).
See also \fref{fig:all2} for the $\b{g}$-vector fans of all connected gentle bound quivers on $2$ vertices.

\begin{figure}[t]
	\capstart
	\centerline{\includegraphics[width=1.1\textwidth]{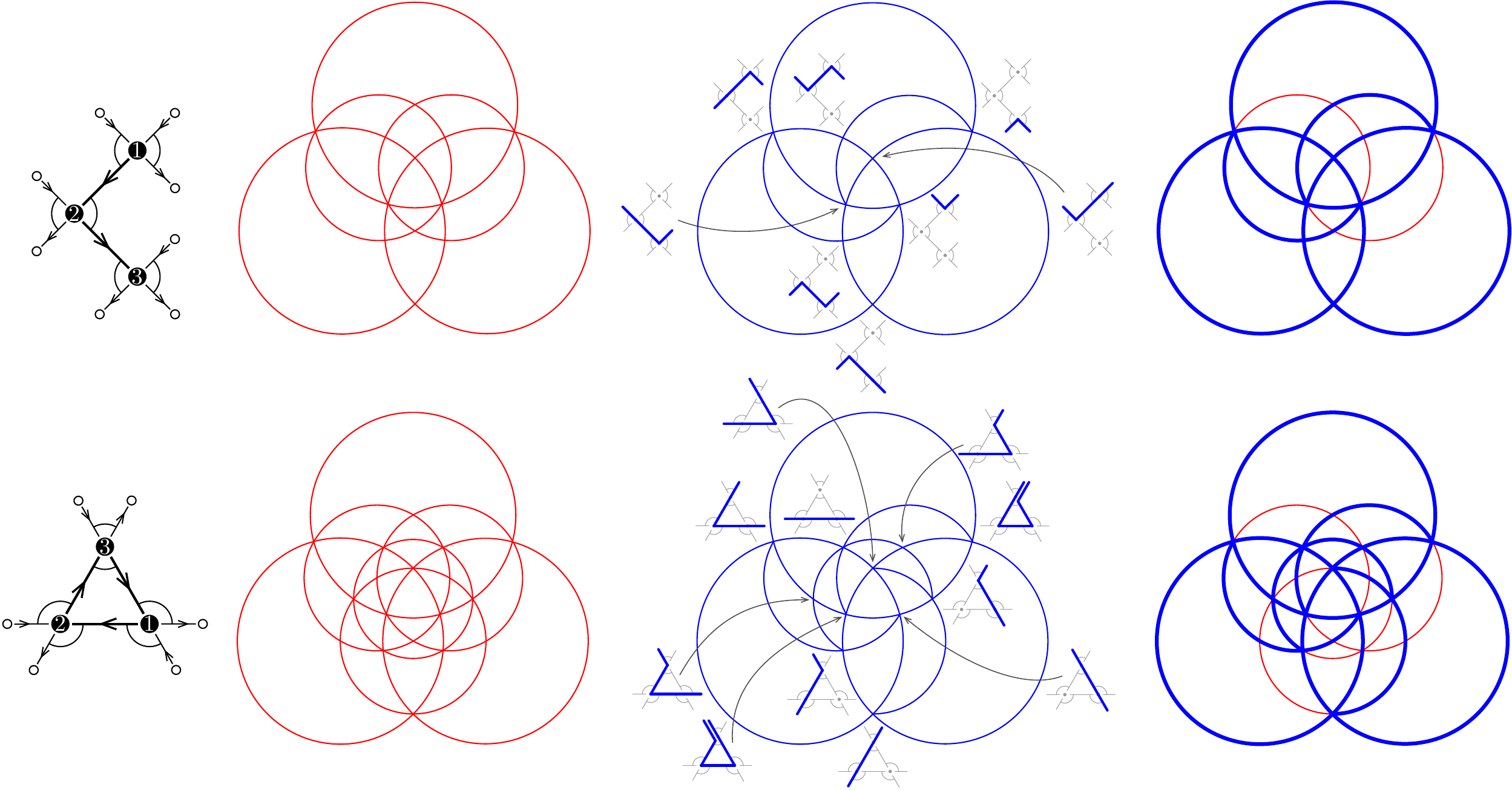}}
	\caption{Stereographic projections of the $\b{c}$-vector fans~$\cvectorFan$ (red, left) and the $\b{g}$-vector fans~$\gvectorFan$ (blue, middle) for two specific quivers. The~$3$-dimensional fan is intersected with the unit sphere and stereographically projected to the plane from the pole in direction~$(1,1,1)$. Note that the $\b{g}$-vector fan is supported by the $\b{c}$-vector fan (right).}
	\label{fig:exmFans}
\end{figure}
\end{example}

\begin{example}
For the dissection and grid bound quivers introduced in Section~\ref{subsec:dissectionGridQuivers}, the $\b{g}$-vector fans already appeared in earlier works:
\begin{itemize}
\item For a dissection~$D$, the $\b{g}$-vector fan~$\gvectorFan[\bar Q(D)]$ was constructed in~\cite{MannevillePilaud-accordion}.
\item For a connected subset~$L$ of~$\Z^2$, the $\b{g}$-vector fan~$\gvectorFan[\bar Q(L)]$ was constructed in~\cite{GarverMcConville-grid}. We note however that our proof, quite different from that of~\cite{GarverMcConville-grid}, will be decisive to obtain polytopal realizations in the next section.
\item For oriented paths, $\b{g}$-vector fans are known as type~$A$ Cambrian fans, constructed by N.~Reading and D.~Speyer~\cite{ReadingSpeyer} in connection with type~$A$ cluster algebras~\cite{FominZelevinsky-ClusterAlgebrasI, FominZelevinsky-ClusterAlgebrasII}.
\end{itemize}
Note that the topmost quiver of \fref{fig:exmFans} is both a dissection and a grid bound quiver, so that its $\b{g}$-vector fan was constructed independently in~\cite{MannevillePilaud-accordion} and~\cite{GarverMcConville-grid}.
\end{example}

\begin{remark}
\label{rem:linDep}
\enlargethispage{.1cm}
Theorem~\ref{thm:gvectorFan} extends to any finite-dimensional algebra~$A$. More precisely:
\begin{itemize}
\item As pointed out to us by L.~Demonet, a similar linear dependence holds for a mutation in the support $\tau$-tilting complex of any finite-dimensional algebra~$A$.
Indeed, the~$\b{g}$-vector of a~$\tau$-rigid pair can be identified with the class of the corresponding $2$-term complex in the Grothendieck group of the bounded homotopy category of finitely generated projective modules over~$A$.
The existence of an exchange triangle $X \rightarrow \bigoplus_{i \in [k]} Y_i \rightarrow X' \rightarrow X[1]$ (see~\cite{AiharaIyama}) immediately gives the equality~$[X] + [X'] = \sum_{i \in [k]} [Y_i]$ in the Grothendieck group, yielding the claimed linear dependence among~$\b{g}$-vectors.
\item By Proposition~\ref{prop:characterizationFan}, this suffices to prove Theorem~\ref{thm:gvectorFan} for any finite-dimensional algebra~$A$ with finite support $\tau$-tilting complex. 
\item This simple argument fails when the support $\tau$-tilting complex is infinite. However, the set~$\gvectorFan$ is still a simplicial fan (not necessarily complete). This can be proven using the (much harder) result~\cite{DehyKeller, DemonetIyamaJasso} that $\tau$-rigid objects are determined by their~$\b{g}$-vectors.
\end{itemize}
\end{remark}

\subsection{The non-kissing associahedron}
\label{subsec:associahedron}

We still consider a gentle bound quiver~$\bar Q$ with finite non-kissing complex~$\RNKC$.
We now construct a polytopal realization of the $\b{g}$-vector fan~$\gvectorFan$.

\begin{definition}\label{def: kissing number}
For two walks~$\omega, \omega' \in \NKWalks^\pm(\bar Q)$, denote by~$\kn(\omega,\omega')$ the number of kisses of~$\omega$ to~$\omega'$.
Note that kisses are counted with multiplicities if~$\omega$ or~$\omega'$ pass twice through the same substring.
The \defn{kissing number} of~$\omega$ and~$\omega'$ is~$\KN(\omega,\omega') \eqdef \kn(\omega,\omega') + \kn(\omega',\omega)$.
Since~$\RNKC$ is finite, we can define the \defn{kissing number} of a walk~$\omega$ on~$\bar Q$ as
\[
\KN(\omega) \eqdef \sum_{\omega'} \KN(\omega,\omega'),
\]
where the sum runs over all walks~$\omega'$ of~$\NKWalks^\pm(\bar Q)$.
\end{definition}

\begin{remark}\label{rem:KNvsTau}
Let~$\omega : \strings_{\ge -1}^\pm(\bar Q) \to \bendingWalks^\pm(\bar Q)$ be the bijection of Lemma~\ref{lem:bijections-STTC-NKC}.
Then one easily checks using Proposition~\ref{prop:morphismsStringModules} that when~$\NKC$ is finite, we have for any~$\sigma \in \strings_{\ge -1}^\pm(\bar Q)$
\[
\KN \big( \omega(\sigma) \big) = \sum_{\rho} \dim \Hom{A} \big( M(\sigma), \tau M(\rho) \big) + \dim \Hom{A} \big( M(\rho), \tau M(\sigma) \big),
\]
where the sum runs over all other strings~$\rho$ of~$\strings_{\ge -1}^\pm(\bar Q)$.
\end{remark}

\begin{lemma}
\label{lem:submodular}
For any adjacent non-kissing facets~$F,F' \in \RNKC$, the values of the kissing number~$\KN$ on the walks~$\omega, \omega', \mu, \nu$ defined in Proposition~\ref{prop:flip} (see \fref{fig:flip}) satisfy the inequality
\[
\KN(\omega) + \KN(\omega') > \KN(\mu) + \KN(\nu).
\]
\end{lemma}

\begin{proof}
Observe in \fref{fig:flip}\,(right) that for any walk~$\lambda$, a kiss of~$\lambda$ with one of (resp.~both) the walks~$\mu$ and~$\nu$ is as well a kiss with one of (resp.~both) the walks~$\omega$ and~$\omega'$.
In other words, $\KN(\lambda, \omega) + \KN(\lambda, \omega') \ge \KN(\lambda, \mu) + \KN(\lambda, \nu)$.
Moreover, the walks~$\omega$ and~$\omega'$ kiss each other but do not kiss~$\mu$ and~$\nu$.
We conclude that~$\KN(\omega) + \KN(\omega') \ge \KN(\mu) + \KN(\nu) + 2$.
\end{proof}

\begin{definition}
Let~$\bar Q$ be a gentle bound quiver with finite non-kissing complex~$\RNKC$.
Define
\begin{enumerate}[(i)]
\item for each non-kissing facet~${F \in \RNKC}$, a point
\[
\point{F} \eqdef \sum_{\omega \in F} \KN(\omega) \, \cvector{\omega}{F},
\]
\item for each walk~${\omega \in \NKWalks^\pm(\bar Q)}$, a halfspace
\[
\HS{\omega} \eqdef \set{\b{x} \in \R^{Q_0}}{\dotprod{\gvector{\omega}}{\b{x}} \le \KN(\omega)}.
\]
\end{enumerate}
\end{definition}

\begin{theorem}
\label{thm:associahedron}
For a gentle bound quiver~$\bar Q$ with finite non-kissing complex~$\RNKC$, the $\b{g}$-vector fan~$\gvectorFan$ is the normal fan of the \defn{$\bar Q$-associahedron}~$\Asso$ defined equivalently~as
\begin{enumerate}[(i)]
\item the convex hull of the points~$\point{F}$ for all non-kissing facets~${F \in \RNKC}$, or
\item the intersection of the halfspaces~$\HS{\omega}$ for all walks~${\omega \in \NKWalks^\pm(\bar Q)}$.
\end{enumerate}
\end{theorem}

\begin{proof}
Consider two adjacent non-kissing facets~$F,F'$ and define the walks~$\omega, \omega', \mu, \nu$ as in Proposition~\ref{prop:flip} (see \fref{fig:flip}).
We have seen in the proof of Theorem~\ref{thm:gvectorFan} that the linear dependence between the $\b{g}$-vectors of~$F \cup F'$ is~$\gvector{\omega} + \gvector{\omega'} = \gvector{\mu} + \gvector{\nu}$.
Moreover, according to Lemma~\ref{lem:submodular}, the kissing number satisfies~$\KN(\omega) + \KN(\omega') > \KN(\mu) + \KN(\nu)$.
Applying the characterization of Proposition~\ref{prop:characterizationPolytopalFan}, we thus immediately obtain that the $\b{g}$-vector fan~$\gvectorFan$ is the normal fan of the polytope defined as the intersection of the halfspaces~$\HS{\omega}$ for all walks~$\omega$ on~$\bar Q$.

Finally, to show the vertex description, we just need to observe that, for any non-kissing facet~$F \in \RNKC$, the point~$\point{F}$ is the intersection of the hyperplanes~$\Hyp{\omega}$ for all~$\omega \in F$.
Indeed, since~$\gvectors{F}$ and~$\cvectors{F}$ form dual bases by Proposition~\ref{prop:gvectorscvectorsDualBases}, we have for any~$\omega \in F$:
\[
\dotprod{\gvector{\omega}}{\point{F}} = \sum_{\mu \in F} \KN(\mu) \, \dotprod{\gvector{\omega}}{\cvector{\mu}{F}} = \sum_{\mu \in F} \KN(\mu) \, \delta_{\omega = \mu} \; = \; \KN(\omega).
\qedhere
\]
\end{proof}

\begin{example}
To illustrate Theorem~\ref{thm:associahedron}, we have represented in \fref{fig:exmAssociahedra}\,(middle) the associahedron~$\Asso$ for two specific gentle bound quivers.
More generally, \fref{fig:allAssociahedra3}\,(middle) gathers the associahedron~$\Asso$ for all possible connected simple gentle bound quivers with~$3$ vertices such that the non-kissing complex~$\RNKC$ is finite (\ie where there is a relation in any cycle, oriented or not).
Their normal fans are represented in \fref{fig:allFans3}\,(middle).
See also \fref{fig:all2} for the associahedra of all connected gentle bound quivers on $2$ vertices.

\begin{figure}[t]
	\capstart
	\centerline{\includegraphics[width=1.1\textwidth]{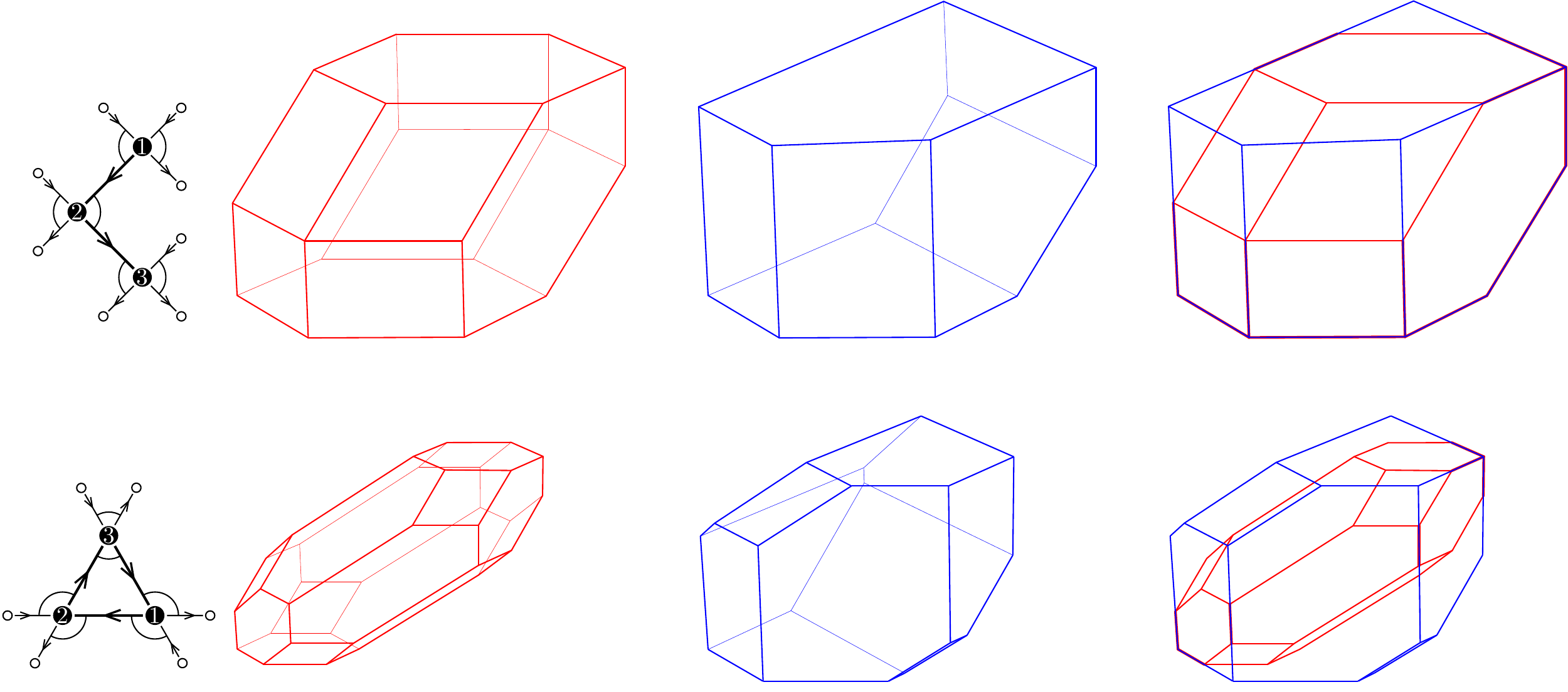}}
	\caption{The zonotope~$\Zono$ (red, left) and the associahedron~$\Asso$ (blue, middle) for two specific quivers. Observe that the facet defining inequalities of~$\Asso$ are not necessarily facet defining inequalities of~$\Zono$ (right).}
	\label{fig:exmAssociahedra}
\end{figure}
\end{example}

\begin{example}
The dissection and grid bound quivers introduced in Section~\ref{subsec:dissectionGridQuivers} give particularly relevant examples:
\begin{itemize}
\item For a dissection~$D$, the $\bar Q(D)$-associahedron of the dissection bound quiver~$\bar Q(D)$ was previously constructed in~\cite{MannevillePilaud-accordion} as a projection of a well-chosen associahedron.
For example, the top line of \fref{fig:exmAssociahedra} already appeared in~\cite{MannevillePilaud-accordion}.
\item In contrast, the non-kissing associahedron was not known for a grid bound quiver. Our construction thus answers a question of A.~Garver and T.~McConville~\cite{GarverMcConville-grid} concerning the polytopality of the $\b{g}$-vector fan~$\gvectorFan[\bar Q(L)]$ for a subset~$L$ of~$\Z^2$. Note that alternative polytopal realizations for the grid-associahedra were constructed in~\cite{SantosStumpWelker} using the dual of a nice triangulation of an order polytope and in~\cite[Sect.~4]{McConville} by a sequence of suspensions and edge subdivisions.
\item For oriented paths, the non-kissing associahedra are classical associahedra, constructed by J.-L.~Loday~\cite{Loday} and C.~Hohlweg and C.~Lange~\cite{HohlwegLange} and revisited in~\cite{HohlwegLangeThomas, Stella, PilaudSantos-brickPolytope, LangePilaud, HohlwegPilaudStella} (among others).
\end{itemize}
\end{example}

\begin{remark}
\label{rem:polytope}
As in Remark~\ref{rem:linDep}, Theorem~\ref{thm:associahedron} in fact holds over any~$\tau$-tilting finite algebra~$A$.
The kissing number has to be replaced by the dimension of some Hom-space, as in Remark~\ref{rem:KNvsTau}.
More precisely, consider the functions
\[
\KN_\ell(X) = \sum_Y \dim \Hom{} \big( X, Y[1] \big)
\qquad\text{and}\qquad
\KN_r(X) = \sum_Y \dim \Hom{} \big( Y, X[1] \big),
\]
where the sums run over all indecomposable $2$-term presilting complexes.
An exchange triangle $X \rightarrow \bigoplus_{i \in [k]} Y_i \rightarrow X' \rightarrow X[1]$ (see~\cite{AiharaIyama}) then gives rise to
\begin{itemize}
\item a linear dependence~$[X] + [X'] = \sum_{i \in [k]} [Y_i]$ in the Grothendieck group (see Remark~\ref{rem:linDep}),
\item an inequality of the form~$\KN(X) + \KN(X') > \sum_{i \in [k]} \KN(Y_i)$ for~$\KN \in \{\KN_\ell, \KN_r\}$.
\end{itemize}
The proof is then similar to that of Theorem~\ref{thm:associahedron} and leads to the same results using either~$\KN_\ell$, $\KN_r$, or more symmetrically their sum~$\KN_\ell + \KN_r$.
Compare to Remark~\ref{rem:KNvsTau}.
\end{remark}

\subsection{Non-kissing associahedron and lattice}
\label{subsec:geomLattice}

The following observation is an immediate consequence of the fact that the normal fan of the $\bar Q$-associahedron~$\Asso$ is the $\b{g}$-vector fan~$\gvectorFan$.

\begin{lemma}
Consider two adjacent non-kissing facets~$F,F'$ and let~$\omega, \omega', \mu$ and~$\nu$ denote the walks involved in the flip as defined in Proposition~\ref{prop:flip} (see \fref{fig:flip}).
Then the difference~${\point{F'} - \point{F}}$ is a strictly negative multiple of~$\cvector{\omega}{F} = -\cvector{\omega'}{F'}$.
More precisely,
\[
\point{F'} - \point{F} = \big( \KN(\mu) + \KN(\nu) - \KN(\omega) - \KN(\omega') \big) \, \cvector{\omega}{F}.
\]
\end{lemma}

\begin{proof}
One can translate Remark~\ref{rem:changeDistinguishedArrows} in terms of $\b{c}$-vectors as follows:
\begin{itemize}
\item For~$\omega$ and~$\omega'$, we have $\cvector{\omega}{F} = -\cvector{\omega'}{F'}$.
\item For~$\mu$ and~$\nu$, two different situations can happen:
\begin{itemize}
\item If~$\mu = \nu$, then it is a bending walk and~$\cvector{\mu}{F'} = \cvector{\mu}{F} + 2\,\cvector{\omega}{F}$.
\item Otherwise, whenever $\mu$ and~$\nu$ are bending walks, we have
\[
\cvector{\mu}{F'} = \cvector{\mu}{F} + \cvector{\omega}{F}
\quad\text{and}\quad
\cvector{\nu}{F'} = \cvector{\nu}{F} + \cvector{\omega}{F}.
\]
\end{itemize}
\item Finally, $\cvector{\lambda}{F} = \cvector{\lambda}{F'}$ for all walks~$\lambda \in (F \cap F') \ssm \{\mu,\nu\}$
\end{itemize}
The formula thus immediately follows from the definition of~$\point{F}$.
Finally, Lemma~\ref{lem:submodular} asserts that~${\KN(\omega) + \KN(\omega') > \KN(\mu) + \KN(\nu)}$, which shows that~${\point{F'} - \point{F}}$ is a strictly negative multiple of~$\cvector{\omega}{F} = -\cvector{\omega'}{F'}$.
\end{proof}

\begin{proposition}
When oriented in the linear direction~$-\one \eqdef (-1, \dots, -1) \in \R^{Q_0}$, the graph of the $\bar Q$-associahedron~$\Asso$ is (isomorphic to) the increasing flip graph~$\NKG$.
\end{proposition}

\begin{proof}
For any two adjacent non-kissing facets~$F,F'$ with~$F \ssm \{\omega\} = F' \ssm \{\omega'\}$, the scalar product
\[
\dotprod{-\one}{\point{F'} - \point{F}} = \big( \KN(\mu) + \KN(\nu) - \KN(\omega) - \KN(\omega') \big) \dotprod{-\one}{\cvector{\omega}{F}}
\]
has the same sign as~$\cvector{\omega}{F}$.
However, the flip~$F \to F'$ is increasing if and only if the \mbox{$\b{c}$-vector}~$\cvector{\omega}{F}$ is positive.
This concludes the proof.
\end{proof}

\subsection{Zonotope}
\label{subsec:zonotope}

We have seen in Remark~\ref{rem:cvectorFan} that the $\b{c}$-vector fan~$\cvectorFan$ coarsens the $\b{g}$-vector fan~$\gvectorFan$.
The goal of this short section is to discuss the possibility to visualize this property at the level of polytopes.
The prototypes are the associahedra of C.~Hohlweg and C.~Lange which can be obtained by deleting inequalities in the facet description of the permtahedron~\cite{HohlwegLange}.
We are interested in a similar behavior for arbitrary gentle bound quivers.

We still consider only gentle bound quivers~$\bar Q$ with finite non-kissing complexes~$\RNKC$.
In particular, there are only finitely many strings, so that the following Minkowski sum is \mbox{well-defined}.

\begin{definition}
The \defn{$\bar Q$-zonotope} is the Minkowski sum
\[
\Zono \eqdef \sum_{\sigma \in \distinguishableStrings^\pm(\bar Q)} [-\multiplicityVector_\sigma, \multiplicityVector_\sigma] \; = \sum_{\b{c} \in \allcvectors} \big|\set{\sigma \in \distinguishableStrings^\pm(\bar Q)}{\b{c} = \pm \multiplicityVector_\sigma} \big| \, \b{c}
\]
where~$\distinguishableStrings^\pm(\bar Q)$ still denotes the set of distinguishable strings~of~$\bar Q$ (see Definition~\ref{def:distinguishedSubstring}), $\multiplicityVector_\sigma$ is the multiplicity vector of the multiset of vertices of~$\sigma$, and~$\allcvectors$ is the set of all $\b{c}$-vectors.
\end{definition}

\begin{example}
\fref{fig:exmAssociahedra}\,(left) illustrates the zonotope~$\Zono$ for two specific gentle bound quivers.
More generally, \fref{fig:allAssociahedra3}\,(top) gathers the zonotopes~$\Zono$ for all possible connected simple gentle bound quivers with~$3$ vertices such that the non-kissing complex~$\RNKC$ is finite.
Their normal fans are represented in \fref{fig:allFans3}\,(top).
See also \fref{fig:all2} for the zonotopes of all connected gentle bound quivers on $2$ vertices.
\end{example}

\begin{remark}
Since~$\allcvectors = \set{\pm \multiplicityVector_\sigma}{\sigma \in \distinguishableStrings^\pm(\bar Q)}$, the normal fan of the zonotope~$\Zono$ is by construction the $\b{c}$-vector fan~$\cvectorFan$.
In particular, the graph of~$\Zono$, oriented in the direction~$-\one$ is the poset of regions of the $\b{c}$-vector fan~$\cvectorFan$ from the positive orthant as defined by A.~Bj\"orner, P.~Edelman and G.~Ziegler in~\cite{BjornerEdelmanZiegler}.
In particular, when~$\cvectorFan$ is simplicial, the oriented graph of~$\Zono$ is the Hasse diagram of a lattice.
However, the $\b{c}$-vector fan is not simplicial in general, as illustrated by Figures~\ref{fig:exmAssociahedra} and~\ref{fig:allAssociahedra3}.
In fact, the oriented graph of~$\Zono$ is not always the Hasse diagram of a lattice, as observed in~\cite[Rem.~6.2]{McConville} (the first counter-example is given by a square quiver).
The lattice of biclosed sets~$\Bicl{\bar Q}$ was designed in~\cite{McConville} to play the role of this missing lattice structure on~$\Zono$.
\end{remark}

\begin{example}
\label{exm:permutahedron}
When~$\bar Q$ is a type~$A$ quiver, the $\bar Q$-zonotope~$\Zono$ is the Minkowski sum of all characteristic vectors of intervals (\ie all type~$A_n$ roots), thus yielding the classical permutahedron.
In this situation, the $\bar Q$-associahedron~$\Asso$ is thus obtained by deleting inequalities in the facet description of the $\bar Q$-zonotope~$\Zono$.
See~\cite{Loday, HohlwegLange, LangePilaud} for details.
\end{example}

\begin{example}
For a dissection quiver~$\bar Q = \bar Q(D)$, T.~Manneville and V.~Pilaud proved in~\cite{MannevillePilaud-accordion} that the $\bar Q$-associahedron~$\Asso$ is obtained by deleting inequalities in the facet description of the $\bar Q$-zonotope~$\Zono$.
\end{example}

Figures~\ref{fig:exmAssociahedra}, \ref{fig:all2} and~\ref{fig:allAssociahedra3} already illustrate that this property might fail for arbitrary gentle bound quivers.
It would be interesting to characterize those quivers for which it holds.

\begin{question}
For which gentle bound quivers~$\bar Q$ are all inequalities of the $\bar Q$-associahedron~$\Asso$ also inequalities of the $\bar Q$-zonotope~$\Zono$?
\end{question}

Although we cannot provide a general answer to this question, we want to exhibit a sufficient (but certainly not necessary) condition for this property.
We need the following definition.

\begin{definition}
We say that two walks~$\omega, \omega'$ of~$\bar Q$ are \defn{mutually kissing} if~$\omega$ kisses~$\omega'$ and~$\omega'$ kisses~$\omega$, that is, $\kn(\omega,\omega') \, \kn(\omega',\omega) > 0$.
\end{definition}

\begin{proposition}
\label{prop:zonotope}
If~$\bar Q$ is a gentle bound quiver with a finite non-kissing complex and no two mutually kissing walks, then all facet inequalities of the $\bar Q$-associahedron~$\Asso$ are also facet inequalities of the $\bar Q$-zonotope~$\Zono$.
In other words, the $\bar Q$-associahedron~$\Asso$ is obtained by deleting inequalities in the facet description of the $\bar Q$-zonotope~$\Zono$.
\end{proposition}

\begin{proof}
Consider a bending walk~$\omega \in \bendingWalks^\pm(\bar Q)$ and the corresponding facet inequality of~$\Asso$ given by~${\dotprod{\gvector{\omega}}{\b{x}} \le \KN(\omega)}$.
We want to show that~$\dotprod{\gvector{\omega}}{\b{x}} \le \KN(\omega)$ is also a facet inequality of~$\Zono$.
For this, we associate to each bending walk~$\omega'$ a negative $\b{c}$-vector~$\b{c}^-(\omega')$ and a positive $\b{c}$-vector~$\b{c}^+(\omega')$ defined by
\[
\b{c}^-(\omega') \eqdef
\begin{cases}
-\multiplicityVector_{\sigma_\bottom(\omega')} & \text{if } \omega' \notin F_\peak \\
\b{0} & \text{otherwise}
\end{cases}
\qquad\text{and}\qquad
\b{c}^+(\omega') \eqdef 
\begin{cases}
\multiplicityVector_{\sigma_\top(\omega')} & \text{if } \omega' \notin F_\deep \\
\b{0} & \text{otherwise}
\end{cases}
\]
where~$\sigma_\bottom(\omega')$ and~$\sigma_\top(\omega')$ are the distinguishable strings defined in Definition~\ref{def:bijectionWalksDistinguishableStrings} and Remark~\ref{rem:bijectionDistinguishableStringsWalks}.
By Proposition~\ref{prop:bijectionDistinguishableStringsWalks}, the $\bar Q$-zonotope~$\Zono$ can be rewritten as the Minkowski sum of these $\b{c}$-vectors:
\[
\Zono \eqdef \sum_{\sigma \in \distinguishableStrings^\pm(\bar Q)} [-\multiplicityVector_\sigma, \multiplicityVector_\sigma] = \sum_{\omega' \in \bendingWalks^\pm(\bar Q)} \big( [\b{0}, \b{c}^-(\omega')] + [\b{0}, \b{c}^+(\omega')] \big).
\]
Therefore, the facet inequality of~$\Zono$ corresponding to the normal vector~$\gvector{\omega}$ is given~by
\[
\dotprod{\gvector{\omega}}{\b{x}} \le \sum_{\omega' \in \bendingWalks^\pm(\bar Q)} \max(\dotprod{\gvector{\omega}}{\b{c}^-(\omega')},0) + \max(\dotprod{\gvector{\omega}}{\b{c}^+(\omega')},0) \\
\]

Consider now a bending walk~$\omega' \in \bendingWalks^\pm(\bar Q)$.
We want to understand when~$\dotprod{\gvector{\omega}}{\b{c}^+(\omega')} > 0$ (the case~$\dotprod{\gvector{\omega}}{\b{c}^-(\omega')} > 0$ is similar).
If~$\dotprod{\gvector{\omega}}{\b{c}^+(\omega')} > 0$, then the walks~$\omega$ and~$\omega'$ share common vertices.
Let~$\tau$ be a maximal common substring of~$\omega$ and~$\omega'$.
If~$\omega$ enters and leaves~$\tau$ in the same direction, then it has as many peaks as deeps along~$\tau$, so that~$\dotprod{\gvector{\omega}}{\multiplicityVector_{\tau}} = 0$.
In contrast, if~$\omega$ enters and leaves~$\tau$ with two outgoing (resp.~incoming) arrows, then~$\omega$ has one more peak than deep (resp.~one more deep than peak) along~$\tau$, so that~$\dotprod{\gvector{\omega}}{\multiplicityVector_{\tau}} = 1$ (resp.~$-1$).
Observe now that if~$\tau$ is not included in~$\sigma \eqdef \sigma_\top(\omega')$, then~$\omega$ must enter and leave the substring~$\tau \ssm \sigma$ with incoming arrows so that~$\dotprod{\gvector{\omega}}{\multiplicityVector_{\tau \ssm \sigma}} = 0$.
We conclude that~$\tau$ contributes~$0$ (resp.~$1$, resp.~$-1$) to~$\dotprod{\gvector{\omega}}{\b{c}^+(\omega)}$ when~$\omega$ and~$\omega'$ are crossing (resp.~when~$\omega$ kisses~$\omega'$, resp.~when~$\omega'$ kisses~$\omega$) at~$\tau$.
Since~$\omega$ and~$\omega'$ are not mutually kissing, all maximal common substrings of~$\omega$ and~$\omega'$ contribute with the same sign.
Therefore, we have~$\max(\dotprod{\gvector{\omega}}{\b{c}^+(\omega)}, 0) = \kn(w,w')$.
We obtain similarly that~$\max(\dotprod{\gvector{\omega}}{\b{c}^-(\omega)}, 0) = \kn(w',w)$, so that 
\[
\max(\dotprod{\gvector{\omega}}{\b{c}^-(\omega')},0) + \max(\dotprod{\gvector{\omega}}{\b{c}^+(\omega')},0) = \kn(w,w') + \kn(w',w) = \KN(\omega,\omega')
\]
and therefore the maximum of~$\dotprod{\gvector{\omega}}{\b{x}}$ over~$\Zono$ is given by
\[
\sum_{\substack{\omega' \in \bendingWalks^\pm(\bar Q)}} \KN(\omega,\omega') = \KN(\omega).
\qedhere
\]
\end{proof}


\subsection{Coordinate sections and projections}
\label{subsec:projections}

We conclude the paper with a natural operation on quivers that illustrates a recent result of~\cite{PilaudPlamondonStella}.
Intuitively, we blind a strict subset of the vertices of a gentle bound quiver, meaning that we forbid peaks and deeps at these blinded vertices.
More formally, this corresponds to the following operation.

\begin{definition}
Consider a gentle bound quiver~$\bar Q = (Q,I)$ and a vertex~$v \in Q_0$.
Denote by~$K_v \eqdef \set{\alpha\beta}{\alpha, \beta \in Q_1 \text{ with } t(\alpha) = s(\beta)}$ the set of paths of length~$2$ with middle vertex~$v$ and let~$I_v \eqdef K_v \cap I$ and~$J_v \eqdef K_v \ssm I$.
We denote by~$\bar Q\blinkers{v} = (Q\blinkers{v}, I\blinkers{v})$ the bound quiver where
\begin{gather*}
(Q\blinkers{v})_0 = Q_0 \ssm \{v\},
\qquad\qquad
(Q\blinkers{v})_1 = \set{\alpha \in Q_1}{s(\alpha) \ne v \ne t(\alpha) } \cup J_v,
\\
I\blinkers{v} = \set{\sigma \in I \ssm I_v}{s(\sigma) \ne v \ne t(\sigma)} \cup \set{\alpha\beta\gamma}{\alpha\beta \in I \text{ and } \beta\gamma \in J_v \; \text{ or } \; \alpha\beta \in J_v \text{ and } \beta\gamma \in I}.
\end{gather*}
We say that~$\bar Q\blinkers{v}$ is obtained by \defn{blinding}~$v$ in~$\bar Q$.
Finally, for a subset~$V = \{v_1, \dots, v_\ell\}$ of~$Q_0$, we let~$\bar Q\blinkers{V} \eqdef ((\bar Q\blinkers{v_1})\blinkers{v_2} \dots )\blinkers{v_\ell}$.
\end{definition}

\begin{remark}
The reader is invited to check that
\begin{itemize}
\item the blinded quiver~$Q\blinkers{v}$ is still a gentle bound quiver.
\item blinding commutes with blossoming: $(Q\blinkers{v})\blossom = (Q\blossom)\blinkers{v}$.
\item blinding commutes with reversing: $\reversed{(Q\blinkers{v})} = (\reversed{Q})\blinkers{v}$.
\item $\bar Q\blinkers{V}$ is independent of the ordering of~$V$.
\end{itemize}
\end{remark}

\begin{proposition}
\label{prop:projection}
Consider a gentle bound quiver~$\bar Q$, a subset~$V \subsetneq Q_0$, and the quiver~$\bar Q\blinkers{V}$.~Then
\begin{enumerate}[(i)]
\item The path algebra~$A\blinkers{V}$ of the blinded quiver~$\bar Q\blinkers{V}$ is the subalgebra of the path algebra~$A$ of~$\bar Q$ given by~$A\blinkers{V} = (1-\varepsilon_V) \cdot A \cdot (1-\varepsilon_V)$, where~$\varepsilon_V \eqdef \sum_{v \in V} \varepsilon_v$.
\item The non-kissing complex~$\NKC[\bar Q\blinkers{V}]$ is isomorphic to the subcomplex of the non-kissing complex~$\NKC[\bar Q]$ induced by walks with no corner at a vertex of~$V$.
\item The $\b{g}$-vector fan~$\gvectorFan[\bar Q\blinkers{V}]$ is the section~$\set{C \in \gvectorFan}{C \subseteq V^\perp}$ of the $\b{g}$-vector fan~$\gvectorFan$ by the coordinate plane~$V^\perp \eqdef \set{\b{x} \in \R^{Q_0}}{\dotprod{\b{e}_v}{\b{x}} = 0 \text{ for all } v \in V}$.
\item When~$\NKC$ is finite, the normal fan of the projection of the non-kissing associahedron~$\Asso$ to the coordinate plane~$V^\perp$ is the $\b{g}$-vector fan~$\gvectorFan[\bar Q\blinkers{V}]$.
\end{enumerate}
\end{proposition}

\begin{proof}
For~(i), observe that the paths in~$\bar Q\blinkers{V}$ correspond to the paths in~$\bar Q$ that do not start or end at a vertex of~$V$ (the correspondence consists in grouping blocks of consecutive letters~$\alpha\beta$ for which~$t(\alpha) = s(\beta) \in V$).

For~(ii), we just need to observe that a walk~$\omega$ of~$\bar Q$ with no corner at a vertex of~$V$ corresponds to a walk~$\omega\blinkers{V}$ of~$\bar Q\blinkers{V}$ (the correspondence consists again in grouping blocks of consecutive letters~$\alpha\beta$ for which~$t(\alpha) = s(\beta) \in V$).
This correspondence clearly sends non-kissing walks to non-kissing walks, and thus induces an isomorphism between the subcomplex of the non-kissing complex~$\NKC[\bar Q]$ induced by walks with no corner at a vertex of~$V$ and the non-kissing complex~$\NKC[\bar Q\blinkers{V}]$.

Moreover, this correspondence sends peaks (resp.~deeps) of~$\omega$ to peaks (resp.~deeps) of~$\omega\blinkers{V}$.
Since the $v$-th coordinate of the $\b{g}$-vector of a walk~$\omega$ is the number of peaks minus the number of deeps of~$\omega$ at~$v$, we obtain that
\[
\gvector{\omega} \in V^\perp
\qquad\text{and}\qquad
\pi_V \big( \gvector{\omega} \big) = \gvector{\omega\blinkers{V}}
\]
for a walk of~$\bar Q$ with no corner at a vertex of~$V$, where~$\pi_V$ denotes the coordinate projection on the hyperplane~$V^\perp$.
This shows~(iii).

Finally, (iv) immediately follows from~(iii) since the normal fan of a projection of a polytope~$P$ on a subspace~$V$ is the section of the normal fan of~$P$ by the subspace~$V$ \cite[Lemma~7.11]{Ziegler-polytopes}.
\end{proof}

\begin{example}
To illustrate Proposition~\ref{prop:projection}, we have represented in \fref{fig:exmKillVertices} all possible coordinate projections of the non-kissing associahedra~$\Asso$ for two specific gentle bound quivers.

\begin{figure}[t]
	\capstart
	\centerline{\includegraphics[width=1.2\textwidth]{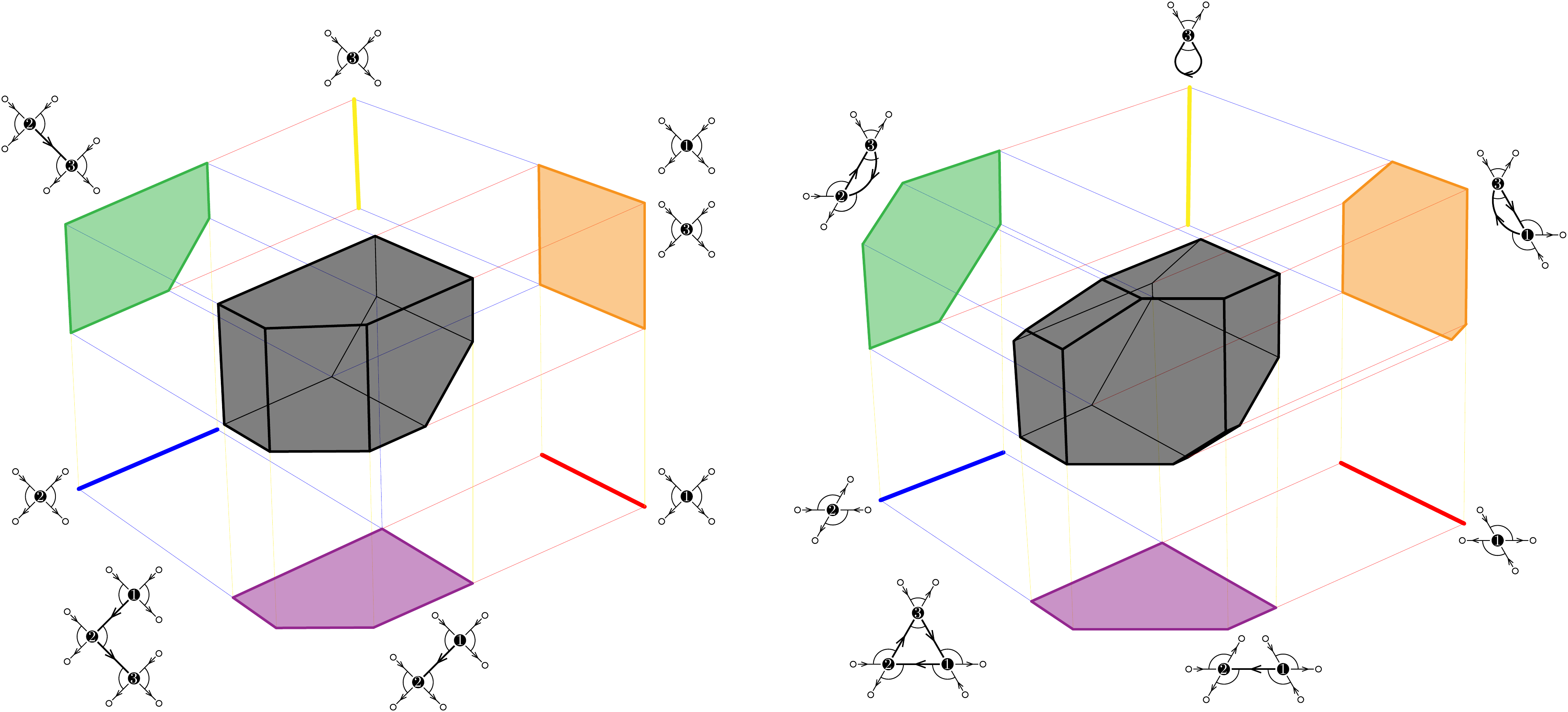}}
	\caption{All possible coordinate projections of the non-kissing associahedra~$\Asso$ for two specific gentle bound quivers.}
	\label{fig:exmKillVertices}
\end{figure}
\end{example}

\begin{example}
Consider a triangulation~$T$ of a convex polygon and a dissection~$D$ whose edges are included in~$T$.
Then the dissection bound quiver~$\bar Q(D)$ coincides with the blinded bound quiver~$\bar Q(T)\blinkers{V}$ where~$V$ are the vertices of~$\bar Q(T)$ corresponding to the diagonals of~$T$ not in~$D$.
In particular, the $\b{g}$-vector fan of any dissection bound quiver~$\bar Q(D)$ is a section of the $\b{g}$-vector fan of a type~$A$ quiver, and is realized by a projection of the associahedron of~\cite{HohlwegPilaudStella}.
See~\cite{PilaudPlamondonStella}.
\end{example}


\captionsetup{width=1.5\textwidth}
\hvFloat[floatPos=p, capWidth=h, capPos=right, capAngle=90, objectAngle=90, capVPos=c, objectPos=c]{figure}
{\begin{minipage}{23.2cm}\vspace*{-1cm}\includegraphics[scale=.25]{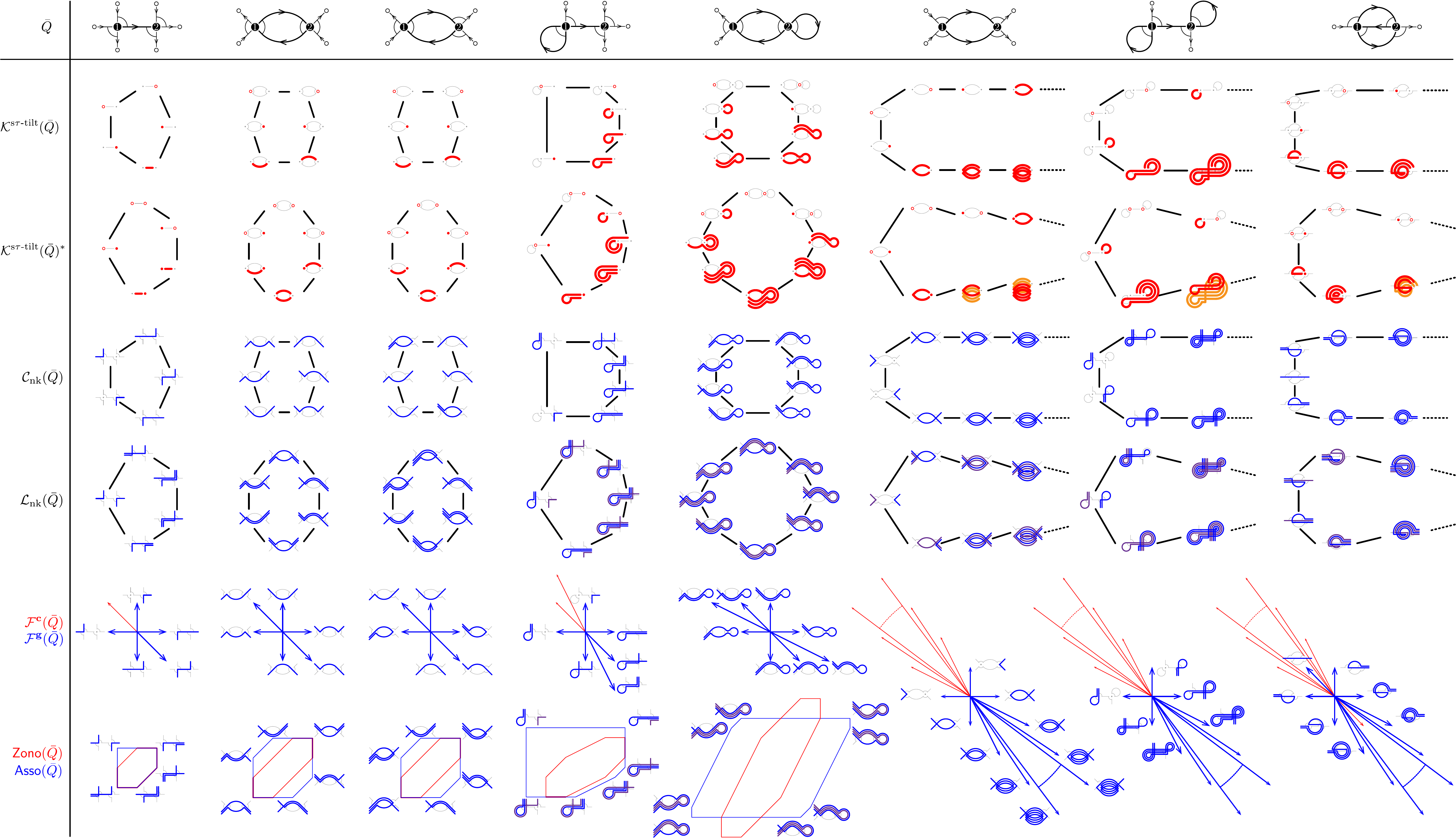}\end{minipage}}
{
For all connected gentle bound quivers~$\bar Q$ on $2$ vertices: the support $\tau$-tilting complex~$\tTC$, its dual graph~$\tTC^*$, the reduced non-kissing complex~$\RNKC$, the non-kissing oriented flip graph~$\NKG$, the $\b{c}$-vector fan~$\cvectorFan$ and the $\b{g}$-vector fan~$\gvectorFan$, and the $\bar Q$-zonotope~$\Zono$ and the $\bar Q$-associahedron~$\Asso$ (when~$\NKC$ is finite).
}
{fig:all2}
\captionsetup{width=\textwidth}

\captionsetup{width=1.5\textwidth}
\hvFloat[floatPos=p, capWidth=h, capPos=right, capAngle=90, objectAngle=90, capVPos=c, objectPos=c]{figure}
{\includegraphics[scale=.27]{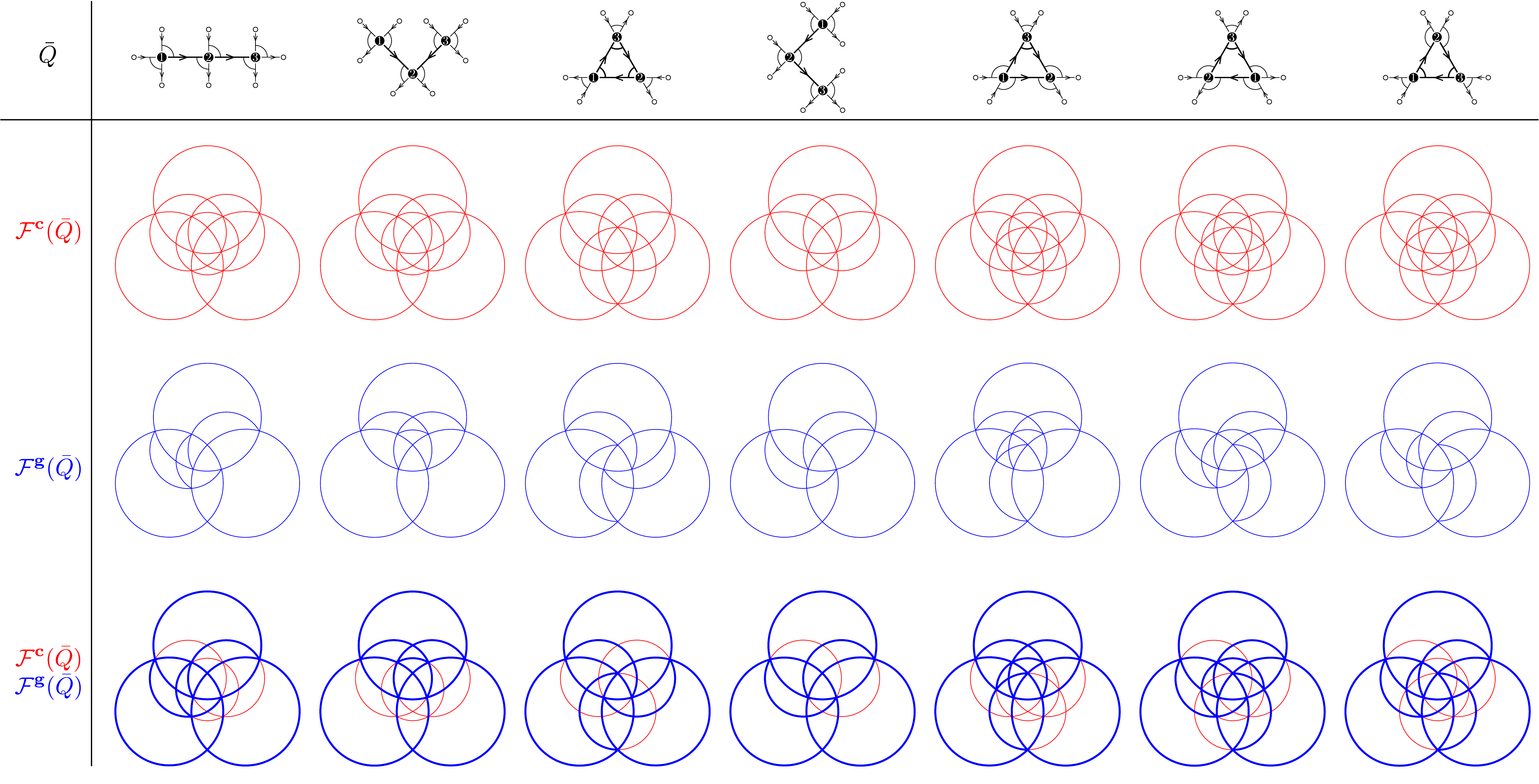}}
{Stereographic projections of the $\b{c}$-vector fan~$\cvectorFan$ (red, top) and the $\b{g}$-vector fan~$\gvectorFan$ (blue, middle) for all connected simple gentle bound quivers~$\bar Q$ on~$3$ vertices with finite non-kissing complex~$\RNKC$. Note that the $\b{g}$-vector fan is supported by the $\b{c}$-vector fan (bottom). The~$3$-dimensional fans are intersected with the unit sphere and stereographically projected to the plane from the pole in direction~$(1,1,1)$. The first three fans are $\b{g}$-vector fans of type~$A$ cluster algebras, the fourth is a fan realization of an accordion complex from~\cite{MannevillePilaud-accordion} or of a non-kissing grid complex from~\cite{GarverMcConville-grid}, the last three are new.}
{fig:allFans3}
\captionsetup{width=\textwidth}

\captionsetup{width=1.5\textwidth}
\hvFloat[floatPos=p, capWidth=h, capPos=right, capAngle=90, objectAngle=90, capVPos=c, objectPos=c]{figure}
{\includegraphics[scale=.27]{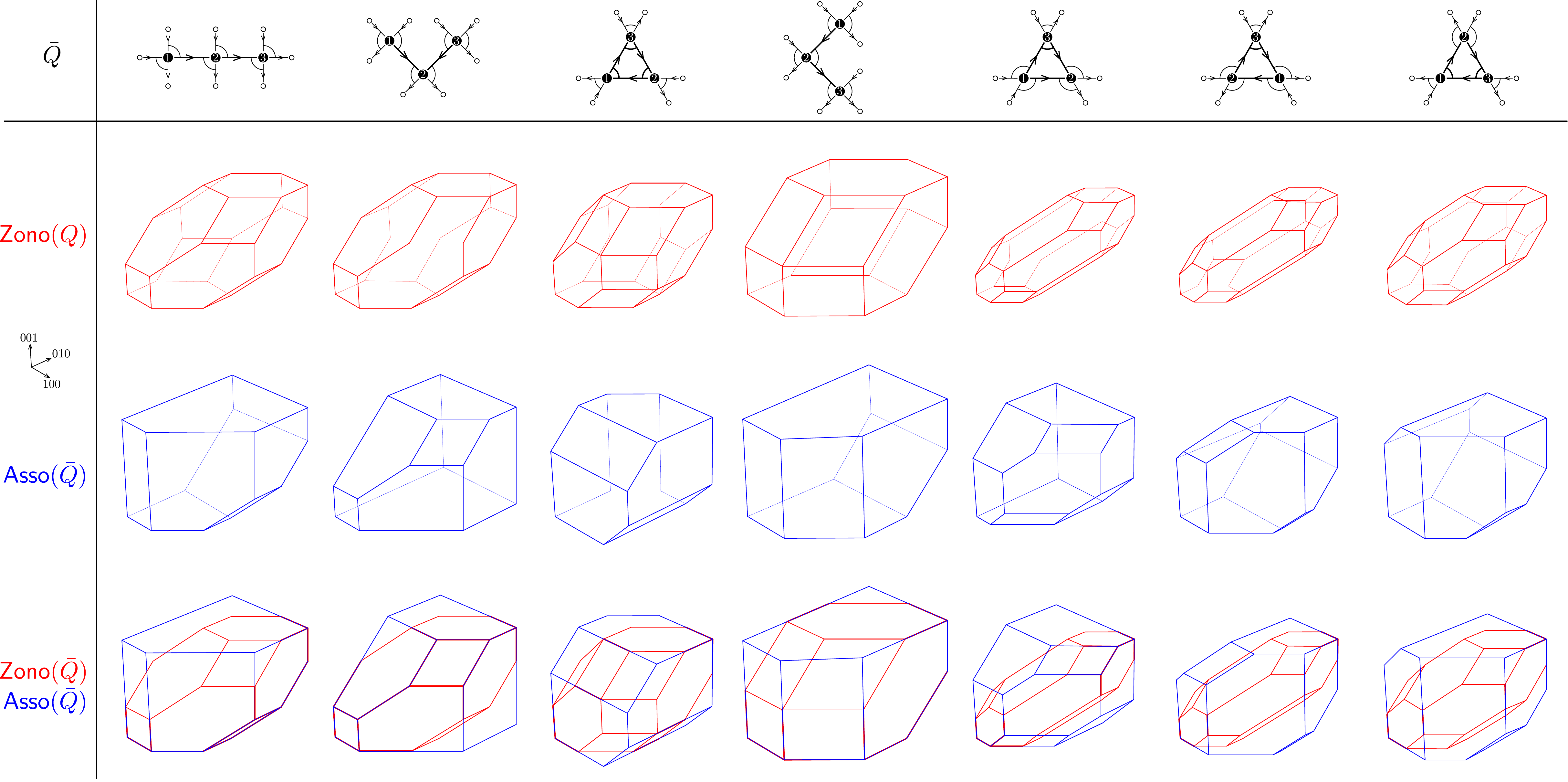}}
{The $\bar Q$-zonotope~$\Zono$ (red, top) and the $\bar Q$-associahedron~$\Asso$ (blue, middle) for all connected simple gentle bound quivers~$\bar Q$ on~$3$ vertices with finite non-kissing complex~$\RNKC$. Observe that the facet defining inequalities of~$\Asso$ are not necessarily facet defining inequalities of~$\Zono$ (bottom). Note that the labels on the quivers have been chosen so that all these polytopes are nicely oriented in the same basis, although the different polytopes are rescaled with independent scaling factors to fit a similar space in the figure. The first is J.-L.~Loday's associahedron~\cite{Loday}, the second is an associahedron of C.~Hohlweg and C.~Lange~\cite{HohlwegLange}, the third is an associahedron recently constructed in~\cite{HohlwegPilaudStella}, the fourth is an accordiohedron of~\cite{MannevillePilaud-accordion}, the last three are new.}
{fig:allAssociahedra3}
\captionsetup{width=\textwidth}


\addtocontents{toc}{\vspace{.3cm}}
\section*{Acknowledgements}

We benefited from various discussions and several useful comments on a preliminary version of this paper.
We are particularly grateful to F.~Chapoton for initially pointing us in the direction of this paper, to S.~Schroll for her valuable comments on Part~\ref{part:algebra}, to ChangJian Fu for his comments leading to a clarification of Definition~\ref{definition: dance and attract}(i)(b), to A.~Garver and T.~McConville for important inputs on Part~\ref{part:lattice} (including pointing out to us the need of Theorem~\ref{thm:characterizationCongruenceUniform2}), to T.~Manneville and S.~Stella for their help on Part~\ref{part:geometry}, and to L.~Demonet for the observation of Remarks~\ref{rem:linDep} and~\ref{rem:polytope}.
We also thank T.~Br\"ustle, O.~Iyama, N.~Reading, H.~Thomas for discussions on their ongoing projects mentioned in the introduction.
This work was carried out while the first author was visiting Paris 7. He would like to thank his colleagues there, and especially \mbox{P.~Le Meur}, for a stimulating atmosphere.
Finally, we thank an anonymous referee for useful suggestions on the presentation of this work.


\bibliographystyle{alpha}
\bibliography{nonKissingComplex}
\label{sec:biblio}

\end{document}

%% file: figures/gcmatrices.tex
\(
\begin{blockarray}{ccccccc}
	& {\color{red} \bullet} & {\color{orange} \bullet} & {\color{yellow} \bullet} & {\color{green} \bullet} & {\color{blue} \bullet} & {\color{violet} \bullet} \\
	\begin{block}{c(cccccc)}
	1 & 0 & 0 & 0 & 0 & 0 & \!\!-1\\
	2 & 0 & 0 & 0 & 0 & \!\!-1\!\! & 0 \\
	3 & 0 & 1 & 0 & 1 & 0 & 0 \\
	4 & 0 & 0 & 0 & \!\!-1\!\! & 0 & 0 \\
	5 & 0 & 0 & 1 & 1 & 1 & 0 \\
	6 & 1 & 0 & 0 & 0 & 0 & 0 \\
	\end{block}
	& & & \multicolumn{2}{c}{$\gvectors{F}$} & &
\end{blockarray}
\qquad
\begin{blockarray}{ccccccc}
	& {\color{red} \bullet} & {\color{orange} \bullet} & {\color{yellow} \bullet} & {\color{green} \bullet} & {\color{blue} \bullet} & {\color{violet} \bullet} \\
	\begin{block}{c(cccccc)}
	1 & 0 & 0 & 0 & 0 & 0 & \!\!-1 \\
	2 & 0 & 0 & 1 & 0 & \!\!-1\!\! & 0 \\
	3 & 0 & 1 & 0 & 0 & 0 & 0 \\
	4 & 0 & 1 & 1 & \!\!-1\!\! & 0 & 0 \\
	5 & 0 & 0 & 1 & 0 & 0 & 0 \\
	6 & 1 & 0 & 0 & 0 & 0 & 0 \\
	\end{block}
	& & & \multicolumn{2}{c}{$\cvectors{F}$} & &
\end{blockarray}
\)